\documentclass[twoside, 12pt, letterpaper]{article}
\usepackage[margin=1in,centering]{geometry}
\usepackage{latexsym,graphicx, epsfig}
\usepackage{amssymb,amsmath,amsthm,amsopn,pifont} 
\usepackage{subfigure}
\usepackage[abs, percent]{overpic}
\usepackage{mathrsfs}
\usepackage{pict2e}
\usepackage[usenames, dvipsnames]{color}
\usepackage[implicit=true, colorlinks=true,linkcolor=Black,citecolor=Black,urlcolor=Black]{hyperref}% pagebackref=true,
\input xy
\xyoption{all}

\newcommand\B{{\mathfrak B}}
\newcommand\D{{\mathbb D}}

\newcommand\cS{{\mathcal S}}

\newcommand\cC{{\mathcal C}}
\newcommand\cE{{\mathcal E}}
\newcommand\cO{{\mathcal O}}
\newcommand\cR{{\mathcal R}}
\newcommand\fR{{\mathfrak R}}
\newcommand\fB{{\mathfrak B}}

\newcommand\cB{{\mathcal B}}

\newcommand\cF{{\mathcal F}}
\newcommand\Tes{{\bf Tes}}
\newcommand\F{{\mathcal F}}

\newcommand\C{{\mathbb C}}
\newcommand\N{{\rm N}}
\newcommand\R{{\mathbb R}}

\newcommand\Z{{\mathbb Z}}
\newcommand\Q{{\mathbb Q}}

\newcommand\ssm{{\smallsetminus}}

\newcommand\e{{\mathbf e}}
\newcommand\g{{\mathbf g}}

\newcommand\tb{{\mathbf t}}

\newcommand\rb{{\mathbf r}}

\newcommand{\mapstoself}%
  {\,\protect\rotatebox[origin=c]{90}{\Huge$\circlearrowleft$}\,}
\def\K{{\mathcal K}}

\def\p{{\mathfrak p}}
\def\e{{\bf e}}
\def\re{{\sf Re}}
\def\v{{\bf v}}
\def\f{{\bf f}}
\def\ocS{{\overline\cS}}

\def\s{{\mathfrak s}}
\def\cB{{\mathcal B}}
\def\cO{{\mathcal O}}
\def\rank{{\rm rank}}

\def\cS{{\mathcal S}}

\def\<{{\langle}}
\def\>{{\rangle}}
\def\E{{\mathcal E}}

\def\I{{\mathcal I}}
\def\bsk{{\bigskip}}
\def\ssk{{\smallskip}}
\def\si{{~\simeq~}}
\def\msk{{\medskip}}
\def\G{{\mathcal G}}

\def\d{{\bf d}}
\def\lg{{(\!(}}
\def\rg{{)\!)}}
\def\({{(\!(}}
\def\){{)\!)}}

\def\<<{{<\!\!<}}
\def\M{{\mathbb M}}

\refstepcounter{equation}

\newtheorem{lem}{Lemma}
\newtheorem{theo}[lem]{Theorem}

\newtheorem{coro}[lem]{Corollary}
\newtheorem{prop}[lem]{Proposition}

\newtheorem{conj}[lem]{\indent Conjecture}

\newtheorem*{lemma*}{Lemma}
\newtheorem*{theorem*}{Theorem}
\newtheorem*{assA*}{Assertion A}
\newtheorem*{assB*}{Assertion B}

\theoremstyle{definition}
\newtheorem{definition}[lem]{\indent Definition}
\newtheorem{rem}[lem]{\indent Remark}
\newtheorem{ex}[lem]{\indent Example}
\refstepcounter{lem}

\renewcommand{\descriptionlabel}[1]%
     {\hspace{\labelsep}\textsf{#1}}

\begin{document}

\title{\bf Cubic Polynomial Maps\\ with Periodic Critical Orbit, Part III:
Tessellations and Orbit Portraits \\ }
\author{\bf Araceli Bonifant and John Milnor}

\maketitle
  
\begin{abstract}
We study the parameter space $\cS_p$ for cubic polynomial maps with a marked
critical point of period $p$. We will outline a fairly complete theory
as to how the dynamics of the map $F$ changes as we move around the parameter
space $\cS_p$. For every escape region $\cE\subset\cS_p$,
every parameter ray in $\cE$ with rational parameter angle 
lands at some uniquely defined point 
in the boundary $\partial{\mathcal E}$. This landing point is necessarily 
either a parabolic map or a Misiurewicz map. The relationship between 
parameter rays and dynamic rays is formalized by the
\textbf{\textit{period $q$  tessellation}} of $\cS_p$, where maps in the
same face of this tessellation 
always have the same period $q$ orbit portrait.
\end{abstract}

\tableofcontents

\vspace{.5cm}
\noindent
{\bf Keywords:} cubic polynomials, parameter rays, tessellations, 
dynamic rays, Misiurewicz maps, parabolic maps, escape region, 
wake, primary edge, secondary edge, orbit portraits.

\vspace{.2cm}
\noindent
{\bf Mathematics Subject Classification (2020):} 37F10,  30C10, 30D05.

\setcounter{equation}{0}
\setcounter{lem}{0}
\section{Introduction}\label{s-in} 

Let $\cS_p$ be the space of all monic centered cubic polynomial
maps $F$ with a \textbf{\textit{ marked critical point}}
$a=a_F$ of period $p\ge 1$. If $v=F(a)$
is the corresponding critical value, then we obtain the normal form
\begin{equation}\label{e-1}
F(z)~=~F_{a,v}(z)~=~z^3-3a^2z+(2a^3+v)\,.
\end{equation}
It will be convenient to identify $\cS_p$ with the smooth affine curve
consisting of all pairs $(a,\,v)\in\C^2$ such that $a$
has period exactly $p$ under iteration of $F$.
Compare  Parts I and II of this paper (\cite{M4} and \cite{BKM}).
In other words, we identify the map $F=F_{a,v}\in \cS_p$ with the pair
$(a,v)\in\C^2$. %, so that we can write $F\in\cS_p$.
Each $\cS_p$ is connected (see \cite{AK}), and consists of a compact
connected \textbf{\textit{connectedness locus}}, together with a finite
number of \textbf{\textit{escape regions}},
each conformally isomorphic to a punctured disk. The genus of $\cS_p$
(or of the smoothly compactified curve $\ocS_p$) grows quite rapidly
with $p$. (Compare \S\ref{s-count}.) 
\smallskip

This paper will outline a fairly complete theory as to
how the dynamics of the
map $F$ changes as we move around the parameter space $\cS_p$. However
there are many questions that we have been unable to answer: We will be
grateful for any comments or suggestions.\medskip

Here is a brief description of what follows.
Section~\ref{s-rays} contains many definitions and basic properties. It
first describes 
 \textbf{\textit{dynamic rays}}, which lie in the complement of 
the filled Julia set of a map in $\cS_p$, and \textbf{\textit{parameter
rays}}, which lie in the various escape regions 
of $\cS_p$. In particular, it shows that every parameter ray with a
rational parameter angle lands at a map which is either parabolic or
Misiurewicz  depending on whether its angle is
\textbf{\textit{co-periodic}}
or not (See Definition~\ref{D-cp} and Theorem~\ref{T-main}). The
\textbf{\textit{ kneading invariant}} of an escape region is described in
Definition~\ref{D-kn} through Remark~\ref{R-J2K}.
Hyperbolic components are discussed in Remark~\ref{R-HC}. Copies of the
quadratic  Mandelbrot set in parameter space in Remark~\ref{R-Mand},  some
basic conjectures and questions in
Remarks~\ref{R-MandelC} and \ref{R-MCquestion}. The section concludes with a
discussion of \textbf{\textit{dual}} critically periodic points in
Remark~\ref{R-dual}.
\smallskip

Section~\ref{s-tess} describes a dynamically useful tessellation
$\Tes_q(\ocS_p)$ of the smoothly compactified parameter curve $\ocS_p$.
Here $p$ and  $q$ can be any positive integers. The \textbf{\textit{edges}}
of this tessellation are the parameter rays of co-period $q$, leading
from an \textbf{\textit{ideal vertex}} to a \textbf{\textit{parabolic
vertex}}. Each face $\F$ of the tessellation determines a period
$q$ \textbf{\textit{orbit portrait}} which describes which dynamic rays
of period $q$ land at the same point of the Julia set for maps in $\F$.
There is a basic distinction between three different kinds of edge:
As we cross a \textbf{\textit{primary edge}} the change in orbit portrait
is strictly monotonic, so that the orbit portrait on one side is a proper
subset of the orbit portrait on the other side. In the case of a
\textbf{\textit{secondary edge}} the change is strictly non-monotonic,
so that neither orbit portrait contains the other. Finally, for an
\textbf{\textit{inactive edge}}, there is no change in orbit portrait.
(Conjecturally the inactive edges are precisely the \textbf{\textit{dead
end edges}}, with no other edge sharing the same parabolic landing point.)
See Definition~\ref{D-3E}, Theorem~\ref{T-edge} and Conjecture~\ref{CJ-main}.
We give a detailed study of just
how the orbit portrait changes as we cross an edge of the tessellation.
The concept of wake is closely 
related. A hyperbolic component $H$ of Type D has a \textbf{\textit{wake}}
$W$ if $H$ is contained in a
simply connected open set $W\subset \cS_p$, which is bounded by two
parameter rays which land at a parabolic boundary
point of $H$, and which lie in the same escape region. (See
Definition~\ref{D-wake}.)

\smallskip

Section \ref{s-near-para} provides a description of how the period $q$
orbit portrait changes as we circle around a parabolic vertex $\p$ of
$\Tes_q(\ocS_p)$. There are three cases to consider, according as the
number of parameter rays landing at $\p$ is two, three, or four.
(There is no change in the orbit portrait
if there is only one ray landing; and conjecturally
there are no cases with more than four rays landing.) We distinguish
between \textbf{\textit{background relations}} which hold for every map
close to $\p$ and \textbf{\textit{distinguishing relations}} which
hold only for some maps near $\p$. In all of the cases we have seen,
the distinguishing relations are completely determined by the angles
of the parameter rays landing at $\p$.  In the four ray case, we  discuss rabbit
and non-rabbit examples. (See Remark~\ref{R-RNR}.)
\smallskip

 Section~\ref{s-count} counts the number 
 of vertices, edges, and faces in each tessellation $\Tes_q(\ocS_p)$. We give
 explicit formulas for the number of edges,  and (conjecturally at
 least) for the number of parabolic vertices. We also give a formula for the
 number of type B components of the curve $\cS_p$. (See Theorem~\ref{T-Bcount}.)
 \smallskip

 Section~\ref{s-Mis} provides a more detailed description of the Misiurewicz
case with special emphasis in Tan Lei similarity in Conjecture~\ref{C-TLS}.
\smallskip

Appendix~\ref{a-para-stab} proves a stability theorem for the landing 
of parameter rays as we approach a suitably restricted parabolic limit. 
\smallskip

Appendix~\ref{a-embdtree}  describes weak Hubbard trees.
\smallskip

The exposition will be continued in Part IV of this paper (see \cite{BM}).
\bigskip

\noindent\textbf{Acknowledgments:} We are grateful to  Jan Kiwi  for very
useful comments, and to Scott Sutherland 
for help with many problems.
Some of the exploration in this paper was done with the program
Dynamics Explorer  from Source Forge
\url{https://sourceforge.net/u/bwboyd/profile/}. We are thankful to
Suzanne and Brian Boyd for helping us to adapt some of our programs
to their Dynamics Explorer platform. \medskip
%and to Scott Sutherland for help with many problems.

%/*TEST ${\mathfrak R}~~{\mathcal R}~~ {\mathscr R} ~~ {\bf R}$ */

\setcounter{lem}{0}
\section{Dynamic Rays and Parameter Rays}\label{s-rays}

The theory of \textbf{\textit{dynamic rays}} $\cR$
in the basin of infinity $\C\ssm K(F)$
for a monic polynomial map  $F$ is classical and well understood.
 (Here $K(F)$ denotes the filled Julia set of $F$.)
In particular, each dynamic ray is an orthogonal trajectory to the family of
equipotentials $\g_F(z)={\rm constant}$. Here  \hbox{$\g_F:\C\to\R$} is the
associated dynamic Green's function, defined by
$$ \g_F(z)~=~\lim_{n\to\infty}\frac{1}{d^n}\log^+|F^{\circ n}(z)|~,$$
where $d\ge 2$ is the degree of $F$. 
Each such ray $\cR$ has a well defined angle $\theta\in\R/\Z$. If $\theta$ is 
rational, then the ray will land at a well defined point of the Julia set 
$\partial K(F)$; unless it crashes into a critical or precritical point
of $F$. (See Remark~\ref{R-DJS}.)

 More explicitly, we can label
these rays, using the \textbf{\textit{B\"ottcher coordinate}}
 $\fB_F(z)\in\C\ssm\overline\D$, which is well defined, for $z$ large, with
$$ \fB_F\big(F(z)\big)~=~ \fB_F(z)^d,\quad{\rm and~with}\quad
 \fB_F(z)=z+O(1/z)\quad{\rm as}\quad |z|\to\infty~.$$
Then\vspace{-.4cm}
  $$\g_F(z)~=~\log|\fB_F(z)|$$ for any $z$
  in the region where $\fB_F$ is defined (See \cite[Definition 9.6]{M3}.). The
  unique dynamic ray through $z$
  can then be labeled as $\cR_F(\theta)\subset\C\ssm K(F)$, where
  $\theta\in\R/\Z$ is the argument of the complex number $\fB_F(z)$ and $K(F)$
  is the Julia set.  Note that
$$ F\big(\cR_F(\theta)\big)~=\cR_F(d\,\theta).$$
Each dynamic ray $\cR_F(\theta)$ can be parametrized by the Green's
function $\g(z)$. 

If the Julia set is connected, 
the B\"ottcher coordinate maps the complement
$\C\ssm K(F)$ holomorphically onto the region $|\fB_F(z)|>1$; and every ray
can be described as a function
\begin{equation}\label{E-ray}
 \g~\mapsto ~\rb(\g)~=~\rb_{F,\theta}(\g)\end{equation}
which is defined and real analytic as a function of the three variables:
$F$ in parameter space,
$\theta\in\R/\Z$, and $\g\in(0,\,+\infty)$. 
Furthermore, each ray tends
to infinity in $\C$ as $\g\to\infty$, and  tends to the Julia set as
$\g\to 0$. 

\medskip

On the other hand, the theory of parameter rays in the parameter space
$\cS_p$  has some special features that require explanation.
Each $\cS_p$ has some finite number of \textbf{\textit{escape regions}}
$\cE_h$ consisting of maps $F$ such that the free critical point $-a$
has unbounded orbit. The complement $\cS_p\ssm\bigcup_h\cE_h$ is the
 \textbf{\textit{connectedness locus}}, which is compact and
connected  and consists of those maps which have connected Julia set. 
Each $\cE_h$ is conformally isomorphic to the
punctured open disk $\D\ssm\{0\}$. Concentric circles in $\D$
correspond to \textbf{\textit{equipotentials}}, where the parameter
Green's function
$${\bf G}(F)~=~\g_F(2a_F)~=~\log|\B(2a_F)|~=~\lim_{n\to\infty}\frac{1}{3^n}\log^+|F^{\circ n}(-a)|$$
takes some constant value (see \cite[Definition 9.6]{M3}). 
The orthogonal trajectories of these equipotentials,
corresponding to radial line segments in the disk, are called
\textbf{\textit{parameter rays}} $\fR$ in $\cE_h$.  The angle $\theta$ of
the dynamic ray passing through $2a$ is called the
\textbf{\textit{co-critical angle}} of $F$.

\begin{rem}[\bf Multiplicity of an Escape Region]\label{R-mult} 
In general, a parameter ray is not uniquely determined
by its parameter angle. Each escape region $\cE_h$ has a well defined
multiplicity $\mu_h\ge 1$; and each parameter angle corresponds to
$\mu_h$ different parameter rays, which are evenly spaced around the
region. (Compare \cite{BKM} where it is shown that $\mu_h$ is always 
a power of two.) Most of the escape regions we discuss explicitly 
have multiplicity $\mu_h=1$. In particular, 
if $p\le 3$ then every escape region in $\cS_p$ has multiplicity one.

However,  when $p\geq 4$, there are certainly escape regions
with multiplicity $\mu>1$.  Here is an example:  
According to  \cite[Table 6.4]{BKM} there are two
escape regions in
$\cS_4$ with kneading invariant $(0,0,1,0)$ (see Definition \ref{D-kn}
below); and each of these has multiplicity two. (The two are $180$
degree rotations of each other.) 
A careful study of Figure 11 of \cite{BKM} shows that each of these two
regions is surrounded by a circle of $26$ neighboring escape regions,
most of which have multiplicity one. For each parameter ray in one of
these two regions there is another ray with the same angle. 
Two such rays usually land near quite different
escape regions, so there is presumably no close  
relationship between the two landing points.

More explicitly (see \cite{BKM}),  in the case $\mu>1$ we will need
a conformal isomorphism $$~\cE_h~\stackrel{\cong}{\longrightarrow} 
~\C\ssm\overline\D~$$ 
given by choosing a smooth branch
$~~ F~\mapsto~\root{\scriptstyle\mu}\of{{\mathfrak B}_F(2a_F)}~~$
 of the $\mu$th-root  of the B\"ottcher coordinate $~\B_F$~
evaluated at $2a=2a_F$. The choice of the $\mu$-root  
is called an \textbf{\textit{~anchoring~}} of the escape region $\cE_h$.
We then define  \textbf{\textit{~parameter angle~}} 
$\phi(F)~$  as
$$  \phi(F)~=~\arg\Big(\sqrt[{\scriptstyle\mu}]{{\mathfrak B}_F(2a_F)}\Big)
~\in~\R/\Z\,.$$
 This angle~ $\phi(F)$,~ measured in the escape region $~\cE_h$,~ is simply
 a label with no invariant dynamical meaning, since it 
 depends on the choice of anchoring. 
However  the \textbf{\textit{co-critical angle~}}
 $$\theta(F)~=~\arg\big({\mathfrak B}_F(2a_F)\big)~\in\R/\Z\,, $$
measured in the dynamical plane, is invariantly defined, 
and is related to the parameter angle  by the equation
\begin{equation}\label{e-theta-phi}
\theta(F)~=~\mu\;\phi(F)\,.
\end{equation}
Of course in the special case $\mu=1$, the parameter angle can be identified
with the co-critical angle. 
\end{rem}
\medskip

\begin{definition}
 For any specified angle $\phi_0$,
 the set consisting of all $~F\in \cE_h~$ with
parameter angle $\phi(F)=\phi_0\,$,~
is called a \textbf{\textit{\,parameter ray\,}} 
${\mathfrak R}_{\cE_h}(\phi_0)\,.$ 
Note that
\begin{equation}\label{e-ray-def}
 F\,\in\, {\mathfrak R}_{\cE_h}(\phi)\qquad\Longrightarrow\qquad
 2a_F\,\in\,\cR_F(\mu\phi)\, = \,\cR_F(\theta)\,.
\end{equation}
\end{definition}

\begin{figure}[htb!]
\centerline{\includegraphics[height=1.3in]{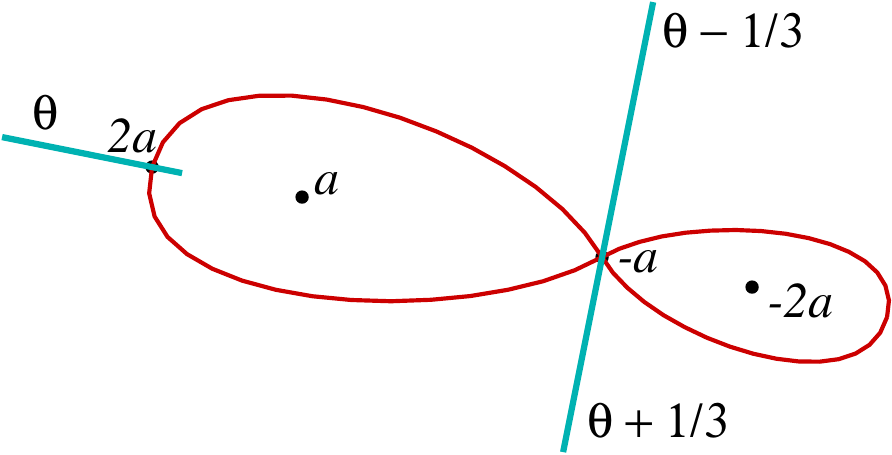}}

\caption[Dynamic plane for a map $F\in\cE_h$]{\textsf{Sketch of the dynamic
    plane for a map $F$ which lies on a parameter ray of angle $\theta$ in some 
    escape region $\cE_h$.
    Here $a=a_F$ is the marked (periodic) critical point; $-a$ is
    the free critical point; and 
    the figure eight curve represents the equipotential~~
    \hbox{$\g_F(z)={\rm constant}=\g_F(-a)=\g_F(2a)$}.~
    Note that the two \textbf{\textit{co-critical points}} $2a$ and $-2a$
    satisfy $F(2a)=F(-a)$ and $F(-2a)=F(a)$.  The  filled Julia set
    $K(F)$ is the disjoint union of parts in the left lobe and parts in the
    right lobe; and hence is not connected.     Since some 
    neighborhood of $2a$ maps bijectively outside the figure eight curve,
    the dynamic $\theta$ ray for $F$ extends smoothly at least a little
    past $2a$. \label{Faa-1}}}
\end{figure}

(Compare Figure~\ref{Faa-1}.) Note that the parameter angle remains constant
as $F$ varies along any
parameter ray. In fact our parameter rays are special cases of stretching
rays, as defined by Branner and Hubbard \cite{BH1}. The different maps
along a stretching ray are quasi-conformally conjugate, under a conjugacy
which multiplies the Green's function by a constant, but preserves the
dynamic angle.
%\bigskip

\begin{figure}[htb!]
\centerline{\includegraphics[height=2.3in]{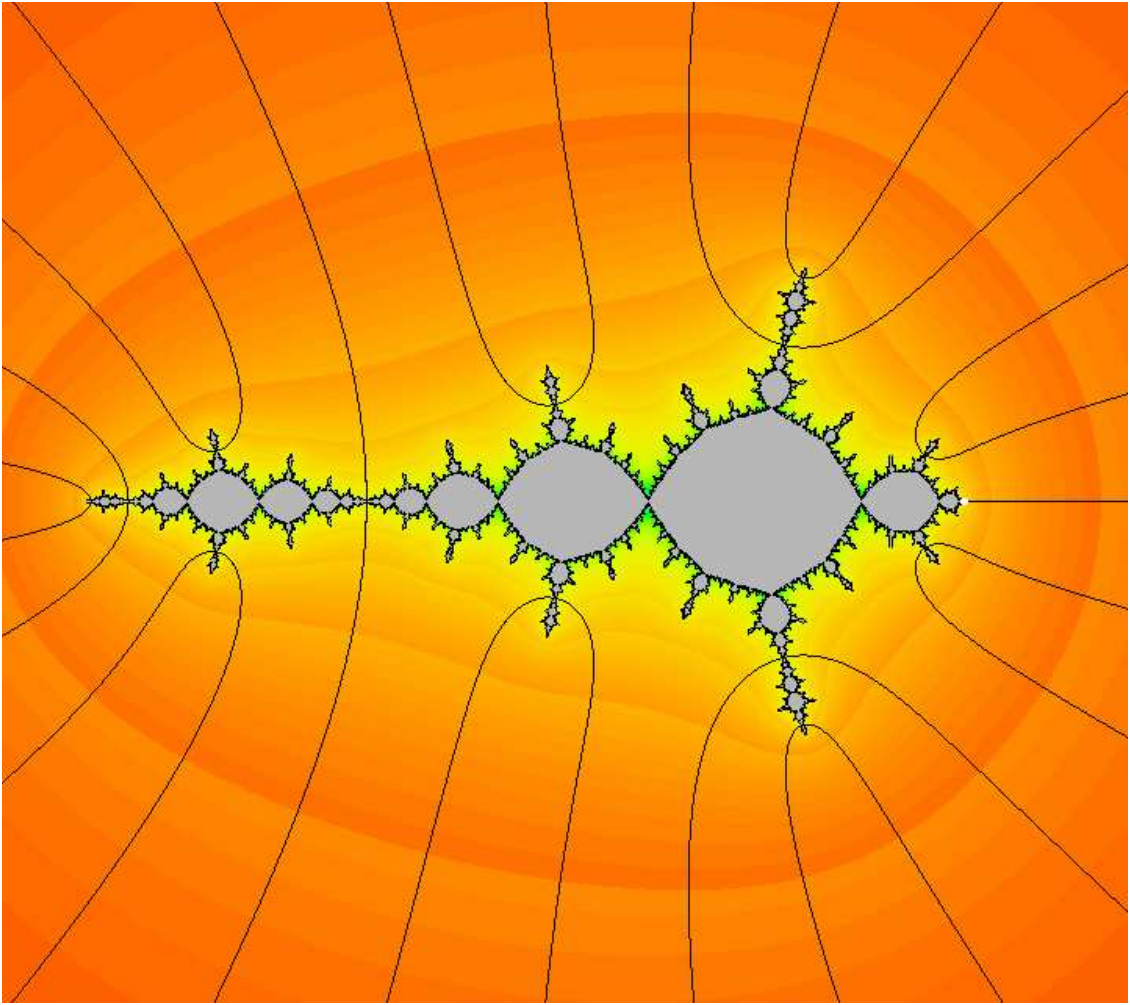}\qquad\includegraphics[height=2.3in]{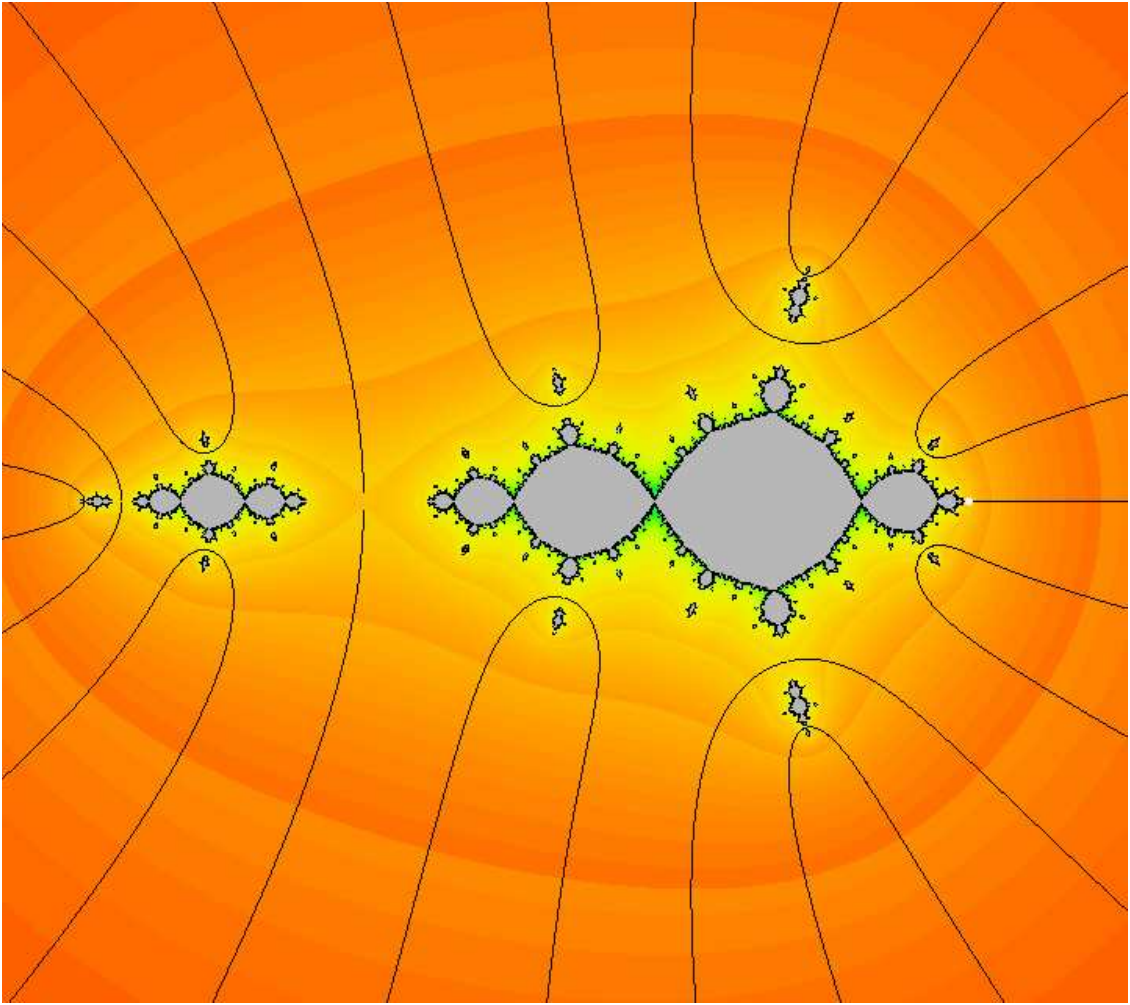}}
  \caption[Julia set for the landing map of the zero parameter ray in the
  ``basilica'' escape region of $\cS_2$]{\textsf{On the left, the Julia set for
the landing map $F_0$ of the zero parameter ray in the ``basilica'' escape
region of $\cS_2$. The dynamic rays of angle $n/27$ are shown. On the right,
the Julia set for a map $F$ on the zero parameter ray for this basilica
region. \label{f-s2cfjul}}}
\end{figure}
\msk

\begin{rem}[\bf Disconnected Julia Sets]\label{R-DJS} In any escape region, 
we are dealing with Julia sets which are not connected, so
we have to be careful in the discussion of dynamic rays. As we follow a dynamic
ray from infinity, it will remain smooth and well defined as long
as we are outside the figure eight equipotential curve. But two rays crash together at
$-a$ (as illustrated in Figure \ref{Faa-1}). Furthermore, $-a$ has countably
many iterated pre-images inside the figure eight, and in many cases two
rays will crash together at each of these. (Compare Figure~\ref{f-s2cfjul}
for many examples of this.) However, there are also cases where only
one ray crashes at a pre-critical point. (Compare Figure~\ref{F-para-pic}-right
for two examples of such ``dead end'' rays.) 
Note that the set of all pre-critical points converges to the filled Julia 
set in the following sense: Given any
neighborhood $N$  of $K(F)$, there are only finitely many
pre-critical points outside of $N$.

These critical and precritical points are precisely the critical  points of
the Green's function in $\C\ssm K(F)$. It is often convenient to parametrize
each dynamic ray as a smooth function $\g\mapsto \rb(\g)$ of the Green's
function. For a generic choice of ray, this function $\rb(\g)$ is defined
for all $\g>0$. However in the countably many exceptional cases where
the ray crashes into a critical or precritical point 
$z_0$, the function $\rb(\g)$ is defined only for $\g>\g(z_0)$.
Note that $\rb(\g)$ can be considered as a real analytic function, not
only of $\g$, but also of the map $F\in\cS_p$ and of the dynamic
angle $\phi$, throughout the interior 
of the region where it is defined.
\end{rem}
\bigskip

Another key idea is the following.

\begin{definition}[\bf Co-periodic Angles]\label{D-cp} An angle
  $\theta\in\R/\Z$ will be called
  \textbf{\textit{co-periodic}} of \textbf{\textit{co-period}}
  $q\ge 1$ if either $\theta+1/3$ or $\theta-1/3$
  is periodic of period $q$ under angle tripling. The elements of the associated
  periodic orbit can be listed as
\begin{equation} \label{E-theta-j}  \{\theta_1,~\theta_2,~\cdots,~ \theta_q\}
\qquad{\rm where} \qquad~\theta_j =3^j\theta\quad{\rm  for} \quad 1\le j\le q~.
\end{equation}
Thus $~~\theta_1\mapsto \theta_2\mapsto\cdots\mapsto \theta_q\mapsto\theta_1~~$
under angle tripling. In other words, if we think of $j$ as a residue class
modulo $q$, then $\theta_{j+1}=3\theta_j$ in all cases.
Here $\theta_q$ is always equal to $\theta\pm 1/3$ for some choice of sign.
  
 It will be convenient to use
the term~ \textbf{\textit{triad}}~ for any triple of angles 
$~ (\theta,~\theta_q,~ \widehat\theta), ~$ such that 
both $\theta$ and $\widehat\theta=\theta\mp 1/3$ are co-periodic, while
$ \theta_q~=~3^q\theta ~=~3^q\widehat\theta$ is periodic.
Thus the three angles $\theta,~\theta_q$
and $\widehat\theta$ are evenly spaced around the circle, and all three map
to $\theta_1$ under tripling. However $\theta_q$ is the only one of the
three which is periodic.
%Compare  Figure~\ref{Faa-1} where the dynamic ray of angle $\theta$ 
%passes through the co-critical point while the rays of 
%angle $\theta_q$ and $\theta^*$ crash together at $-a$.
 We will describe the two co-periodic 
angles in a triad (or the associated dynamics rays) 
as \textbf{\textit{twin co-periodic angles}} (or \textbf{\textit{twin rays}}).
\ssk

If $F$ is a map 
belonging to a parameter ray $\fR$ of angle $\theta$ in some escape region,
then in the dynamic plane for $F$ the rays of angle $\theta_q$ and 
$\widehat\theta$ will crash together at the escaping critical point $-a$,
while the dynamic ray of angle $\theta$ will pass through the co-critical
point $2a$. (Compare Figure~\ref{Faa-1}.) But note that $\theta$ is not
uniquely determined by the  $\theta_j$. In fact the correspondence
$\theta\mapsto \theta_1$ is two-to-one. (Compare Lemma~\ref{L-2to1}.)
\end{definition}\msk

\begin{rem}[\bf Notational convention]\label{R-NC} 
 The smallest common denominator
for angles of period $q$ under tripling will be denoted by
$$d~=~d(q)~=~3^q-1~.$$
In fact the angles of the form $\theta=m/d$ where $m$ is an integer modulo
$d=d(q)$ are precisely those which satisfy $3^q\theta\equiv\theta$ modulo $\Z$.
%Thus every such angle can be written as a fraction $m/d$,
%where $m$ is an integer modulo $d$.
Similarly every angle $\theta$ of co-period $q$ can be written as 
a fraction of the form 
$\theta=m/(3d)$. Here two conditions must be satisfied:

\begin{itemize}
\item[(1)] $m$ must be congruent to either one or two modulo three, and
%not\equiv 0~~({\rm mod}~3$, and

\item[(2)] it must not be possible to write $\theta$ in this form
for any smaller value of $q$.
\end{itemize}
\noindent(As an example, it might seem that $\frac{13}{3\;d(3)}$
has co-period three. However the computation
$$ \frac{13}{3\;d(3)}~ =~\frac{13}{78}~=~\frac{1}{6}~=~\frac{1}{3\;d(1)}$$ 
% $13/d(3)=13/26=1/2$
shows that it actually has co-period one.)\ssk

It is not hard to check that a fraction $0<m/n<1$ in lowest terms is
co-periodic of some co-period if and only if $n$ is divisible by three
but not by nine.
\end{rem}\medskip

\begin{definition}[\bf Parabolic Maps and Misiurewicz Maps]\label{D-PM}
  A map $F \in\cS_p$ will be called \textbf{\textit{parabolic}} if it has
  a (necessarily unique) parabolic orbit. The \textbf{\textit{ray period}}
  $q$ of a parabolic map is the period of any dynamic ray which lands
  at a parabolic point. This is always equal to the period of the
  cycle of parabolic basins; but it may be a multiple of the period of the
  parabolic point itself. (Compare \cite[Thm.~18.13]{M3}.)
  \medskip
%where $k=q=5$ see Figure~\ref{F-J2K} and for an example where $q=4$ and $k=2$ see
%Figure~\ref{F-43J}.)

  Imitating the terminology of Douady and Hubbard for quadratic maps,
  we will say that $F$ is a \textbf{\textit{Misiurewicz map}} if the orbit
  of the free critical point $-a_F$ is eventually periodic repelling.
% For the maps we consider, with only one
%free critical point, there can be at most one parabolic orbit. Any map
%with a parabolic orbit will be described as a ``parabolic map'', and
%necessarily has connected Julia set.  ( %For many examples with $p=q>k=1$
%  see the parabolic boundary points of rabbit regions in \S\ref{s-rab}.
\end{definition} 
\medskip

{\bf Note.} For simplicity in the exposition from now on we will
assume that $\mu=1$, and denote by $\theta$ the angle corresponding to both,
the  parameter ray ${\mathfrak R}(\phi)$ and dynamical ray
${\mathcal R}(\theta)$. The reader should be aware that when
dealing with an escape region of multiplicity $\mu >1$ adjustments should
be made according to equation~(\ref{e-theta-phi}).
\medskip

\begin{theo}[\bf Landing Theorem]\label{T-main}
Every parameter ray $\fR$ with rational parameter angle $\theta$
lands at a uniquely defined map $F=F_{\fR}$ which belongs to the boundary
of its escape region in $\cS_p$. There two possible cases:

\begin{description}
\item[{\bf Case 1: Parabolic Maps.}] If the angle $\theta$ is co-periodic of
  co-period $q$,  then the free critical orbit for $F_\fR$ converges
  to a parabolic orbit of ray period $q$.
  {\rm(For a more detailed description, see Lemma \ref{L-para},
    and for examples, see Figures~$\ref{Fac-julpar}$ 
    and $\ref{F-s3per2jul}$ as well as Figures~\ref{F-s5rab-par} and
    \ref{F-4wall}.)}

\item[\bf Case 2: Misiurewicz Maps.] If $\theta$ is rational but not
co-periodic, then the free critical orbit for $F_\fR$ 
is eventually periodic repelling.
In this case the dynamic $\theta$ ray lands at the
co-critical point $2a$; and the non-periodic rays
of angle $\theta\pm 1/3$ both land at $-a$. 
{\rm (Compare Figures~$\ref{F-S2juls}$ and $\ref{F-cap-ara}$.)} 
As a special case, if the angle $\theta$ is actually periodic, then it is not
co-periodic, and $2a$ itself is a 
repelling periodic point. {\rm (Compare Figure~\ref{F-percase}.)}
\end{description}
\end{theo}

\smallskip

\begin{figure}[htb!]
%\begin{center}
%\begin{overpic}[width=3.5in, tics=10]{per1parinS2.png}
%\put(257,83){0}
%\put(255,169){1/12}
%\put(182,257){\bf 1/6}
%\put(85,257){1/4}
%\put(13,257){1/3}
%\put(-25,178){5/12}
%\put(-23,120){\bf 1/2}
%\put(-25,49){7/12}
%\put(28,-12){2/3}
%\put(111,-12){3/4}
%\put(188,-12){\bf 5/6}
%\put(257,30){11/12}
%\put(53,192){\color{red} $\bullet$}
%\put(21.5,154.7){\color{red} $\bullet$}
%\put(65,130){\color{red} $\bullet$}
%\put(148,81.5){\color{red} $\bullet$}
%\put(190,57){\color{red} $\bullet$}
%\put(148,30){$\textbf{\textit{a}}$}
%\put(216,108){${\bf 2}\textbf{\textit{a}}$}
%\put(141,139.5){$\textbf{\textit{-a}}$}
%\put(0,98){${\bf -2}\textbf{\textit{a}}$}
%\put(10,212){$\textbf{\textit{v}}$}
%\linethickness{1.5pt}
%\put(20,213){\color{blue}\vector(3,-1.5){35}}
%\put(152,40){\color{blue}\vector(0,1.2){40}}
%\put(222,103){\color{blue}\vector(-1,-1.5){25}}
%\put(138,141.5){\color{blue}\vector(-1,-0.1){65}}
%\put(12,108){\color{blue}\vector(0.5,1.8){13}}
% \end{overpic}
   \centerline{\includegraphics[width=3.5in]{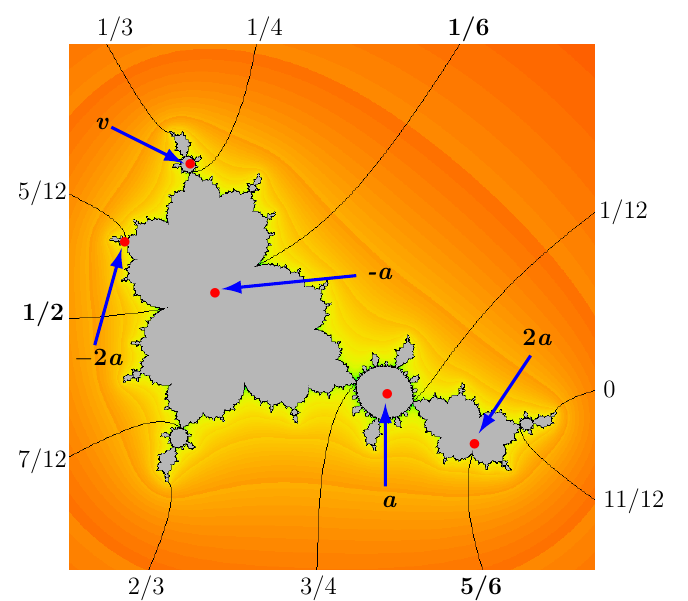}}
  \vspace{-.4cm}
  \caption[Parabolic example of period one.]{\textsf{A parabolic example with
      ray period one:
 %\index{Figure~\ref{Fac-julpar}, parabolic example period one}
    Julia set for the landing map of the
    $5/6$ parameter ray in the inner escape region $\cE_{\sf in}$
    of $\cS_2$. Here the dynamic rays of angle \hbox{$5/6+1/3=1/6$} 
    and $5/6-1/3=1/2$ land on the
    boundary of the fixed parabolic component, with the $1/2$ ray
 landing at the parabolic fixed point.  (See 
    Figure~\ref{f2} lower-left for the $5/6$ ray in parameter space.)
In this example we have period $q=1$,
so that the critical basin $U(-a)$ is equal to the critical value basin
$U(v')$~.    \label{Fac-julpar}}}
%\end{center}
\end{figure}     
\smallskip

\begin{figure}[htb!]
%\begin{center}
%\begin{overpic}[width=3.5in, tics=10]{per1parinS2-mod.png}
%\put(257,80){0}
%\put(255,168){1/12}
%\put(180,257){\bf 1/6}
%\put(80,257){1/4}
%\put(6,257){1/3}
%\put(-27,171){5/12}
%\put(-24,115){\bf 1/2}
%\put(-27,49){7/12}
%\put(25,-12){2/3}
%\put(106,-12){3/4}
%\put(185,-12){\bf 5/6}
%\put(257,24){11/12}
%\end{overpic}
% \end{center}
\centerline{\includegraphics[width=4in]{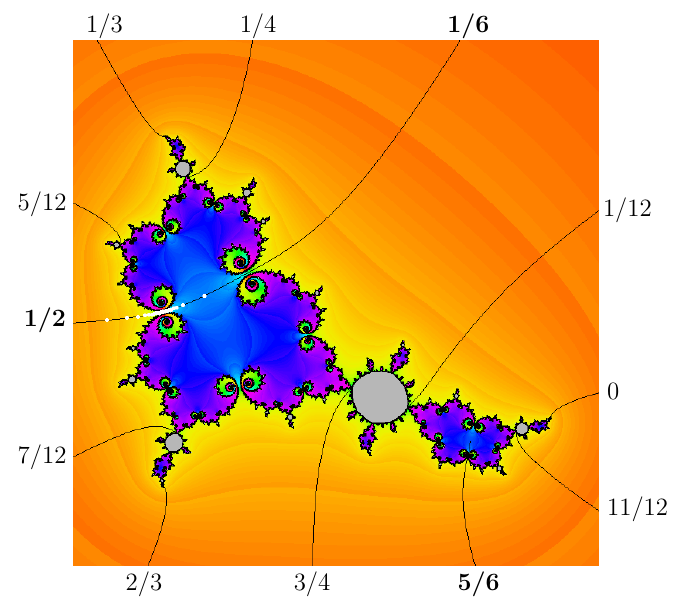}}
\vspace{-.4cm}
\caption[Perturbed figure~\ref{Fac-julpar}.]{\label{F-near-par} \textsf{Modified
    version of Figure~\ref{Fac-julpar} showing
    the Julia set for a map belonging to the parameter ray of angle $5/6$
    in the inner escape region of Figure~\ref{f2}. 
    The dynamic rays of angle $\theta-1/3=1/2$ and $\theta+1/3\equiv 1/6$
    now crash together at the free critical point $-a$; while the $5/6$ ray
    passes through $2a$. 
    The first $22$ forward images of $-a$ are marked by white dots,
    lying on the ($1/2$)-ray. 
(Most of these dots are jammed close together since the orbit escapes very 
slowly through the narrow channel between two repelling fixed points.)
}}
\end{figure}
\smallskip

\begin{figure}[htb!]
\begin{center}
\begin{overpic}[width=3in, tics=10]{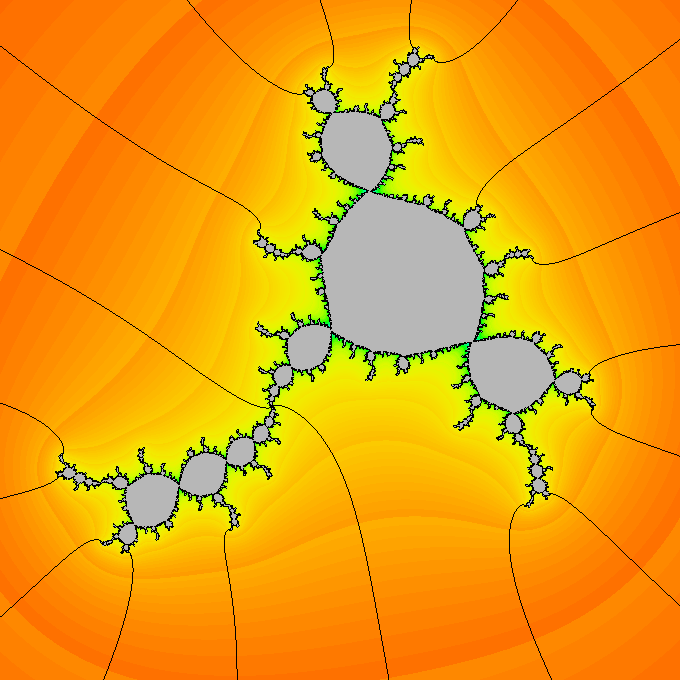}
\put(220,104){0}
\put(220,144){1/18}
\put(220,200){\bf 1/9}
\put(160,222){1/6}
\put(127,222){2/9}
\put(93,222){5/18}
\put(50,222){\bf 1/3}
\put(-27,200){7/18}
\put(-25,136){\bf 4/9}
\put(-22,84){1/2}
\put(-22,53){5/9}
\put(-32,15){11/18}
\put(24,-14){2/3}
\put(60,-14){13/18}
\put(113,-14){\bf 7/9}
\put(170,-14){5/6}
\put(220,20){8/9}
\put(220,66){17/18}
\put(61,62){\color{red} $\bullet$}
\put(147,147){\color{red} $\bullet$}
\put(125,125){\color{red} $\bullet$}
\put(111,167){\color{red} $\bullet$}
\put(82,83){\color{red} $\bullet$}
\put(41,13){${\bf -2}\textbf{\textit{a}}$}
\put(77,21){$\textbf{\textit{-a}}$}
\put(196,162){${\bf 2}\textbf{\textit{a}}$}
\put(118,52){$\textbf{\textit{a}}$}
\put(175,192){$\textbf{\textit{v}}$}
\linethickness{1.5pt}
\put(125,62){\color{blue}\vector(0.3,4){5}}
\put(195,165){\color{blue}\vector(-3,-1){40}}
\put(88,30){\color{blue}\vector(0,1){55}}
\put(57,24){\color{blue}\vector(0.3,1.5){8}}
\put(172,193){\color{blue}\vector(-2,-0.8){50}}
\end{overpic}\qquad\qquad
\begin{overpic}[height=3in]{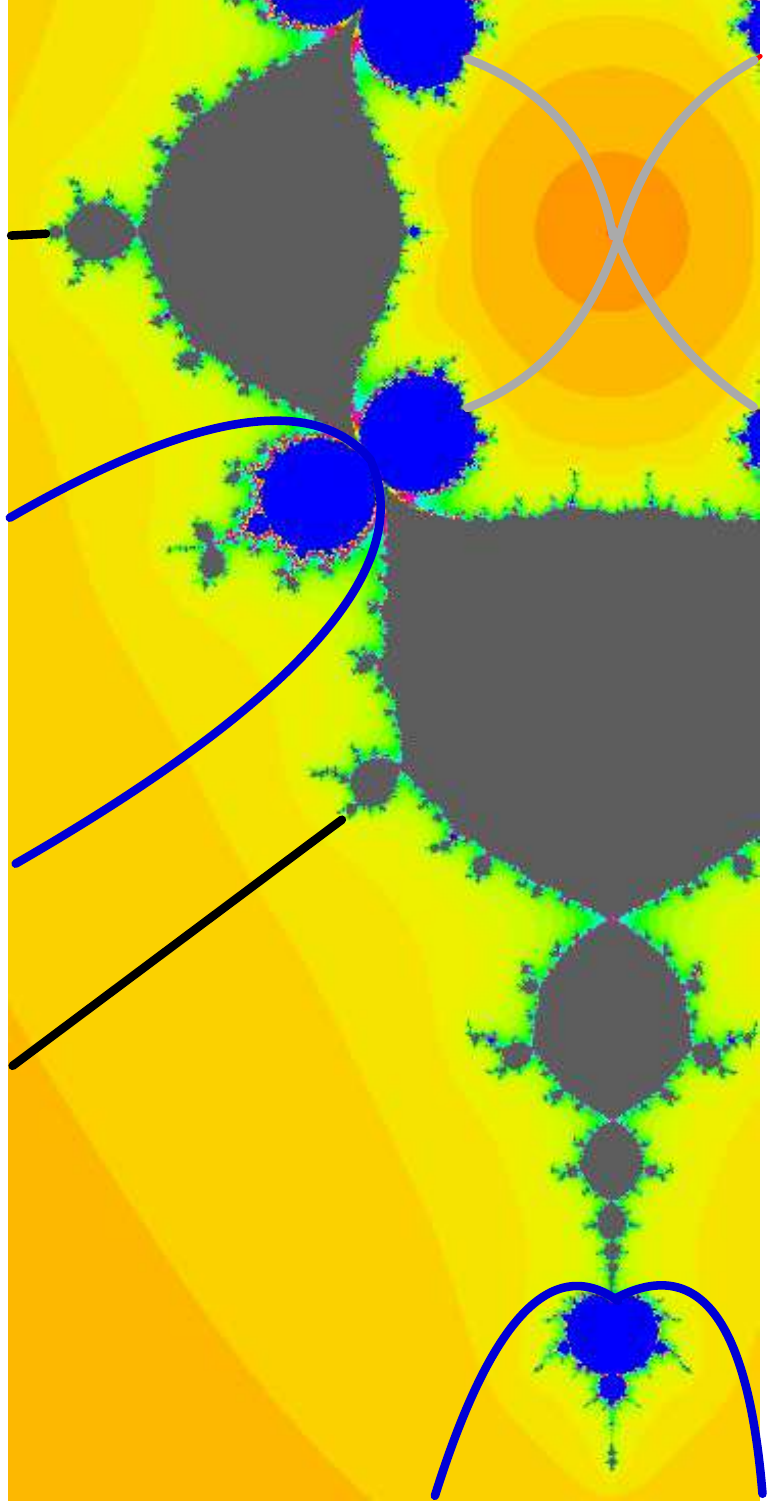}
  \put(20,63){1/9}
  \put(50,-14){1/6}
  \put(100,-14){1/3}
  \put(0, 106){2/24}
  \put(0, 156){1/24}
  \put(0,188){0}
\end{overpic}
\medskip

\caption[A Misiurewicz example.]{\label{F-S2juls}\textsf{A Misiurewicz 
    example.
 %\index{Figure~\ref{F-S2juls}, critically finite example}
On the left,  Julia set for the landing map of the $1/9$ \break parameter
ray in the outer (basilica) escape region $\cE_{\sf out}$ of
$\cS_2$. Note that the dynamical $1/9$ ray lands at the co-critical point,
which is eventually periodic repelling
since $1/9\mapsto 1/3\mapsto 0 ~({\rm mod}~\Z)$. 
The two dynamic rays of angles $1/9\pm 1/3$ land on the
free critical point $-a$.  The $1/9$ parameter ray is shown in the
figure on the right. Note that $1/9$
 is not co-periodic. (Some co-periodic rays are also shown. Compare
 Figures \ref{f2} and \ref{F-S2rays} in Section \ref{s-tess}.) Here the landing
 point of the $1/9$ dynamic ray looks very much like the landing point of
 the $1/9$ parameter ray. An analogous statement seems to be true in all
 Misiurewicz cases. Compare the discussion of
 \textbf{\textit{Tan Lei similarity}} 
 in Conjecture \ref{C-TLS}.
  For the coloring of parameter space pictures, 
 see Remark~\ref{R-HC}; and for the choice of local coordinates in
 $\cS_p$ see \cite[Sec.2]{BKM}~.}}%the Appendices \ref{A-CC} and \ref{A-S3}.
\end{center}
\end{figure}
\smallskip

\begin{figure}[htb!]
  \centerline{\includegraphics[width=4.2in]{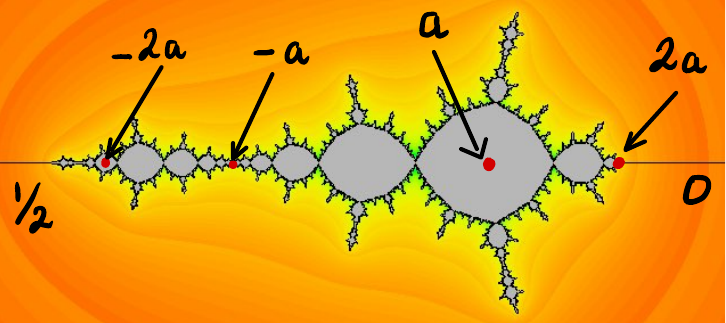}}
  \caption[Simpler Misiurewicz example.]{\label{F-percase}\sf A simpler
    Misiurewicz example:
  %\index{Figure~\ref{F-percase} simpler Misiurewicz ex.}
    Julia set for the landing point of the zero parameter ray in the outer
    escape region of  $\cS_2$. Here the free  critical point $-a$ maps
    directly to the repelling fixed point $2a$, the landing point of the
    $\theta=0$ ray. As in the previous figure, the landing
    point of the $\theta$ ray in the dynamic plane looks much like the 
  corresponding landing point in parameter space. (Again see Figure~\ref{f2}).}
  \end{figure}

(Compare \cite[\S5]{AK}.) The proof of Theorem \ref{T-main} begins as follows.

\begin{lem}[{\bf Preliminary Landing Statement}]\label{L-ratland} 
Every parameter ray $\fR\subset\cE_h$ with rational parameter angle $\theta$
lands at a well defined map $F_\fR$ in the boundary $\partial\cE_h$
which is either parabolic or Misiurewicz.
%\begin{itemize}
%\item[{\bf(a)}] \textbf{\textit{parabolic}}, in the sense that $F_\fR$ has a
%parabolic periodic orbit, or: \medskip
%\item[{\bf(b)}] A \textbf{\textit{Misiurewicz map}},  such
%that the orbit of $-a$ is eventually periodic repelling.\end{itemize}
Furthermore, in the Misiurewicz case, the $\theta$ dynamic ray lands at
$2a$, and the two rays of angle $\theta\pm 1/3$ land at $-a$.
\end{lem}\smallskip

Note that these statements depend only on the angle of the parameter 
ray, and not on which particular escape region it is contained in. \msk

The proof of this lemma will make use of the following. 
\smallskip 

\begin{lem}[{\bf Landing stability for repelling  points (Douady-Hubbard)}]
  \label{L-land-stab} 
For $f_0$ in the space of monic polynomial maps of some fixed degree $n\ge 2$,
 suppose that the periodic
 dynamic ray $\cR_{f_0}(\theta)$ lands on a repelling periodic point $z_0$
of the Julia set. Then for all
$f$ in a neighborhood of $f_0$ the corresponding ray $\cR_{f}(\theta)$
lands on a repelling point  which depends holomorphically on $f$. Furthermore,
the entire ray $\cR_{f}$ converges uniformly to $\cR_{f_0}$
$($using the standard  parametrization by the Green's function$)$ as $f$ tends
to~$f_0$.
\end{lem}
\smallskip

See for example \cite[Prop. 8.5]{DH1}. This contrasts with the
behavior at a parabolic landing point. (Compare the Appendix.) 
\medskip

\begin{proof}[\bf Proof of Lemma \ref{L-ratland}]
The parameter ray ${\mathfrak R}$  
must have at least one accumulation point $\widehat F$ in  the
boundary $\partial\cE_h$. Since the parameter angle $\theta$ is rational,
and since $\widehat{F}$ belongs to the connectedness locus, the associated 
dynamical ray $\cR_{\widehat F}(\theta)$  lands at some  periodic or
pre-periodic point $z_0(\widehat{F})$ in the Julia set $J(\widehat F)$.
In fact, if we choose a multiple $3^k\theta$ which is periodic under angle
tripling, then the dynamic ray $\widehat{F}^{\circ k}(\cR)$ will land at a
periodic point
$z_k(\widehat F)=\widehat{F}^{\circ k}\left(z_0(\widehat{F})\right)$ which is
either parabolic or repelling. (See for example \cite[\S18.10]{M3}.)
If $z_k(\widehat F)$ is a parabolic point, then the map $\widehat{F}$ is
parabolic by definition.

On the other hand, suppose that
this periodic point $z_k(\widehat F)$  is repelling,
then we will prove that the landing point $z_0(\widehat F)$ is 
equal to the co-critical point $2a=2a_{\widehat F}$, so that
$\widehat F$ is Misiurewicz.
Let $F$ vary over a small neighborhood of $\widehat F$ in $\cS_p$. By
Lemma~\ref{L-land-stab}, the entire dynamical ray
 $\cR_{F}(\theta)$ will vary continuously with $F$. In particular,
its landing point $z_0(F)$ will vary continuously.
If we choose $F$ to lie on the parameter ray
 ${\mathfrak R}_{\cE_h}(\theta)$,
then the co-critical point $2a_{F}$ will lie on this dynamical
 ray $\cR_{F}(\theta)$. (Compare Figure~\ref{Faa-1}.)
 Choose a sequence of maps $F_j$ tending to $\widehat{F}$ and all
 belonging to the parameter ray  ${\mathfrak R}_{\cE_h}(\theta)$,
recall from Lemma \ref{L-land-stab} that the dynamic ray  
$\cR_{F_j}(\theta)$ converges uniformly to $\cR_{\widehat F}(\theta)$, when
these rays are parametrized by the respective Green's functions, the
 Green's function $\g_{F_j}(2a_{F_j})$ 
will tend to $\g_{\widehat{F}}(2a_{\widehat{F}})=0$, hence the distance between
$2a_{F_j}$ and the landing point $z_0({F_j})$ must tend to zero.
Taking the limit as $F_j$ tends to $\widehat F$, it follows that
$2a_{\widehat F}$ must be precisely equal to the
 landing point $z_0(\widehat F)$, which is eventually periodic
 repelling. Since $\widehat F(2a)$ is always equal to the critical value 
$\widehat F(-a)$, this proves that $\widehat F$ is Misiurewicz.
(Here we are only concerned with the orbit of $-a$,
since the critical point $a$ has finite periodic orbit by definition.) 
\medskip
 
Thus every accumulation point for the ray
  ${\mathfrak R}$ in the boundary $\partial\cE_h$ must be
either Misiurewicz or parabolic. Since the set of critically
finite maps in $\cS_p$  
is a countable union of zero-dimensional algebraic
subvarieties, it is a countable set; and similarly the set of parabolic maps
is countable. But the set of all accumulation points of a parameter ray
in $\partial \cE_h$ is necessarily a connected set, hence it must consist
of a single map $F_\fR$. This proves Lemma~\ref{L-ratland}.
\end{proof}

\bigskip

\begin{proof}[{\bf Proof of the Landing Theorem~\ref{T-main}}]\hfill{}

{\bf Step 1.} Suppose that the
parameter angle $\theta$ is co-periodic. If the landing map $F_\fR$  were
Misiurewicz, then by Lemma~\ref{L-ratland} the dynamic ray of angle
$\theta$ for
$F_\fR$ would land on the co-critical point $2a$. Therefore the ray of angle
$3\theta$ which is periodic would land on the critical value $F_\fR(2a)=
F_\fR(-a)$. Thus both critical orbits would be periodic; hence $F_\fR$
would be in the interior of the connectedness locus. This is impossible
since the landing point must be a boundary point of the connectedness locus.
Thus $F_\fR$ must be a parabolic map. It then follows from the work of
Fatou and Julia that the orbit of at least one critical point (which can
only be $-a$) must converge to the parabolic orbit.\medskip

{\bf Step 2.}
  We must prove the following statement:

  \begin{lem}\label{L-paraland} If the landing map $F_\fR$ is parabolic,
    with a parabolic basin of period $q$, then 
    the parameter angle $\theta$ for $\fR$ must be co-periodic of
    co-period $q$. 
    Furthermore, for any map $F$ on the ray $\fR$, the entire orbit of
    the free critical point under $F^{\circ q}$
    must be contained in a dynamic ray of period $q$, with  angle either
    $\theta+ 1/3$ or $\theta-1/3$. 
\end{lem}\smallskip

We will need the following.\ssk

\begin{definition}[\bf The Root Point of a Parabolic Basin] \label{D-root}
Given a cycle of $q$ parabolic basins, the parabolic map $F^{\circ q}$ sends
each basin to itself with degree two. Hence the boundary of each
basin contains a unique fixed point of $F^{\circ q}$, called its
\textbf{\textit{root point}}. These root points form a parabolic orbit
of some period $k$ which divides $q$. 
In the case where  $k$  is less than $q$, note that each such root point
will be shared between $q/k$ different basin boundaries.
\end{definition}

{\bf Proof of Lemma \ref{L-paraland}.}
  Let $U$ be the parabolic basin for $F_\fR$ which contains the free
  critical point $-a$, and let $q\ge 1$ be the period of this
  basin, so that $F_\fR^{\circ q}$ maps $U$ onto itself with degree
  two.\ssk

  To simplify the discussion, we will first consider the special
  case $q=1$, so that\break $U=F_\fR(U)$.
 Thus the orbit of $-a$ under $F_\fR$ will converge
  to the parabolic root point. (Compare Figure \ref{F-para-pic}-left.)
 If we approximate this landing map
  by a nearby map $F$ on the ray $\fR$, then the orbit of $-a$ under
  $F$ will pass through a narrow ``escape channel'' near the former
  parabolic point, and then diverge towards infinity. 
  If the parameter angle is $\theta$, recall that the two dynamic
  rays of angle   $\theta\pm 1/3$ must crash together at the
  free critical point   (Figure \ref{F-para-pic}-right).\ssk

  We want to prove that $F$ must map one of these two crashing rays
  into itself. This will prove that $\theta$ is 
  co-periodic.   Furthermore, since the $F$-invariant ray contains the
  free critical point, it will prove that this ray contains the 
  entire free critical orbit.\msk

  Suppose to the contrary that neither crashing ray maps to itself under
  $F$. Since both of the angles $\theta\pm 1/3$ map to $3\theta$ under
  tripling, it follows that the ray of angle $3\theta$ must land at $F(-a)$.
  Similarly for each $k>0$ the ray of angle $3^k$ will land at
  $F^{\circ k}(-a)$.   But every
  rational angle is eventually periodic under tripling, so this contradicts
  the hypothesis that $-a$ eventually escapes to infinity. This
  proves the Lemma in the case $q=1$.\medskip

  If the parabolic basin has period $q>1$, then we can apply an analogous
  argument to the iterate $F^{\circ q}$ of degree $3^q$. Details are
  straightforward, and will be 
  left to the reader. This completes the proof of Lemma \ref{L-paraland}
  and Theorem \ref{T-main}.
\end{proof}\msk

\begin{rem}[\bf The quadratic analog]\label{R-qa} A completely analogous
  argument proves the corresponding statement for quadratic maps: 
  \begin{quote}\sf If a
parameter ray of angle $\theta$ lands at a parabolic point of the Mandelbrot
set, then $\theta$ is periodic under doubling. Furthermore,  the dynamic ray
of angle $\theta$ for any map $f$ belonging to this parameter ray contains
the entire orbit of the critical point under $f^{\circ q}$,  
where $q$ is the period of $\theta$.\end{quote}
Of course these results are well known. (Compare \cite[Chapter 8, Section 2.2
and Chapter 13]{DH1} and  Schleicher \cite[Theorem 1.1]{Sch}). 
However this method of proof seems much easier.
\end{rem}\msk

\begin{figure}[htb!]
  % \begin{minipage}{1.7in}
  \begin{overpic}[width=1.7in]{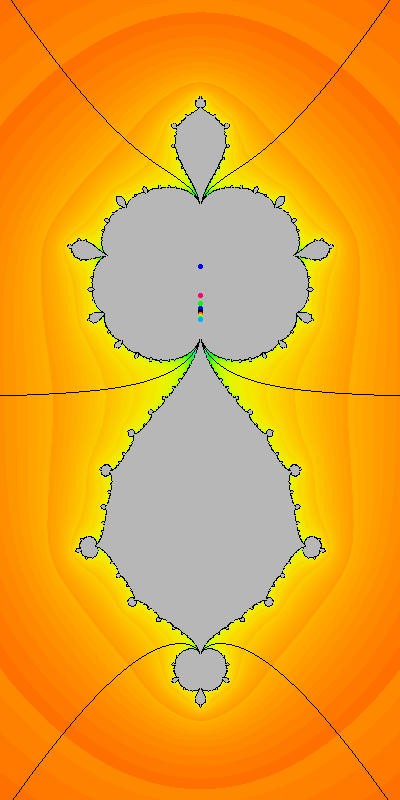}
    \put(105,20){5/6}
  \put(105,212){1/6}
  \put(0,212){1/3}
  \put(65, 160){$-a$}
  \put(0, 110){1/2}
  \put(0, 20){2/3}
  \put(117,110){0}
\end{overpic}\qquad\qquad\begin{overpic}[width=3.5in]{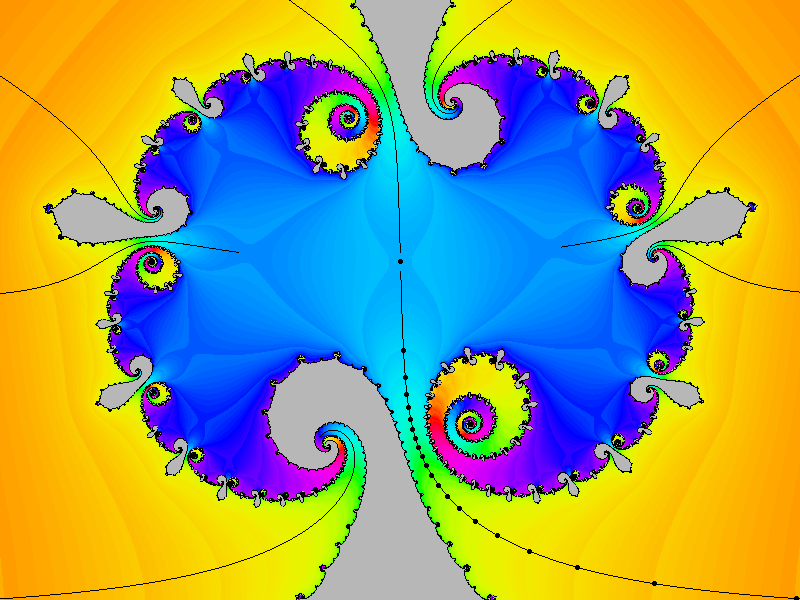}
  \put(245,3){0}
  \put(230,87){1/18}
  \put(230,165){1/9}
  \put(158,178){1/6}
  \put(70,178){1/3}
  \put(0,167){7/18}
  \put(0, 87){4/9}
  \put(0, 3){1/2}
  \put(130,104){$-a$}
    %\centerline{\includegraphics[width=1.7in]{para-pic.png}}
    %\end{minipage}\qquad \qquad\begin{minipage}{3.7in}
      %\centerline{\includegraphics[width=3.5in]{cfix-1000.png}}
    \end{overpic}
\caption[Julia sets showing rays crashing together]
{\label{F-para-pic}\sf On the left: Julia set for the map
  $z\mapsto z^3+ i\,z^2+z$ in $\cS_1$, showing the free critical orbit
  converging from above to the parabolic fixed point at $z=0$ within the basin
  $U$. On the right: detail of the Julia set
for a map on the $\theta=2/3$ parameter ray, showing that the dynamic
rays of angle $\theta+1/3=0$ and $\theta-1/3=1/3$ 
crash together at the free critical point (picture by Scott Sutherland).}
  
  \centerline{\includegraphics[width=2.9in]{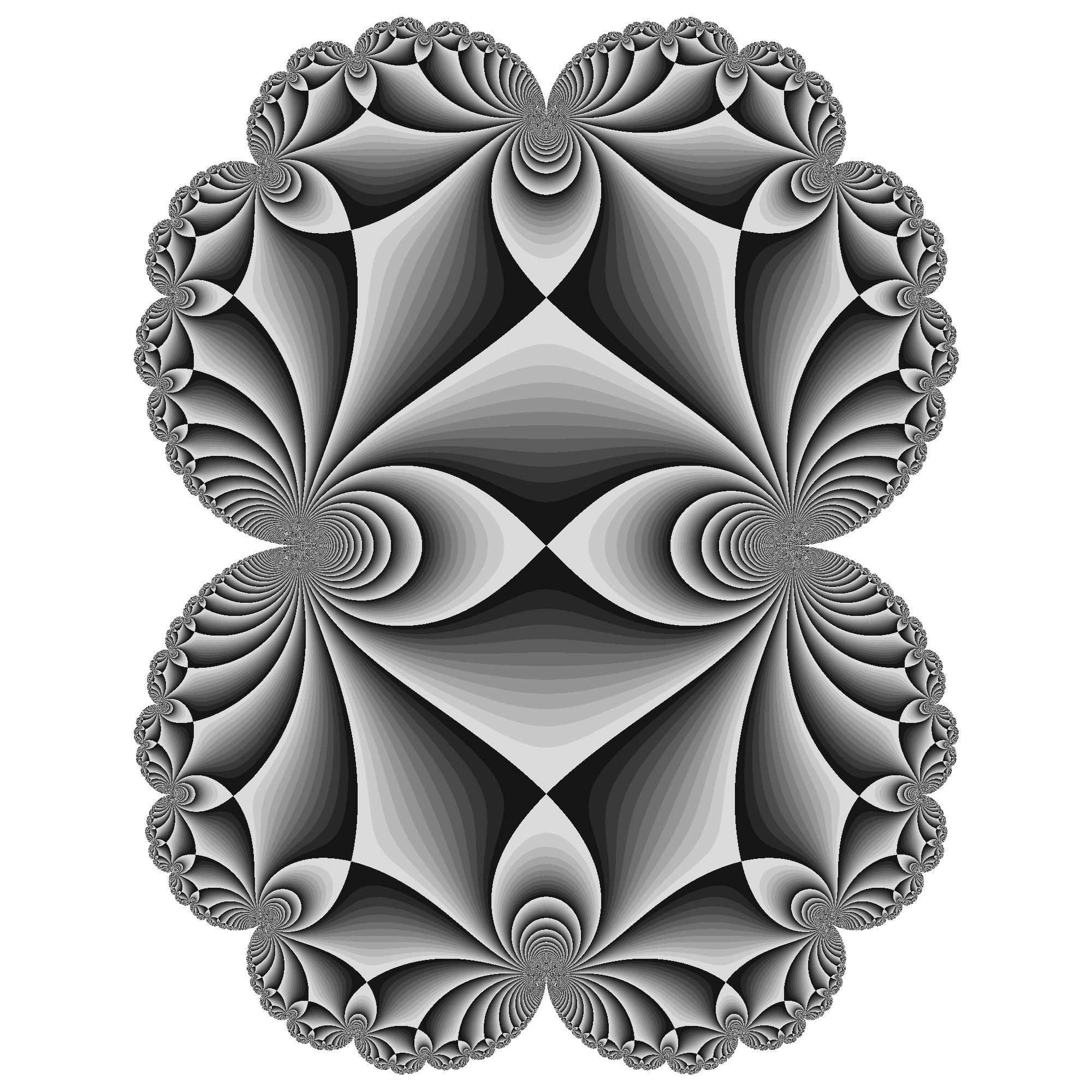}}
  \caption[The standard cauliflower]
  {\label{F-caulpic}\sf A picture of the standard cauliflower
    $z\mapsto z^2+z$ showing the real part $\re\big(\Phi(z)\big)$
    of the Fatou coordinate,
    modulo $\Z$. Here the shading varies from white for~~ zero$\!+$ modulo $\Z$,
    to black for~~ one$\!-$ modulo $\Z$.
    Thus the boundary curves between black and white indicate
    places where the real part of $\Phi$ is an integer. The points
    where two such curves cross represent the critical point and its
    iterated pre-images.}
\end{figure}

\begin{rem}[\bf Cyclic order]\label{R-cyc} If three or more parameter rays
land at a common point (either parabolic or Misiurewicz), then we believe
that the cyclic order of the rays about their landing point is always
the same as the cyclic order of their parameter angles.
(This is purely a local statement, and does not depend on whether the
rays belong to the same escape region or to different escape regions.)
For the Misiurewicz case compare Conjecture~\ref{C-TLS}; while for the
parabolic case compare the more detailed discussion in Section~\ref{s-near-para}.
\end{rem}
\msk

We next provide a more detailed description of the parabolic case,
which will play an essential role in later sections.
For the following, the symbol $F$ will always
denote a map which is parabolic, and hence belongs to the connectedness locus.
\medskip

\begin{rem}[\bf Dynamics in the parabolic basin] \label{F-par-basin}
Dynamics within a parabolic basin can be rather complicated, and is best
understood in terms of Fatou coordinates. In the case of an $F$-invariant
basin with a single simple critical point, the map is always topologically
conjugate to the standard quadratic example which is illustrated in Figure
\ref{F-caulpic}. Here the interior of $U$ is divided into simply
connected open ``cells'', bounded by curves which are black on one side
and white on the other. Each such cell maps biholomorphically onto a different
cell, and two cells have the same image if and only if they are diametrically
opposite with respect to the central critical point. Each cell which has the
parabolic fixed point as a boundary point maps to an immediate neighbor,
either the right hand neighbor if it is along the $x$-axis, or the left hand
neighbor otherwise.
\end{rem}

\begin{rem}[\bf Parabolic Multipliers]\label{R-pbm} 
One basic invariant of a parabolic map $F$ in $S_p$
is its \textbf{\textit{multiplier}} $\lambda$.

\begin{quote}{\bf Definition.} If the parabolic orbit for $F$ has
  period $k$, then $\lambda$ is the first derivative of $F^{\circ k}$ at
  any orbit point.\end{quote}

\noindent This is always a root of unity. More precisely,
  if the ray period is $q$ then $\lambda$ is a $(q/k)$-th root of unity.
  For further discussion see Conjecture MC2 in Remark \ref{R-MandelC}. \end{rem}
\msk

\begin{theo}[\bf Parabolic Maps]\label{L-para}
Suppose that $F$  has a parabolic orbit of ray period $q$.
Let $U(-a)$  be the Fatou component of $F$ which contain the free critical
  point $-a$. Then, with notation as in Definition~$\ref{D-cp}$, the
 dynamic ray of angle $\theta_q$ for $F$ lands at the \textbf{\textit{root
      point}} of $U(-a)$; that is at the unique boundary point which is fixed
  by $F^{\circ q}$. {\rm $($See Definition~\ref{D-root}.$)$}
  Similarly the non-periodic ray of angle  $\widehat{\theta}$ lands at
  the unique boundary point of $U(-a)$,  which is not equal to the root point,
  and yet maps to the root point under the $2$-fold branched covering map
  $F^{\circ q}$ from $U(-a)$ to itself.
  Similarly, if we define the root point of the Fatou component $U(2a)$
  to be the unique point which maps to the root point of $U(v')$ under $F$,
  then the non-periodic dynamic ray of angle
  $\theta$ for $F$ lands at the root point of  $U(2a)$.
\end{theo}

\begin{proof}[{\bf Proof}] The following diagram represents the
  various Fatou components
which must be considered in the case $q>1$.
\begin{equation}\label{E-Umaps}
\xymatrix@C=4em{U(2a)\ar[r]^{F} & U(v')\ar[r]<8pt>^{F^{\circ q-1}} & 
  U(-a)\ar@{->>}[l]^{F}}    \end{equation}
(Here we  are only  concerned  with the 
orbit of the free critical point: the marked critical point is not
involved in this discussion.)
Thus the Fatou component $U(2a)$ containing $2a$ maps isomorphically onto
the component $U(v')$ under $F$, and then isomorphically
onto the component
$U(-a)$ under $F^{\circ q-1}$. Then $F$ sends it back to the component
$U(v')$ by a 2-fold branched covering map. For the special case $q=1$,
since $U(v')=U(-a)$ there is a simpler diagram.
     $$ \xymatrix@C=4em{U(2a)\ar[r]^{F} & U(-a)}\mapstoself_{F}$$
Since $\theta$ has co-period $q$, we know that the angle
$\theta_1=3\theta$ is periodic of period $q$. %According to Assertion~\ref{AssA},
 It follows from Definition~\ref{D-root} that the dynamic ray of angle
$\theta_1$ lands at the root point of the Fatou component $U(v')$ containing
$v'=F(-a)$. It follows inductively that
the dynamic ray of angle $\theta_j$ lands at the root point of the component
containing $F^{\circ j}(-a)$. Since the cycle of components has period $q$
it follows that the  ray of angle $\theta_q$ lands at the root point of the
component $U(-a)$ itself.

The root point of $U(v')$ has three pre-images under $F$.
One is the root point of  $U(2a)$, one is the root point of $U(-a)$,
and the third is a diametrically opposite point of the boundary of $U(-a)$.
Correspondingly the dynamic ray of angle $\theta_1$ has three pre-images
under $F$. The ray of angle $\theta$ lands at the root point of $U(2a)$;
the ray of angle $\theta_q$ lands at the root point of $U(-a)$; while
the ray of angle  $\widehat{\theta}$ lands at the diametrically
opposite point of the boundary of $U(-a)$. \end{proof} \medskip

\begin{rem} It might seem natural to expect that as a sequence of
  maps $F_h$ on the parameter
ray $\fR$ converge to the landing map $F_\fR$, the corresponding dynamic rays
of angle $\theta\pm 1/3$ would converge to uniquely defined paths
from infinity to $-a$; but this may not be true.
The limit outside of the filled Julia set $K(F_\fR)$ is uniquely defined.
However, depending on the Lavaurs phase, there may be infinitely many possible
limit maps inside of this filled Julia set, depending on which sequence of
points of $\fR$ we choose. (See \cite{BMS}.)
\end{rem}

\begin{figure}[htb!]
\centerline{\includegraphics[width=3in]{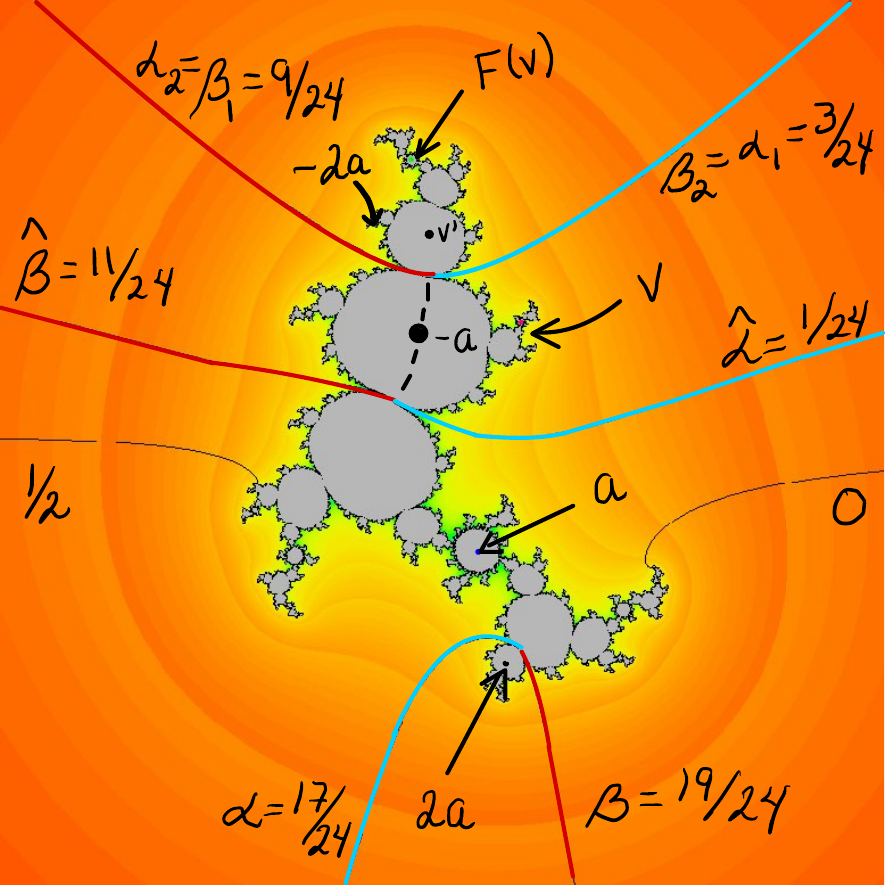}}
\caption[Julia parabolic map in $\cS_3$]{\label{F-s3per2jul}\sf Julia set for
  a parabolic map in $\cS_3$
%\index{Figure~\ref{F-s3per2jul} Julia parabolic map in $\cS_3$}
  which is the landing point for two different parameter rays of co-period two
  with angles $\alpha=17/24$ and $\beta=19/24$.
  In this example, there is a single parabolic fixed point, which is the root
  point of both the $-a$ and the $F(-a)=v'$ components. Furthermore, in this
  example the parameter                               
  angles  $\alpha$ and $\beta$ belong to the same  grand orbit,
  so that $\alpha_1  =\beta_2$ and $\beta_1=\alpha_2$. The         
  corresponding point $F$ in the parameter space
  $\cS_3$ is visible towards the upper left in Figure \ref{F-t2s3}.
  There is an  analogous picture for an arbitrary parabolic map in any $\cS_p$,
  although the corresponding number of parameter rays will vary (conjecturally
  always between one and four). 
Compare Remark \ref{R-J2K}. 
}
\end{figure}

\medskip

These results have an interesting consequence as follows.
 \smallskip
 
\begin{coro}\label{C-cop} Suppose that several co-periodic parameter rays
with parameter angles $\theta(j)$ 
land at a common parabolic map $F\in \cS_p$.
Then the  dynamic rays of angle $\theta(j)$ for $F$  all land at the
root point $z_F$ of the co-critical component $U(2a)$ in the filled 
Julia set. $($Here %\rnote{Shortened}
these parameter rays need not all lie in the same escape region.$)$
%Furthermore if three or more parameter rays land at a common 
%parabolic point $F$, then the corresponding dynamics rays for $F$
%land in the same cyclic order around $z_F$.
%Similarly the dynamic rays of angle $3^k\theta(j)$
%land in the same cyclic order around the periodic point $F^{\circ k}(z_F)$
%for $1\le k\le q$, since $F^{\circ k}$ is a local diffeomorphism at $z_F$
%for $k\le q$.
\end{coro}

\begin{proof} This follows immediately from Theorem~\ref{L-para}. 
%  (assuming Remark \ref{R-cyc}).
\end{proof} 
\smallskip

\begin{coro}\label{C-distinct} No two co-periodic parameter rays
landing at a common point can have the same parameter angle. 
{\rm (This is true whether the rays lie in the same escape region or different
escape regions.)}
\end{coro}

\begin{proof}  This follows immediately from Corollary \ref{C-cop}.
 \end{proof}

{\bf Examples:}
For each fixed co-periodic angle $\theta$, we will see in Section~\ref{s-count}
that there are ${\bf d}_p$ distinct parameter rays in $\cS_p$
with parameter angle $\theta$, where $\d_p$ is the degree of $\cS_p$
(asymptotic to $3^{p-1}$). Thus, for each such $\theta$ there correspond
$\d_p$ distinct parabolic landing points. For examples in the case $p=2$
with  ${\bf d}_2=2$, see Figures~\ref{f2}, and \ref{F-S2rays} or
\ref{f-po}. For many examples with $p=3$ and  ${\bf d}_3=8$, see
Figures~\ref{F-t1s3} 
through ~\ref{F-T3S3}. In all of these cases,
there is just one such ray in each escape region. However, for $p\ge 4$
there are regions of multiplicity $\mu>1$, each containing $\mu$ such
rays.\bsk

{\bf A convenient notation.}  We will often use the notation 
$\p$ for a parabolic map, when we are thinking of it simply as a point
in the parameter space $\cS_p$. In such cases, we will use the notation
$F_\p:\C\to\C$ if we are thinking of it explicitly as a cubic map from
the complex plane to itself.\smallskip

As an example, in the lower right of  Figure \ref{F-S2rays}, the 
four parameter rays with angles  $~\theta~=~10/24$, $11/24$,
$ 14/24$ and $17/24$ belong to two different escape
 regions, but still lie in positive cyclic order around their common
 landing point $\p$. The dynamic rays with the same angles land at a
 common point $z_0$ in the Julia set for $F_\p$. Here $z_0$ is the root
 point of the hyperbolic component containing $2a$,  as shown in 
 Figure \ref{f-s2par}. The associated periodic angles
 $~3\,\theta~=~1/4,~ 3/8,~ 3/4,~ 1/8$  are also in positive cyclic order,
 and land at the image $F_\p(z_0)$, which is the root point 
 of the component containing $-a$.  (Notice that in this particular case
 the Fatou components of $-a$ and of $v'$ share
 the same root point.) 
 \medskip

 In many cases the converse of Corollary \ref{C-cop} is true, so that
 a dynamic ray lands at the co-critical root point if and only if
 there is a parameter ray with the same angle landing at the given parabolic
 map. But there are also many counterexamples.  For example
 Figure~\ref{F-s5rab-par} shows four rays
 (in three different escape regions) landing at a parabolic map
 in $\cS_5$; and Figure~\ref{F-4wall} shows the four
corresponding rays in the dynamic plane. But there are six additional rays
in the dynamic plane, corresponding to six additional accesses to the root
point. (Compare Remark~\ref{R-access}.) Note that this parabolic map is on
the common boundary between a rabbit region, and two different escape
regions. % as described in Section~\ref{s-rab},

Conjecturally, such rabbit examples are the only counterexamples to
the converse statement. 
And even in these counterexamples, for every dynamic ray landing at
the co-critical root, there seems to be a parameter ray with the same angle
which lands fairly close to the parabolic point. Furthermore, these
other rays seem to 
play an important role in understanding the local dynamics
(see \cite{BM}).
% \rnote{Jack, Do you agree with my fixing of the paragraph starting with
%   ``On the converse statement? A-}
% \change On the converse statement we present in Remark~\ref{R-rabconj}
% a conjectural picture of what happens with the rays landing at the root
% point of the hyperbolic component containing the point $2a$. We described in
% Remark~\ref{R-rabconj} that the union of those rays together with their end
% points, forms a figure eight curve with the $\alpha$ and $\beta$ rays forming
% a loop around the primary wake containing $D_j$, and with $\gamma$, $\delta$
% and the remaining
% $2p-4$ rays forming a loop around $D_j'$.
 %Remark ~\ref{R-rabconj}, as well as
%Figures~\ref{F-s3rab-new},  \ref{F-s4rab-par} and \ref{F-s5rab2}.
 \medskip

% those are the ones that together with the rays $\gamma$ and $\delta$
%form the loop of the $D_j'$ hyperbolic component.
% which do not have any counterpart in parameter space.

Note also that the local angles between parameter rays at their landing
point are quite different from the corresponding angles between dynamic
rays (which are not even well defined unless the multiplier is real). 
This is completely analogous to what happens for
the quadratic family. For example, since 
 the  parameter rays of angle $1/7$
 and $2/7$ for the Mandelbrot set land at the ``fat rabbit'' map $F_0$,
 it follows that the $1/7$ and $2/7$ dynamic rays for $F_0$
 land at a common point in $J(F_0)$.
 But the $4/7$ dynamic ray also lands at this same point. Similarly the
 correspondence does not preserve local angles in the quadratic case: 
The  180$^\circ$ angle between parameter rays at their landing point
is not equal to  the apparent 120$^\circ$ angles between dynamic rays at
their landing point.
 (The latter angles are not actually well defined since the dynamic rays
spiral in to their landing point.) 
\medskip

 In \S\ref{s-tess} we will prove that every parabolic map in ${\mathcal S}_p$
 is the landing point of at least one co-periodic ray. (See 
 Theorem~\ref{T-decomp} and Corollary~\ref{C-allpara}.) %The corresponding 
% statement for Misiurewicz maps is discussed in  \S\ref{s-crit-fin}.
 \medskip

 To conclude this section, we will describe several important concepts.

\begin{definition}[\bf Kneading]\label{D-kn} For every
  escape region $\cE_h\subset \cS_p$ there is an associated infinite~
  \textbf{\textit{kneading sequence}}\vspace{-.2cm}
  $$(\kappa_1,~\kappa_2,~\kappa_3,~\cdots)~,$$
where $\kappa_j=0$  if 
$F^{\circ j}(a)$ belongs to the lobe of the figure eight containing $a$,
and where $\kappa_j=1$ if it belongs to the other lobe which contains $-2a$.
(Compare Figure \ref{Faa-1}.) Here $F$ can be any map
in $\cE_h$. Since the marked critical point $a$ has period
$p$, it follows that
$$ \kappa_p=0,~~~{\rm and~that}\qquad \kappa_j=\kappa_{p+j}\quad
{\rm for~all} \quad j~.$$  (Note that the minimal period of this
sequence may be some divisor of $p$.)

In practice, it is usually
more convenient to list only the first $p$ entries
$(\kappa_1,~\cdots,~\kappa_p)$ which we will call the 
\textbf{\textit{kneading invariant}} of the escape region. (Here $\kappa_p$
is always zero.) We will often abbreviate this
notation, writing for example 110 in place of $(1,1,0)$.
(Note that the various kneading invariants
$(1,0)$ in $\cS_2$ or $(1,0,1,0)$ in $\cS_4$ or $(1,0,1,0,1,0)$ in $\cS_6$
all give rise to the same kneading sequence of period two; which can be denoted
briefly by $\overline{10}$.)
\end{definition} \medskip

\begin{rem}[\bf The Mandelbrot Set and $\cS_p$]\label{R-Mand}
  The classical quadratic  Mandelbrot set, which we denote by  $\M$, 
  plays two quite different roles in   the study of $\cS_p$.\smallskip

\begin{itemize}  
\item[{\bf(1)}] Every $\cS_p$ contains infinitely many copies of $\M$,
which play a very important role. These will be studied in Remark~\ref{R-MandelC}
  below.\smallskip

\item[{\bf (2)}]
Branner and Hubbard \cite{BH1} showed that every escape region
$\cE$ in $\cS_p$ has a unique \textbf{\textit{ associated
critically periodic quadratic map}} with the following property:

\begin{quote}\sf For every map $F$ in $\cE$, every connected
  component of the Julia set $J(F)$ is either a point, or a copy 
  of the Julia set for the associated quadratic map.\end{quote}
\end{itemize}
\end{rem}
\medskip

If the kneading sequence for an escape region in $\cS_p$
has minimal period $k$, then the period for the
critical point of the associated quadratic map is $p/k$. (Compare
Definition~\ref{D-PM}  %\ref{d-rp}
as well as \cite[Theorem 5.15]{M4}.) In
particular, if the kneading invariant is $(0,0,\cdots,0)$ so that 
$k=1$, then the quadratic map has critical period $p$. 
\medskip

\begin{figure}[htb!]
\centerline{\includegraphics[width=3.8in]{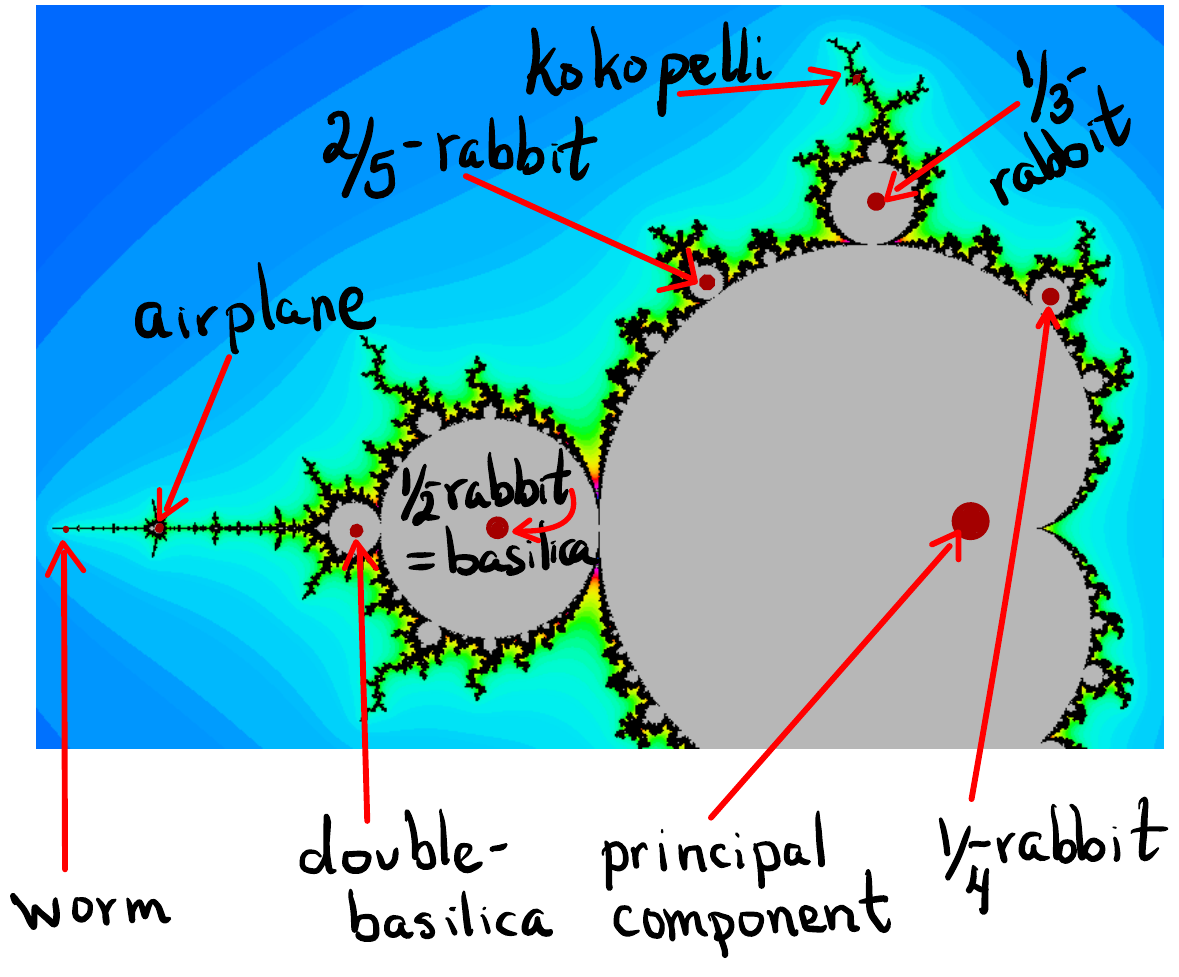}}
\vspace{-.3cm}
\caption[Mandelbrot set showing some critically periodic points]{\sf The
  classical Mandelbrot set $\M$, showing some of the critically
%\index{Figure~\ref{F-manpic}, Mandelbrot set}
  periodic points that we will refer to. (Compare \cite[Table 5.15]{BKM}.
  Here the ``worm'' is the left-most period four component.) The immediate
  satellites of the main hyperbolic component are referred to as 
  \textbf{\textit{rabbit}} components. There is one rabbit 
  component for each rational number $0<n/p<1$.
\label{F-manpic}}
\end{figure}

\begin{definition}[\bf Zero-Kneading Regions] 
We will say that an escape region in $\cS_p$ belongs
to the \textbf{\textit{zero-kneading family}} if its kneading invariant 
$(0,\,0,\,\cdots 0)$ consists of $p$ zeros, or in other words if its
kneading sequence is $\overline 0$, of period one. 
In this case, Branner and Hubbard showed that 
the associated quadratic map determines the
escape region uniquely. Zero-kneading
regions are always invariant under 180 degree rotation
(\cite[Lemma 5.16]{BKM}). They are invariant under complex conjugation if and
only if the associated quadratic map has real coefficients.

In particular, for each of the components emphasized in Figure~\ref{F-manpic}, 
there is a corresponding escape region in some $\cS_p$. The main hyperbolic
component corresponds to the unique escape region in $\cS_1$. 
There is the ``basilica'' region in $\cS_2$. There are
the ``airplane'' and two ``rabbit'' regions in $\cS_3$. In 
$\cS_4$ there are the ``Kokopelli'' region  (see Figure~\ref{F-44})
% ~\ref{F-kocan} and
and its complex conjugate, as well as the
``$(\pm1/4)$-rabbit'' regions, % (Figure~\ref{F-rabcan}),
the ``double-basilica'' region, and  %(Figure~\ref{F-Dcan}), and
the ``worm''. %(Figure~\ref{F-worm}).
Compare  \cite[Ex.~5.14]{BKM}.

Every $\cS_p$ with $p>1$
contains two or more rabbit regions, corresponding to the period $p$ 
immediate satellites  of the Mandelbrot set.   For more about
zero-kneading regions see \cite{BM} (the sequel to this paper).\bsk

For $p\le 4$ the following statement is true: 
 \smallskip

\begin{quote}
%  \begin{itemize}
%\item[$\bullet$] {\sf Every escape region either has at least one
%zero kneading region as an immediate neighbor (sharing at least one 
%common boundary point), or is itself a zero kneading region.}\smallskip
{\sf Two escape regions with a common boundary point
always have different kneading sequences. 
In particular, no two zero kneading regions can share a boundary point.}
\end{quote}\medskip

\noindent % We have no idea whether the first statement is true for larger $p$.
It seems likely that this statement 
is true in every $\cS_p$. As a step in this direction:
\cite[Theorem 6.8]{AK} shows
that if the common boundary point is also a boundary  point of a hyperbolic
component of Type B$(m,n)$ then the $m\!$th component of the kneading
sequence always flips as we cross from one escape region to the other,
but no other component changes. (For the definition of Type  B$(m,n)$ see
Remark~\ref{R-HC}.) In \cite{AK}, Arfeux-Kiwi provide a similar but much more
complicated rule in the case of a Type A component. \medskip

{\bf Non-Zero Kneading.} In all cases with non-zero
kneading, since the kneading period $k$ is greater than one,
the period for the associated map is some proper divisor
$p/k<p$. In particular, in many cases we will have
$p/k=1$, which means that all Julia sets in the escape region will contain
countably many connected components which are 
topological circles. In particular, if $p$ is a
prime number then this will certainly happen in all cases with non-zero
kneading.
\end{definition}\smallskip

\begin{rem}[\bf From Julia Set to Kneading Invariant]\label{R-J2K}
  Any parameter ray $\fR$ of co-period $q$ lands at some parabolic map $\p$.
  Given the parameter angle $\theta$, and the 
  Julia set $J(F_\p)$, we will show that it is not difficult
to compute the kneading invariant for the escape region which contains $\fR$.
  (Compare \cite[Definition 5.3]{AK}; which also includes 
 an analogous discussion for
  parameter rays which are rational but not co-periodic.) \medskip

Recall from Theorem~\ref{L-para} 
that the dynamic ray of angle $\theta$ must land at
  the root point of the Fatou component which contains $2a$. Hence for
  $1\le j\le q$ the ray of angle $\theta_j=3^j\theta$ lands at the root point
  of the component which contains $F_\p^{\circ j}(2a) =F_\p^{\circ j}(-a)$.
  In particular, the ray of angle $\theta_q$ lands at the root point of the
  component $U=U(-a)$ which contains the free critical point $-a$.
  Note that the map $F_\p$
  from $U$ to $F_\p(U)$ is a 2-fold branched covering map
  (with reference to Equation$($\!\ref{E-Umaps}$)$).

  The parameter angle $\theta$
  of co-period $q$ gives rise to a
\textbf{\textit{triad}} $$(\theta,~~\theta_q,~~\widehat{\theta})$$
consisting of three equally spaced
dynamic angles, where $\theta$ and  $\widehat{\theta}$ are co-periodic, but
$\theta_q$ is periodic of period $q$. In the dynamic plane
  for the parabolic landing point,
  the $\theta$ ray lands at the root point of $U(2a)$, the $\theta_q$
  ray lands at the root point of $U(-a)$, and the $\widehat{\theta}$
  ray lands
  at the ``opposite'' boundary point of $U(-a)$, which maps to the root
  point under the 2-fold covering map $F_\p^{\circ q}$. (Compare
  Figure~\ref{F-s3per2jul}.)  
\end{rem}

\begin{figure} [htb!]
  \centerline{\includegraphics[width=2.4in]{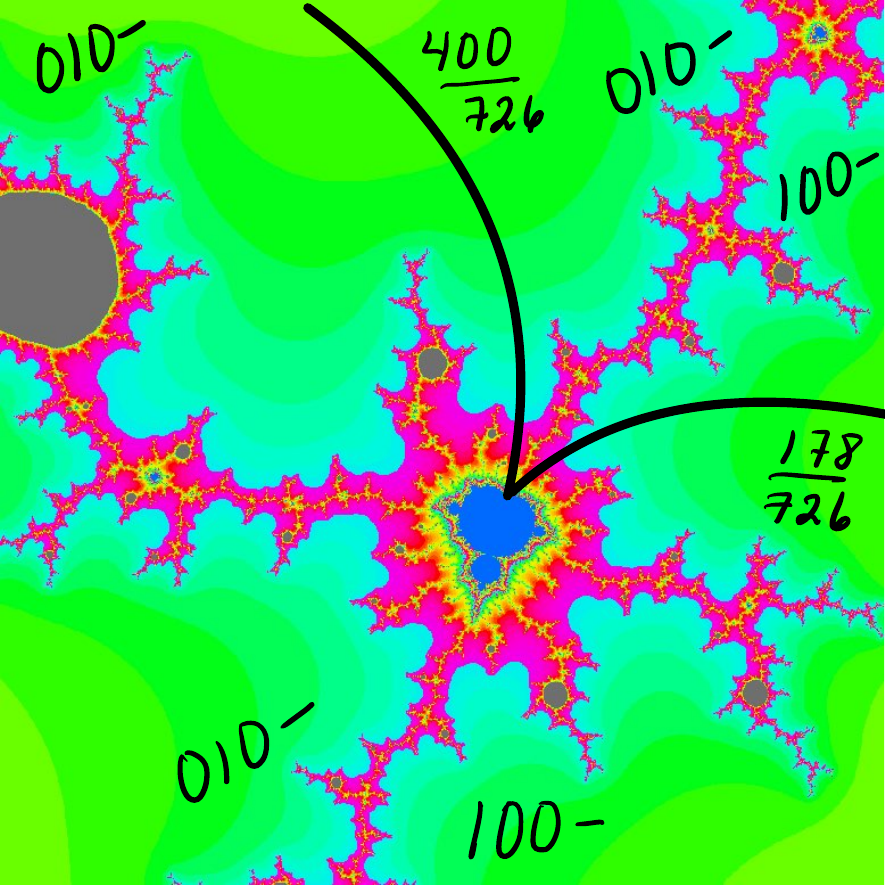}}\medskip
  \centerline{\includegraphics[width=2.6in]{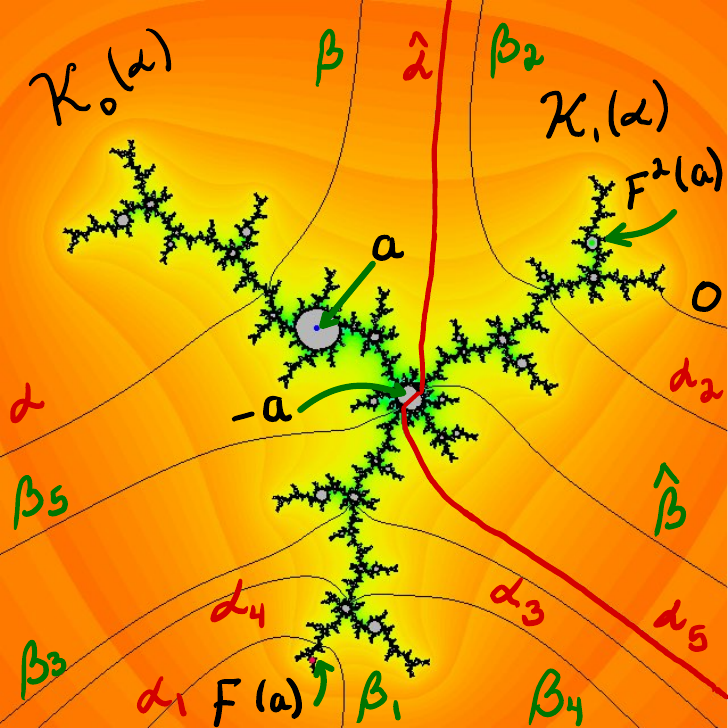}\quad
      \includegraphics[width=2.6in]{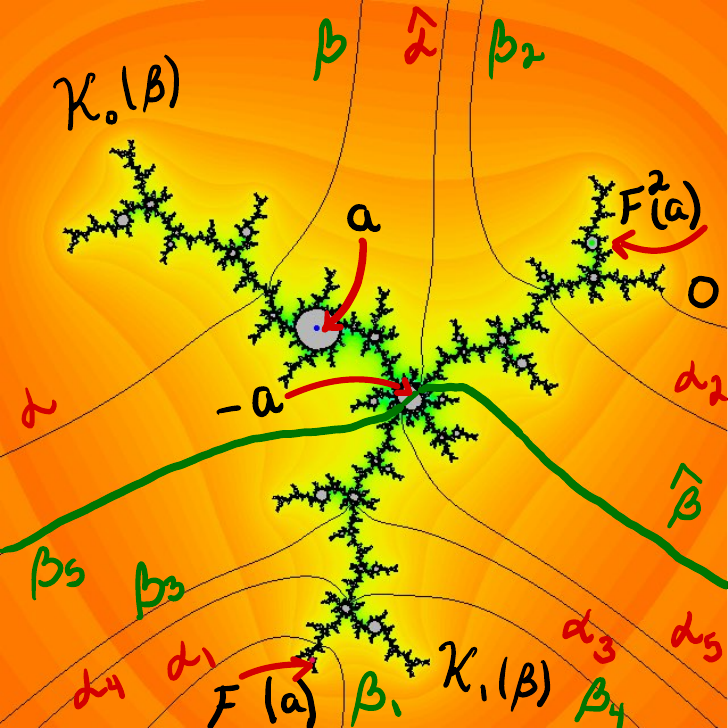}}
    \caption[Small Mandelbrot copy in $\cS_3$ with walls]{\label{F-J2K}\sf Above:
      A small Mandelbrot copy in $\cS_3$
    %\index{Figure~\ref{F-J2K} small Mandelbrot copy in $\cS_3$ with walls}
      lying along the boundary between the $010-$ and $100-$ regions.
      Below: The dynamic plane for the parabolic landing point, with
      the $\alpha$-wall on the left and the $\beta$-wall on the right, where
      $\alpha$ has numerator $400$ and $\beta$ has numerator $178$. 
      The corresponding kneading invariants, $010$ on the left and $100$
      on the right, can be computed using Lemma \ref{L-J2K}.}      
  \end{figure}\medskip

  \begin{definition}[\bf Walls and Kneading Walls] \label{D-wall} 
    By a \textbf{\textit{wall}} across $\C$ we
    will mean a closed subset homeomorphic to $\R$. Any wall separates 
    $\C$ into two complimentary simply-connected open sets.\medskip

    Let $\p$ be a parabolic point of ray period $q$ in $\cS_p$, and let
    $\theta$ be the angle of some parameter ray which lands on $\p$. Then
    in the dynamic plane for the map $F=F_\p$, the dynamic ray of angle
    $\theta$ will land at the root point of the $U(2a_\p)$ component. The
     union of the $\theta_q$ and  $\widehat{\theta}$ dynamic rays,
    together with an arc within $U(-a_\p)$ joining their landing points,
    will be called the $\theta$-\textbf{\textit{kneading wall}} across the
    dynamic plane. This wall bisects the plane, dividing it into an open set
    $\K_0(\theta)$ containing the marked critical point $a_\p$, and an open
    set ${\mathcal K}_1(\theta)$ containing the marked co-critical point
    $-2a_\p$. 
  Every point of the marked periodic orbit must belong to one of these
  two open sets.\end{definition}

See Figures~\ref{F-s3per2jul}, \ref{f-s2par},  \ref{F-4wall}, 
for examples where the root point of $U(-a_\p)$ is a fixed point of $F$; and see
Figures~\ref{F-J2K} and \ref{F-43}-right for examples where it is not a fixed
point.

\begin{lem}\label{L-J2K}
The kneading invariant $(\kappa_1,\ldots\kappa_q)$ for the escape region which
contains this parameter ray of angle $\theta$ can be computed as follows. We
have $\kappa_j=0$ if and only if $F^{\circ j}_\p(a)\in \K_0$, 
and $\kappa_j=1$ if and only if $F^{\circ j}_\p(a)\in \K_1$.
\end{lem}

\begin{proof} (Compare \cite[Definition 5.3]{AK}.) 
The proof is easily supplied: As we approximate $F_\p$ closely by a map $F$
which lies on the parameter ray $\fR$, the  dynamic rays of
angle $\theta_q$ and $\widehat{\theta}$ will deform continuously
outside of the open
set $U$, but will crash together inside $U$. Since the marked periodic orbit
will also deform continuously, the conclusion follows.
\end{proof}\smallskip

As an example, applying this lemma to Figure~\ref{F-s3per2jul}, it follows
that the $\alpha=17/24$ 
parameter ray in an escape region of $\cS_3$ with kneading invariant $110$,
shares a landing point with the $\beta=19/24$ parameter ray in a region
with invariant $010$. See Figure \ref{F-t2s3}, upper left, for
the corresponding point in parameter space. 
\medskip

\begin{rem}[\bf Hyperbolic Components]\label{R-HC}
  The most basic distinction in our pictures of $\cS_p$ is between
 the \textbf{\textit{connectedness locus}}, consisting of
 maps with connected Julia set, and the various escape regions.
 The connectedness locus 
  contains four essentially different types of hyperbolic components. 
 (Compare \cite[\S1A]{M4}.) If $F$ is any map belonging to the  
hyperbolic component $H$, then $H$ belongs to one of four types, as follows.

\begin{quote}
\textbf{\textit{~Type A~}}
if both critical points of $F$ belong to the same Fatou component.
\smallskip

\textbf{\textit{~Type B~}} if the critical points belong to different
components of a common  cycle of Fatou components. More
explicitly, it has \textbf{\textit{Type B$(m,n)$}} if
$$\xymatrix@C=4em{a~~ \ar[r]<8pt>^{F_0^{\circ m}}& -a~~ \ar[r]<8pt>^{F_0^{\circ n}}& a}$$

%%$$a~~\stackrel{ F_0^{\circ m}}{\longmapsto}~~ -a~~
%%\stackrel{ F_0^{\circ n}}{\longmapsto}~~ a~,$$
%$$ F_0^{\circ m}(a)~=~-a\qquad {\rm and}\qquad F_0^{\circ n}(-a) ~=~a~,$$
  where $F_0$ is the center point of the hyperbolic component, and where
  $m+n=p$ with $m,\,n>0$.
\smallskip

\textbf{\textit{~Type C~}} if some forward image $F^{\circ n}(-a)$
belongs to the cycle of attracting Fatou components containing $a$, but the
free critical point $-a$ itself does not. \smallskip

\textbf{\textit{~Type D~}} if the critical points belong to disjoint
cycles of attracting Fatou components. More explicitly we will call this
\textbf{\textit{Type D$(q)$}}  if the free critical point $-a$ belongs
to a component of period $q$.
\smallskip
\end{quote}

\noindent  In our figures, the components of Type A, B and C are
brown, while those of Type D are blue. 
 The escape regions have lighter colors, tending towards green (or red) 
 near the border.  In all cases, the background is
 computer drawn, but any parameter rays are hand drawn approximations.

 There are only finitely many 
components of Type A and B in each $\cS_p$, while there are infinitely
many of Type C and D. The components of Type A, B, and especially D,
play an important role in the organization of each $\cS_p$. (See for
example  Section~\ref{s-near-para}.) %and \ref{s-rab}.) 
Components of Type C usually appear as small attachments to the much
larger A or B components, or as isolated disks. For examples see 
Figure~\ref{F-J2K} (top-left),  or Figure~\ref{F-S2cheb} (left).
\end{rem} \medskip

\begin{rem}[\bf Root Points in Parameter Space]\label{R-roots2}
  By definition a \textbf{\textit{root point}} of a hyperbolic component in
  $\cS_p$ is a parabolic boundary point with minimal ray period. Here
  components of Type D have one root point; those  of Type A have 
  two root points; and  those of Type B have three root points.
As one example, the archetypal Type A component is the central
lemon shaped component in $\cS_1$, with root points at the upper and
lower tips. (Compare Figure~\ref{F-t2s1}.)
As far as we know, those of Type C can never have a parabolic boundary point.

Assuming Conjecture MC1 below,  
it follows that components of Type A and B always have attached components
of Type D. We believe that components of Type C never have such attachments.
(For details in the Type A and B cases
see  \cite{M1} and \cite[Theorem 6.6]{AK}. Similar arguments prove the
corresponding  statement for components of Type D.)
\end{rem}\medskip

\begin{rem}[\bf Copies of $\M$ in $\cS_p$: Three Basic Conjectures]\label{R-MandelC}
  McMullen \cite{Mc}\footnote{Special cases of \cite[Theorem 1.1]{Mc} were
    studied by Douady and Hubbard in \cite[pgs. 332-336]{DH3} and by Eckmann and
 Epstein in \cite{EE}.} has shown that families of rational maps often contain
many \hbox{hybrid-conjugate} copies of the quadratic Mandelbrot set $\M$. (Here
``\hbox{hybrid-conjugate}'' means that the copies are conformally isomorphic to
$\M$ throughout the interior;
  but only quasi-conformally isomorphic on the boundary.) In our case,
  we conjecture a somewhat more explicit statement.\msk

  \begin{description}
  \item[{\bf Conjecture MC1.}] {\sf Every parabolic point in $\cS_p$
      is contained in a unique maximal  \hbox{hybrid-conjugate} copy of
    $\M$. Similarly every hyperbolic component of Type D in $\cS_p$
is contained in a unique maximal \hbox{hybrid-conjugate} copy of $\M$.} \msk

\noindent It follows for example that every parabolic point is the root
point of a unique component of type D, and that parabolic points 
are everywhere dense on the boundary of each such component.

Now consider the \textbf{\textit{multiplier}} of a parabolic point, as defined
in Remark~\ref{R-pbm}.\msk

\item[{\bf Conjecture MC2.}] {\sf Let $F_\p$ be a parabolic map in a copy
 $M$ of $\M$. Then the multiplier of $F_\p$ is equal to the multiplier  of the
 corresponding parabolic map in $\M$.}

\noindent  As an example, this implies that the multiplier of $F_\p$ is
equal to $+1$ if and only if $\p$ is the cusp point of its copy of $\M$.
%\rnote{Does this statement check with examples? YES! A.}
\end{description}\msk

By the \textbf{\textit{period}} of a critically periodic map in $\cS_p$
we will mean the  period of its free critical point. By the
\textbf{\textit{period}} of a Mandelbrot copy $M$ we will mean the period
of the center of its main hyperbolic component.
\msk

\begin{description}
 \item[{\bf Conjecture MC3.}] {\sf If $F$ is the critically periodic center
  point of a component in $M$, then the period of $F$ is the product of
  the period of $M$, and the period of the corresponding critically periodic
  point in $\M$. (In particular if $M$ has period $q$, then the period  of the
  center point of any component of $M$ will be some multiple of $q$.) }
\end{description}

\end{rem}

%{\bf Question.} Is each Mandelbrot copy a complete copy of the entire
%Mandelbrot set, or are some of the extremities missing in some cases.
%For example, does every Mandelbrot copy extend all the way out to the
%Chebyshev point at the very tip?  Figure ? shows an example of a Mandebrot
%copy in $\cS_2$, together with a picture of the Julia set for its Chebyshev
%point. However for $p>2$ the computations become more difficult, and we have not
%been able to locate such a point.

%\begin{quote}{\bf Conjecture MC1.} Every parabolic point  $\p\in\cS_p$
%  is contained in a unique\break maximal quasiconformal copy
%  $M$  of all or part of the classical Mandelbrot set. In the cases where $M$
%  is surrounded by only one escape region $\cE$, this is a complete copy
%  of the entire Mandelbrot set. In every case, $M$ contains a complete copy
%\rnote{Araceli: Do you believe this? Is molecule the accepted term? Reference?}
%  of the central ``molecule''   consisting of the main hyperbolic component
%  together with the union of all iterated  immediate satelites. In particular,
%  the cusp point of $\M$ always corresponds to a cusp point in the copy.
%\end{quote}\msk

\begin{rem}[\bf A Question]\label{R-MCquestion} If $F$ is the center point of
  a hyperbolic component contained in such a copy $M\subset\cS_p$, 
    how is the Hubbard tree of $F$ related to the Hubbard tree of the
    associated critically periodic quadratic map?\msk
    
%    constructed, up to equivalence,  from the Hubbard tree of $F$ by
%    the construction described below.
     Giving a sufficiently precise answer to this question, it might be
    possible to give a constructive proof of Conjecture MC1. 
    But before discussing the question, we need a precise definition which is
    weak enough so that it makes sense to compare quadratic and cubic Hubbard
    trees. (See Appendix~\ref{a-embdtree}.)
\end{rem}
    \msk

\begin{figure}[htb!]
  \centerline{\includegraphics[width=3in]{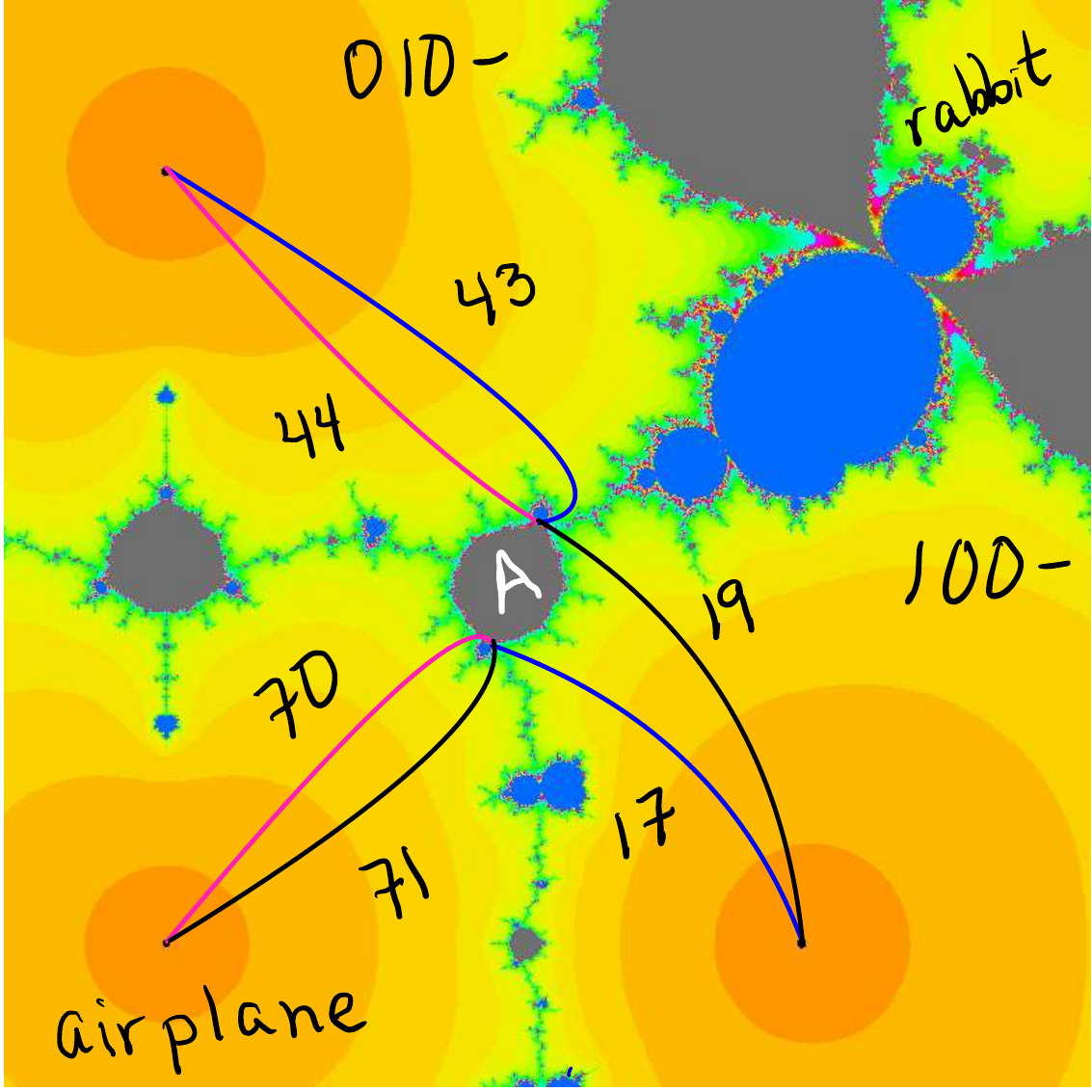}}
  \caption [Detail of $\cS_3$]  {\label{F-MC-crash}
    \sf Detail of $\cS_3$. (Compare Figures~\ref{F-t3s3air-010} through
    \ref{F-air}.) Four different escape regions are visible.
    Here the large Mandelbrot copy crashes into a Type A component
    at the root point of its airplane
    component, so that the tip of the copy gets bent
    upward. It is interesting to note that the small D component
    at the bottom of the A component is the dual of the ``airplane'' component
    at the top, so that the Julia sets of the center points are the same,
    but the roles of $+a$ and $-a$ are interchanged. Here twin 
    rays have been given the same color. (Compare Remark~\ref{R-dual}.) All
    angles have denominator 78 =3\,d(3). (See Remark \ref{R-NC}.)}
  \medskip

 \centerline{\includegraphics[width=4.2in]{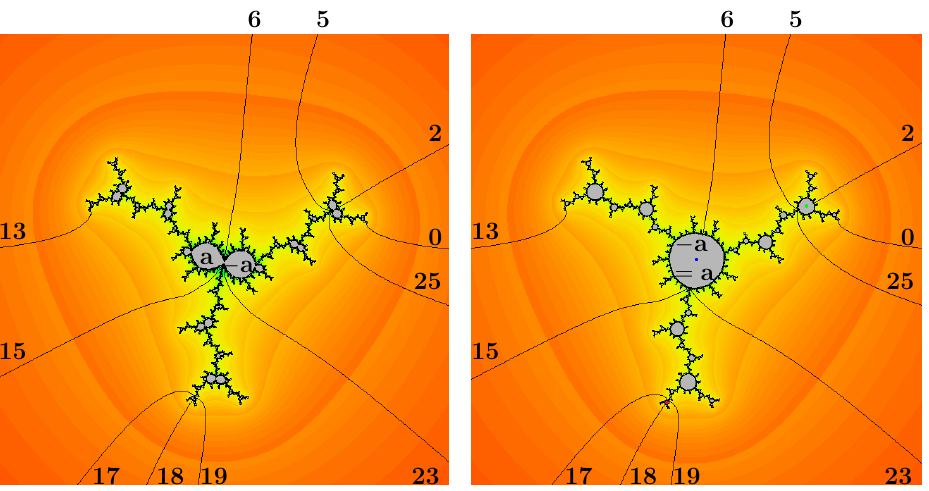}} 
\caption[Julia sets for a Type A component and the attached Type D
components.]
{\sf \label{F-centerA-airplane}  On the left, Julia set for the center of
  the Type D component which is  attached to the top of the 
  Type A component in the lower left of Figure~\ref{F-MC-crash}. The lower 
Type D component has the same Julia set but with $+a$ and $-a$ interchanged.
  On the right, the Julia set for the center of this type A component, with
  $+a=-a$.  All angles have denominator $d(3)=26$. Note that 
  the upper root point is the landing point of the $(43,44, 19)/78$ 
  parameter rays, while the lower root points is the landing point
  for the $(17, 70, 71)/78$ rays. Both of these triples correspond to the
  angles $(17, 18, 19)/26$ for the critical value root point of the Julia
  set on the left, where 
  $(17,\,18,\, 19) \mapsto (25,\,2,\,5) \mapsto(23,\,6,\,15)$.
}

 \end{figure}\msk

\begin{rem}[\bf Dual Critically Periodic Points in Parameter Space] 
\label{R-dual}
Suppose that both critical points of a cubic polynomial are periodic. 
We have a choice as to which one to mark if the map 
represents the center point of a 
hyperbolic component of Type B or D. (For center points of A components,
the two critical points coincide, so there is no choice.) If one choice
of marked point give rise to a point $F\in\cS_p$, then the other choice
gives rise to a \textbf{\textit{dual}} point $F^*\in \cS_q$, where 
$q$ is the period of the other critical point. Note that $F$ and $F^* $
are identical as maps from $\C$ to $\C$. In particular, the Julia set
$J(F)$ is identical to the Julia set $J(F^*)$. The only difference
is the choice of which critical point to mark. In the case of an A
component, we say that the center point is \textbf{\textit{self-dual}}.

If $(a,v)$ are the coordinates of a critically periodic point
in $\cS_p$, then the dual center point in $\cS_q$ has coordinates
$$          (a', v')~ =~ (-a,~ F(-a))~ = ~(-a,~ 4 a^3+ v).$$
Here $p=q$ in the case of a map of Type A, B or D$(p)$. However for a
component of type D$(q)$ in $\cS_p$ with $p\ne q$, the dual point will be
of Type D$(p)$ in the space $\cS_q$. \bigskip

\begin{figure}[htb!]
  \begin{center}
  \begin{overpic}[width=2.7in]{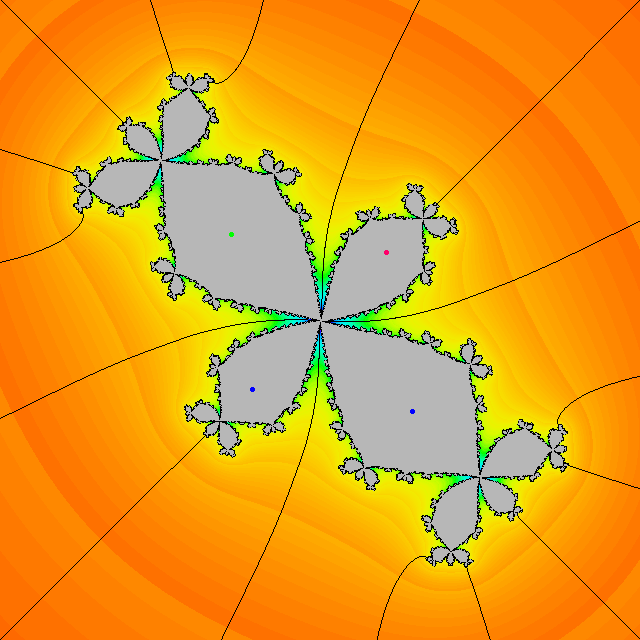}
      \put(200,78){\bf 0}
      \put(200,130){\bf 1/16}
      \put(200,190){\bf 2/16}
      \put(130,200){\bf 3/16}
      \put(78,200){\bf 4/16}
      \put(30,200){\bf 5/16}
      \put(-20,200){\bf 6/16}
      \put(-30,146){\bf 7/16}
      \put(-30,108){\bf 1/2}
      \put(-30,58){\bf 9/16}
      \put(-30,-13){\bf 10/16}
      \put(51,-13){\bf 11/16}
      \put(100,-13){\bf 12/16}
      \put(147,-13){\bf 13/16}
      \put(198,-13){\bf 14/16}
      \put(200,39){\bf 15/16}
      \put(128,71){${\mathbf a}$}
      \put(60,127){${\mathbf{-a}}$}
    \end{overpic}
    \medskip
  %\centerline{\includegraphics[width=3in]{selfdualB.png}}
  \caption[A weakly self-dual Julia set of Type B$(2,2)$]
{\label{F-selfdual}\sf A weakly self-dual Julia set of Type B$(2,2)$
  in $\cS_4$. %(Compare Lemma~\ref{L-Bduality}.)
  Here the two largest components are centered at $+a$ and $-a$. 
  The 180 degree rotation
  of this set interchanges the orbits of $+a$ and $-a$. Note that 
  $~~1/16\mapsto 3/16\mapsto 9/16\mapsto 11/16\mapsto 1/16$.}
\end{center}
\end{figure}

For centers of Type~B$(r,s)$ (where $r,~s>0$ and $r+s=p$), the dual is of
Type~B$(s,r)$. In the special case $r=s=p/2$, the component is 
\textbf{\textit{weakly self-dual}} in the sense that the 180 degree rotation
of the Julia set interchanges the two critical points $a$ and $-a$.
(Compare Figure \ref{F-selfdual}.
The involution $ ~ (a,v) \leftrightarrow (a', v')~$ has an actual fixed point
only for the center of an A component.)

The center of a D component can never be self-dual. For this 
would imply that it has two different root points, which is impossible.
(Compare Remark \ref{R-roots2}.)

The self-duality of an A component may seem uninteresting, but in fact
it has significant consequences. According to Conjecture MC1, each of the
two root points is also the root point of a component of type D.
As we pass from one root point 
to the opposite one along a path through the center, the critical
points $+a$ and $-a$ cross through each other and are interchanged.
As a consequence the D components at one root point are dual to
the D components at the other one. (See for example Figures~\ref{F-MC-crash}
and \ref{F-centerA-airplane}.) 

For other examples of this duality see Remark~\ref{R-dual-ex}, as well as
the many examples in \cite{BM}. %Section~\ref{s-rab}.

Recall from  Definition~\ref{D-cp} that every co-periodic angle $\theta$ 
has a unique \textbf{\textit{twin}} co-periodic angle of the form
$\theta\pm 1/3$. If two components
 $H$ and $H^*$ are dual to each, then the angles of parameter
 rays which land on a root point of $H$ 
 are twins of the angles of rays which land on a root point of 
$H^*$. In particular, if $H$ is self-dual,
then there are parameter rays with twin angles landing on the two root
points of $H$.

This provides an easy way to distinguish between components of Type A and
Type B in any parameter picture in which ray angles are provided: Choose
any ray landing at a root point. There is a twin ray landing on the 
opposite root point of the same component if and only if it has Type~A.

% For example, the angles \{44/78, 70/78\}; \rnote{??}
%\{43/78, 17/78\} and
%  \{71/78, 19/78\} that appear in the center of Figure~\ref{F-2sf} are twins
%  of each other. The dynamical angles 6/26, 15/26 and 23/26 in
%  Figure~\ref{F-centerA-airplane}-left correspond to the parameter rays
%  70/78, 71/78  and 17/78 respectively.

\end{rem}

\setcounter{lem}{0}
\section{Tessellations and Orbit Portraits.}\label{s-tess}
For each integer $q\ge 1$ we will describe a partition of the smoothly
compactified 
parameter curves $\overline\cS_p$ which controls the behavior of
dynamic rays of period $q$ for the associated Julia sets.
(Compare Definition \ref{D-op} and Theorem \ref{T-decomp}.)
\smallskip

   \begin{definition} \label{D-copdec} Let $q$ be any positive integer.
     The \textbf{\textit{period $q$ tessellation}} of the Riemann surface
     $\overline\cS_p$ is constructed as follows. 
 The collection of all parameter rays of co-period $q$ in all escape regions,
together with their landing points, decomposes $\cS_p$ into a finite number
of connected open sets $~\F_k$, which we will call the 
\textbf{\textit{faces}} of the tessellation. (Note that there are only
finitely many escape regions in $\cS_p$,
and only finitely many parameter rays of co-period $q$ in each escape region.
(Compare \cite[Corollary 5.12]{M4}; and see Section~\ref{s-count} for
some specific numbers.) By definition, the \textbf{\textit{edges}}
of the tessellation consist of all parameter rays of 
co-period $q$; while the \textbf{\textit{vertices}} consist of:

\begin{enumerate}
\item[(1)] \textbf{\textit{ parabolic vertices:\/}}
  the landing points of these rays;  and 

\item[(2)] \textbf{\textit{ ideal vertices}} (or puncture points):
the points of $\overline\cS_p\ssm\cS_p$.
\end{enumerate}

\noindent Note that every edge joins a parabolic vertex to an ideal vertex.
We will use the notation  
$$\Tes_q(\overline\cS_p)~=~\big(\Tes^{(0)},~\Tes^{(1)},~\Tes^{(2)}\big)$$ for
this tessellation,
where $\Tes^{(0)}$ is the set of vertices, $\Tes^{(1)}$ is the set of edges,
and $\Tes^{(2)}$ is the set of faces.
\end{definition}
\medskip

\begin{rem}\label{R-simplycon} 
  Here we are stretching the usual concept of tessellation,
which assumes that all faces are simply-connected. Our tessellations
may well have faces which are not simply-connected. (Compare Figure~\ref{f2}.)
In fact we see  in  \cite{BM} %Corollary~\ref{C-notcon} (Section~\ref{s-rab}) 
that the tessellation $\Tes_q(\ocS_p)$
has non-simply-connected faces, and also has non-connected one-skeleton,
whenever $1<p\ne q$. On the other hand, it seems quite possible that,
in all cases with $p=q$, the one-skeleton
is connected, and every face is simply-connected.
However we have no idea how to prove such a statement.
\end{rem}

 These tessellations are extremely complicated for 
 all but the smallest $p$ and $q$. (Compare Table~\ref{t3} in
  Section~\ref{s-count}.)
\medskip

\begin{rem}\label{r-mpt} Each parameter ray in an escape region 
$\cE_j$ is a leaf in the smooth foliation of  $\cE_j$ by its parameter rays. 
Since parameter rays can never intersect each other, we could equally well
construct such a tessellation using any finite collection of co-periods 
$q$. Such \textbf{\textit{mixed period tessellations}} provide more
information by correlating behavior for several different periods. However
we discuss only the simpler single period tessellations, which together
provide all of the essential information, since a
 mixed period tessellation is completely determined by the corresponding
 single period tessellations. 
\end{rem}
\medskip

Note that rays from two different escape regions $\cE_i$ and $\cE_j$
may well land at a common boundary\footnote{Such a common boundary
$\partial\cE_i\cap\partial\cE_j$ (if non-empty) is always a compact subset of
the connectedness locus. We suspect that this set is always totally
disconnected. (See Figure~\ref{F-t2s3} or \ref{F-air} for examples.)} 
point in $\partial\cE_i\cap\partial\cE_j$. Similarly, a face  
$\F_k\in\Tes^{(2)}_q(\overline\cS_p)$ may well 
  intersect several different escape regions, and will always intersect 
the connectedness locus. (Compare Figure~\ref{F-S2rays}. )
\bigskip

\begin{rem}[\bf Parametrization of $\cS_p$]\label{R-coords}
To illustrate these tessellations, we must choose some parametrization 
of $\cS_p$,  of large open subsets of $\cS_p$. For $p\ne 3$ will  use
the local  ``\textbf{\textit{canonical coordinate}}'' $\tb$ around any point
of the curve $\cS_p$. (See 
\cite[Section 2]{BKM} or \cite[Appendix~D]{BM}.)
 On the other hand, in the special case of $\cS_3$,
it is more useful to use a \textbf{\textit{torus coordinate}}, that is a
suitably chosen affine coordinate on the universal covering surface of
$\ocS_3$, which is conformally isomorphic to $\C$.  (See \cite[Appendix~C]{BM}.)
In both cases, this parametrization is conformal, so that it preserves
the shapes of small regions in $\cS_p$.\end{rem}
\medskip
 
\begin{figure}[htb!]
\begin{center}
\begin{overpic}[width=2.25in, tics=10]{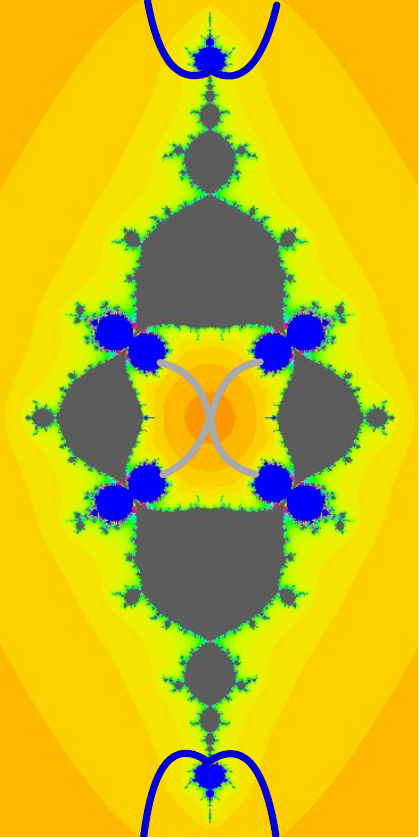}

\put(36,315){\bf 5/6}
\put(110,315){\bf 2/3}
\put(1,158){\bf 0}
\put(157,158){\bf 1/2}
\put(35,10){\bf 1/6}
\put(110,10){\bf 1/3}
\put(55,168){\bf 1/6}
\put(87,168){\bf 1/3}
\put(55,148){\bf 5/6}
\put(87,148){\bf 2/3}
\end{overpic}\qquad\qquad\begin{overpic}[width=1.9in, tics=10]{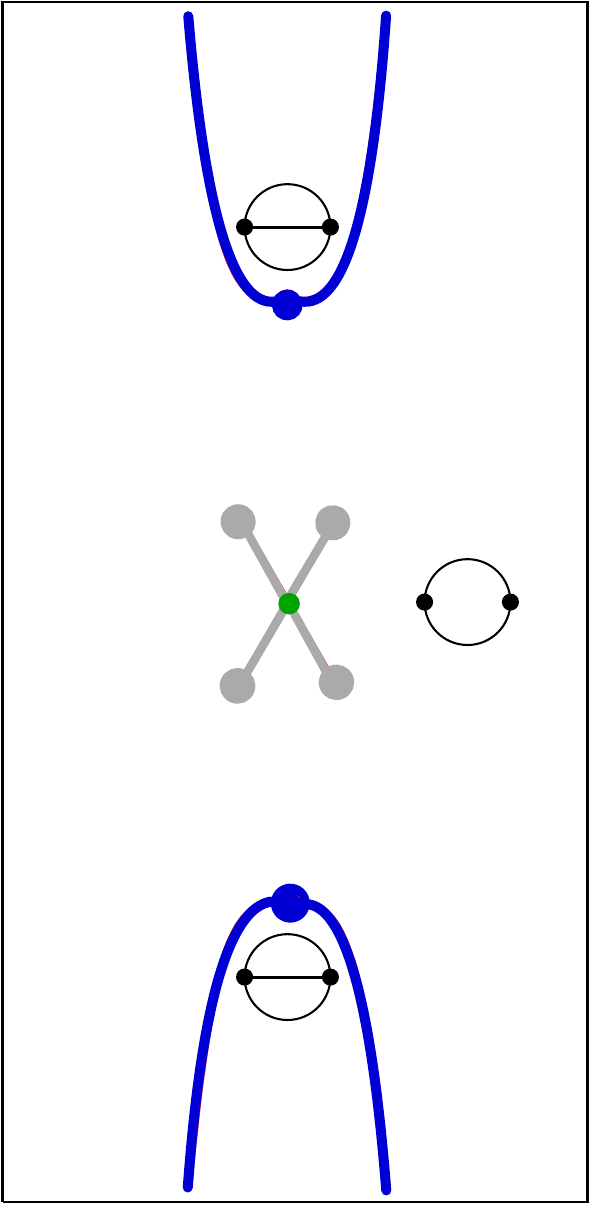}
\put(58, 265){$\F_1$}  
  \put(21,265){\bf 5/6}
  \put(94,265){\bf 2/3}
 \put(58,10){$\F_3$}
\put(21,10){\bf 1/6}
\put(94,10){\bf 1/3}
\put(15,180){$\F_2$}
\put(32,150){\bf 1/6}
\put(82,150){\bf 1/3}
\put(32,120){\bf 5/6}
\put(82,120){\bf 2/3}
\end{overpic}
\caption[$\Tes_1(\overline\cS_2)$]{\textsf{On the left, the curve $\cS_2$,
    using the canonical  
 coordinate  $\tb$. The inner escape region with kneading
invariant $(1,0)$ forms  
 a neighborhood of the ideal vertex at the center of the figure, while
 the outer ``basilica'' escape  region with kneading invariant $(0,0)$
is a neighborhood of the ideal vertex at infinity.
 Both regions have multiplicity $\mu=1$. The four outer rays of 
 co-period one divide the plane into three regions; but the four inner rays
 do not further separate the plane. (These rays are colored blue and gray
 respectively. Compare Definition \ref{D-3E}.) 
 On the right,  a cartoon showing the resulting tessellation
 $\Tes_1(\overline\cS_2)$.  
 The orbit portrait associated with each face is indicated by a circle,
 with the points of angle zero and $1/2$ joined  whenever the corresponding
 rays have a common landing point. The six large 
 dots represent parabolic vertices, while the green dot in the center
 represents one of the two ideal points. (The other ideal point is at
 infinity.) Note that the middle  
 face intersects both of the escape  regions. 
 \label{f2}
}}
\end{center}
\end{figure}

\begin{ex} The space $\cS_1$ is isomorphic to $\C$, parametrized by
the canonical coordinate $\tb=a$, where $a$ is the marked critical point. The
space $\cS_2$ is isomorphic to $\C\ssm\{0\}$, parametrized by
$$\tb~=~ \frac{1}{3(v-a)}~,$$
where $v$ is the marked critical value, with $a\leftrightarrow v$.
See Figures~\ref{f2} and~\ref{F-S2rays}, where the 
inner escape region $\cE_{\sf in}$ parametrizes maps with kneading
invariant $(1,0)$, so that $a$ and $F(a)=v$ lie in different
lobes  of the figure eight curve of Figure~\ref{Faa-1}. Every nontrivial
component of the filled 
Julia set is a topological copy of the closed unit disk
(Compare Definition \ref{D-kn}.) 
 The outer escape region
 $\cE_{\sf out}$ parametrizes maps with kneading invariant
$(0,0)$, so that the orbit of $a$ lies entirely in  
one lobe of the figure eight. In this case every nontrivial
component of the filled  Julia set is a copy of the
basilica.\footnote{Compare Remark \ref{R-Mand}, as well as 
  \cite[Table~6.8 and Figure~9]{BKM}.}\medskip

 For $p>2$ the canonical coordinate
can only be defined as a local coordinate, since it ramifies at most ideal
points. (See  \cite[Appendix~C]{BM} for further discussion.)
\end{ex} \medskip

\begin{rem}[{\bf Symmetries}]\label{R-sym}
  Let $\G$ be the commutative group consisting of the four transformations
  $$\xymatrix{z~ \ar[r] & ~\pm z} \qquad {\rm and} \qquad
  \xymatrix{z~ \ar[r] & ~\pm \overline z}$$
  of the plane of complex numbers. Thus $\G$ consists of three involutions,
  which we can describe as:

  \begin{itemize}

\item[$\bullet$] \textbf{\textit{up-down reflection}}
  $\xymatrix{z~ \ar@{<->}[r] & ~\overline z}$, 

\item[$\bullet$] \textbf{\textit{left-right reflection}}
  $\xymatrix{z~ \ar@{<->}[r] &  ~-\overline z}$, and

\item[$\bullet$] \textbf{\textit{180 degree
      rotation}} $\xymatrix{z~ \ar@{<->}[r] & ~-z}$;
  \end{itemize}

\noindent    together with their product which is
  the identity transformation. Note that two of these four transformations
  are holomorphic and two are antiholomorphic.

  This group $\G$ acts naturally on each $\cS_p$. If we think of $\cS_p$
  as an affine variety with coordinates $(a,v)$, then the action of a group
  element $g$ is given by $g(a,v)=\big(g(a),\,g(v)\big)$. On the other hand
  if we think of the elements of $\cS_p$ as polynomial maps $F:\C\to\C$, then
  it is given\footnote{If $F(z)=z^3+\alpha z+\beta$, then conjugating by
    $z\mapsto\overline z$ sends $F$ to 
    $z\mapsto z^3+\overline\alpha z+\overline\beta$; while 
 conjugation by $z\mapsto -z$ sends it to $z\mapsto z^3+\alpha z-\beta$.}
  by $F\mapsto g\circ F\circ g$.

  For the associated Julia sets the appropriate formula is
  $$J(g\circ F\circ g)~=~g\big(J(F)\big)~.$$
  (Compare  Figure \ref{F-4para}.)   For either the angles of dynamic
  rays in a Julia set or the co-critical angles of a parameter ray in
  $\cS_p$, the corresponding action is given by

  \begin{quote}
    $\theta \mapsto -\theta$ \qquad \quad  for up-down reflection,

   $\theta\mapsto 1/2-\theta$ \quad  for left-right reflection, and

    $\theta\mapsto 1/2+\theta$ \quad for 180 degree rotation;
  \end{quote}
  \noindent where $\theta\in\R/\Z$.\medskip

  {\bf Note.} All of the escape regions we are familiar  with have at
  least one non-trivial symmetry, so that there is at most one other
  region in $\ocS_p$ which is a holomorphic or antiholomorphic copy.
  However it seems quite possible that for higher values of $p$ there exist
  escape regions with no symmetries other than the identity map,
  so that there could be four distinct regions which are copies of each other.
\end{rem}\medskip

The parameter  $\tb$ for $\cS_1$ or $\cS_2$, or for the universal 
covering space of $\ocS_3$ will always be chosen so that the action of $\G$ on
this space  corresponds to its action on the complex number  $\tb$.
For $p>3$, a similar  statement is true for the local parametrizations.
These automorphisms always preserve the tessellations.

\begin{figure}[htb!]
\begin{center}
  \begin{minipage}{2in}
 \begin{overpic}[width=2in]{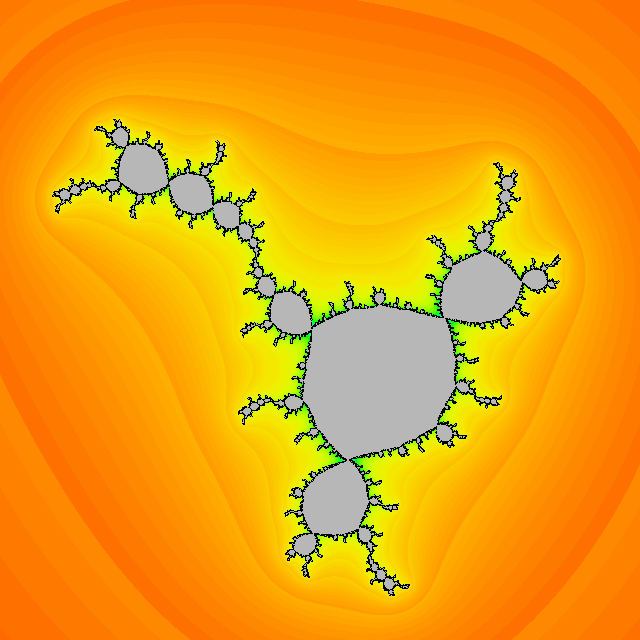}
\put(80,55){$a$}
\put(60,84){$-a$}
\end{overpic}
\end{minipage}\quad
\begin{minipage}{2in}
\begin{overpic}[width=2in]{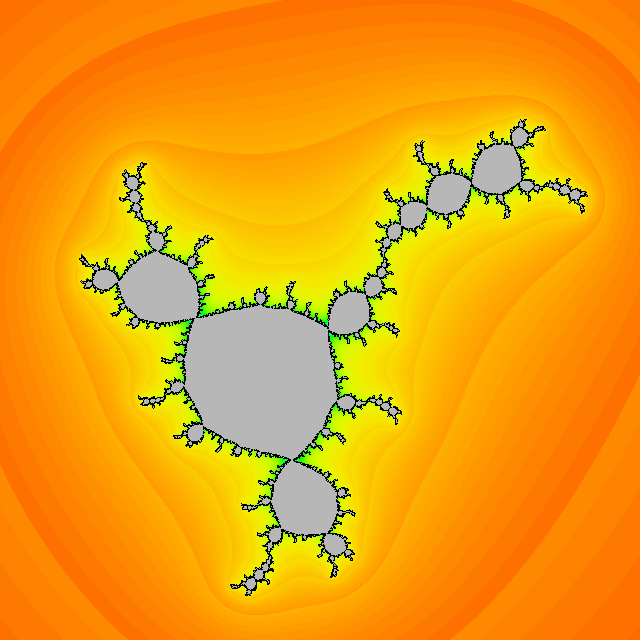}
  \put(55,55){$a$}
  \put(65,84){$-a$}
\end{overpic}
\end{minipage}
\smallskip

\begin{minipage}{2in}
\begin{overpic}[width=2in]{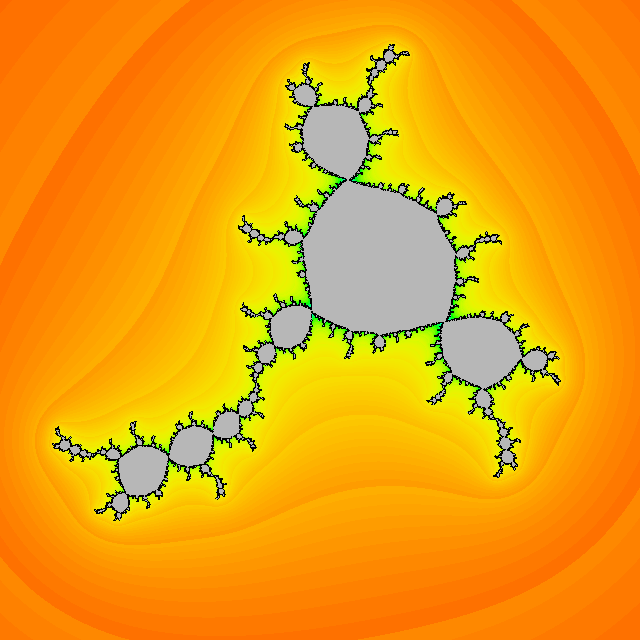}
\put(80,82){$a$}
\put(37,55){$-a$}
\end{overpic}
\end{minipage}\quad
\begin{minipage}{2in}
\begin{overpic}[width=2in]{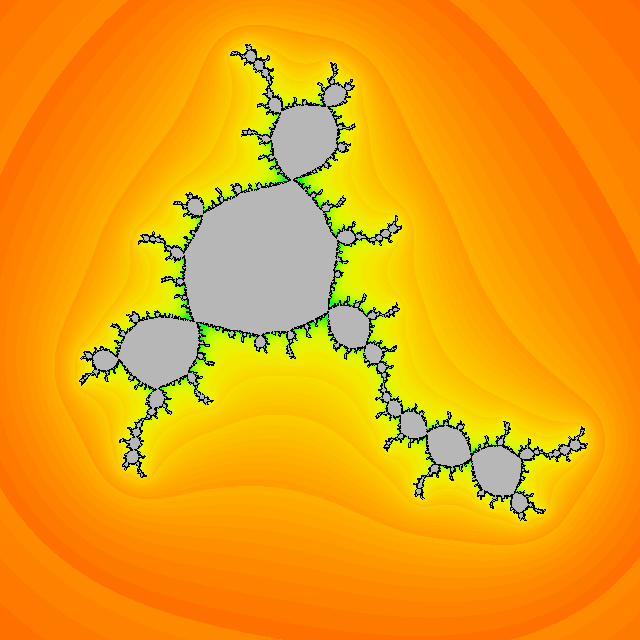}
\put(55,83){$a$}
\put(65,52){$-a$}
\end{overpic}
\end{minipage}
\end{center}
\caption[Four parabolic maps in $\cS_2$]{\textsf{For most maps $F\in\cS_p$,
    the four images $g\circ F\circ g$
  under conjugation by elements of $\G$ are distinct, 
  and the four corresponding Julia sets
$g\big(J(F)\big)$ are also distinct. This figure shows a typical example,
consisting
of Julia sets for  four Misiurewicz maps in $\cS_2$ which are on the
boundary of the basilica escape region. (See Figure~\ref{F-S2juls} in
Section \ref{s-rays} for more detail of the lower left figure, and see Figure
\ref{f2} or \ref{F-S2rays} for the full parameter space.)
 The group $\G$ permutes these four maps, and correspondingly
 permutes their Julia sets. Thus this entire figure is invariant under the
 action of $\G$. 
 \label{F-4para}}}
\end{figure}

\begin{figure}[htb!]
  \begin{center}
   \begin{overpic}[width=2in, tics=10]{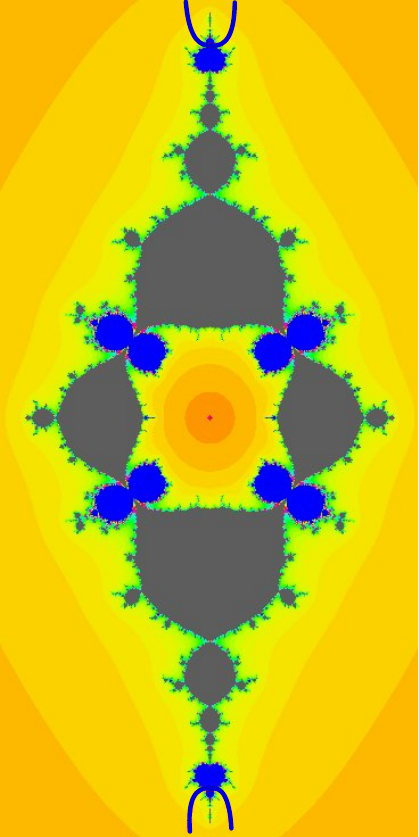}
\put(52,5){\bf 5}
\put(86,5){\bf 7}
\put(86,278){\bf 17}
\put(45,278){\bf 19}
\end{overpic}\quad\begin{overpic}[width=3.6in, tics=10]{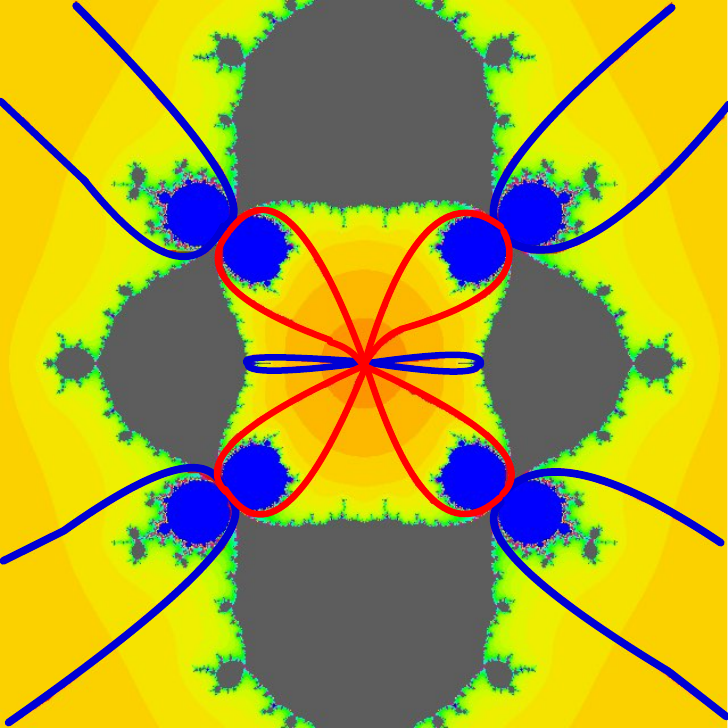}
\put(4,70){\bf 1}
\put(25,5){\bf 2}
\put(220,10){\bf 10}
\put(235,82){\bf 11}
\put(230,178){\bf 13}
\put(200,237){\bf 14}
\put(42,247){\bf 22}
\put(3,192){\bf 23}
\put(92,134){\bf 1}
\put(99,148){\bf 2}
\put(110,172){\bf 5}
\put(138,168){\bf 7}
\put(142,147){\bf 10}
\put(155,135){\bf 11}
\put(157,117){\bf 13}
\put(144,106){\bf 14}
\put(130,87){\bf 17}
\put(108,78){\bf 19}
\put(105,108){\bf 22}
\put(90,117){\bf 23}
\end{overpic}
\end{center}
\caption[${\rm Tes}_2(\ocS_2)$]{\textsf{Showing the $24$ parameter rays of
    co-period $2$ for $\cS_2$
    (twelve inside and twelve outside).
     Four of these rays are shown on the left, and the remaining twenty
     on the right. These rays divide the plane into sixteen period $2$
  faces.  Compare Figure \ref{f-po}. For the distinction between red
  and blue rays, see Definition \ref{D-3E}.  All angles with denominator
  $3d(2)=24$.
     \label{F-S2rays}\vspace{-.3cm}}}
\end{figure}

\begin{figure}[htb!]
  \begin{center}
    \begin{overpic}[width=3.3in, tics=10]{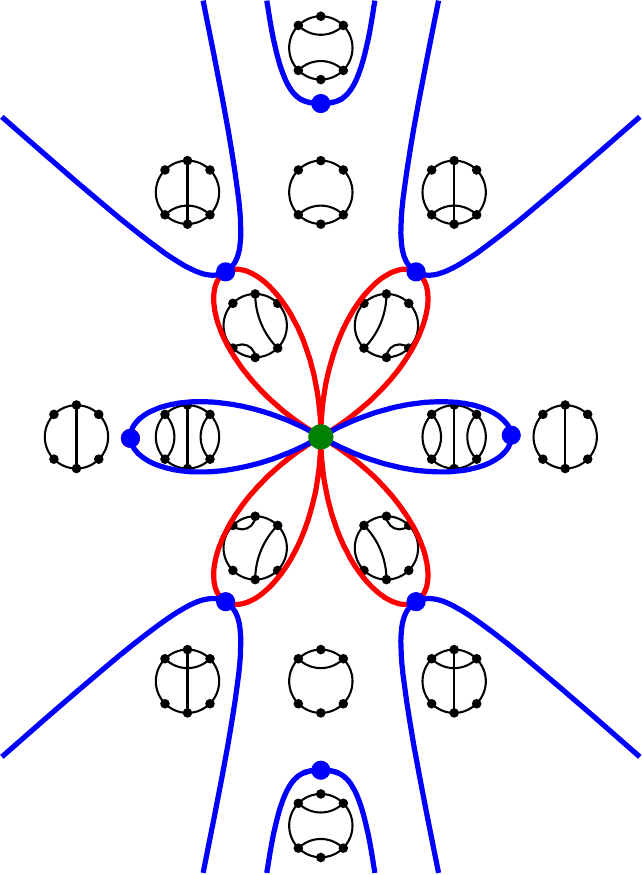}
\put(94,-10){\bf 5}
\put(135,-10){\bf 7}
\put(132,330){\bf 17}
\put(87,330){\bf 19}
\put(7,60){\bf 1}
\put(70,30){\bf 2}
\put(158,30){\bf 10}
\put(220,60){\bf 11}
\put(205,243){\bf 13}
\put(156,280){\bf 14}
\put(65,280){\bf 22}
\put(21,241){\bf 23}
\put(52,176){\bf 1}
\put(70,200){\bf 2}
\put(100,220){\bf 5}
\put(128,215){\bf 7}
\put(159,200){\bf 10}
\put(160,178){\bf 11}
\put(160,138){\bf 13}
\put(158,118){\bf 14}
\put(125,95){\bf 17}
\put(106,105){\bf 19}
\put(65,118){\bf 22}
\put(55,138){\bf 23}
\end{overpic}
\end{center}
\caption[Cartoon of  $\Tes_2(\ocS_2)$]{\label{f-po}\textsf{Cartoon showing the
    tessellation
%\index{Figure~\ref{f-po}, cartoon of  $\Tes_2(\ocS_2)$}
    $\Tes_2(\overline\cS_2)$, together with the period $2$ orbit portrait
    associated with each of the sixteen faces.  Each of the co-periodic
    angles shown in this cartoon has denominator $24= 3\,d(2)$. 
In this particular case, %    with Figure \ref{F-t2s1},
    each face has a well defined critically periodic center point, and two faces
    have the same orbit portrait only if their center points have the same
  Julia set. (Compare Remark \ref{R-dual-ex}.)}}
\end{figure}

\begin{figure}[htb!]
  \centerline{\includegraphics[width =2.7in]{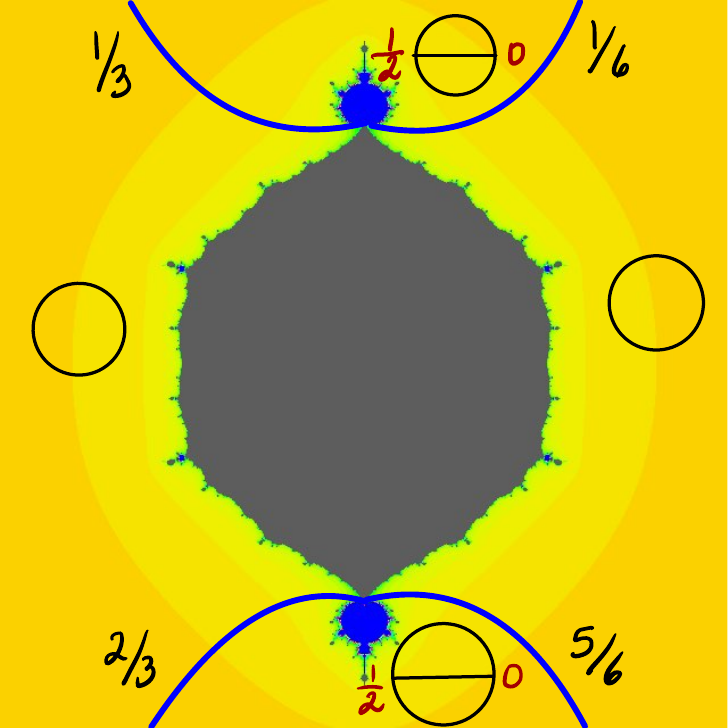}\quad\includegraphics[width =2.7in]{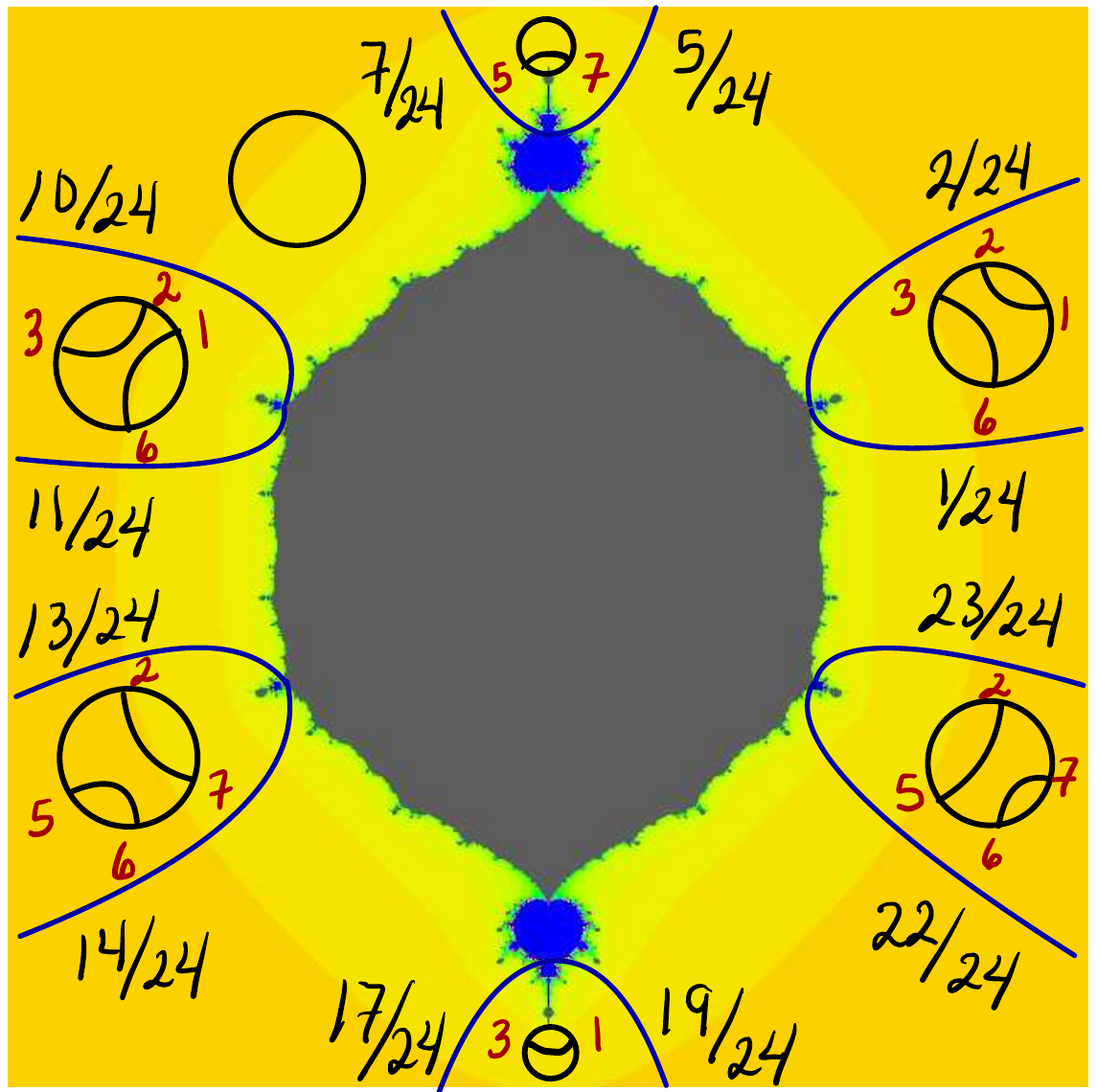}}
  \caption[Pictures of $\Tes_1(\ocS_1)$ and $\Tes_2(\ocS_1)$ ]{\sf On the left,
    a  picture of $\Tes_1(\ocS_1)$. Notice that
   the dynamic rays $0$ and $1/2$ land together for maps in 
   the two primary wakes.
   On the right, a picture of $\Tes_2(\ocS_1)$
   (with common denominator $d(2)=8$ for orbit portraits). In both cases, 
   the largest face containing the central region
   has trivial orbit portrait. 
   The orbit portraits on the right are identical 
 with those in the inner region of Figure \ref{f-po}, illustrating the
 duality between vertices of $\Tes_2(\ocS_1)$ and  
  $\Tes_1(\ocS_2)$.
\label{F-t2s1}}
\medskip

\centerline{\includegraphics[height=2.2in]{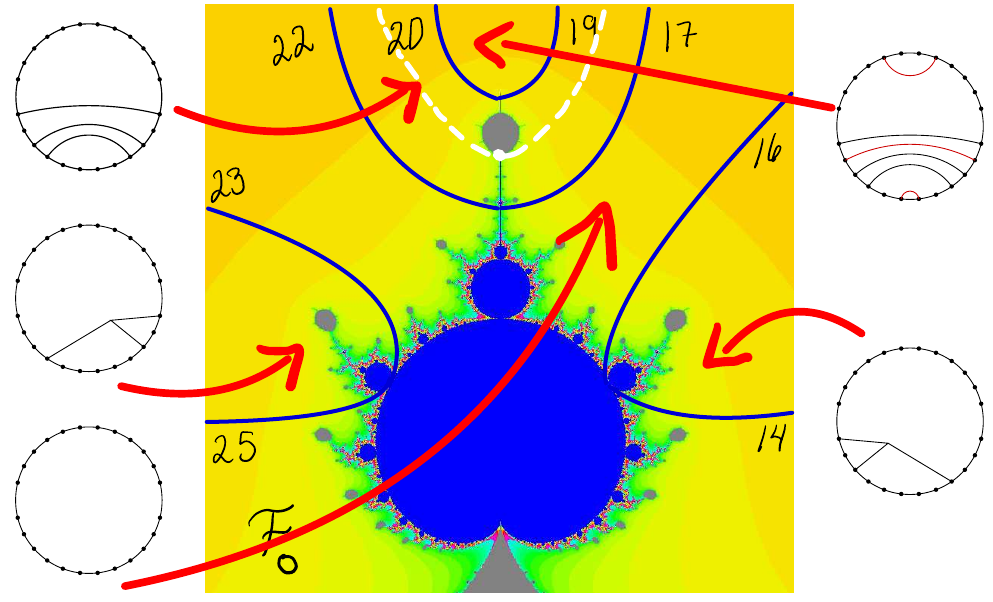}\qquad\includegraphics[width=1.8in]{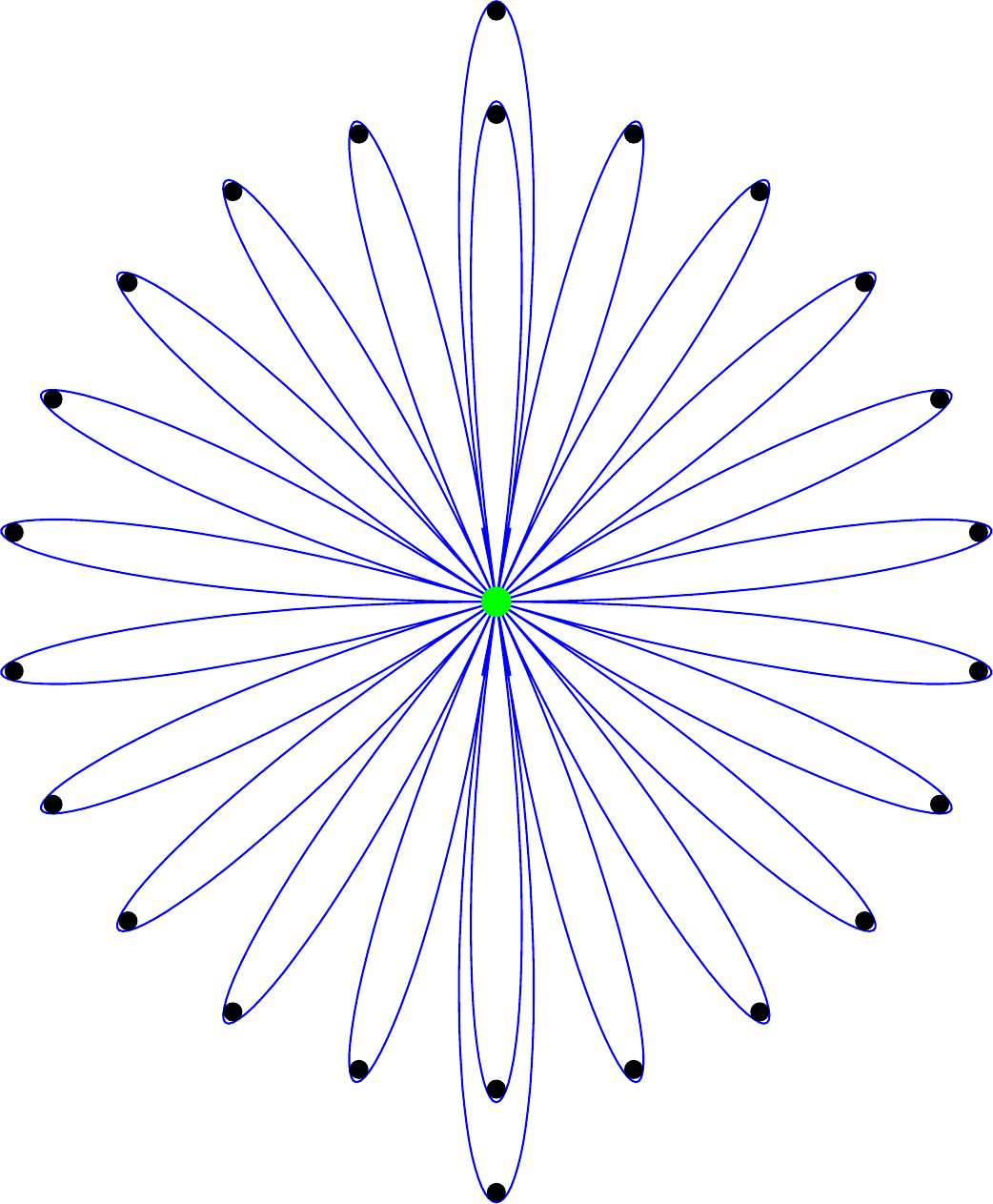}}
\caption[Pictures of $\Tes_3(\ocS_1)$]{\sf Left: The top part of the 
  tessellation $\Tes_3(\ocS_1)$
  showing five faces, together with their period $3$ orbit portraits. (The
common denominator is $3\,d(3)=78$ for parameter rays, and $d(3)=26$ 
 for dynamic rays.)
The rest of the connectedness locus in $\cS_1$ can be seen above in
Figure~\ref{F-t2s1}.
The Type D components at the base of the two top faces
are very small, and would be visible only under
substantial magnification. For the dotted white rays, see
Figure~\ref{F-cap-ara} and Remark~\ref{R-not-wake}.) \hfill\break
Right: An inverted picture of $\Tes_3(\ocS_1)$ 
with the ideal point at the center, showing only the $48$ parameter rays.
Conjecturally the same pattern of co-period $3$ rays landing together would
occur  in {\bf any} zero-kneading region. Compare Figures
\ref{F-air} and \ref{F-3rab} in Sections \ref{s-tess} and \ref{s-near-para}. 
 There is an analogous conjecture for 
any co-period $q$. Compare the outer part of Figure \ref{F-S2rays}
with Figure \ref{F-t2s1}-right for the case $q=2$. 
(See \cite{BM} for a more detailed discussion.)\label{F-t3s1}}
\end{figure}

\begin{definition}[{\bf Orbit Portraits}]\label{D-op} 
  Let $F\in \cS_p$ be a map such that every dynamic ray of period $q$
  lands at a (necessarily periodic) point of the Julia set. By definition 
the \textbf{\textit{period $q$ orbit portrait}} $\cO_q(F)$ is
the following equivalence relation between angles of period $q$ under tripling:

\begin{quote}\sf Two such angles $\phi$ and $\psi$ are equivalent, written 
$\phi\simeq\psi$,  if and only if the corresponding dynamic rays have a common
landing point in $J(F)$.\end{quote}\end{definition}

\noindent(Compare~\cite{M2}.) This definition makes sense for every
  $F$ in the  connectedness locus $\cC(\cS_p)$.
  It also makes sense and is independent of $F$ throughout 
  each  face $\F_k$ of the tessellation $\Tes_q(\overline\cS_p)$.  
  However, in the special case where $F$ belongs to a parameter ray of
  co-period $q$, the orbit portrait is not defined, since at least
  one period  $q$ dynamic ray must crash into a pre-periodic point and
  hence not land anywhere. In most cases, the orbit portrait
  will jump discontinuously as we cross such a ray. (Compare
  Corollaries~\ref{C-jump} and Theorem~\ref{T-OPjump}.) %and Definition~\ref{D-3E}.)

More generally we can consider \textbf{\textit{mixed period orbit
    portraits}}. Let $0<q_1<\cdots<q_k$ be a finite list of periods. %\footnote{
%Of course one could also define the \textbf{\textit{total orbit portrait}},
%allowing all possible periods; but we have no idea how to compute such things.
%Is the total portrait for a critically finite polynomial map always finite?}
For a generic $F\in\cS_p$ we can consider the associated orbit portrait
$\cO_{q_1,\cdots,q_k}(F)$ in which we allow dynamic rays for all periods
which are in the list. (Of course, dynamic rays of different period
can never have a common landing point; hence their angles
can never be equivalent.)
\medskip

Such an equivalence relation  between angles can be conveniently described
by providing an unordered list of all of the equivalence classes 
which contain more than one element. As in Remark \ref{R-NC} we will 
 write angles which have period
$q$ under multiplication by $3$ as fractions of
the form $n/d$ with common denominator $d=3^{q}-1$. 
As an example, if $q=2$ so that \hbox{$d(2)=8$} then  the angles with 
period precisely $q$ can be listed as
$1/8,\;2/8,\;3/8,\;5/8,\;6/8,\;7/8$. If the $1/8$ and $2/8$ rays land at
one point and the $3/8$ and $6/8$ rays land at a different point, then
we will write
$$\cO_2(F)~=~\{1/8\simeq 2/8,~3/8\simeq 6/8\}~. $$
In practice we will often abbreviate the right hand side 
by writing $\{1\simeq 2,~3\simeq 6\}/8$. 
If all four of these rays land at a single point,
we will write $\{1\simeq 2\simeq 3\simeq 6 \}/8$.

Such an orbit portrait can be described graphically by identifying
each fraction $n/d$ which has period $q$ under tripling
with the point $e^{2\pi i\, n/d}$ on
  the unit circle, and then joining two such points by a path within
  the unit disk whenever the corresponding dynamic rays land at a common
  point in the Julia set.  This can always be done so that two such
  paths intersect only when all of the associated dynamic rays land at
a common point. As examples, the two orbit portraits described in
the preceding paragraph corresponds to the two circle diagrams to 
the ``south-east'' in \autoref{f-po}. \smallskip

  Another way of expressing this condition is the following. By definition,
  % \begin{definition}\label{D-link} T
two or more disjoint compact subsets $A_j$ of the unit circle  are
\textbf{\textit{unlinked}} if we can choose disjoint, compact connected
sets $A'_j$ in the unit disk so that $A_j\subset A'_j$.
 In fact, if the $A_j$ are unlinked, then we can always choose
 $A'_j$ to be the smallest convex set which contains $A_j$.

 \begin{lem}\label{L-unlink}
A necessary property of any orbit portrait (whether single period
or mixed period) is that its various equivalence classes must be unlinked.
\end{lem}

The proof is straightforward.\qed\msk

Such a diagram in the unit disk
can be thought of as the period $q$ part of the
  Thurston lamination for $F$. See Figures~\ref{f2} and \ref{f-po} 
which show the orbit portraits associated with the various faces
of the period one and two tessellations for the curve  $\cS_2$;  and
Figure~\ref{F-t2s1}  for the various faces of the corresponding tessellations for
the curve  $\cS_1$.
%\end{definition}
\msk

{\bf Note.} Although we have described an orbit portrait as a geometric object,
it can also be thought of as a finite dynamical system, consisting of the
tripling map from a suitable set of angles of fixed co-period $q$ to itself.
Note that two orbit portraits which look very similar topologically,
may have quite different dynamics. As an example, the orbit portrait
$\{1\simeq 2\simeq 3\simeq 6\}/8$ looks topologically
very much like $\{1\simeq 3\simeq 9\simeq 27\}/80$. But the first consists
of two orbits of period two while the second represents a single
orbit of period four. \msk

\begin{rem} \label{R-edge-conj}
 In  Figures~\ref{f2}, \ref{f-po} and \ref{F-t2s1} note that:~
  {\sf Any two faces with an
edge in common have different orbit portrait.} (A single boundary point
in common is not enough.) See the Edge Monotonicity Conjecture \ref{CJ-main} 
which implies that the same statement is true in every tessellation
$\Tes_q(\ocS_p)$.
\end{rem}
\msk

\begin{figure}[htb!]
  \centerline{\includegraphics[width=3.5in]{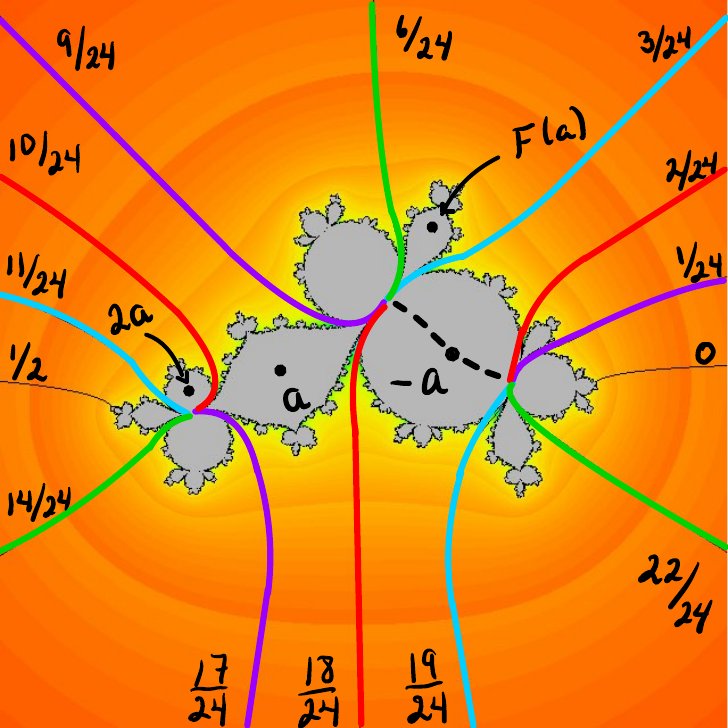}}
  \caption[Julia set of parabolic map in $\cS_2$, with kneading walls.]{\label{f-s2par}\textsf{Julia set for a parabolic map
    $F\in\cS_2$, with rays of angle $10/24$, $11/24$, $14/24$ and
    $17/24$ landing at the root point of the  co-critical component $U(2a)$.
    These are the same as the angles of the parameter rays landing 
    at the corresponding parabolic map. (See the the bottom right quadrant
    of Figure~\ref{F-S2rays}.) Here the four ray colors correspond to the
    associated \textbf{\textit{triads}}\break
    $(\theta,~\theta+1/3,~\theta+2/3)$ of equally spaced angles.
    The associated periodic rays of angle $3/4$, $1/8$, $1/4$, and $3/8$ 
    land at the root point of the $U(-a)$ component. (In this particular case,
    this is also the root point of the $U(a)$ and $U\big(F(a)\big)$
    components.) The remaining ray in each triad lands at the diagonally
    opposite boundary point of $U(-a)$.  Here the walls corresponding to
    the pairs $(1/24, 9/24)$ and  $(6/24, 22/24)$ separate $a$ from $F(a)$,
    hence the corresponding parameter rays of angle $17/24$ and $14/24$
lie in the escape region with
 kneading invariant $(1,0)$. On the other hand, the $(2/24, 18/24)$
    and $(3/24, 19/24)$ do not separate, hence the $10/24$ and $11/24$
parameter rays lie in the $(0,0)$ region, as shown in Figure~\ref{F-S2rays}.
There is an analogous figure for any parabolic point $\p\in\cS_p$, with
a triad for each parameter ray landing at $\p$. (Compare Remark~\ref{R-J2K}.)
%    $~(9/24, 1/24)$, $~(18/24,\, 2/24)$, $~(19/24,\, 3/24)$ and
 %   $~(22/24, 6/24)$ all of them with positive orientation.
}}\vspace{-.3cm}
\end{figure}

\begin{figure}[htb!]
\begin{center}
\begin{overpic}[width=3.5in, tics=10]{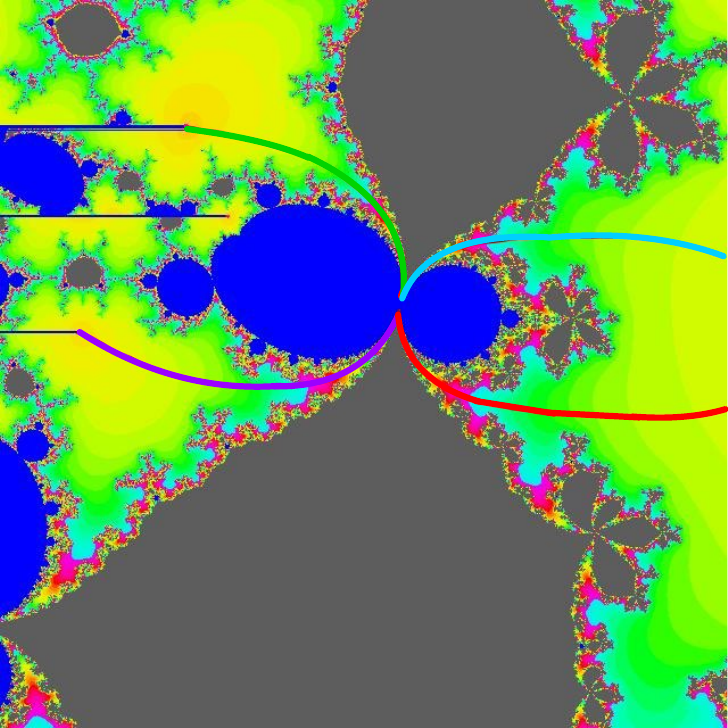}
\put(205,173){\bf 335/726}
\put(225,150){\bf 3/5-}
\put(218,135){\textbf{\textit{rabbit}}}
\put(218,120){\textbf{\textit{region}}}
\put(205,97){\bf 334/726}
\put(32,110){\bf 575/726}
\put(25,94){\bf 11110}
\put(67,208){\bf 344/726}
\put(71,237){\bf 01000}
\put(13,168){\bf 11010}
\linethickness{1.5pt}
\put(-20,146){\color{black}\vector(7,2){95}}
\put(-38,150){\bf 01010}
\end{overpic}
\caption[Part of the parameter curve $\cS_5$]{\textsf{Part of the parameter
    curve $\cS_5$, showing a parabolic point $\p$
    which is on the boundary of three different escape regions,
    and showing the four rays of co-period $5$ which land on it.
   These rays have been given four different colors, matching the corresponding
    colors in the Julia set picture (Figure \ref{F-4wall}).
    The two right hand rays are primary and lie in the $3/5$ rabbit escape
    region; while the two on the left are secondary, and lie in two
    different escape regions. 
    In fact there are four different  escape regions on the left,
 labeled by their kneading invariants. 
    (The three horizontal slits which terminate at ideal vertices
 are needed since the canonical coordinate is ramified at these 
    points.)
 The four rays are labeled by their associated parameter angles. 
Note that each of the three escape regions which have $\p$
as boundary point contains at least one of these rays. There are 
actually hundreds of co-period $5$ rays in each escape region
(see Table \ref{t1a} in Section \ref{s-count}), so the full
picture of the tessellation within this square would be very complicated,
even though each hyperbolic component is contained in a single face.
For a more detailed picture of this region, see \cite{BM}. 
\label{F-s5rab-par}}}
\end{center}
\end{figure}

\medskip 
\begin{figure} [htb!]
  \centerline{\includegraphics[width=5in]{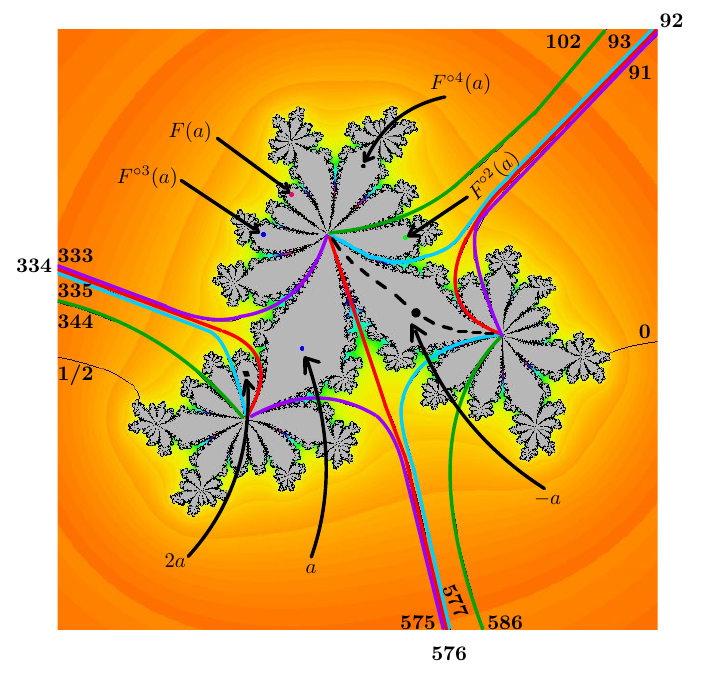}}
\vspace{-.7cm}
\caption[Julia set for Figure \ref{F-s5rab-par}]{\label{F-4wall} \small \sf
  Julia set for the center of the right
  hand Type D component in Figure \ref{F-s5rab-par}. (The corresponding picture
for its parabolic root point 
$\p$ would be homeomorphic; but its components would be fatter and
harder to separate.) Let $\theta(1),\,\ldots,\,\theta(4)$ be the angles for
the four parameter rays landing at $\p$. Then the corresponding dynamic
rays of angle $\theta(j)$
land, in the same cyclic order, at the root point of the $U(2a)$ component
to the lower left of the present figure. Similarly the corresponding periodic
rays of angle $\theta_5(j)$ land at the root point of the $U(-a)$
component towards the top of the figure; while the associated co-periodic
rays of angle $\widehat{\theta}_5(j)$ land at the opposite
boundary point of $U(-a)$, to the lower right. The last eight rays, landing
on the boundary of $U(-a)$, give rise to four different walls across this
dynamic plane. The kneading invariants of the corresponding parameter rays
can easily be computed, using Lemma \ref{L-J2K} together with these walls.
As one example,  for the red wall, consisting of the
$92$ and $576$ rays together with the dotted line,
 the entire orbit of $a$ is contained in one complementary
 component; so the kneading invariant of the red $334/726$ parameter ray
 is $00000$, as shown in Figure~\ref{F-s5rab-par}.
 On the other hand the purple wall, consisting of the 
 $91$ and $333$ rays connected by the dotted line, 
separates the marked point $a$ from the rest of its orbit.
Therefore the purple $575/726$ parameter ray has kneading invariant $11110$.
There are actually ten different rays landing at the co-critical
root point.  For a discussion of how to tell
which dynamic rays actually correspond to parameter rays, see
\cite{BM}. The primary parameter rays  correspond to the  co-periodic rays 
at both sides of the hyperbolic component that contains $2a$ and that
land at the root point of this component. The secondary parameter rays
are the next closer ones to this component landing at 
the root point of  the hyperbolic component of $2a$. Note that the dual
map has an identical Julia set, but with the roles of $+a$ and $-a$
interchanged. (See Remark \ref{R-dual}.)}
\end{figure}

\begin{definition}[{\bf Portrait Size}]\label{D-size} 
  By the  \textbf{\textit{size}} $\s\ge 0$ of an orbit portrait
  we will mean the number of unordered pairs of distinct angles which are
  equivalent to each other. This is a convenient measure of the
  complexity of the portrait.
  Note that a collection of $n$ mutually equivalent angles makes a
  contribution of $~n(n-1)/2~$ to the size. Thus a portrait consisting 
  of four rays meeting at a point has size six.  As examples,
  the various portraits in \autoref{f-po} have sizes $1$, $2$, $3$, and $6$.
\end{definition}
\medskip

\begin{definition}[{\bf Portrait Compatibility and Amalgamations}]\label{D-sum}
By the \textbf{\textit{amalgamation}} of two equivalence relations on the
same set we will mean the smallest equivalence relation which contains both
of them. For example the amalgamation of the orbit portraits
$\{1/8\simeq 3/8\}$ and $\{2/8\simeq 6/8\}$ is a well behaved
orbit portrait $\{1/8\simeq 2/8\simeq 3/8\simeq 6/8\}$. (All three are visible
in Figure~\ref{f-po}.) However, the amalgamation of two orbit portraits although
well
defined as an equivalence relation, is not always  a possible orbit portrait.
In order to characterize which equivalence relations can be orbit portraits,
we will need the following.

\begin{lem}\label{L-rays} Suppose that three or more dynamic rays of period 
  $q$ land on a periodic point $z$ of period $q'<q$. Then the cyclic order
  of the angles of these rays must be preserved  under multiplication
  by $3^{q'}$. Hence these rays have a common well defined rotation number
  under the map $F^{\circ q'}$.
\end{lem}

\begin{proof}  %(Compare \cite{BM}.)
  This is clear since the map 
  $F^{\circ q'}$ maps a neighborhood of $z$ diffeomorphically, preserving the
  cyclic order of the rays landing there.\end{proof}
\bigskip

\begin{ex}\label{ex-nOP} 
The orbit portrait $\{1/8\simeq 2/8 \simeq 3/8 \simeq 6/8\}$
described above is well behaved: The corresponding four rays must land at
a common fixed point, and cyclic order is preserved by the tripling map
$\{1/8,\, 2/8,\, 3/8,\, 6/8\} \mapsto \{3/8,\, 6/8,\, 1/8,\, 2/8\}$. 
However the larger equivalence relation
$\{1/8\simeq 2/8 \simeq 3/8\simeq 5/8\simeq 6/8\simeq 7/8\}$
cannot be an orbit portrait since the tripling map sends it to
$\{3/8,\, 6/8,\, 1/8,\, 7/8,\, 2/8,\, 5/8\}$, completely confusing the cyclic
order. More generally, for any $q\ge 2$ it is not hard 
to check that the relation in which {\bf all} 
period $q$ orbits are equivalent can never be an orbit portrait.  \end{ex}
\medskip

By a \textbf{\textit{formal orbit portrait}} we will mean any equivalence
relation which satisfies the conditions of both Lemma \ref{L-unlink} and Lemma
\ref{L-rays}. Conjecturally every formal orbit portrait is realized as an actual
orbit portrait in $\cS_p$ for all sufficiently large $p$. (See Remark
\ref{R-OPques}.)

\msk

 Two or more orbit portraits of the same period $q$ will be called
 \textbf{\textit{compatible}} if they are
  contained in some common formal orbit portrait.  This definition also makes 
  sense for orbit portraits of different periods, providing that we allow
  mixed period orbit portraits.
  However note that two orbit portraits of different period are automatically
  compatible with each other. 
  Evidently a collection of orbit portraits are compatible if
and only if their amalgamation is a formal orbit portrait. 

In many interesting cases, the orbit portrait of a face is the
amalgamation of the orbit portraits
of two neighboring faces.  This seems to happen  when the parabolic point
  ${\mathfrak p}$ of period $p$
  is in the boundary of two different escape regions in ${\mathcal S}_q$.  
  Conjecturally, ${\mathfrak p}$ is the landing point of 3 or at most 4
  parameter rays. (See Section~\ref{s-near-para}.) For simple examples of this
  see Figure~\ref{f-po}, and
for more complicated examples see Figure~\ref{f-T3OP}. \msk

{\bf Caution.} If we choose two arbitrary period $q$ 
  orbit portraits, the amalgamation,
although well defined as an equivalence relation, may not correspond to any
possible orbit portrait.   As an example, consider the three orbit portraits 
  $$ \cO_1=\{1/8\simeq 3/8\},~~~~\cO_2=\{2/8\simeq 6/8\},~~~~\cO_3=\{5/8\simeq 7/8\}~.  $$
  The two-fold amalgamations $\cO_i\uplus\cO_j$
  all correspond to possible orbit portraits, which are visible in the
  outer part of Figure~\ref{f-po}.   However, as noted in Example \ref{ex-nOP},
  the three-fold amalgamation with
  $\{1/8\simeq 2/8 \simeq 3/8\simeq 5/8\simeq 6/8\simeq 7/8\}$ 
  does not correspond to any possible orbit portrait.
\end{definition}\smallskip

Note that the period $q$ orbit portraits of a fixed map $F$ 
for different periods $q$ are always compatible with each other,
so that their amalgamation is a well defined mixed period orbit portrait.
As an example, the orbit portrait  $\{0\simeq 1/2\}$
in Figure \ref{f2} is certainly compatible with the orbit portrait 
$\{1/8\simeq 3/8,\; 5/8\simeq  7/8\}$ in Figure \ref{f-po}.    \medskip

\begin{rem}[\bf Problems]\label{R-OPques}
This discussion raises at least two questions:\ssk
\begin{itemize}
\item[$\bullet$] Is every formal orbit realized by an actual orbit portrait in
some $\cS_p$, or are there additional constraints?\ssk

\item[$\bullet$] If it is realized in
a given $\cS_p$ can it always be realized in $\cS_{p+1}$ also?\msk
\end{itemize}

\begin{figure}[htb!]
  \centerline{\includegraphics[height=4.5cm]{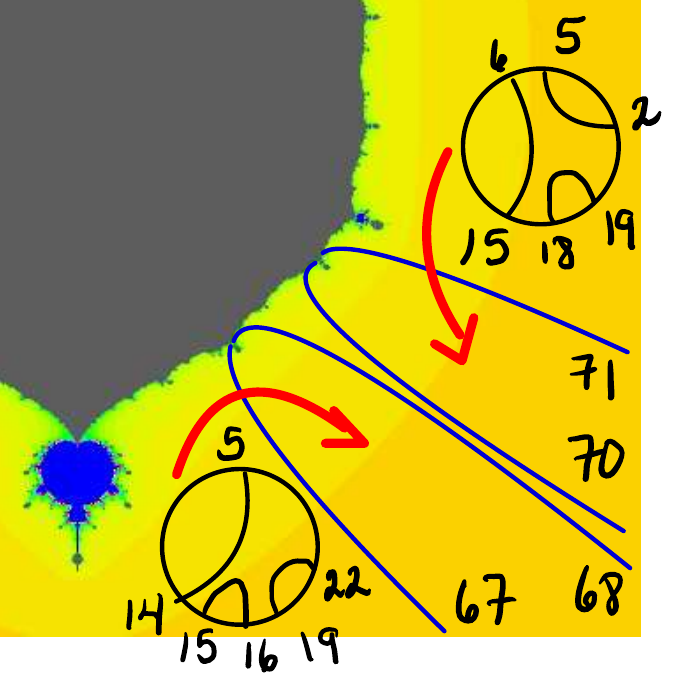}\qquad\qquad\qquad \includegraphics[width=3.7cm]{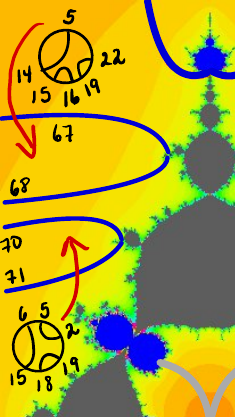}}
  \caption[Bottom part of $\cS_1$ and top left of $\cS_2$]{\label{Orb-S1S3} \sf On the left, the bottom right of $\cS_1$ displaying
    the co-periodic angles $70/78$ and $71/78$ and its corresponding  orbit portrait
    $\{2\simeq 5,\; 6\simeq 15,\; 18\simeq 19\}/26$. On the right, the top left part
    of $\cS_2$ displaying the same orbit portrait and the same co-periodic angles.
    The same orbit portrait appears also in $\cS_3$, see Figure~\ref{F-2sf}.
    In this picture we also display the orbit portraits of the primary
    wakes in $\cS_1$ and $\cS_2$ corresponding to the co-periodic angles $67/78$ and
    $68/78$. These orbit portraits  are very different from the orbit portrait
    of the  primary wake delimited by the same co-periodic angles in the airplane
    region in $\cS_3$.}
\end{figure}

\noindent  Here are some examples. We see from Figures~\ref{f-po},
~\ref{F-t2s1}-right and \ref{f-OPt2s3} that $\{2\simeq 7,\; 5\simeq 6\}/8$
occurs as an actual orbit portrait in all three $\cS_1$,\; $\cS_2$
and $\cS_3$. Similarly we see from Figures~\ref{Orb-S1S3} and \ref{F-2sf}
that $\{2\simeq 5,\; 6\simeq 15,\; 18\simeq 19\}/26$ also occurs as an orbit
portrait in  all three $\cS_1$,\; $\cS_2$
and $\cS_3$. 
However the orbit portrait
$$\{\{5\simeq 14\simeq 25\}/26,\;\, \{15\simeq 16\simeq 23\}/26,\;\, \{17\simeq
19\simeq 22\}/26\}$$
appears in $\cS_3$ but doesn't appear in $\cS_1$, $\cS_2$ or in $\cS_4$.
\end{rem}\msk    

\begin{rem}[\bf Examples of duality]\label{R-dual-ex} 
 According to Remark~\ref{R-dual}, 
  there is a natural one-to-one correspondence between hyperbolic components
  of Type D$(q)$ in $\cS_p$ and hyperbolic components of Type D$(p)$ 
  in $\cS_q$. The center points of two dual components actually represent
  the same polynomial map, and hence have the same Julia set.
  It follows of course that they have the same period $h$ orbit portrait
  for any choice of $h$.
  As an example, there are six components of Type D$(2)$ in $\cS_1$,
  which are visible in Figure~\ref{F-t2s1}.
  Hence there are six corresponding components of Type D$(1)$ in $\cS_2$.
  These six components lie at the endpoints of the eight parameter rays
  in Figure~\ref{f2}; but it is easier to identify them 
  by their period two orbit portraits, as shown in
  the central part of Figure \ref{f-po}. 
  %Four   of these D components are visible in the center part of Figure
%  \ref{F-S2rays}-right,   and the other two at the top and bottom of
%  Figure \ref{f2} or \ref{F-S2rays}-left.) But
  There are also eight components of Type D$(2)$ in $\cS_2$ (see the outer
  parts of Figures~\ref{F-S2rays} and \ref{f-po}).  This yields 
  four dual pairs with the same orbit portrait. Finally there are  two 
  large components of type B to the left and right of Figure~\ref{F-S2rays}
which are dual to each other,   and hence have the same orbit portrait.
\end{rem}
\ssk

\begin{definition}[{\bf Edge Numbers}]\label{D-pm}
We return to the study of tessellations.
One very simple invariant 
of a face $\F_k$ of a tessellation $\Tes_q(\ocS_p)$
is the number of its edges, counted with multiplicity. (By definition,
an edge has multiplicity two if it is the only edge landing at a cusp
point, so that we must traverse it twice as we follow the boundary
of the associated face.) This edge number is always even, since ideal and
parabolic vertices alternate as we traverse
any component of the boundary of $\F_k$. Note that the sum of edge numbers
over all faces is equal to twice the total  number of edges.  
As an example, in Figure~\ref{f2} the middle face $\F_2$ has twelve
edges counted with multiplicity, since each of the four edges in 
the middle is counted twice; while the top and bottom faces
have two edges each.

As a very rough principle, faces with many edges will tend to have
relatively small orbit portraits, since each
additional  parabolic boundary point will impose 
an additional upper bound on the associated orbit portrait. In fact the
orbit portrait for any point of the face is contained in the orbit portrait
for any parabolic boundary point. (Compare Corollary \ref{usc} in
Appendix~\ref{a-para-stab} as well as the proof of the
Edge Theorem~\ref{T-edge}.) In particular, faces with only two
edges will tend to have larger orbit portraits. 
\end{definition}
\smallskip

\begin{definition}[\bf Wakes]\label{D-wake} The concept of wake is
closely related to this discussion. Let $H\subset\cS_p$ be a hyperbolic
component of Type D$(q,p)$.  By definition, $H$ has a \textbf{\textit{wake}}
$W$ if $H$ is  contained in  a simply-connected open subset $W\subset\cS_p$
which is bounded by two edges of $\Tes_q(\ocS_p)$ which
\begin{itemize}
\item[(1)] lie in the same escape region $\cE$, and

\item[(2)] share a common parabolic landing point $\p\in\partial H$.
\end{itemize}
  \smallskip

\noindent
A wake may consist of a single face of $\Tes_q(\ocS_p)$, in which case
it is described as a \textbf{\textit{minimal wake}}. 
Many examples of minimal wakes can be seen in
 Figures~\ref{f2}, \ref{F-S2rays} and \ref{F-t2s1}.

 But in many cases
wakes may be \textbf{\textit{nested}}, with one contained in another.
For many examples of nested wakes, see Figure \ref{F-t3s3air-010}.
Any wake which is not minimal must be made up out of two or more of the
faces of the tessellation, together with all of the common boundary
edges between two of these faces.

It will be convenient to say that a hyperbolic component $H$ and an escape
region $\cE$ are \textbf{\textit{adjacent}} if they have infinitely many
shared boundary points; or in other words if the intersection 
$\overline H\cap\overline\cE$ is infinite. Since every boundary point of a
hyperbolic component must be a boundary point of at least one escape
region, it follows that every hyperbolic component is adjacent to at
least one escape region. \medskip

Here are two conjectural statements, which are true in every case that we
have observed:
\begin{itemize}
\item[(3)] {\sf Any two parameter rays which are contained in the
    same escape region $\cE$, and have a common parabolic landing point,
    bound a uniquely determined wake $W=W(H)$ which intersects no escape 
    region other than $\cE$. In particular, the component $H$ of Type 
    D is uniquely determined.
}

\item[(4)] {\sf Every Type D component which is adjacent to 
  only one escape region $\cE$ must have a wake in $\cE$.}
\end{itemize}
On the other hand, it follows from $(3)$ that a component $H$ which
is adjacent to two or more escape regions can never have a wake. 
See Figure~\ref{f-M(53)bdry} for examples of Type D components which
are adjacent to two different escape regions.\medskip

 Since no two parameter rays can intersect each other, it
also follows from $(3)$ that:
\begin{itemize}                                            
\item[(5)] {\sf  Any two wakes must be either disjoint or nested.}
\end{itemize}
\noindent (This is true even if the wakes are associated with different
values of $q$. As an example, the period two wake labeled by angle
$17/24$ and $19/24$ at the top of  Figure~\ref{F-S2rays} is a subwake of the
period one wake with angles $2/3$ and $5/6$ at the top of
Figure~\ref{F-manpic}.) 
\end{definition}\medskip

The complexity of the tessellation $\Tes_q(\ocS_p)$ increases 
rapidly as either $p$ or $q$ increase. For period $q=3$,
there are $48$ angles of co-period three, all 
of the form $j/26\pm 1/3$ with denominator $3\,d(3)=78$. 
(Compare Table~\ref{t3} in Section \ref{s-count}.)
\medskip

$\Tes_3(\ocS_1)$ consists of a central lemon shaped hyperbolic component
which is contained in a large face $\cF_0$ with trivial orbit portrait.
It is surrounded by two pairs consisting of a wake and minimal subwake,
together with $20$ other minimal wakes. 
(Compare Figure~\ref{F-t3s1}-left for the top part of this 
tessellation, and Figure~\ref{F-t3s1}-right for a cartoon of this tessellation.)

$\Tes_3(\ocS_2)$ consists of an annulus shaped face with trivial orbit
portrait which separates the two escape regions, together with $24$
faces in the inner $(1,0)$ escape region,  and $24$ in the outer
basilica region. More precisely, the inner region contains six pairs
consisting of wake and minimal subwake, plus twelve additional minimal
wakes; while the outer region contain only two
such pairs, plus twenty additional minimal wakes.

See Remark~\ref{R-S3} for a detailed description of
$\Tes_1(\overline{\cS_3})$, $\Tes_2(\overline{\cS_3})$ and
$\Tes_3(\overline{\cS_3})$. In particular note
that the statistics for the outer region of $\Tes_3(\ocS_2)$ are identical
with those for the outer part of $\Tes_3(\ocS_1)$, or for the airplane or
rabbit regions of $\Tes_3(\ocS_3)$. (Compare the Zero Kneading Conjecture
in \cite{BM}, as %Remark~\ref{R-trivk}, as
well as Figures~\ref{F-t3s1}, \ref{F-air}, and \ref{F-3rab}.)

\medskip

 We will prove the following sharper 
 version of Lemma~\ref{L-land-stab}. Let $F$ vary over any one of the faces
 $\F_k$ of $\Tes_q(\ocS_p)$.  

\begin{figure}[htb!]
  \begin{center}
    \begin{minipage}{1.7in}\includegraphics[width=1.7in]{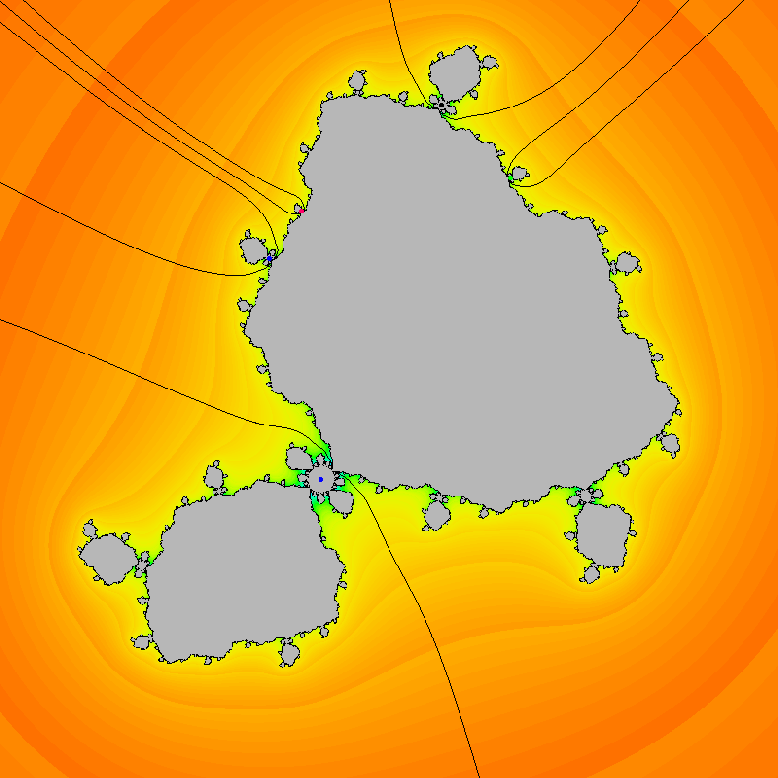}
    \end{minipage}
    \begin{minipage}{1.7in}\includegraphics[width=1.7in]{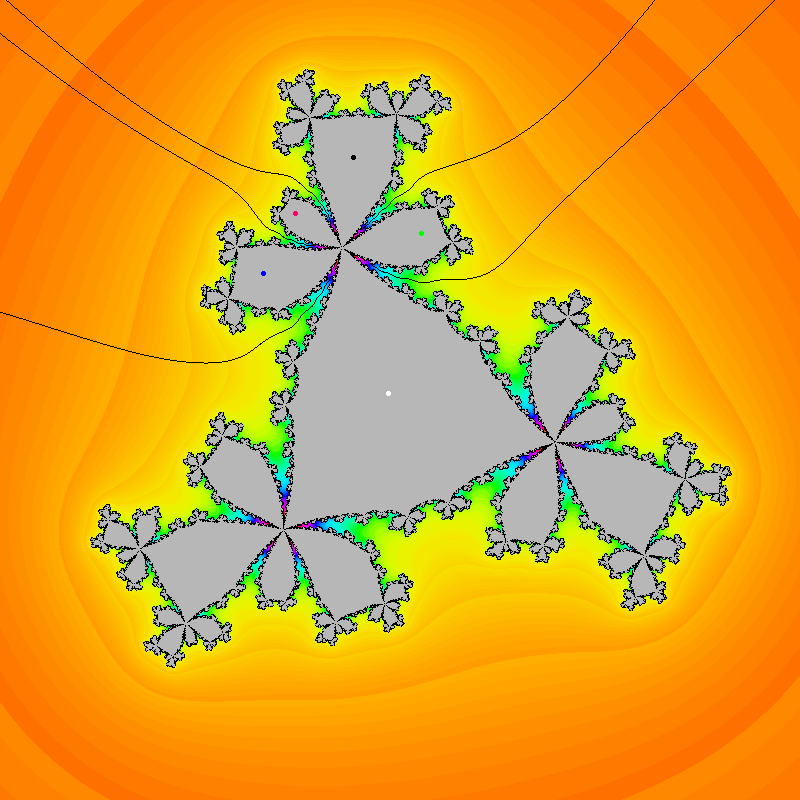}
      \vskip .1cm
      
    \includegraphics[width=1.7in]{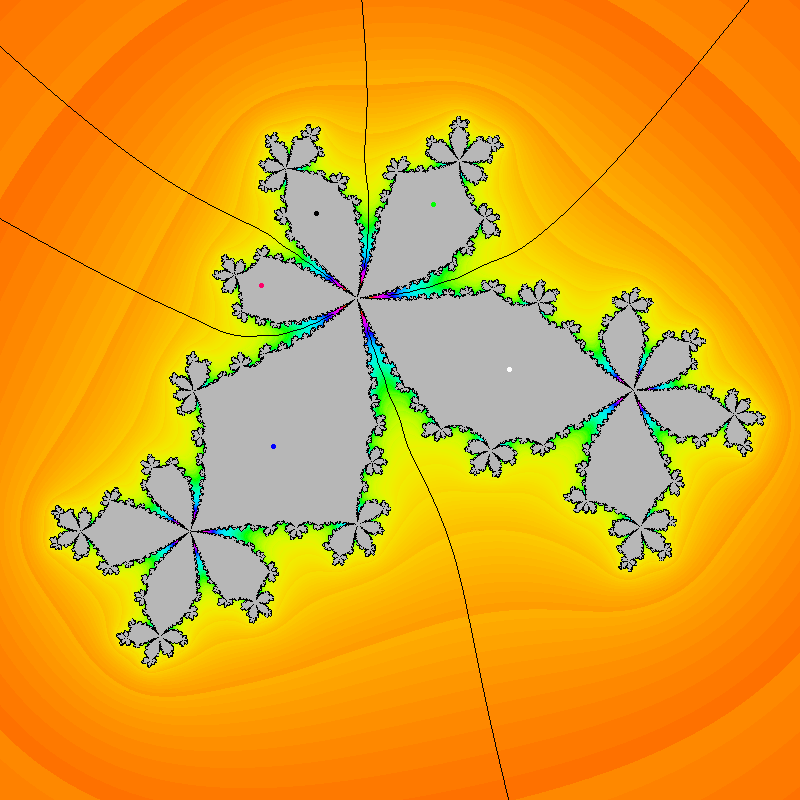} \end{minipage}
    \begin{minipage}{1.7in}\includegraphics[width=1.7in]{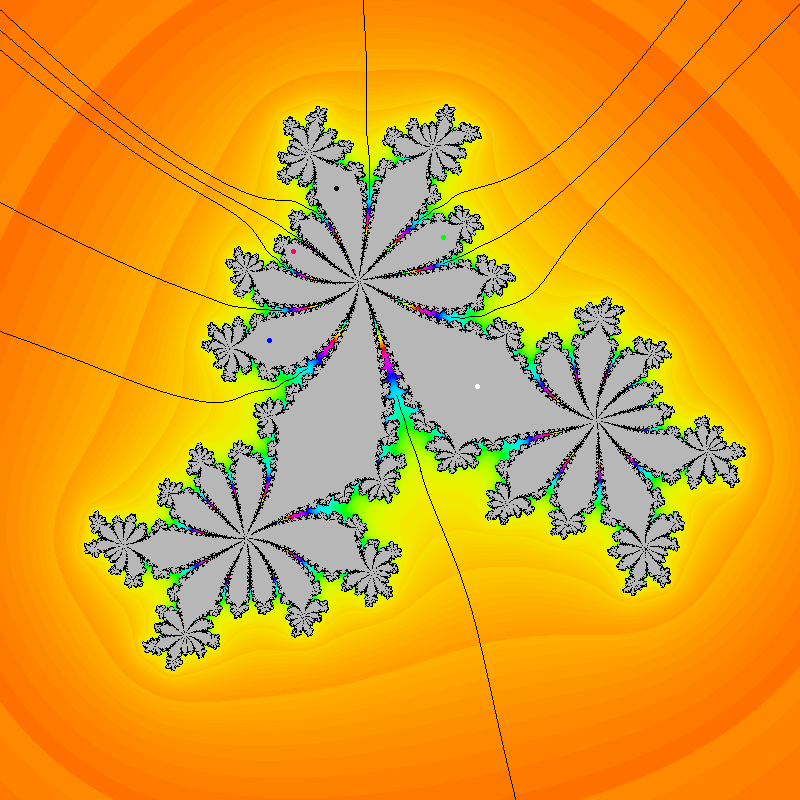}
    \end{minipage} 
    \caption[Julia sets for maps in $\Tes_5(\overline\cS_5)$ of components in  Figure \ref{F-s5rab-par}]{\sf Julia sets for maps in the four hyperbolic
components surrounding the parabolic map in Figure \ref{F-s5rab-par}; and belong
to four different faces of $\Tes_5(\overline\cS_5)$. The left component
of Type D and period one has trivial period $5$ orbit portrait. The top
and bottom figures, of Type A and  B respectively, each have a cycle of
$5$ periodic rays landing at the middle fixed point; but these are different
cycles. For the right hand figure, of Type D and period $5$, both
of these cycles of rays land at the same point. The Julia set
of the parabolic map (Figure \ref{F-4wall}) is topologically equivalent
to the right hand figure. (Here we can hope to understand
the local tessellation; but it
would hardly be possible to understand the full $\Tes_5(\overline\cS_5)$,
with its $38400$  edges. Compare Section \ref{s-count}.) \label{F-4Js}}
\end{center}
\end{figure} 

\begin{theo}[{\bf Tessellation Landing Theorem}]\label{T-decomp}
  For each $F\in \F_k$, and each angle $\theta_0\in\Q/\Z$ of period $q$
  under tripling, the  dynamic ray
  $\cR_F(\theta_0)$ lands at a repelling periodic point
  $z(F)\in  J(F)\subset\C$.
Furthermore, the correspondence $ F\mapsto z(F)$ defines a holomorphic 
function $\F_k\to\C$.

The orbit portrait for dynamic rays of period $q$
is the same for all maps $F\in \F_k$.
\end{theo}
\medskip

In particular, any two hyperbolic components contained in the same face must
have the same period $q$ orbit portrait, even if they are surrounded by
different  escape regions.  
\smallskip

\begin{proof}[{\bf Proof of Theorem \ref{T-decomp}.}]
  Since $\theta_0$ has period $q$ under tripling, there are just
  three possibilities for a dynamic ray  $\cR_F(\theta_0)$. Either it
  lands on a repelling periodic
  point, or a parabolic point of ray period $q$, or
  it crashes into $-a$ or some pre-image of $-a$. In the last case,
  some forward image  of $\cR_F(\theta_0)$ would hit $-a$ itself.
  But this would mean that $F$ lies on a parameter ray of co-period
  $q$, or in other words to  the edge of the period $q$ foliation.
  Since $F$ belongs to a face by hypothesis, this Case cannot occur.

There are only finitely
many parabolic maps of ray period $q$ in $\cS_p$. This means there
is at worst a finite subset $X\subset \F_k$ of such points. Thus for
all $F\in \F_k\ssm X$ the $\theta_0$ ray must land on a repelling fixed
point of the iterate $F^{\circ q}$. The multiplier $\lambda_F$ of
this fixed point depends holomorphically on $F$. Thus we have a
holomorphic map $F\mapsto \lambda_F$ from $\F_k\ssm X$ to
$\C\ssm\overline\D$ for each choice of $\theta_0$.

But it is not hard to check that the multipliers of fixed points of
$F^{\circ q}$ are bounded on any compact subset of $\cS_p$. Therefore
the ``singularities'' of the map $F\mapsto \lambda_F$
at the points of $X$ are removable. In other words,  we can
extend to a holomorphic function from $\F_k$ to $\C\ssm\D$.
In fact, by the minimum modulus principle, it must actually take
values in $\C\ssm\overline\D$.

But this leads to a contradiction.
Given any  $F_0\in X$, choose a sequence of maps 
$F_j\in \F_k\ssm X$ which converge to $F_0$. Passing to an 
infinite subsequence, we can assume that the landing points of the
associated dynamical $\theta_0$-rays converge to a limit point $z_0$.
Evidently
the multiplier $\lambda$ of $F_0$ at $z_0$ satisfies $|\lambda|>1$.
This contradicts the hypothesis that the landing point of the
$\theta_0$ ray is  parabolic. This shows that the set $X$ must
be empty;
and completes the proof of Theorem~\ref{T-decomp}.
\end{proof}
\medskip

\begin{coro}\label{C-allpara}
Every parabolic map $F_0\in\cS_p$ is the landing point
of at least one co-periodic ray. 
\end{coro}

\begin{proof} Choose a parabolic periodic  point $z_0$ for $F_0$,
  and a dynamic 
ray $\cR_{F_0}(\theta)$ landing at $z_0$. (For the existence of such
  a ray,  see the beginning of \S\ref{s-rays}.) Let $q$ be the period
of this ray. Then $F_0$ cannot belong to one of the sets $\F_k$ in the
co-period $q$ tessellation of $\cS_p$. In fact, by Theorem  \ref{T-decomp},
any landing point of a period $q$ dynamical ray in $\F_k$ is repelling.
It follows that $F_0$ must be one of the
vertices of $\Tes_q(\overline\cS_p)$, the period $q$ tessellation;
and hence must be the landing point of some co-period $q$ parameter ray.
\end{proof}
\medskip

We have seen in Theorem \ref{T-decomp} that the
\hbox{correspondence} \hbox{$F\mapsto \cO_q(F)$} is constant
throughout each period $q$ region $\F_k$.
\medskip

\begin{figure}[htb!]{
  \newcommand{\cskp}{-.4ex}
  \newcommand{\lskp}{1ex}
  \newcommand{\mvright}[1]{\hspace*{.1\textwidth}#1\hspace*{-.1\textwidth}}
\centering
  \begin{tabular*}{.5\textwidth}%
    {cc}
 \includegraphics[width=3in]{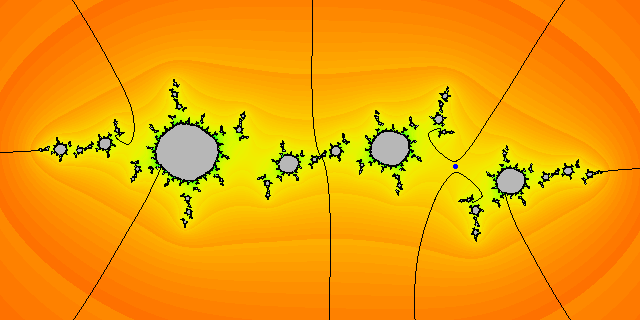} &
\includegraphics[width=3in]{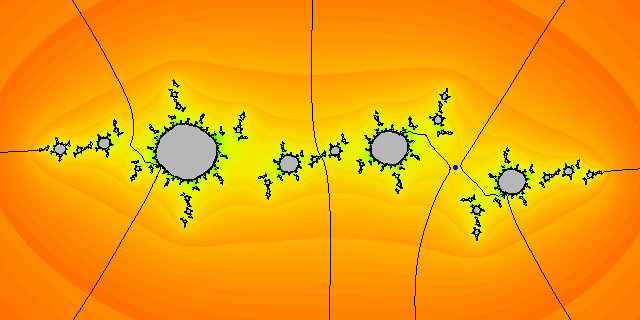}
      \\[\cskp]
{\small \sf On parameter ray $11/24-\epsilon$.} &    
{\small\sf  On parameter ray \hbox{$11/24+\epsilon$.}}
  \end{tabular*}
}
\caption[Julia sets in parameter rays \hbox{$11/24\pm\epsilon$}]{\sf The
  parameter angle $11/24=1/8+1/3$ is co-periodic. On the left,
  for a map on the parameter ray of angle  $11/24-\epsilon$, the
  $1/8$ dynamic ray (which is periodic of period two) bounces up,
 away from the free critical point
  $-a$, and lands on the Julia set. Similarly its
  pre-image, the 3/8 ray is pushed up, while the other pre-image, the
$19/24$ ray is pushed down, so that both land on the Julia set. For a map
  precisely on the $11/24$ parameter ray, the $1/8$ and $19/24$ rays would
  crash together. However, after we cross the $11/24$ ray, the $1/8$ and $3/8$
  rays both bounce to the right,   so as to share a landing point with the
  $7/8$ and $5/8$ rays respectively. Thus the period two
  orbit portrait jumps in size from one to  three. Note that only
  the $1/8$ ray and its iterated preimages jump discontinuously as we
cross the $11/24$ parameter ray. All other dynamic rays vary continuously.
\label{F-11-24}}
\bigskip \bigskip

  \newcommand{\cskp}{-.4ex}
  \begin{tabular*}{.5\textwidth}%
    {ccc}
 {\includegraphics[width=2.8in]{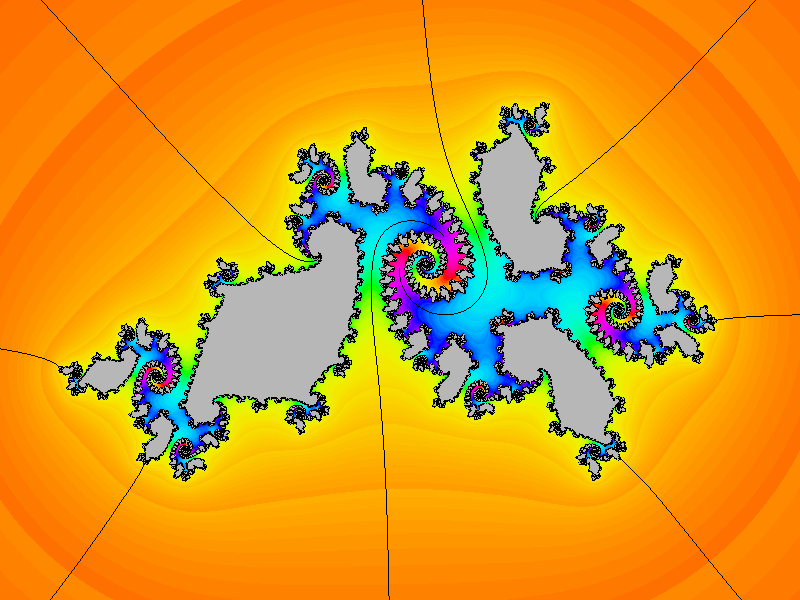}} & &
{\includegraphics[width=2.8in]{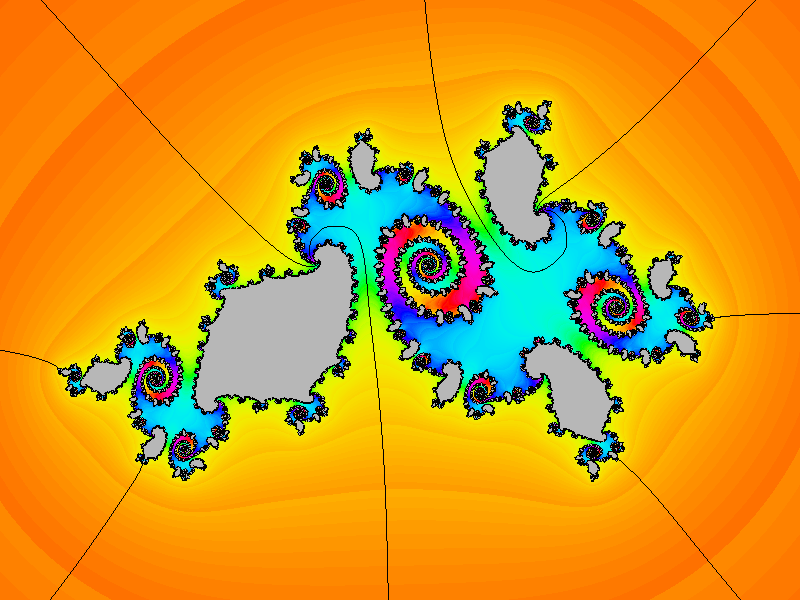}}
      \\[\cskp]
{\small \sf On parameter ray $7/12-\epsilon$.} & &   
{\small \sf  On parameter ray \hbox{$7/12+\epsilon$.}}
  \end{tabular*}
\caption[Julia sets on parameter rays $7/12\pm \epsilon$]{\sf The landing
  points of both the $1/4$ and $3/4$ rays jump  discontinuously  as we cross
  the $7/12$ ray. \label{F-7-12}}
\end{figure}
\medskip

\begin{lem}[{\bf Crossing a Parameter Ray}]\label{L-crossing}
  If the map $F$ lies precisely on a parameter ray with co-periodic angle
  $\psi$, then the  two dynamic rays of angle $\psi\pm 1/3$ crash together
  at the
critical point $-a$. This pair of rays divides the dynamic plane into two
simply-connected open sets $U$ and $V$. Now consider a one-parameter family
of maps which cross over this parameter ray. Just before the crossing, the two
dynamics rays will bounce in opposite directions, so that one lands at a
periodic
point in $U$ and the other lands in a periodic point in $V$. After the
crossing,
the roles are reversed, so that the dynamic ray which formerly landed in $U$
will now land in $V$ and vice versa.\end{lem}

We will not provide any detailed argument; but
examination of examples such as Figures~\ref{F-11-24} through
\ref{F-7-12a} illustrate this behavior.\qed\bigskip

It is important to emphasize that a period $q$ dynamic ray for a map $F$ in
an escape region nearly always
lands on a periodic point of ray period $q$. This fails to be true only
in the cases where the map $F$ lies precisely on a parameter ray of co-period $q$,
or on the iterated pre-image of such a ray, so that the dynamic ray
crashes into a critical or pre-critical point. 
Similarly periodic points nearly always vary continuously as functions of the
map $F$. This fails to be true only when $F$ is a parabolic map.

What can happen as we cross a parameter ray of co-period $q$ is that the
rotation number of the former landing point of the ray changes, so that
the ray period increases, and
this point no longer plays any role in orbit portraits of period $\le q$.
\medskip

\begin{definition}\label{D-GO}
  The \textbf{\textit{grand orbit}} of a co-periodic angle
  $\psi$ under angle tripling will be denoted by $\( \psi\)$. 
This consists of all of the periodic points $\psi_j$, together with all
of their iterated pre-images under angle tripling. (If the period $q$ has been
specified, then a convenient identifying label for the grand orbit is the
minimum of the angles $0<\psi_j<1$, 
multiplied by the common denominator $3^q-1$.)\end{definition}

\begin{coro}[\bf The cycle of jumping rays]\label{C-jump}
As we cross any parameter ray with angle $\psi$ of co-period $q$, the
landing point of a dynamic ray will jump discontinuously if and only if
its dynamic angle belongs to the grand orbit $(\! (\psi)\! )$. 
In particular, a periodic ray will jump if and only if its angle is
one of the $\psi_j$. For all periodic or
preperiodic angles which are not in this grand orbit, the landing point
will vary continuously as a function of the point $F$ in $\cS_p$. 
\end{coro}

The proof is straightforward. \qed \medskip

 As an immediate consequence of Lemma \ref{L-crossing},
 we obtain the following statement. 

\begin{theo}\label{T-OPjump}
Let $\mathfrak R$ be a parameter ray with angle $\psi$ of co-period $q$.
Suppose that for maps on one side of $\mathfrak R$ there
is a dynamic ray with angle belonging to the cycle $\{\psi_j\}$
which shares a landing point with a dynamic ray of some other angle
(necessarily also of period $q$). 
Then after crossing the parameter ray these two dynamic rays will no longer
share a landing point. Thus the period $q$ orbit portrait will definitely
change.\end{theo}

\begin{proof}
There are three possible cases. If the other angle is not in the same
periodic cycle, then its ray will not jump, and the conclusion follows
immediately from Lemma~\ref{L-crossing}.
  
Now suppose that the two rays belong to the same periodic cycle, and hence
both jump. Then at the exact crossing
point we will have two different pairs of dynamic rays crashing together,
as illustrated in Figure~\ref{F-7-12a} (upper right), dividing the plane
into three regions.  We claim that these two periodic dynamics rays
must be  oriented in opposite directions, as in Figures~\ref{F-7-12}
and \ref{F-7-12a}. Hence as we cross the parameter ray in
one direction both periodic dynamic rays will jump into the middle region,
where they will find a common landing point by hypothesis. But crossing
in the other direction they will jump into the two outer regions,
and can never land together.

If on the other hand these two rays were oriented in the same direction,
then after crossing to either side of the parameter ray they would always
jump into different   regions, and could  never share a landing point.
\end{proof}\medskip

{\bf Note.} 
The concept of an edge $E$ of the tessellation $\Tes_q(\ocS_p)$
is of course completely equivalent to the concept of a parameter ray
$\fR$ of co-period $q$ in $\cS_p$. However, since we want to discuss
tessellations, it will be convenient to switch to the edge terminology.

\medskip

\begin{definition}[{\bf The three kinds of edge in $\Tes_q(\ocS_p)$}]\label{D-3E}
Using the discussion above, we can separate the edges $E$ of the
tessellation $\Tes_q(\ocS_p)$ into three
fundamentally different cases, as follows. Let $\psi$ be the co-periodic
angle of $E$.\smallskip

\begin{description}
\item[{\bf Case 0.}] $E$ is \textbf{\textit{ inactive:}} For neighboring
maps on either side of $E$, no dynamic ray with angle in the
cycle $\{\psi_l\}$ shares a landing
point with any other dynamic ray. In other words, throughout some 
neighborhood of $E$ there is no non-trivial orbit
relation involving any angle in this cycle.

{\sf It follows that the 
period $q$ orbit portrait is the same on both sides of $E$.}
\smallskip

\item[{\bf Case 1.}] $E$ is a \textbf{\textit{primary edge:}} For maps
on one side of $E$ every ray with angle in $\{\psi_j\}$
shares a landing point with at least one other dynamic ray
(which itself may or may not have angle in this cycle); but
for maps on the other side, no ray with angle in $\{\psi_j\}$
shares a landing point with any other ray.

{\sf It follows that the change of orbit portraits between the two
sides is strictly \textbf{\textit{monotone}}: 
The orbit portrait on one side is a proper subset of the orbit portrait
on the other side.}\smallskip 

\item[{\bf Case 2.}] $E$ is a \textbf{\textit{secondary edge:}} Every
ray with angle in $\{\psi_j\}$ on either side shares a landing
point with at least one other ray. (Therefore the other ray or rays
on one side must be disjoint from those on the other side.)

{\sf It follows that the change of orbit portrait is 
strictly \textbf{\textit{non-monotone}}: Neither orbit portrait
contains the other.}
\end{description}\end{definition}
\msk

Many examples of each of the three cases will be provided in the pages
which follow; but here are some preliminary examples.\smallskip

\begin{description}
\item[{\bf Inactive Edges.}] These will usually  be colored gray. 
  As examples there are four
  such edges in Figure~\ref{f2}, and many others in Figures~\ref{F-t1s3},
  \ref{f-OPt1s3} and \ref{f-OPt2s3}.
\medskip

\item[{\bf Primary Edges}] will usually\footnote{One exceptional case
    is Figure \ref{F-s5rab-par} where the ray colors are chosen to match the
    colors of the corresponding triad in the associated Julia set picture.}
 be colored blue. These occur in every
 tessellation. See 
 for example Figures~\ref{f2}, \ref{F-S2rays}, \ref{F-t2s3},  \ref{F-T3S3}.
\item[{\bf Secondary Edges}] will usually  be colored red. These occur in
  every tessellation with \hbox{$q=p>1$}, and for many with $q\ne p$ and
  $p,\,q>1$. See Figures~ \ref{F-S2rays}, \ref{F-t2s3}, \ref{F-T3S3}. 
\end{description}

\medskip 

In particular, a wake will be called either primary or secondary according
as its edges are primary or secondary. In  $\Tes_2(\ocS_2)$
(Figures~\ref{F-S2rays} 
and \ref{f-po}), the six wakes in the outer basilica region are primary,
while the inner region with kneading $(1,0)$ contains two primary wakes 
and four secondary wakes. In $\Tes_2(\ocS_3)$, Figure~\ref{F-t2s3} 
shows $32$ wakes, of which $28$ are primary. (The four secondary wakes
are all located in the $100\pm$ regions, near the $x$-axis.) 
Figure~\ref{F-t3s3air-010} shows several examples of sub-wakes, which
are necessarily primary. 
\medskip

\begin{rem}
Empirically it is very easy to distinguish between the three different
kinds of edge without computing a single orbit portrait, just by 
looking at the local geometry near the parabolic
endpoint. See Theorem \ref{T-edge} and Conjecture \ref{CJ-main}
below. 

First note that, in practice every parabolic vertex $\p$ in
$\Tes_q(\ocS_p)$ is the \textbf{\textit{root point}} of a unique hyperbolic
component $H=H_\p$ which has Type D$(q,p)$. (This means that each map $F\in H$
has an attracting orbit of period $q$ which is disjoint from the
marked periodic orbit of period $p$.) Such a component $H$ is conformally
isomorphic to the open unit disk $\D$ by mapping each $F\in H$ to the
multiplier of this period $q$ orbit. (See for example \cite{M1}.) In
our cases this conformal isomorphism always seems to extend to a
homeomorphism from
the closure $\overline H$ onto the closed unit disk $\overline\D$.
By definition, the \textbf{\textit{root point}} $\p$ of $H$ is the unique
parabolic boundary point for which this multiplier is $+1$.
\end{rem}
\bigskip

By Corollary \ref{C-allpara} every parabolic vertex of the tessellation
is the endpoint of at least one edge.
\bigskip

\bigskip

\begin{theo}[\bf Edge Theorem]\label{T-edge}  \hfill{}

\begin{itemize}
\item[{\bf(1)}] If there is only one edge $E$ of $\Tes_q(\ocS_p)$
with $\p$ an endpoint, then $E$ is an inactive edge.

\item[{\bf (2)}] If there are two or more edges landing at $\p$,
which is the root point 
of $H$, then the two which are closest to $H$ are primary edges.
\end{itemize}\end{theo}

In fact Statement (1) is obvious since both sides of $E$ belong to the
same face. The proof of Statement (2) will be given later in this section.
Here is an example to illustrate
Statement (2): In Figure~\ref{F-s5rab-par}, taking $H$ to be the small disk
enclosed by the  blue and red rays, we see that these two rays must be
primary.
 \medskip

We will assume  Conjecture MC1 in Remark \ref{R-MandelC},
which asserts that:

\begin{quote}\sf Every parabolic
point $\p$ of ray period $q$ is the root point of one and only one 
hyperbolic component $H=H_\p$ of Type D.\end{quote}

\noindent We will also assume the following.

\begin{conj}[\bf Edge Monotonicity Conjecture]\label{CJ-main} There can be at 
most four edges with a given parabolic endpoint $\p$.
If there are two or more such edges, then the two edges which are closest
to $H_\p$ are primary edges, so that the orbit portrait increases monotonically
as we cross into the face containing $H_\p$. Any further edges are secondary.
Furthermore a primary edge and a secondary edge with the same parabolic endpoint
always lie in different escape regions.
\end{conj}

One easy consequence of this conjecture is the following:

\begin{quote}\sf The parabolic endpoint of every
secondary ray must be a common boundary point between two or more 
escape regions.\end{quote}

\noindent For example, there can be no secondary edges in $\cS_1$
since there is only one escape region. Another easy consequence is the
following: 

\begin{quote}\sf If $\p$ is a parabolic boundary point of the escape region
  $\cE$ then at most two parameter rays from $\cE$ can land at $\p$. (There is
  always at least one such ray. See  Arfeux and Kiwi \cite{AK}
  Theorem 5.8; and compare Remark \ref{R-access}  and Lemma~\ref{L-acc}.) 
\end{quote}

\noindent Another important consequence is that $E$ is inactive only if it
is a ``dead end'' edge which does not share its parabolic endpoint with
any other edge. A completely equivalent statement is the following:
\smallskip

\centerline{\sf Any two faces which have an edge in common must have
different orbit portraits.}
\smallskip

\noindent Furthermore:

\centerline{\sf Any face with trivial orbit portrait must be bounded 
    by primary edges.}\ssk

\noindent  (See Figures~\ref{f2}, \ref{f-OPt1s3} and \ref{f-OPt2s3}.)\bigskip

 Arfeux and Kiwi \cite[Section 5]{AK} have proved that every parabolic
boundary point of an escape region is the landing point of at
least one parameter ray in that region. %(Compare Lemma~\ref{L-acc}.)
It seems likely that every parabolic point which is not the cusp point 
of a Mandelbrot copy is the landing point of at least two parameter rays.
We can try to make this more precise as follows.
\bigskip

\def\fA{{\mathfrak A}}

\begin{rem}[\bf Accessibility]\label{R-access} Let $\p$ be a boundary
point of an escape region $\cE\subset\cS_p$. By definition, $\p$ is
\textbf{\textit{accessible}} from $\cE$ if there exists a Jordan path in
$\cE\cup\{\p\}$ with endpoint $\p$. More precisely, we can define an
\textbf{\textit{access}} to $\p$ from $\cE$ as a homotopy class of such paths.
\end{rem}
% \rnote{Do we need a reference to [AK] here? -J}

\begin{lem}\label{L-acc} For every access to $\p$ from $\cE$, there exists
  one and only one parameter ray\footnote{Caution: The corresponding statement
    for dynamic rays is false if we allow Julia sets which are not
    connected. In Figure~\ref{F-q8ex}, we show a connected Julia set where the are
    infinitely many
    accesses to the central point, but only four rays landing there.}
in $\cE$, necessarily co-periodic,  which lands on $\p$.
\end{lem}

\begin{proof} Choose such an access path $\alpha:[0,1]\to \cE\cup\{\p\}$ 
  with $\alpha(0)=\p$. For every $t>0$ there exists a unique angle $\theta$ 
  such that the $\theta$ parameter ray passes through  $\alpha(t)$. Choose a
  sequence $\{t_j\}$ converging to zero. After passing to a countable
  subsequence, we may assume that the associated angles $\theta_j$ converge
  to a limit $\theta_0$. Evidently this limit ray must land on $\p$.
  It follows that the angle $\theta_0$ is co-periodic. If it were
  irrational, then there would be an irrational dynamic rays landing
  on the parabolic orbit, which is impossible. Therefore, by
  Theorem~\ref{T-main},   it is co-periodic. It is unique since there cannot
  be two rays landing at the same point without some part of the
  connectedness locus between them.  \end{proof}

%{\bf Another possible conjecture.} Here is another
%  statement that is true in every case that we have seen.
%  \medskip
  
%\begin{conj}\label{conjA} {\bf (Secondary Rays)} 
%  There are no secondary rays in zero kneading regions.\footnote{Common
%  boundary points between a zero kneading region in $\cS_p$ and any other region
%  always seem to have ray
%  period $p$, so that is only necessary to consider rays of co-period $p$.}
%\end{conj}
% ~Compare Figures~\ref{F-3rab}, \ref{F-rabcan}, \ref{F-kocan},
%\ref{F-DBcan} and \ref{F-worm}. (For the special case of Rabbit regions,
%this is proved in Section~\ref{s-rab}.) In fact it seems possible that an
%escape region in $\cS_p$ contains secondary rays of co-period $p$ 
%if and only if it has non-zero kneading, (Compare Figures~\ref{F-S2rays} and
%\ref{F-T3S3}.)
 \medskip
 
% \begin{conj}\label{Cj-Man}{\bf (Mandelbrot Copies)}
%%   \rnote{Duplication: delete this ??}
%   Every parabolic point in $\cS_p$ belongs
%   to a unique maximal copy of the classical Mandelbrot set.
%  \end{conj}~~~ See
%McMullen \cite{Mc} for quasiconformal copies of the Mandelbrot set
%in families of rational maps.

The proof of Theorem~\ref{T-edge} will depend on two preliminary statements.

\begin{lem}\label{L-rootOP}
  The period $q$ orbit portrait for the root point $\p$ of a
hyperbolic component $H$ of Type D$(q)$
is identical to the period $q$ orbit portrait
for any map $F$ in the interior of $H$.\end{lem}
\smallskip

\noindent In other words, two or more
dynamic rays for the map $\p$ land together at a periodic point for this map
if and only if the corresponding rays for any $F\in H$ will land together
at a periodic point of $F$. This can be proved by Ha\"isinski
surgery.\footnote{See \cite{H1}, \cite{H2}.  For example
  Kawahira in the sketch of \cite[Theorem 2.1]{Ka}  says that ``To check the
  dynamical stability on the
Julia sets, we check the stability of the ray equivalence, which is defined
by shared landing points of external rays.''}\qed
\medskip

We will also need the Upper Semi-Continuity Corollary of
Appendix~\ref{a-para-stab} which states that:

\begin{quote}\sf The orbit portrait for any map in a neighborhood of $\p$
is a subset of the orbit portrait for $\p$.\end{quote}\smallskip

\begin{proof}[Proof of Theorem \ref{T-edge}]
We know that the orbit portrait must change as we move out of the face
containing $H$ by crossing one of the two bounding edges near $\p$. Since
the change is monotone by upper semi-continuity, both of these rays
must be primary.\end{proof}\medskip

Although the fact that an edge is secondary tells us exactly which
dynamic rays jump,
it does not provide much information about where they jump to. For a
more explicit discussion of the precise changes in orbit portrait as we 
cross an edge of the tessellation, see Section \ref{s-near-para}.
\medskip

  Diagrams such as Figure~\ref{f-po}
  provide a great deal of information about changes in the orbit portrait
  as we cross an edge (= parameter ray). Note that there are twenty four
  edges shown in this figure. Sixteen of these are primary, so
 that the orbit 
  portrait is replaced by a proper subset of itself as we cross the ray in
  the right direction. But for the remaining eight, all in the inner
  escape region, there is a more complicated behavior.
   (Again compare Section~\ref{s-near-para}.) 
  \medskip

 How does the change actually take place? Two 
  of these twenty four cases are illustrated in Figures~\ref{F-11-24}
  through \ref{F-7-12a}.  
If we cross $11/24$ parameter
ray  of the internal escape region of $\cS_2$, the change seems relatively 
easy to understand. (Compare Figure~\ref{F-11-24}.) Note that
$$ \frac{11}{24}~=~\frac{1}{8}+\frac{1}{3}~,$$
where $1/8$  has period two under tripling.
As we cross this ray in the direction of increasing parameter angle,
the landing point of the $7/8$ dynamic ray  varies continuously; 
but the landing point of the $1/8$ rays jumps discontinuously to join it.
Thus the orbit portrait jumps to a larger orbit portrait
which contains the original one.
More generally, all of the preimages of the $1/8$
ray under iterated angle tripling behave in the same way. For example
the landing point of the $3/8$ ray jumps to the landing point of the $5/8$
ray, which varies continuously.  However, the landing
points of all rays which are not iterated pre-images of $1/8$ vary
continuously. For example the $1/4$ and $3/4$ 
rays land together throughout this transition. However the period two orbit
portrait jumps to a larger orbit portrait which contains the original one.

  The transition as we cross the parameter ray of
co-critical angle
$$ \theta~=~ \frac{14}{24}~=~\frac{7}{12}~=~\frac{1}{4}+\frac{1}{3} $$
in the inner $\overline{10}$ escape region\footnote{There is similar behavior
at the inner $2/24$, $5/24$, $7/24$, $10/24$, $17/24$, $19/24$, and $22/24$
  rays. All of the crossings in the outer (basilica) region are monotone.}
seems harder to understand. 
(See Figure \ref{F-7-12}.)
Only the dynamic rays of angle $1/4$ and $3/4$ (and their preimages)
jump discontinuously, and yet we obtain a new orbit portrait which neither
contains nor is contained in the original one.

Figure~\ref{F-7-12a} helps to provide a more detailed understanding of what is
happening. Here we use a different coloring algorithm: Within the basin of
infinity,   the color depends on the value of the $F$-invariant angular
variable
$$  \log_3\big(\g_F(z)\big) \quad ({\rm mod}~~\Z)~.$$
Here $\g_F(z)$ stands for the Green's function for $F$,
which is multiplied by  $d=3$ if we replace $z$ by $F(z)$.
This coloring is useful since it makes it easier to see the critical
and precritical points, which are located in the centers of red cross shaped
configurations.

Each map  in the central escape region has three repelling fixed points.
There is the landing
point of the zero ray to the right and the $1/2$ ray to the left, and there is
also a central fixed point. Generically this central
fixed point is the landing point
of finitely many dynamic rays, which map to each other with a well defined
rational rotation number. (Compare the Rotation Number Section in \cite{BM}.)

There are infinitely many precritical points in each of the two spiral
channels leading to the central repelling point.
For parameter angles in the open interval $5/12< \theta<7/12$ (top left in
Figure~\ref{F-7-12a}), the $1/4$ dynamic ray misses
these precritical points, which act as local repellors. If $\theta=7/12$,
then the $1/4$ dynamic ray crashes into the free critical point (top right).
However if $\theta>7/12$, then the $1/4$ ray is pushed to the right
and lands at the same point as the $1/8$ ray. 
For each rational rotation
number $\rho=m/n$ in an interval $(1/2,~1/2+\epsilon)$ there is:

\begin{itemize}
\item[(1)] an angle $\theta_\rho$ of rotation number $\rho$ under tripling, and

\item[(2)] an entire open interval of $\theta$ values around $\theta_\rho$
such that the $n$ dynamic rays with angle of the form $3^j\theta_\rho$
manage to wiggle through the spiral to land at the central fixed point.
\end{itemize}

\begin{figure} [htb!]
  \centerline{\includegraphics[width=3in]{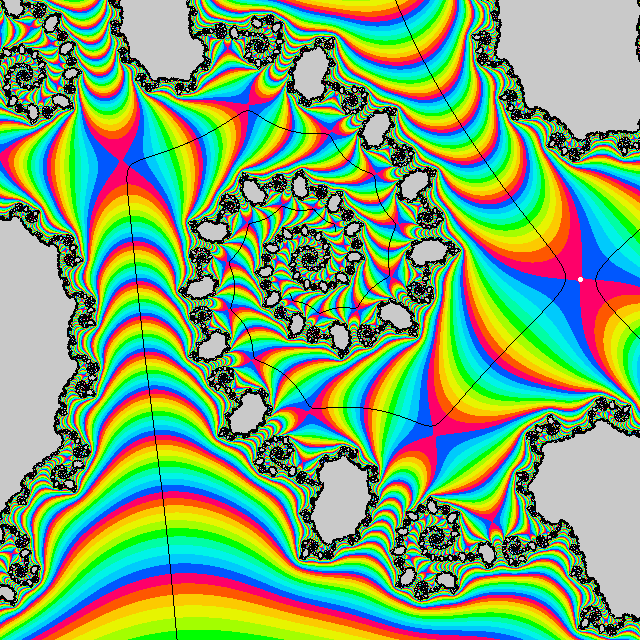}\qquad
    \includegraphics[width=3in]{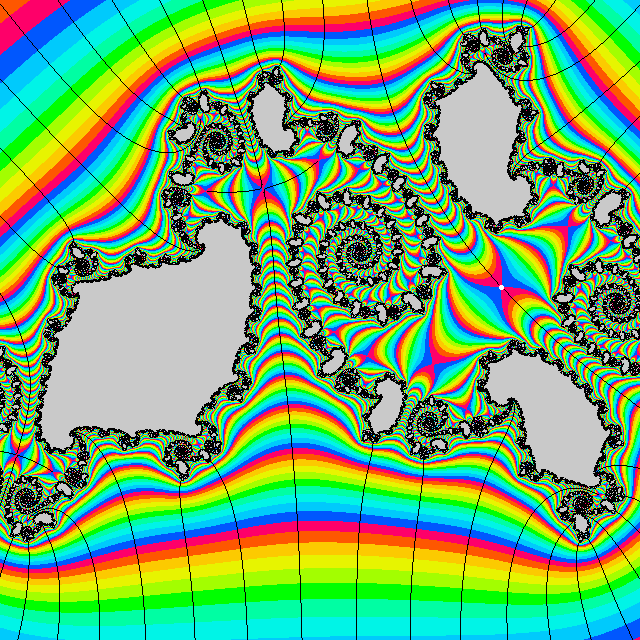}}\bigskip
  \bigskip
  
  \centerline{\includegraphics[width=3.7in]{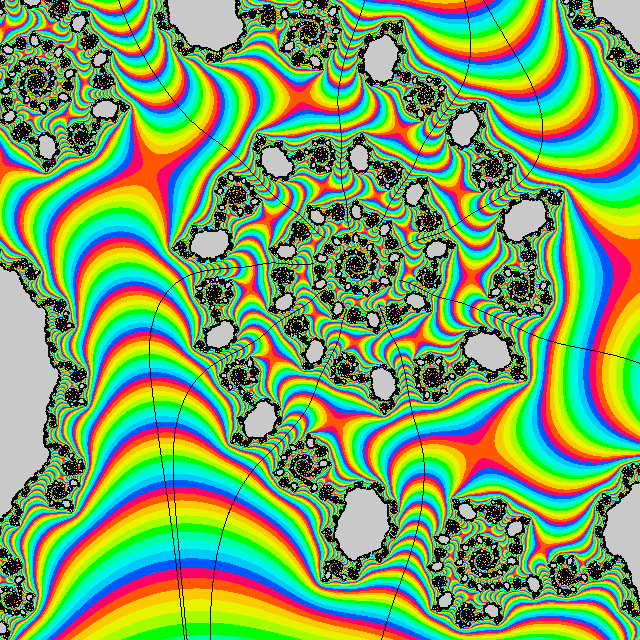}}
  \caption[Detail of Figure~\ref{F-7-12}]{\sf A more detailed view of the
    crossing in Figure~\ref{F-7-12},
   using a different coloring scheme. Top Left: For
    $\theta< 7/12$ the dynamics rays of angle $1/4$ and $11/12$ come close
    to the free critical point $-a$ and shoot off in opposite directions, so
    that the $1/4$ ray (like the $3/4$ ray) spirals into the central fixed
    point, which has rotation number $\rho=1/2$.
    Right: if  $\theta= 7/12$, the rays of angle $1/4$ and $11/12$ crash
    together at $-a$. 
    Below: For $\theta>7/12$, the $1/4$ and $3/4$ rays land on a period two
    orbit, while the central fixed point has a rotation number
    $\rho>1/2$ which varies continuously with $\theta$. 
     \label{F-7-12a}}
  \end{figure}

\noindent(The figure illustrates the case $m/n=5/9$.) At both
endpoints of such an interval of constant rotation number, these $n$ rays
crash into $n$ pre-critical points (visible as the centers of red X shaped
configurations). For $\theta$ values in  the complementary Cantor
set, if we exclude those which are endpoints of intervals of constancy,
the rotation number is irrational, and infinitely many dynamic
rays land on the central fixed point.

However, unless $\theta$ is one of the
countably many co-periodic angles, some dynamics rays will manage to wobble
though the various narrow channels to land at the central fixed point, with
rotation number $1/2<\rho<1$.
The lower part of Figure \ref{F-7-12a} shows the special case where the
rotation number is $5/9$, so that nine rays land.
\medskip

{\bf Note.} We have described what is happening in Figure \ref{F-7-12a}
in some detail; but there is actually a similar 
problem in understanding Figure \ref{F-7-12}. 
The problem in both cases is that the period $q$ orbit portraits are only
aware of periodic orbits of ray period $q$. If the ray period changes,
then, although orbit itself will usually vary continuously, 
anything involving that orbit will disappear from the
orbit portrait. However, the ray period can
change, only if the  rotation number of the first return map to the periodic
point changes. Since rotation numbers change continuously, it follows that
any change must involve VERY large denominators, and hence VERY large ray
periods.

In Figure \ref{F-7-12a} the change involves a fixed point.  We will
see in the Rotation Number Section of \cite{BM} that the rotation numbers of
fixed points are relatively
easy to deal with, but in Figure \ref{F-7-12} the change involves a period
two orbit. In fact on the left side of this figure 
the 1/8 and 3/8 rays land on a period two orbit. On he right side, this
period two orbit still exists; but its rotation number has changed.
The first return map $F^{\circ 2}$ now has degree nine;
and the computational problems involved in understanding
exactly what happens become much more difficult.

Thus to understand what is happening, we need to
consider not only period two orbits, but orbits of arbitrary period.
The central repelling fixed point remains ``visible from infinity''
as the co-critical angle increases. 
That is, there are always dynamic rays which land at this point (except at
the countably many co-periodic angles, where the appropriate dynamic
rays crash into  precritical points). In fact, for an entire Cantor set
of co-critical angles, there are infinitely many  dynamic rays landing at
the fixed point. The concept which unifies all of these cases is
the \textbf{\textit{rotation number}} $\rho$ of this fixed point, which varies
continuously as a function of $\theta$, with an interval of constancy for
each rational rotation number.  For further details,
see \cite{BM}.

\begin{figure}[htb!]
  \begin{center}
  \begin{overpic}[width=2.8in]{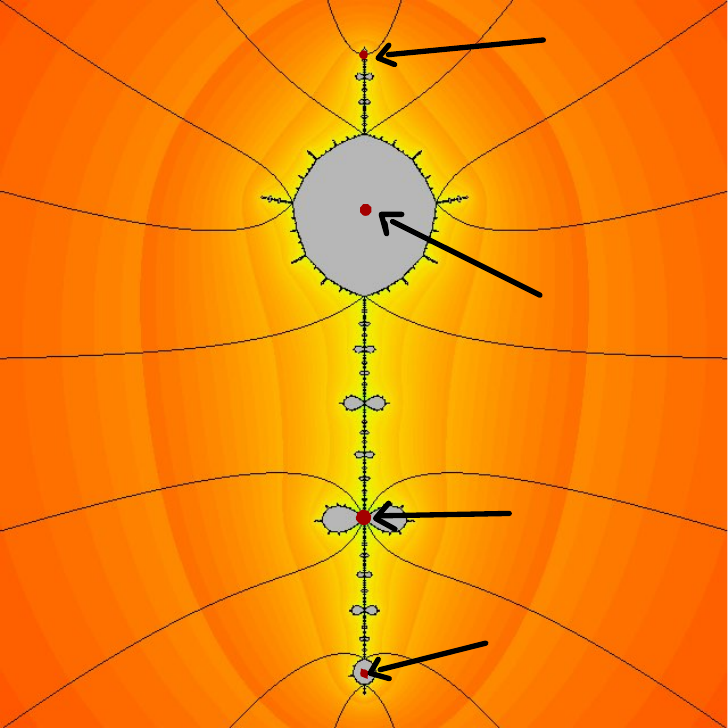}
    \put(153,190){$2a$}
    \put(153,116){$a$}
    \put(145,55){$-a$}
    \put(140,20){$-2a$}
    \put(203,100){$0$}
    \put(203,144){$1$}
    \put(203,200){$2$}
    \put(142,207){$3$}
    \put(110,207){$4$}
    \put(85,207){$5$}
    \put(55,207){$6$}
    \put(-12,200){$7$}
    \put(-12,145){$8$}
    \put(-12,98){$9$}
    \put(-17,50){$10$}
    \put(-17,0){$11$}
    \put(55,-10){$12$}
    \put(80,-10){$13$}
    \put(100,-10){$14$}
    \put(135,-10){$15$}
    \put(203,0){$16$}
    \put(203,50){$17$}
\end{overpic}
\caption[Julia for a Misiurewicz map $F$ in $\cS_1$]{\sf The Julia set for
  a Misiurewicz map $F$
  in $\cS_1$. The $4/18$ and $5/18$ parameter rays in $\cS_1$ (indicated
  by dotted white curves in  Figure \ref{F-t3s1}) land at this map.
  Note that these angles are rational, but not co-periodic.
These two rays bound a simply connected open subset of
$\ocS_1$; but this set is not a wake. See Remark \ref{R-not-wake}. Here
for example\\ 
\centerline{\quad $\{4,\,10,\,16\}/18\,\mapsto\,12/18\,\mapsto\,0\,\mapstoself$ \qquad
and \quad $\{5,\,11,\,17\}/18\,\mapsto\,15/18\,\mapsto\,9/18\,\mapstoself~$.}
\label{F-cap-ara}}
\end{center}
  \end{figure}

\begin{figure}[htb!]
  \centerline{\includegraphics[width=5.8in]{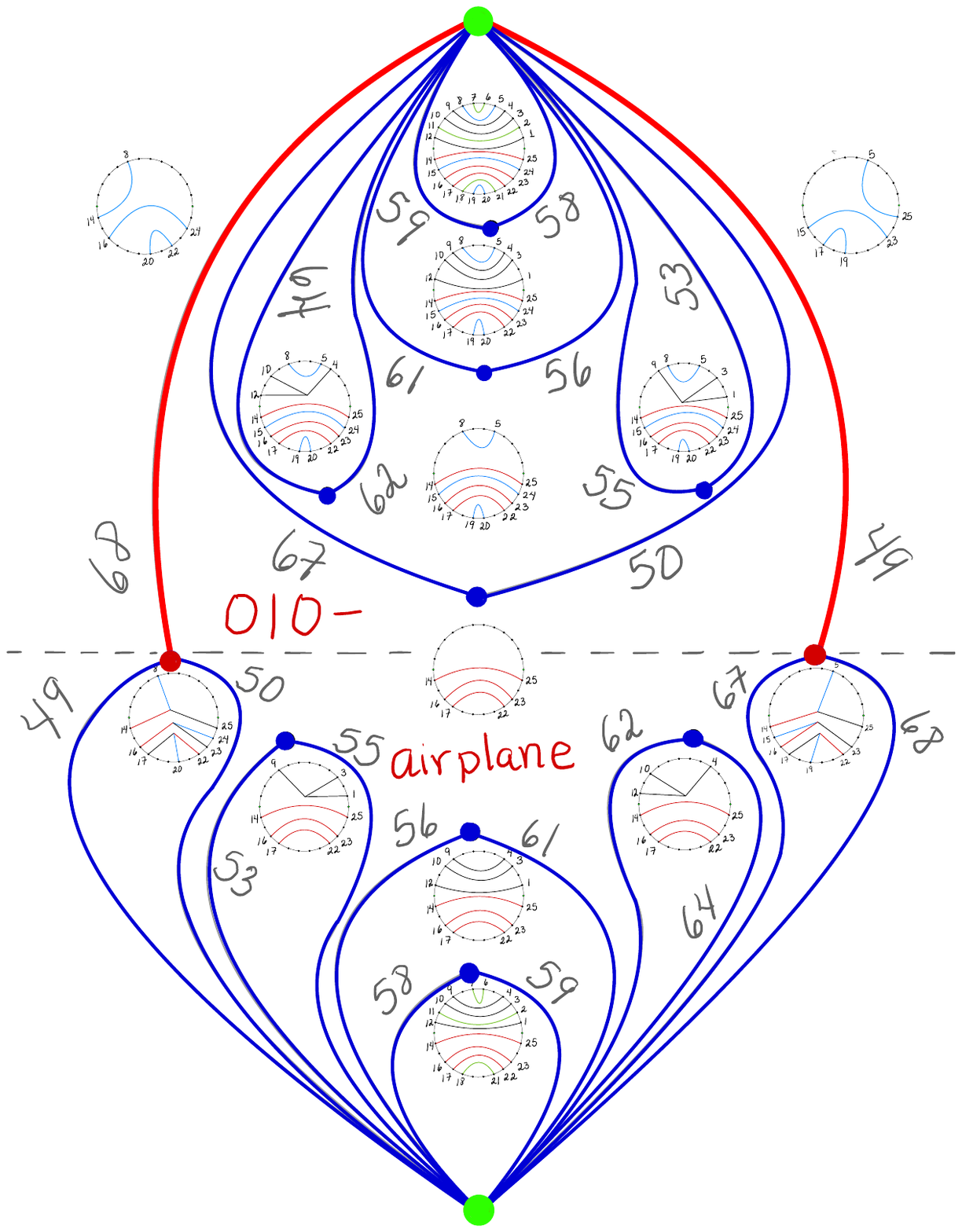}}
  \caption[Period $3$ orbit portraits of $\ocS_3$ between $010-$ and airplane]{\label{F-t3s3air-010} \sf Period $3$ orbit portraits in a part
   of $\ocS_3$ between the $010-$ escape region above, and the airplane
   region below. 
(Here $010$ is an abbreviation for the kneading invariant $(0,1,0)$.
   Compare Figures \ref{F-t1s3} through \ref{f-T3OP}.)
Note the five nested
 wakes above, and the two nested wakes among the six wakes below.
 Only the central orbit portrait represents a shared face. (Denominators 26
 and 78.)
}
\end{figure}

\begin{figure}[htb!]
  \centerline{\includegraphics[width=4.6in]{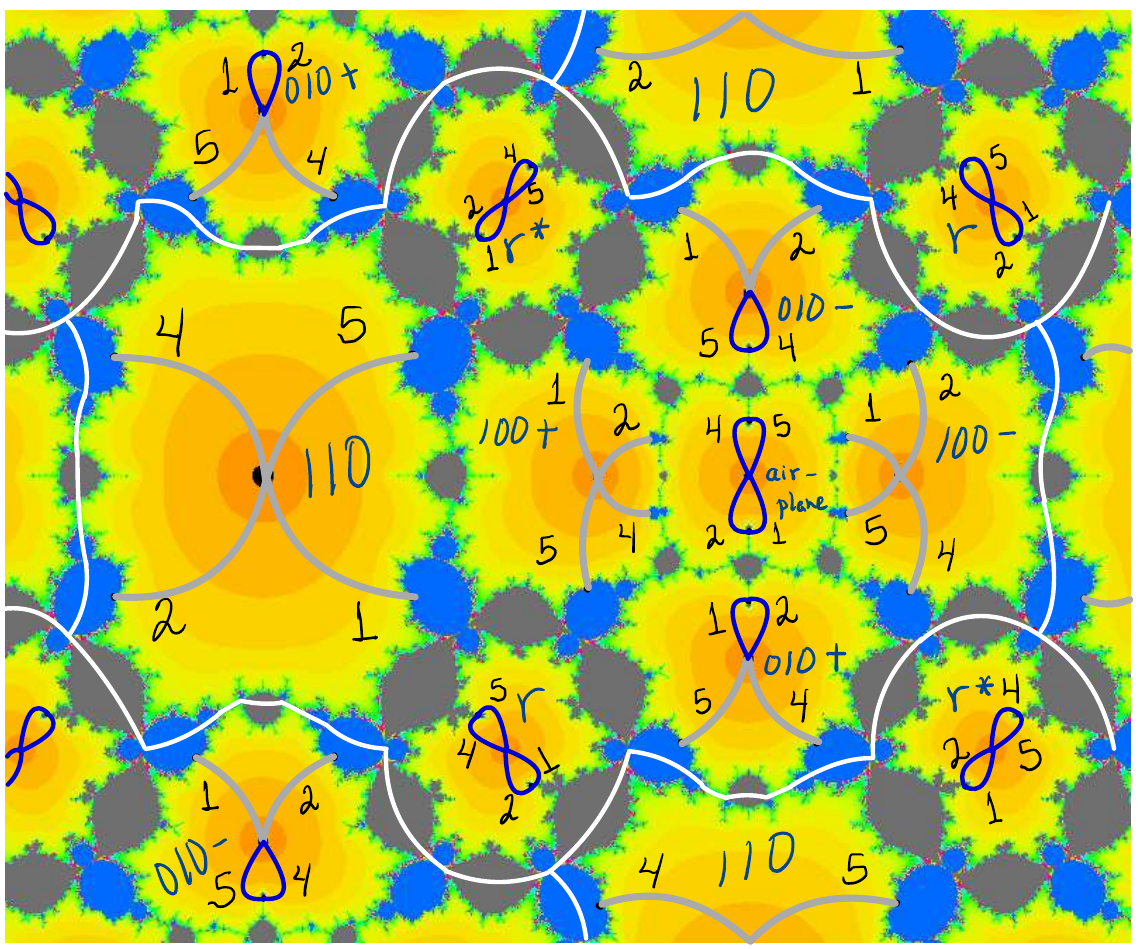}}
  \vspace{-.2cm}
  \caption[$\Tes_1(\ocS_3)$]{\sf The tessellation $\Tes_1(\ocS_3)$, lifted to
    the universal
    covering space of $\ocS_3$. (Compare \cite[Appendix]{BM}.)  A fundamental
      domain which contains one copy of each escape region
      is outlined in white,  with $110, ~100+$, the airplane, and
      $100-$ in the central row,
    with the co-rabbit and $010-$ above, and with the rabbit and $010+$ below.
(Note that this white curve cuts across faces of the tessellation.) 
 The orbit portrait is non-trivial for maps 
in each of the eight  simply-connected wakes, 
 but is trivial within the large non simply-connected shared face.
\label{F-t1s3}
}

\bigskip

  \centerline{\includegraphics[width=4.8in]{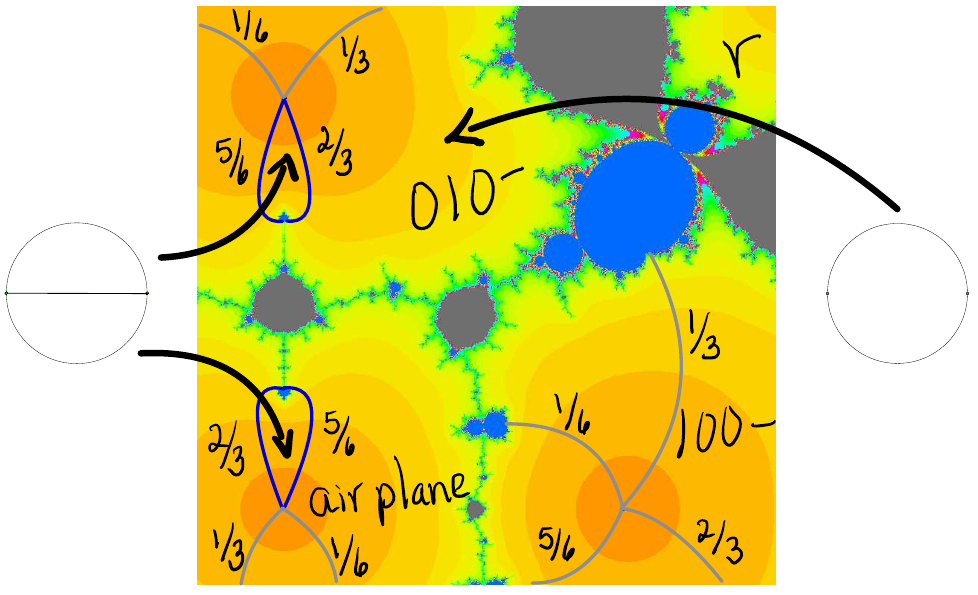}}
  \vspace{-.2cm}
  \caption[Part of period one tessellation of $\cS_3$]{\sf Magnified picture 
    of a region in the upper right of Figure \ref{F-t1s3}, showing the
    associated orbit portraits.\label{f-OPt1s3}  }
 \end{figure}

 \begin{figure}[htb!]
   \centerline{\includegraphics[width=4.5in]{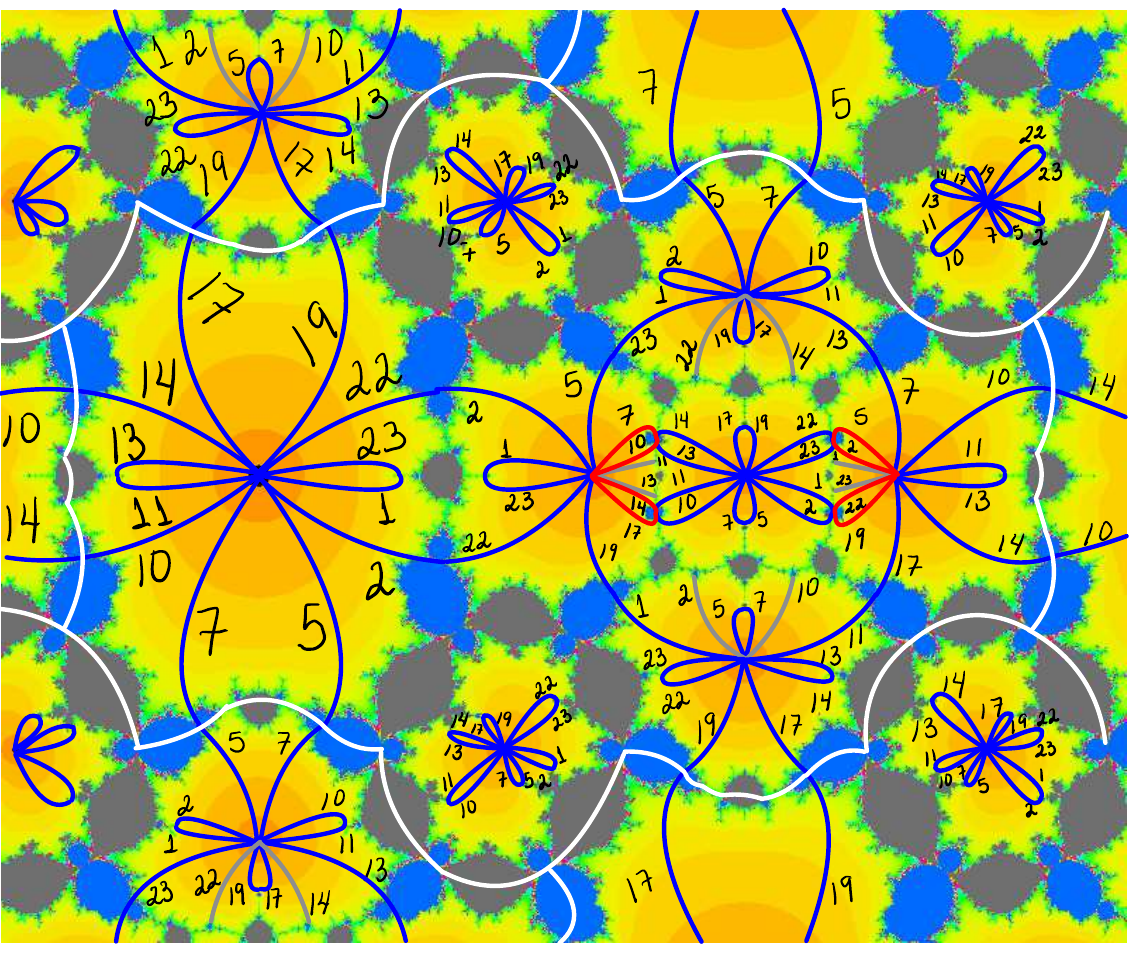}}
   \vspace{-.4cm}
  \caption[$\Tes_2(\overline\cS_3)$ ]{\sf The tessellation
    $\Tes_2(\overline\cS_3)$ contains $32$ wakes, together with $8$ 
    simply-connected shared faces, and two shared faces surrounding the 
    two rabbit regions which are not simply-connected. As usual, primary edges
    are blue and secondary edges are red.
\label{F-t2s3}}

  \medskip

  \centerline{\includegraphics[width=4.5in]{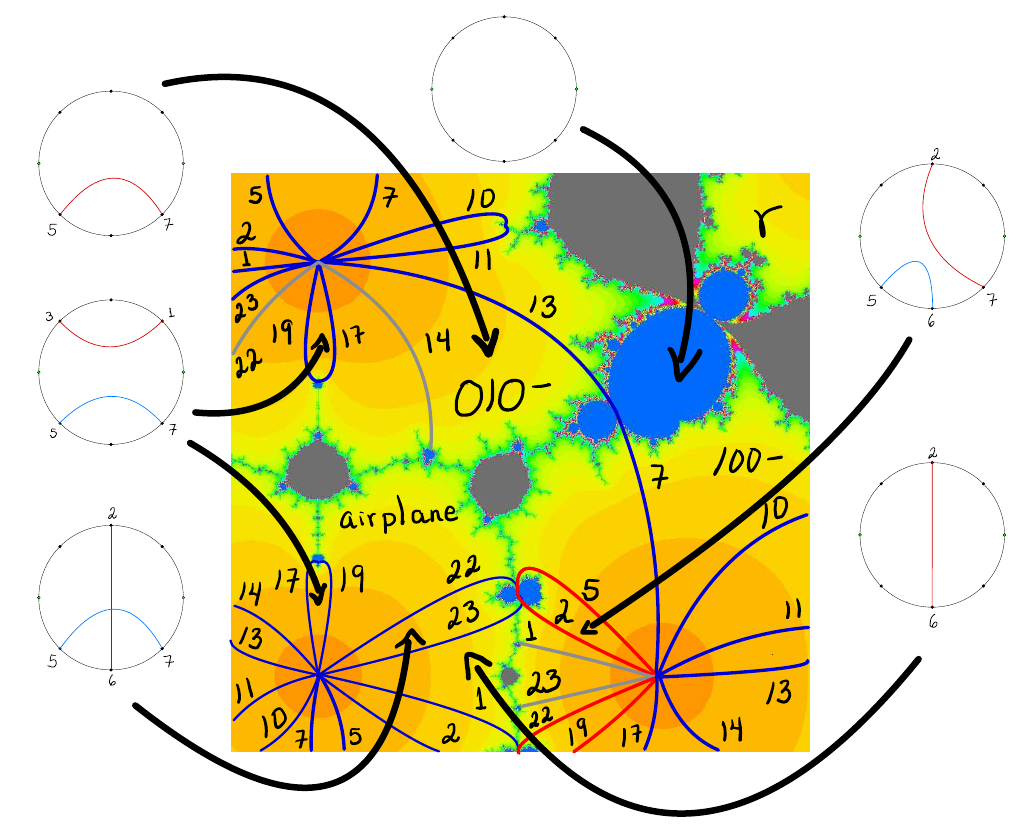}}
  \vspace{-.4cm}
  \caption[Part of period two tessellation of $\cS_3$]{\sf Again a 
    magnified picture of a region to the upper right, showing the associated
 orbit portraits. Each of these  non-trivial orbit portraits
    is found also    in $\Tes_2(\ocS_2)$, Figure~\ref{f-po}.
    \label{f-OPt2s3}}
\end{figure}

If we assume the Edge Monotonicity Conjecture \ref{CJ-main}, then it follows that 
every wake is either \textbf{\textit{primary}}, bounded by two primary edges,
or \textbf{\textit{secondary}}, bounded by two secondary edges. Furthermore
it follows that every primary wake contains a  component $H$ of Type D which
has $\p$ as its root point. (Similarly a secondary wake always seems to
contain a component of Type D with $\p$ as a boundary point; but never as
its root point.)

\begin{rem}\label{R-subwake} Still assuming  Conjecture \ref{CJ-main},
if $W$ is a subwake of some other wake $W'$, then 
$W$ is necessarily a primary wake, since the parabolic endpoint
$\p$ for $W$ has an open neighborhood $W'$ which intersects only one escape
region.\medskip

Evidently the orbit portrait is not constant throughout a wake unless it
is a minimal wake: The orbit portrait always becomes strictly larger every
time we pass to a sub-wake.
Examples are easy to find. Thus Figures~\ref{F-S2rays} and \ref{f-po}
clearly show twelve wakes, each consisting of a single face.\medskip

Although we are primarily concerned with wakes for a given period $q$, it is
certainly possible to consider nested wakes which do not 
have the same value of $q$. Such wakes of different period still have
the basic constraint that they must be either nested or disjoint. 
For example the period $q=1$ wake bounded by the $2/3$ and $5/6$ rays
in Figure~\ref{f2} contains the period $q=2$ wake bounded 
by the $17/24$ and $19/24$ rays in Figure~\ref{F-S2rays} or \ref{f-po}.
\end{rem}
\msk

\begin{rem}[{\bf Shared Faces}]\label{R-face-class} 
  Every tessellation contains two essentially different kinds of face.
  For $p>1$, the distinction is between  \textbf{\textit{shared faces}},
  which intersect two or more escape regions, and the remaining faces which
  intersect only one escape region. If we assume that any two edges in an
  escape region with a common landing point bound a wake (see Definition
  \ref{D-wake} $(3)$), then it is not hard to prove 
  the following:
  \begin{quote}\sf A face of $\Tes_q(\cS_p)$ with $p>1$ is shared if
 and only if it is neither a wake nor a sub-face of a wake.\end{quote}

\noindent As a very rough rule, the shared faces tend to have smaller
  orbit portraits, but larger numbers of edges. 

  The case $p=1$ is different. There is only one escape region, and
  every tessellation    $\Tes_q(\ocS_1)$ contains a unique 
``exceptional'' face, which contains the
map  $F(z)=z^3$, and which has trivial orbit portrait. Every other
  face is either a wake or a sub-face of a wake, and has non-trivial
  orbit portrait. This dichotomy is perhaps more visible if
  we invert $\ocS_1$, placing the ideal point at zero and the map $F(z)=z^3$
  at infinity, as in  Figure~\ref{F-t2s1}-right. \smallskip
  %\ref{F-t3s1a}.\smallskip

We can also distinguish between \textbf{\textit{shared vertices}}
which are boundary points of at least two escape regions, and
\textbf{\textit{unshared}} vertices,
which are boundary points of only one escape region.    \end{rem}
%\msk

\begin{rem}[\bf Not a wake]\label{R-not-wake}
  Note that a wake must be bounded by co-periodic
  parameter rays. A region bounded by two rays which are not co-periodic is
  not called a wake. As an example, the $4/18$ and
  $5/18$ rays in $\cS_1$ (the dotted white curves in  Figure \ref{F-t3s1})
land at a Misiurewicz map which  is the bottom boundary point
  of an oval shaped hyperbolic component of Type C. Since these rays and their
  landing point are all contained in the interior of a single face of
  $\Tes_3(\ocS_1)$, the period 3 orbit portrait does not change as we cross
  these rays. In fact a similar argument shows that the period $q$ orbit
  portrait does not change for any $q$. The Julia set for this landing
  point is shown in Figure \ref{F-cap-ara}.\end{rem}\smallskip

The rest of this section will discuss many detailed pictures of
tessellations; while the theoretical discussion will resume in Section
\ref{s-near-para}.\smallskip

\begin{rem}[{\bf The torus $\ocS_3$}]\label{R-S3} 
  Before describing tessellations of $\ocS_3$,
  it may be useful to describe this space briefly. (See 
  Figures~\ref{F-t3s3air-010} through \ref{f-T3OP}, and see
  \cite[Appendix D]{BM}
  for details.)  There are eight escape regions. Three of these
belong to the zero kneading 
 family of Remark~\ref{R-Mand}, and correspond
to the rabbit, co-rabbit, and airplane. There is also one escape region
with kneading invariant $110$, and two each with kneading invariants
$100$ and $010$. 
 The three involutions of Remark~\ref{R-sym} interchange these eight escape
 regions as follows. The airplane and $110$  regions map to themselves under
 all three involutions, and hence have a full group of symmetries. Each of the
remaining six is invariant under only one of the three involutions.
Thus the $180^\circ$ rotation maps each rabbit region to itself but
interchanges the remaining four in pairs, with
$$\xymatrix{010- ~~\ar@{<->}[r] & ~~010+} \qquad{\rm and}\qquad
\xymatrix{100- ~~ \ar@{<->}[r] &~~ 100+~}.$$
Similarly, complex conjugation maps each $010\pm$ region  to itself, but
interchanges the two rabbits and the two $100\pm$ regions; while
left-right reflection interchanges the rabbits and the $010\pm$ regions.
\end{rem}
\medskip

{\bf The tessellation $\Tes_1(\ocS_3)$.} %(Figures~\ref{F-t1s3} and
% \ref{f-OPt1s3}).
In Figure~\ref{F-t1s3} we show the period 1 tessellation of
the full torus; while Figure~\ref{f-OPt1s3} concentrates on an smaller region
in the upper right of the airplane region. In most of Figure~\ref{f-OPt1s3},
the period one orbit portrait is trivial; but inside the three wakes to the
left, the portrait is  $\{\{0,1/2\}\}$. The boundary rays  of these wakes are
primary rays  (Definition~\ref{D-3E});  but   the other rays shown are all
  non-separating.
  \smallskip
  
This tessellation  is similar to the period one tessellations of
$\ocS_1$ and $\ocS_2$ in the following sense: For
each of these cases, there is a one-to-one correspondence between
escape regions and connected components of the 1-skeleton, with each such
component contained in the closure of a corresponding escape
region.\footnote{It seems possible that the analogous statement is true 
for every $\Tes_1(\ocS_p)$. However, it would fail if there were a Mandelbrot
copy with cusp point of period one which lies along the boundary between two
escape regions (as in Figure \ref{f-M(53)bdry} for the period five case). If
this happens, then the two regions would be connected within $\Tes_1(\ocS_p)$.}
For any $\Tes_1(\ocS_p)$ where this is true,
it follows easily that the face $\F_0$ which separates the various
components of the 1-skeleton satisfies
$$ H_1(\F_0)~\stackrel{\simeq}{\longrightarrow}~ H_1(\cS_p) 
~\stackrel{\rm onto}{\longrightarrow}~H_1(\ocS_p)~.$$
For example,
if $p=3$, $H_1(\F_0)$ is isomorphic to the group $H_1$ of an $8$-fold
punctured torus, which has rank $9$. (The first puncture does not change the
rank; but each additional puncture adds one.)

If the escape region has multiplicity one (as it always does when $p\le 3$),
note that a component of the $1$-skeleton of $\Tes_1$
which is contained in its closure 
can only have the shape of an $X$, an $\boldsymbol\alpha$, or of a figure
eight.

\medskip

\begin{figure}[htb!]
  \centerline{\includegraphics[width=4.1in]{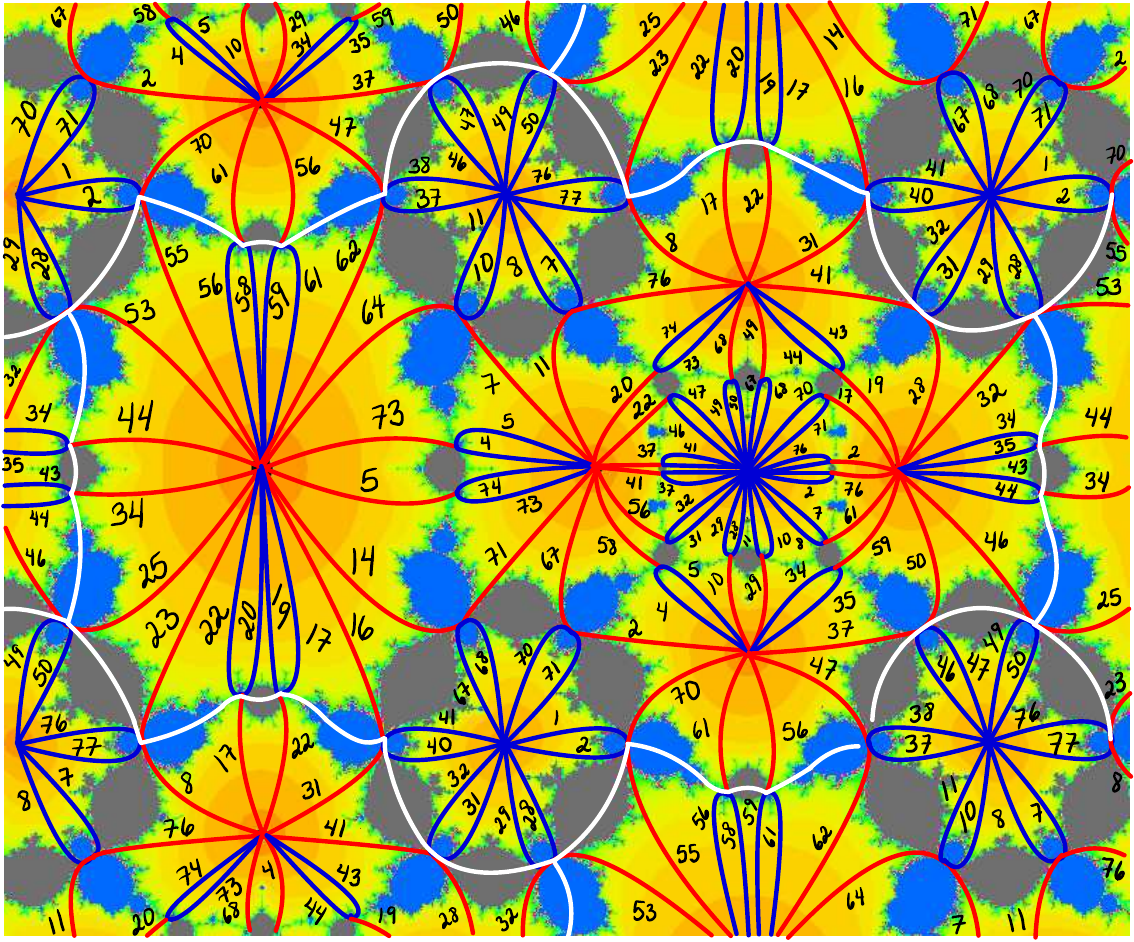}}
  \vspace{-.2cm}
  \caption[$\Tes_3(\ocS_3)$]{\label{F-T3S3}  \sf Simplified picture of
    $\Tes_3(\ocS_3)$.
    Only the rays which land on
    a common boundary point between two or more escape regions are 
    shown. 
   Here the common denominator for parameter angles is $78$.}
\vspace{-.5 cm}

  \centerline{\includegraphics[width=4.2in]{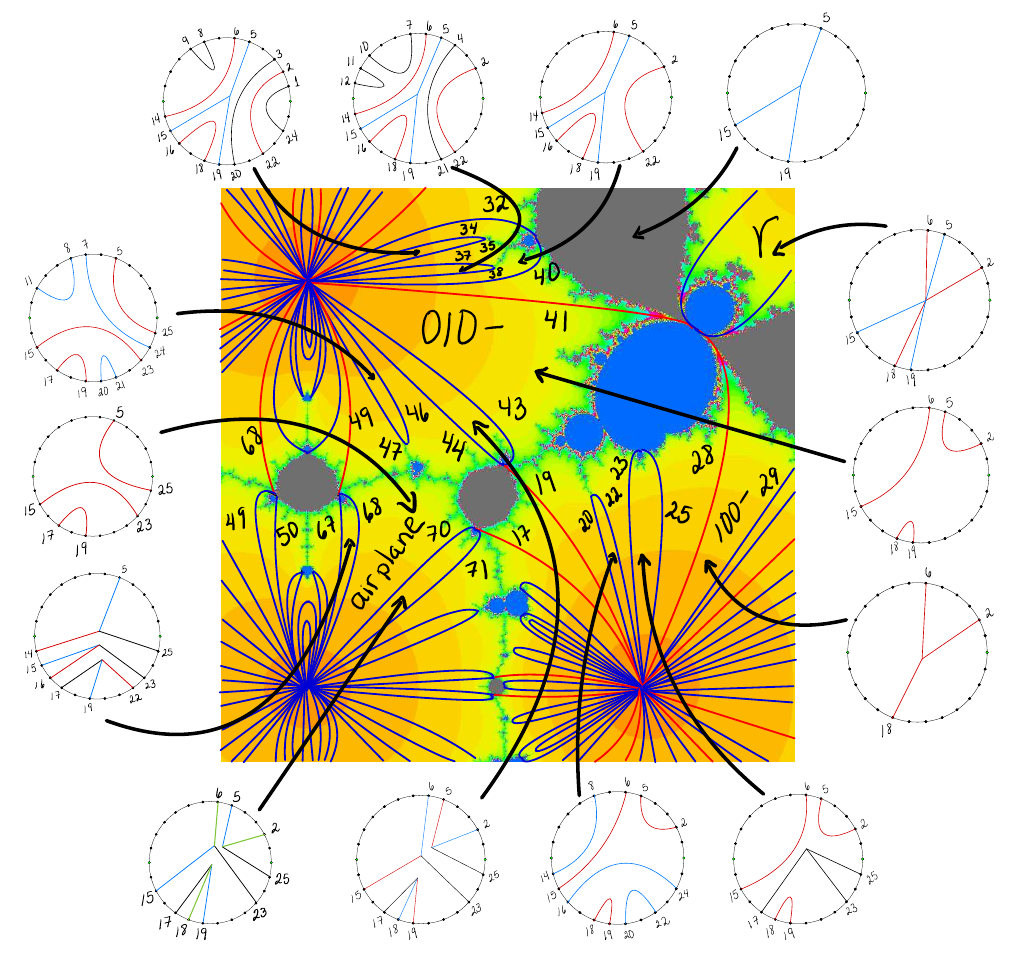}}
  \vspace{-.5cm}
  \caption[Part of period three tessellation of $\cS_3$]{\sf Period three 
    tessellation and orbit portraits  for an area to the upper right
    The wake in the rabbit region, and the wake which lies between
    the $43$ and $44$ rays in $010-$ provide examples    
for which the orbit portrait is the amalgamation of the orbit portraits
for the two neighboring shared faces.  $($See Proposition~\ref{P-amalg}.$)$ \label{F-2sf}}
  \end{figure}

\medskip

  \begin{figure}[htb!]
    \centerline{\includegraphics[width=4.5in]{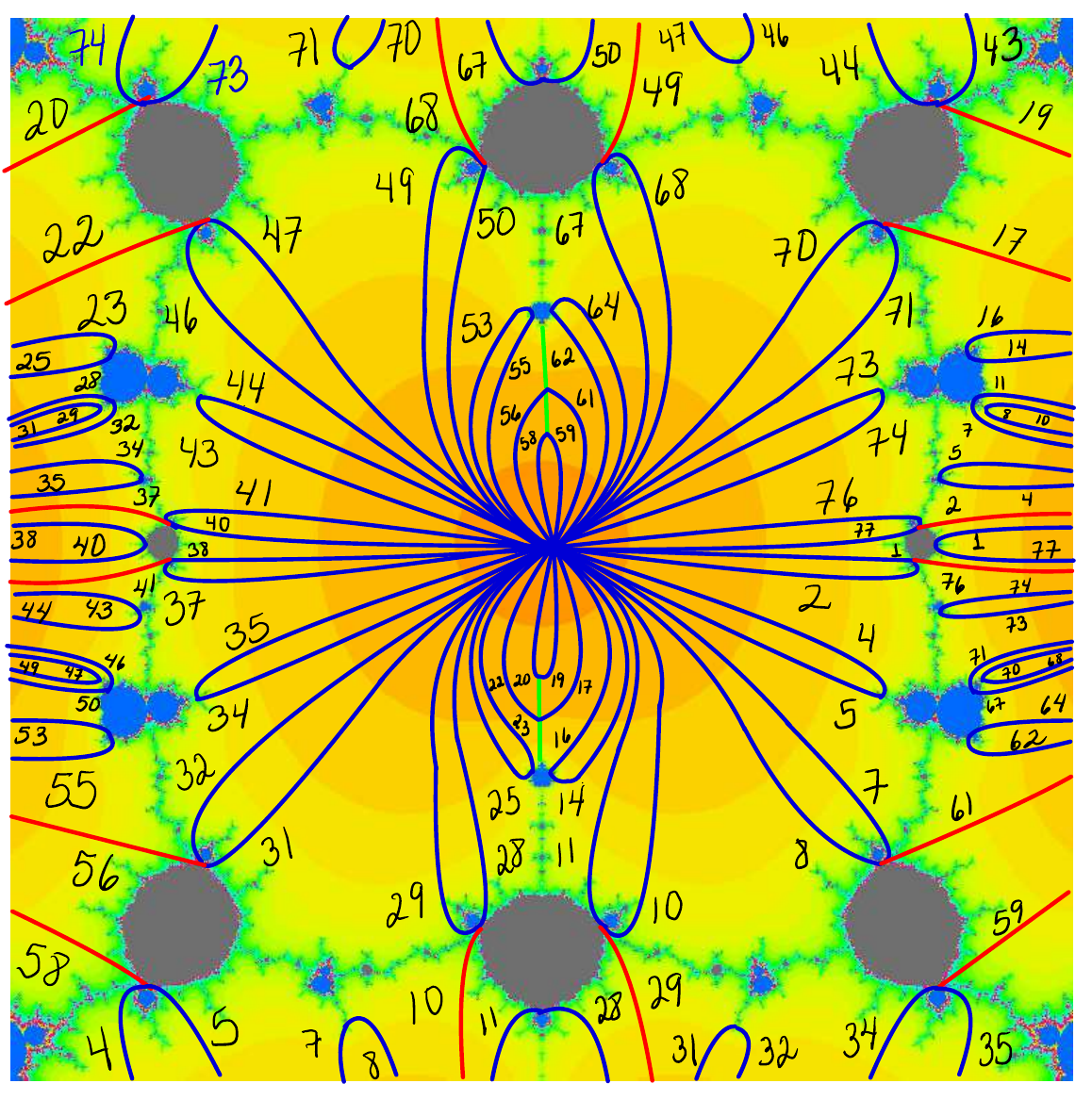}}
    \caption[$\Tes_3(\ocS_3)$-airplane]{\label{F-air} \sf Detail of the
      airplane region of
      $\Tes_3(\ocS_3)$ showing all of the co-period three rays.}
  \end{figure}\medskip

{\bf The tessellation $\Tes_2(\ocS_3)$.}
In Figure~\ref{F-t2s3} we show the 2 tessellation of the full
torus; while in Figure~\ref{f-OPt2s3} we concentrate on a smaller region in
the upper right of the airplane region. In this figure, the two
thin wakes on the left with co-periodic angles 17/24 and 19/24 are primary,
and lie inside the corresponding period one wakes 4/6 and 5/6 shown in
Figure~\ref{f-OPt1s3}. The fatter wake in the airplane region with angles
22/24 and 23/24, is also primary: Its orbit portrait is the amalgamation
of the orbit portraits
of the two neighboring faces (Definition~\ref{D-sum}). However, the nearby
wake in the $100-$ region is secondary. Note that the face with trivial orbit
portrait is not simply connected, and completely surrounds the rabbit region.

Consider the large Mandelbrot copy in $\cS_3$ which lies between the
$010-$, ~$100+$, and rabbit regions, as shown in these  figures. 
The cusp point of this copy is  near the bottom right of the
  copy, and is the landing point of a non-separating
  ray in the period one tessellation. (See Figures~\ref{F-t1s3} or
  \ref{f-OPt1s3}.)
  The $1/2$ limb  of this copy is to the lower left, and its root point is
  the landing point of two rays in the period two tessellation which are
  primary; but do not bound any wake. (See Figures~\ref{F-t2s3} or
  \ref{f-OPt2s3}.)
  
This tessellation has two faces which are not
simply-connected. In fact there is one such face surrounding each of the
two rabbit regions, and each of these has the topology of an annulus.

Here the common denominator for the edge angles is $3\,d(2)=24$ (as is true 
for any $\Tes_2(\ocS_p)$).
\medskip

  \begin{figure}[htb!]
  \centerline{\includegraphics[width=5in]{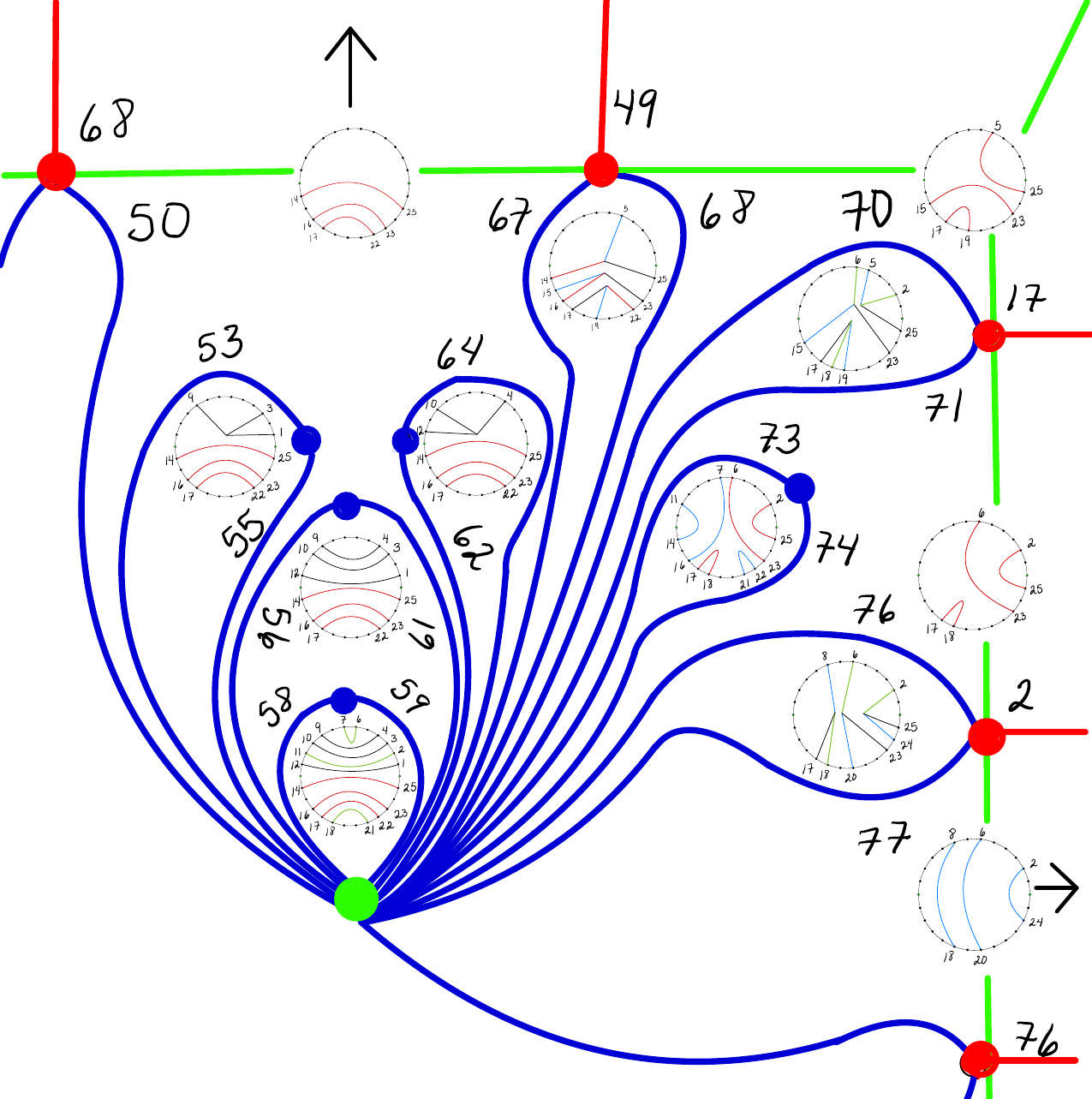}}
  \caption[Orbit portraits top right quadrant airplane region in $\Tes_3(\ocS_3)$]{\label{f-T3OP} \sf These orbit portraits correspond to faces
    in the top right quadrant of the airplane region of $\Tes_3(\ocS_3)$,
    with the $010-$ region  above and the $100-$ region to the right.
Compare Figures~\ref{F-t3s3air-010}, \ref{F-air}
  and \ref{F-2sf}.   Note that the three portraits on a vertical line
through the ideal point are invariant under left-right
  reflection. Corresponding portraits for the upper left  quadrant could
  be obtained by left-right reflection around this vertical line. Similarly
  the portrait on a horizontal line through the ideal point is invariant
  under up-down reflection (= complex conjugation); and portraits for
  the two lower quadrants could be obtained by up-down reflection in this
  horizontal line. Three of the wakes in each quadrant 
  border on two different
  shared faces, and have orbit portrait equal to the amalgamation of 
  the portraits for these two faces. (See Definition \ref{D-sum}.)}
\end{figure} 
\medskip

{\bf The tessellation$\Tes_3( \ocS_3)$} (Figures \ref{F-T3S3}, and  \ref{F-2sf}
through \ref{f-T3OP}). This is  
the most complex tessellation we have tried to understand in detail.
Figure \ref{F-T3S3} shows a simplified picture of $\Tes_3(\ocS_3)$, where
an edge of the tessellation is shown only if its parabolic endpoint belongs
to the boundaries of two (or three) different escape regions. This figure is
enough to show that every face of $\Tes_3(\ocS_3)$is simply-connected; but
of course does not give a complete picture of the tessellation.\smallskip

Figure~\ref{F-2sf} shows the orbit portraits in the area between the airplane
and the rabbit region.  The $1/3$-limb of the big Mandelbrot copy lying between
the $010-$ and $100-$ regions 
   lies in the rabbit region, and its root point is the landing point
  of four rays in the period three tessellation. The two rays in the rabbit
  region   are primary and bound a primary wake, while the other two are
  secondary  and do not bound any wake. 

Figure \ref{F-air} provides a complete picture for a large open subset of
$\Tes_3(\ocS_3)$ centered around the ``airplane'' escape region. There
are $36$ faces 
of the tessellation which intersect the
airplane region; but up to symmetry, as described in  Remark~\ref{r-mpt},
only eleven of these are distinct. In fact it suffices to study the eleven
which lie in the upper right hand quadrant of the figure.
\ssk

Figure \ref{f-T3OP} shows the orbit portrait for each of these eleven faces.
Four of these are shared faces (two shared with the $100-$ region, one with
the $010-$ region, and one shared with both). The remaining seven 
 consist of six wakes and one sub-face of a wake.
The shared faces all have orbit portraits of size three, although
each shared face has an edge in common with each of its neighboring 
shared faces.\footnote{It is
curious that the twelve shared faces around the airplane, together with
their common edges, form an annulus surrounding the ideal point.}
The remaining  faces (wakes and sub-faces of wakes) have portraits
of size either six or nine. The number of edges of these faces vary 
from two for the wakes and four for the
sub-face of a wake to six, eight, or twelve for the shared faces. (For shared
faces, we must also count edges outside of the 
of the airplane region, which are visible only in Figure \ref{F-air}.)

\bigskip

However further details definitely depend on the choice of escape region.
For example, corresponding faces in different escape regions often have
different orbit portraits. This is because the surroundings often
impose different background relations. (Compare Section \ref{s-near-para}.)

\setcounter{lem}{0}
\section{Around a Parabolic Point}\label{s-near-para}
This section will help to describe $\Tes_q(\cS_p)$ by giving a precise
description of the changes in the period $q$ orbit
portrait as we circle around a parabolic  vertex of the tessellation. 
% point of two or more parameter rays.
Much of the discussion will assume  that the Edge Monotonicity
Conjecture \ref{CJ-main} is satisfied. 

First a definition. If an orbit relation $\phi\simeq\psi$ is true for all maps
in a neighborhood of $\p$, then it will be called a \textbf{\textit{background
    orbit relation}} around $\p$. 
It is essential to specify the point $\p$ since different parabolic 
points will often have different background relations. 
As examples, in Figure~\ref{f-T3OP} the five points where
exactly two rays meet all have background relations, while the five points
where three rays meet do not.

Here we are concerned only with 
 orbits relations of period $q$. For $q'\ne q$ 
the point $\p$ will be an interior point of some face of the $q'$
tessellation, and the period $q'$ orbit portrait will be constant throughout
this face.
If the point $\p$ is contained in a wake $W$, then  every relation which holds
for all points of $W$ is certainly a background relation around $\p$. 

This section will often ignore background relations and concentrate on
the \textbf{\textit{distinguishing relations}} which hold 
for some maps near $\p$ but not for all.\medskip

{\bf Notation:} It will be convenient to use Greek letters, such as 
$\theta$, for parameter angles of co-period $q$, 
and to use the notation $\{\theta_j\}$ for
the associated cycle of periodic dynamic angles, where $j$ is an integer
modulo $q$, and where $\theta_j=3^k\theta$ for any integer $k\ge 1$ which
represents $j$ mod $q$. Thus if
$$ \theta=\frac{n}{3(3^q-1)}\quad {\rm then}\quad
\theta_1=\frac{n}{3^q-1}~.$$
\def\ul{\underline}
However we must be careful with this notation since
the correspondence  $\theta\mapsto\theta_j$ is two-to-one. Here is
a more precise statement.

\begin{lem}\label{L-2to1} 
For every periodic critical value angle $\theta_1=n/(3^q-1)$ there are
two possible choices for the associated co-periodic angle $\theta
=n'/3(3^q-1)$. The numerator $n'$ must satisfy
$$ n'\equiv n \quad ({\rm mod}~~ 3^q-1)~, \qquad{\rm and}\quad n'
\equiv \pm 1~~({\rm mod}~~ 3)~.$$
\end{lem}

Thus the two possible choices of $\theta$ correspond to the two possible
choices of sign $\pm 1$. 
The proof is easily supplied: The first congruence must be satisfied so that
$3\theta=\theta_1$, and the second must be satisfied so that $\theta$
is co-periodic rather than periodic. (Recall that the two choices are referred
to as {\textbf{\textit twin}} co-periodic angles. Of course both choices 
belong to the same grand orbit.) 
\qed\medskip

A completely equivalent statement would be that 
\begin{equation}\label{E-fig1}
  \theta~=~\theta_q~\mp~ 1/3~=~3^{q-1}\theta_1~\mp~ 1/3~,
\end{equation}
where $\mp$ is the opposite of the sign $\pm$ above. (Compare
Figure~\ref{Faa-1}.)
Here $3^{q-1}\theta_1=\theta_q$
is the angle of the dynamic ray landing at the free critical point,
since $F^{\circ q-1}$
maps the free critical value $v'$ to the free critical point $-a$. If we multiply
Equation~(\ref{E-fig1}) (which can be thought of as a congruence modulo one)
by $3(3^q-1)$ then we obtain the congruence
$$n'~\equiv~ 3^q n~\mp~ (3^q-1) ~~\big({\rm mod}~~3(3^q-1)\big)~.$$ This last
congruence must hold both modulo three and modulo $3^q-1$, yielding the
two congruences of the lemma. 
\bigskip

{\bf Face Notation.} If there are $n$ parameter rays landing at the parabolic
point $\p$, then the $n$ corresponding faces around $\p$ will be numbered in
counter-clockwise order as $\cF_k=\cF_k(\p)$ where $1\le k\le n$. Here the
\textbf{\textit{number one face}} $\cF_1$ will always be the one which contain the
unique component $H_\p$ of type D which has $\p$ as its root point. \bigskip

\begin{figure}[htb!]
\begin{center}
  \begin{overpic}[width=2.7in]{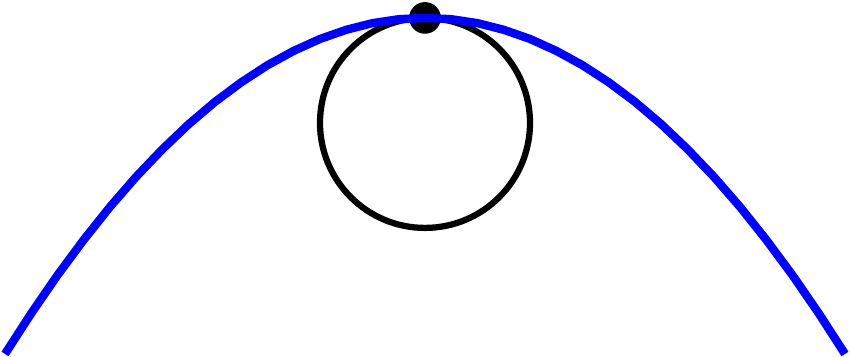}
    \put(45,80){${\mathcal F}_2$}
    \put(45,20){${\mathcal F}_1$}
    \put(89,50){$H_{\mathfrak p}$}
    \put(85,87){${\mathfrak p}$}
    \put(0,20){$\alpha$}
    \put(185,20){$\beta$}
 \end{overpic}
 \caption[Cartoon for the Two Ray Case]{\label{F-2ray} \sf Cartoon for the Two Ray Case}
 \end{center}
\end{figure}

\begin{figure}[htb!] 
  \centerline{
    \includegraphics[width=2.2in]{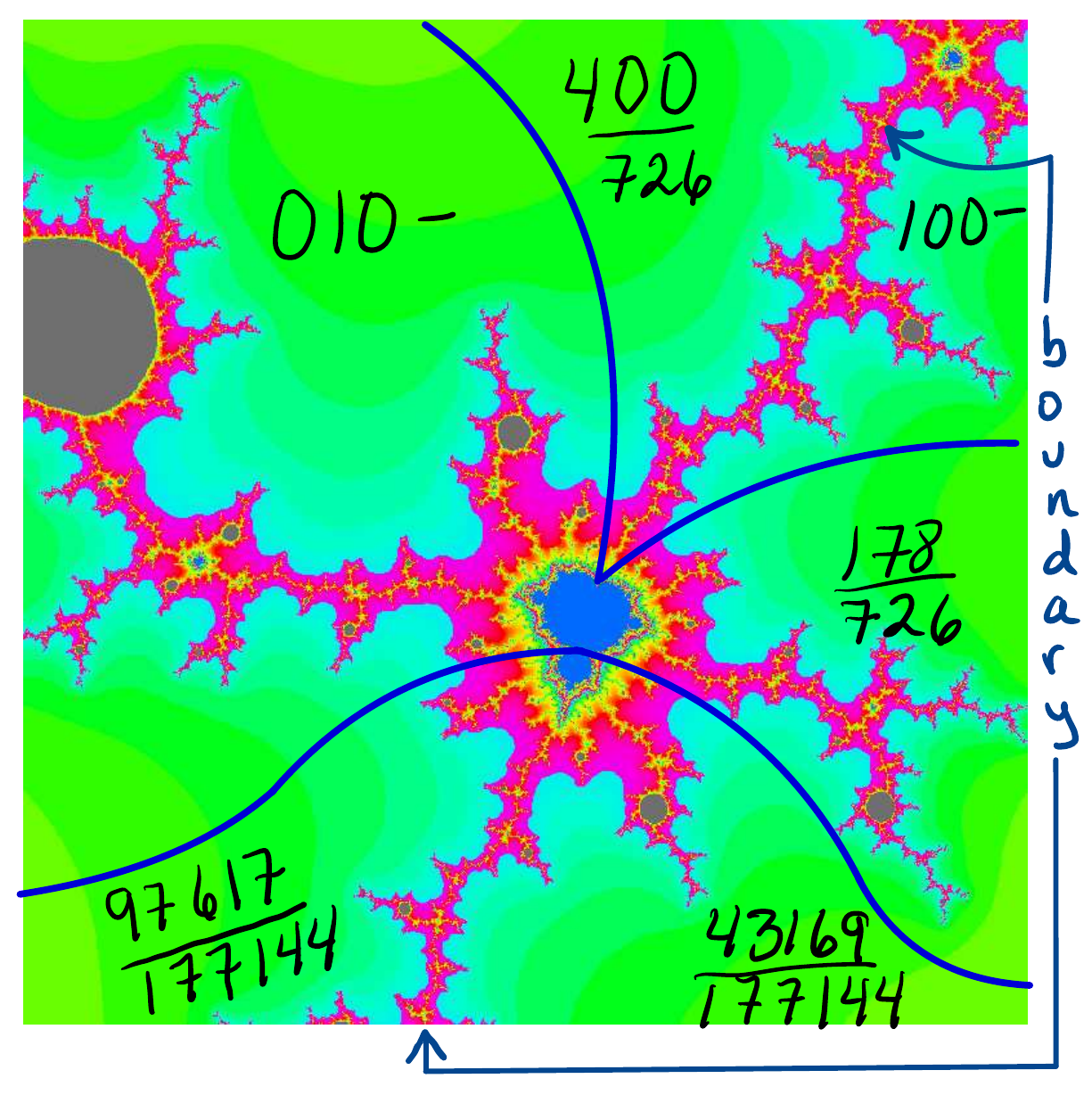}
   \includegraphics[width=2.2in]{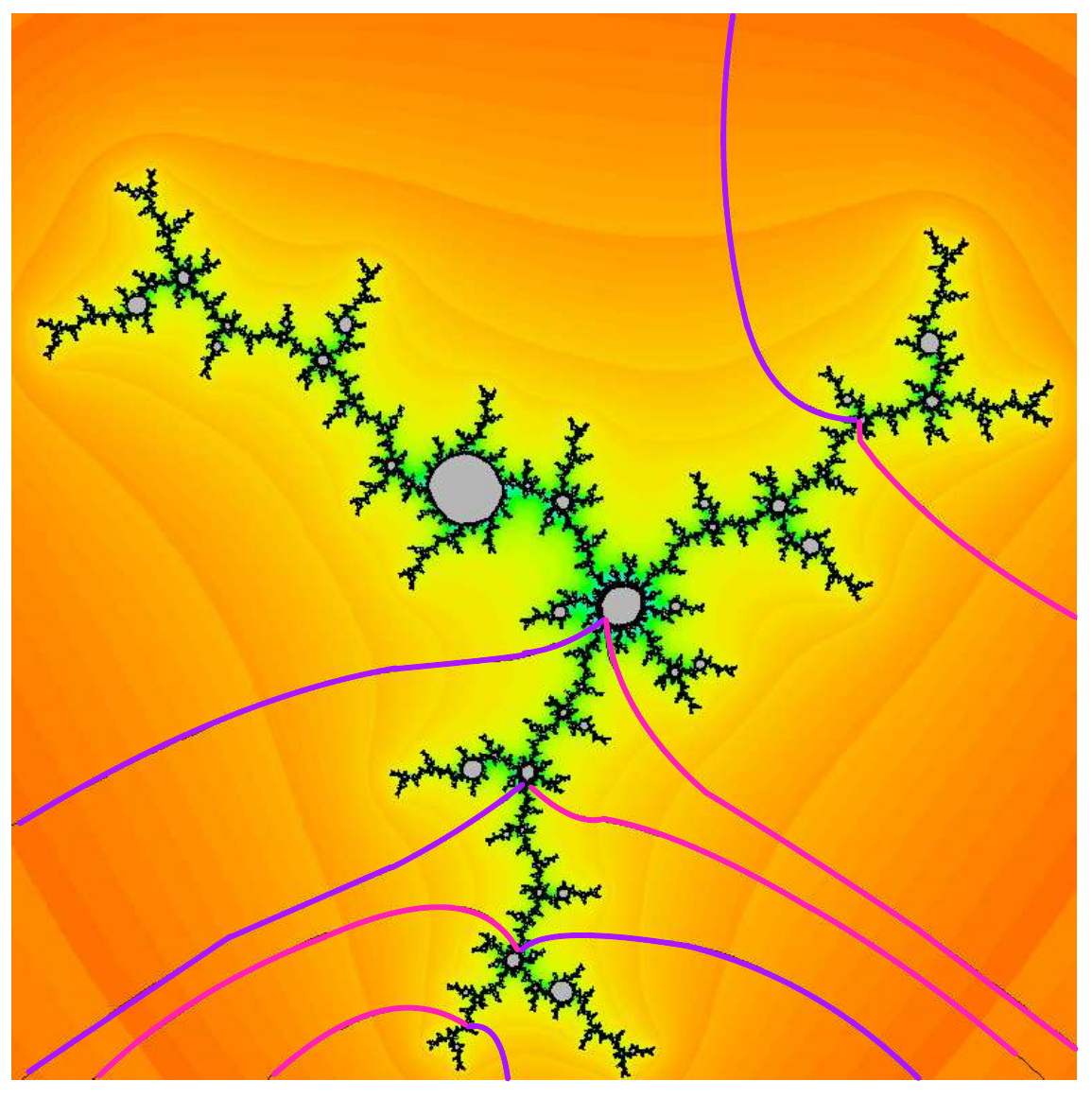}
    \includegraphics[width=2.2in]{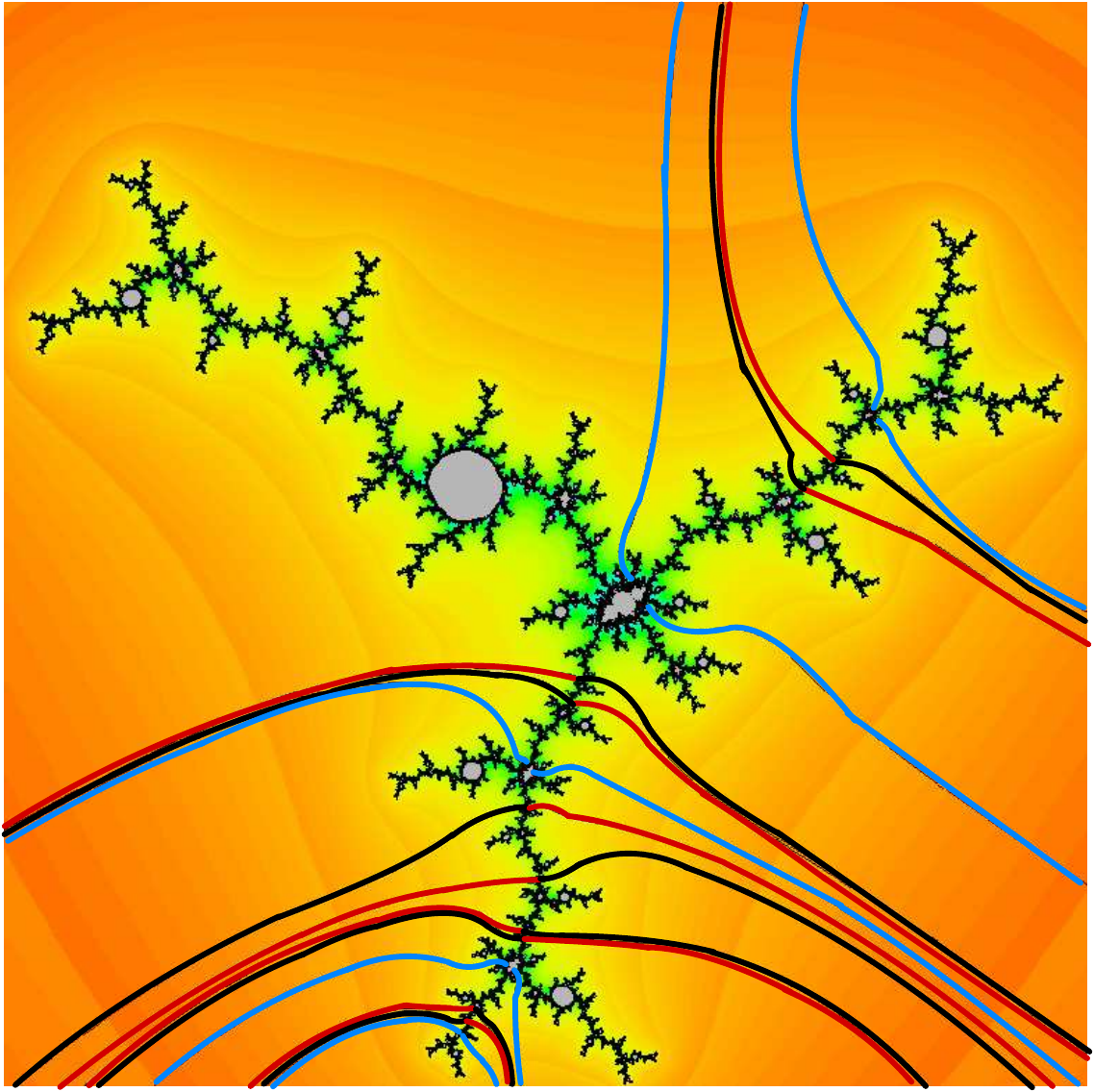}}
  \caption[Copy of Type M$(5,3)$ which lies \textbf{\textit{along}}  the boundary between the $010-$ region]{\label{f-M(53)bdry} \sf Illustrating the
    two ray case.  {\bf On the left:}  magnified picture of a
  Mandelbrot  copy of Type M$(5,3)$ which lies \textbf{\textit{along}}
  the boundary between the   $010-$ region on the left and  $100-$ region
  on the right, as in Figure \ref{F-J2K}. There are many 
  pairs of primary rays, one from the left and one from the right, which
  land together at points of $M$. We have drawn the pair of co-period 
  five landing at the cusp point, and the pair of co-period ten landing
  in the middle. (The others all have higher co-period.)\break
  \noindent {\bf Middle:} Julia set picture showing the period five orbit
  portrait for
  a map below the cusp point. (The corresponding portrait for points above the
  co-period five rays is   trivial.)\break
  \noindent {\bf Right:} The period ten orbit portrait has size ten above the
  co-period ten
  parameter rays and size fifteen below. This figure illustrates the larger
  orbit portrait. The rays in the size ten part, which is unchanged in the
  crossing,
  are colored red and black. Thus the landing orbit has period ten, and also
  ray period ten. The dynamic rays in the part which appears only below the
  co-period ten are colored blue. They land on an orbit which has period five,
  but ray period ten.\break
  \noindent ({\bf The most conspicuous smaller  Mandelbrot copies} in
  the left hand  figure are of Type M$(8,3)$ on the left, and M$(7,3)$
  upper right. The black disks, including the large one on the left, are
  capture components. This figure is located in the short stretch of
  common boundary between the  $010-$ and $100-$ regions in
  Figures~\ref{f-OPt1s3} through \ref{F-2sf}.)}
\end{figure}

\subsection*{\bf The Two Ray Case}

If only two parameter rays land at $\p$, the discussion is fairly easy.
We will label these two rays by their angles as
in Figure \ref{F-2ray}. Since the $\alpha$ and $\beta$ parameter
rays land together at $\p$, it follows as in Theorem \ref{T-edge}
that the $\alpha_j$ and $\beta_j$ dynamic rays land together for all
$F\in H_\p$ and hence for all $F$ in the face $\cF_1$. Thus we have the
orbit portrait relations 
\begin{equation}\label{E-2ray}
  \alpha_j\simeq \beta_j\quad{\rm for ~every~~} j \in \Z/q\end{equation}
throughout $\cF_1$. On the other hand, since the $\alpha$ and
$\beta$ edges are primary, the relations (\ref{E-2ray})
disappear when we cross to the face $\cF_2$.
%It follows that these relations, together with the background relations,
%generate all possible distinguishing relations. In fact,
In all the cases  we have observed, the relations (\ref{E-2ray}) are the only
distinguishing relations. (In principle there could be others. If there is a
background relation $\phi\simeq\psi$,  and if the landing point
%\rnote{Prove this is impossible ??}
of $\alpha_j$ jumped to the landing point of $\phi$, 
then two new distinguishing relations
$\alpha_j\simeq\phi$ and $\beta_j \simeq\psi$ would be generated. But this
 does not happen in any case we have seen.)

%\rnote{Jack, should we mention that in Fig. \ref{F-t3s3air-010}. The primary rays \{58, 59\} land together, belong to different orbits and  the orbits dissapear. The rays \{56, 61\} land together, belong to different
% orbits but the orbit relations  of the ray that we do not jump land together forming a tripod. A.}
\medskip

{\bf Caution.} This numbering of faces is only a local description with respect
to $\p$.
%There may be several different parabolic points in the boundary of any given face,
%and for example the number one face around $\p$ may well be the number two face around
%near the point $\p$. The face $\cF_0$ may have other parabolic boundary points.
%In this case it may happen that $\cF_0$ is still a primary face as seen
%from other boundary points; but this is not always true.
As an example, in Figure~\ref{f-T3OP} the face surrounded
by rays 56, 58, 59 and  61 is the number one face as seen from the parabolic
point above, but the number two face as seen from the parabolic point below.
 %(Four examples can be seen in Figure \ref{F-t2s3}.
%\note{With respect to: ``Four examples can be seen in Figure 34'' Are you talking about the circle of 4 rays encircling the airplane region? Should we say this
%  clearly? But this is not a primary wake! A.})
%But it can also happen that it is a secondary face as seen from other boundary
%points. This certainly happens in the case of nested wakes, as seen in Figures
%\ref{F-t3s3air-010} or \ref{f-T3OP}.
\medskip
\def\({{(\!(}}
\def\){{)\!)}}

Here $\{\alpha_j\}$ and $\{\beta_j\}$ may be the same periodic orbit. 
(Equivalently, the grand orbit $\( \alpha\)$ may be equal to the grand orbit
$\( \beta\)$.) For an example, see the
parameter rays with angle $\alpha=53/78$ and $\beta=55/78$ in
Figure~\ref{f-T3OP}. These correspond to dynamic angles $\alpha_1
=1/26$ and $\beta_1=3/26=\alpha_2$, in the same periodic orbit. It follows
that $\alpha_1\simeq\alpha_2\simeq\alpha_3$, so that there is an equivalence
class with three elements. 

On the other hand, the nearby parameter angles $\alpha=56/78$ and $\beta=61/78$
correspond to angles $\alpha_1=4/26$ and $\beta_1=9/26$, which are not in the 
same periodic orbit. In this case, each equivalence class has at most two
elements.

Although in many two ray cases the $\alpha$ and $\beta$ rays
form the boundary of a wake, this is not always the case.
See Figure \ref{f-M(53)bdry} for a configuration where these
two parameter rays lie in different escape regions.
\medskip

\begin{figure}[htb!]
  \begin{center}
    \begin{overpic}[width=2.5in]{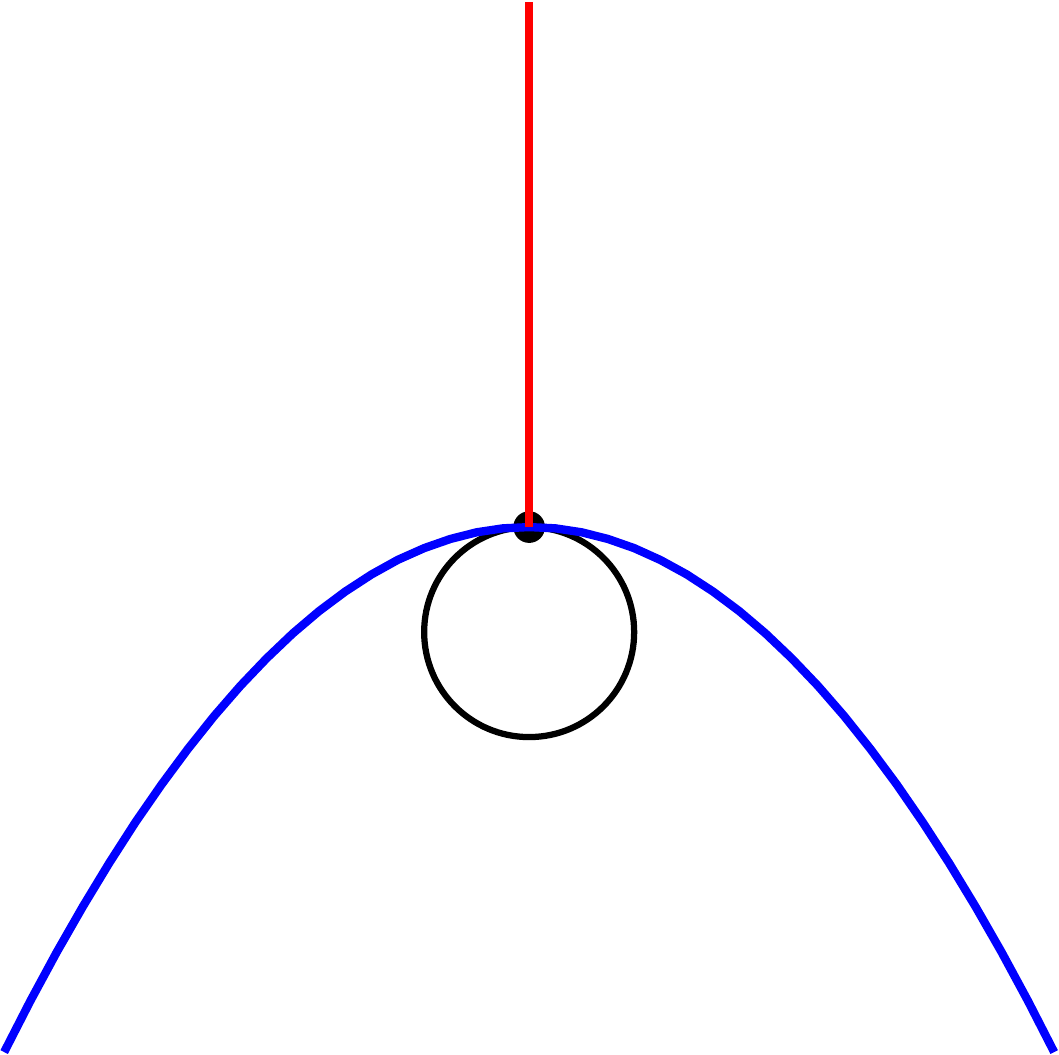}
      \put(-12,10){$\alpha$}
      \put(180,10){$\beta$}
      \put(50,20){${\mathcal F}_1$}
      \put(80,70){$H_{\mathfrak p}$}
      \put(95,95){${\mathfrak p}$}
      \put(30,150){${\mathcal F}_{3}$}
      \put(140,150){${\mathcal F}_2$}
      \put(75,130){$\gamma$}
      \end{overpic}
      \caption[Cartoon of the 3 ray case]{\sf Cartoon of the three faces
        surrounding a parabolic point where
  three rays land.  \label{F-rc3}}
\end{center}
\end{figure}
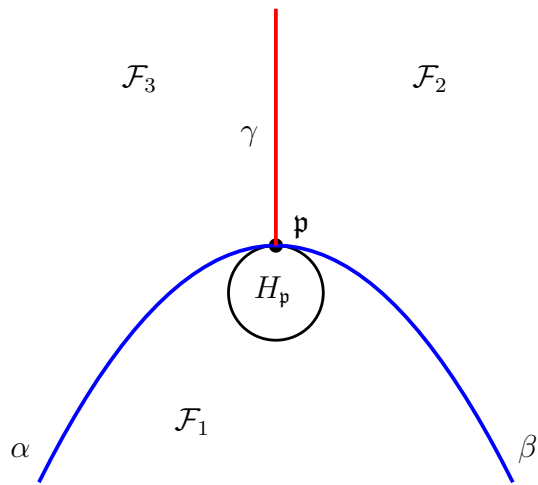

\subsection*{\bf The Three Ray Case.}
If there are three parameter rays of co-period $q$ landing at a point
$\p\in\cS_p$, labeled
as in Figure \ref{F-rc3}, then the period $q$ orbit relations in $\cF_1$ are clearly generated by
\begin{equation}\label{E-3R-1}
 \alpha_j~\simeq~ \beta_j~\simeq~ \gamma_j\qquad{\rm for~all}~~ j \in \Z/q~,
\end{equation}
(together with any background relations which may exist). 
{\sf However, in all of the three ray cases we have seen:}\medskip

 % {\bf(a)} $\p$ is a root point of a component of  Type A or B,% which
%  implies that $p=q$.
 % {\bf(a)} There are no background relations; and

  \begin{itemize}

  \item[{\bf(a)}] {\sf There are no period $q$ background relations at $\p$;}
    \smallskip

\item[{\bf(b)}] {\sf the grand orbits of $\alpha,~\beta$   and $\gamma$ are distinct:}
\begin{equation}\label{E-GO3}
\lg\alpha\rg ~\ne~ \lg\beta\rg~\ne~ \lg\gamma\rg~\ne~ \lg\alpha\rg~; ~~{\sf and}
\end{equation}

\item[{\bf(c)}] {\sf  $\p$ is the root point of a Mandelbrot copy of period $q$ 
  which is enclosed between the two primary rays with angle $\alpha$ and $\beta$.}
\end{itemize}
\medskip

\begin{figure}[htb!]
  \centerline{\includegraphics[width=5.3in]{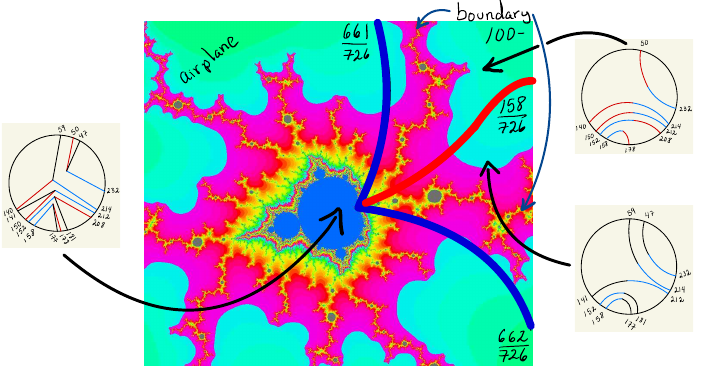}}
  \caption[Little Mandelbrot copy of period $q=5$ which lies in a
  primary wake in the airplane region of $\cS_3$.] 
  {\label{M53-rays} \sf A three ray example.     Magnified picture of a
    little Mandelbrot copy of period 5 which lies in a minimal primary 
    wake in the airplane region of $\cS_3$, but with cusp point on the boundary
    of the $100-$ region. Thus the cusp is the landing point of two primary
    rays in the airplane region and one secondary ray in the $100-$ region.
    Note that the orbit portrait in the principal hyperbolic component
    is the amalgamation of the two orbit portraits on the right.
    This figure is located near the top of the stretch of common boundary
    between the airplane and $100-$ regions which is visible 
    in Figures~\ref{f-OPt1s3} through \ref{F-air};  very close to the 
    ray of angle $71/78$ which lands at the bottom root point of the type A
    component in the upper right of Figure \ref{F-air}.}

  \bigskip

\centerline{\includegraphics[width=2.4in]{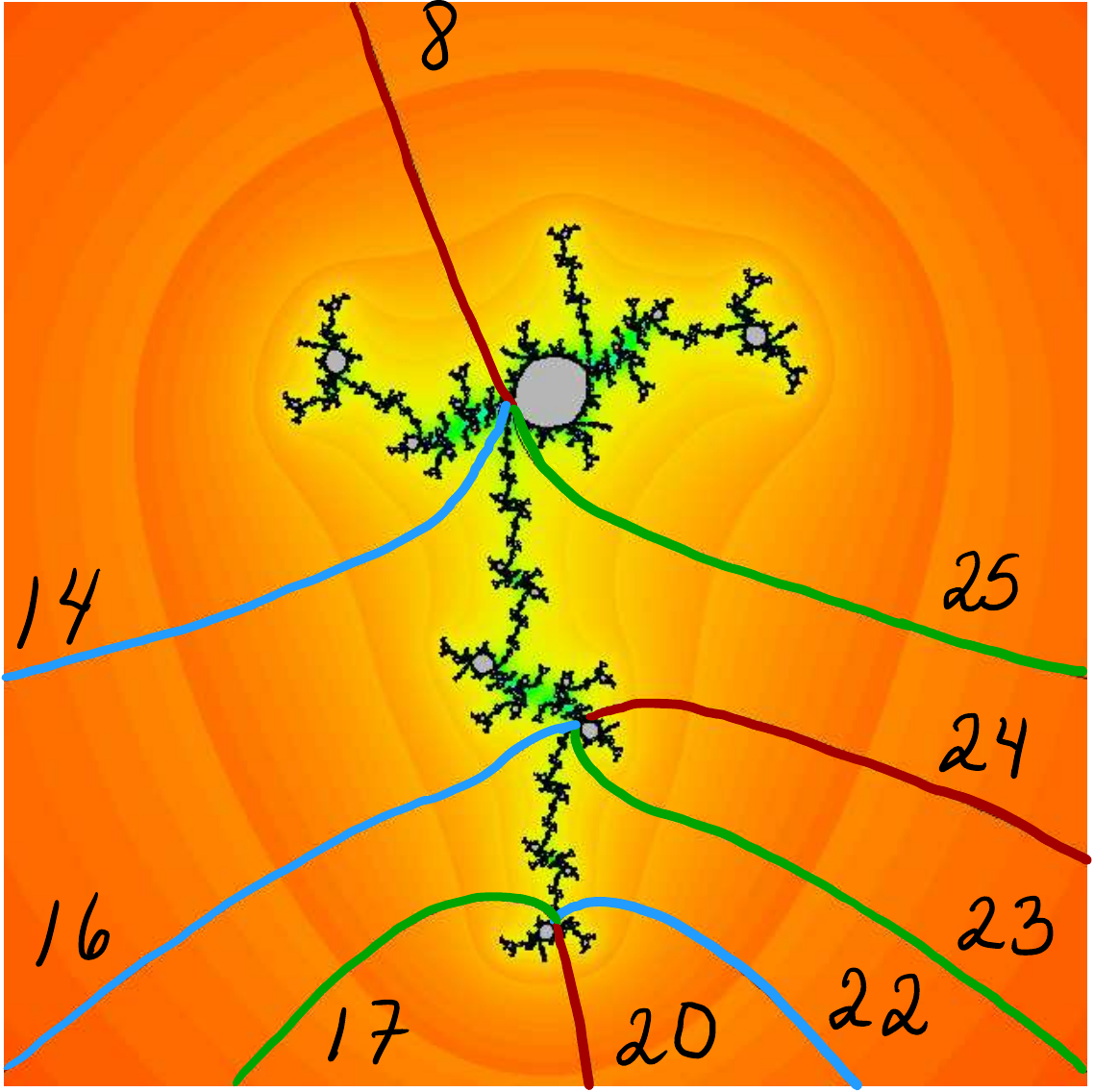}}
\caption[Julia set for 3-ray case  between the $110$ and $010-$ regions of $\Tes_3(\ocS_3)$]{\label{F-33jul} \sf Julia set for the $\cF_0$ face of a three ray
parabolic point which lies between the $110$ and $010-$ regions of
$\Tes_3(\ocS_3)$, with $\alpha=20/78,~\beta=22/78$ and
$\gamma=17/78$. This point is visible twice in Figure \ref{F-T3S3},
both to the
upper right and to the lower left. In this example, $\alpha$ and $\beta$ are
contiguous as co-periodic angles, since the intermediate angle $21/78$ is
not co-periodic. Note that $\alpha_1=20/26\mapsto\alpha_2=8/26\mapsto
\alpha_3=24/26$. Thus for example $-a$ is at the landing point of the 24 ray
while $v'=F(-a)$ is at the landing point of the 20 ray.}    \end{figure}

\begin{figure}[htb!]
  \centerline{\includegraphics[width=2.4in]{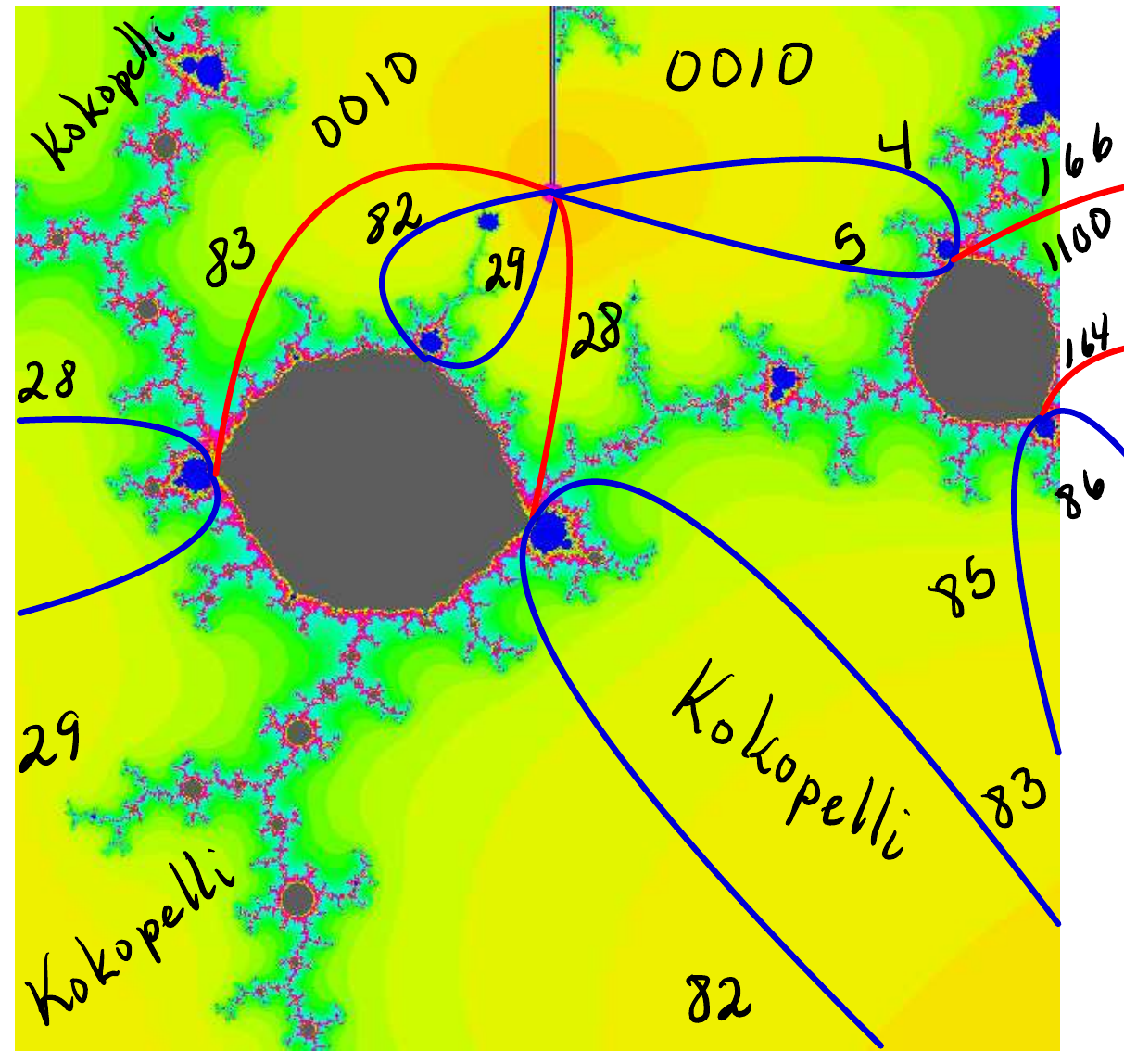}}
  \centerline{\includegraphics[width=4in]{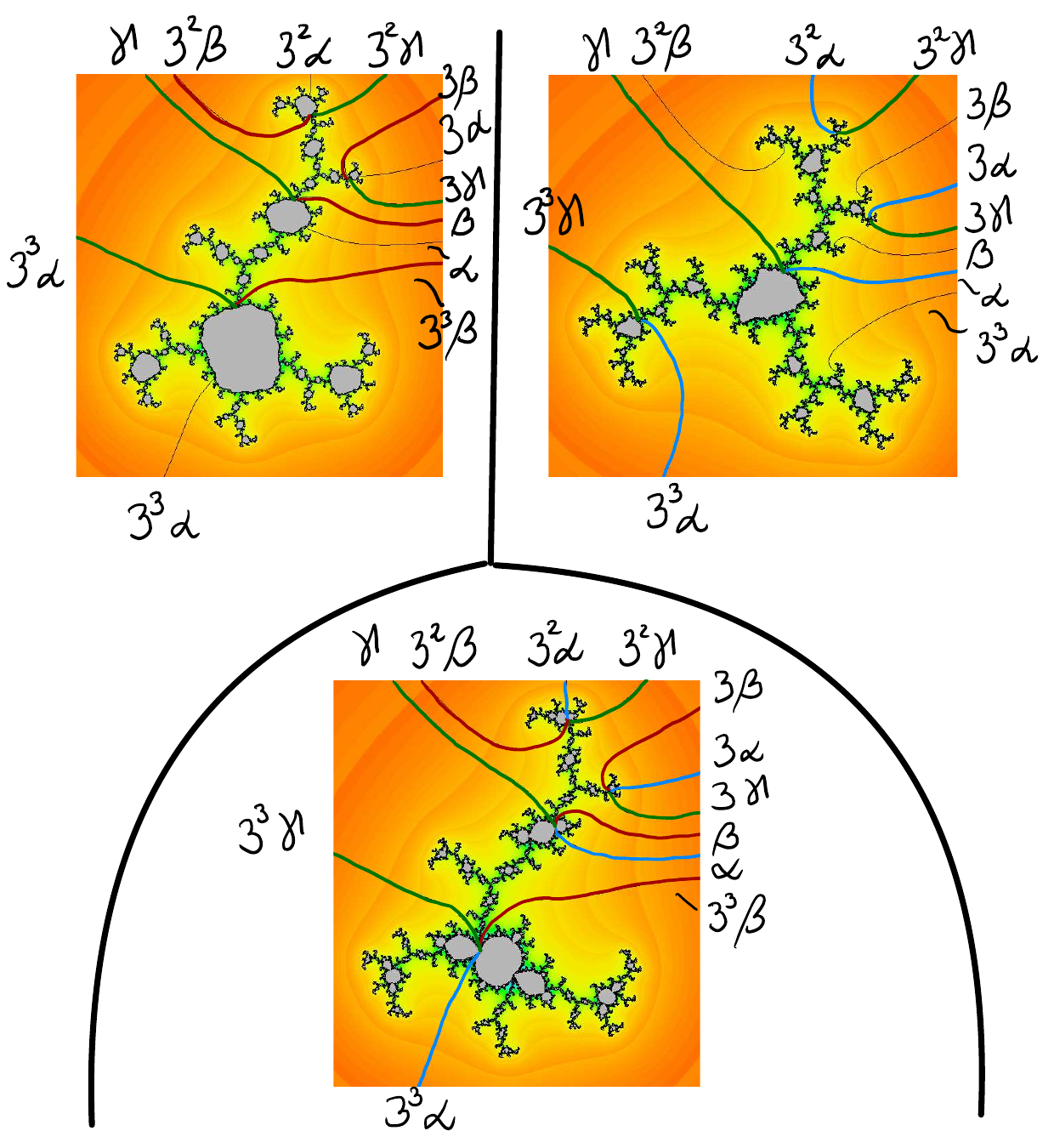}}
  \caption[3-ray case in part of $\Tes_4(\ocS_4)$]{\label{F-44}\sf  Above: a
 3-ray point  in $\Tes_4(\ocS_4)$. Below:
  Julia sets for the three faces around the parabolic vertex
near the center of the top figure. Here
  $(\underline\alpha, \underline\beta,~\underline\gamma)~=~(82, 83,28)/240$,~~
  and $~~(\alpha,~\beta,~\gamma)~=~(2,~3,~28)/80$.}
\end{figure}

\begin{lem}\label{L-3ray}
  It follows from {\bf (a)} and {\bf (b)} that the period $q$ 
  orbit portrait for the face $\F_1$
  consists of $q$ disjoint triples $\alpha_i\simeq\beta_i\simeq\gamma_i$.
  Similarly the orbit portrait for $\F_2$ consists of $q$ disjoint
  pairs $\alpha_i\simeq\gamma_i$, while the orbit portrait for $\F_3$
  consists of $q$ disjoint pairs $\beta_i\simeq\gamma_i$.
\end{lem}

\begin{proof} As base point in $\F_1$ we can take the center point of the
  component $H_\p$. Then the free critical point will be periodic of period
  $q$. The dynamic rays $\alpha_i, \beta_i,$ and $\gamma_i$ will
  land on $F^{\circ i}(-a)$ for $1\le i\le q$. Now as we cross the $\beta$
  ray into $\F_2$, the landing points of the $\beta_i$ rays and only these
  rays must jump. 
  But the Parabolic Stability Theorem of Appendix~\ref{a-para-stab}
  implies that the
  new orbit relations will be a subset of the old. The only way this
  can happen is for all relations involving the $\beta_i$ to disappear.
  The corresponding statement for $\F_3$ follows similarly.
    \end{proof}

    In particular, as expected, it follows that the orbit portrait changes
    monotonically as we cross the two primary rays, and non-monotonically
    as we cross the secondary ray. Here is another important consequence.

    \begin{coro}\label{C-amalg} It follows that the orbit portrait for
      $\cF_1$ is the
amalgamation of the orbit portraits of the two side faces $\cF_2$ and
$\cF_{3}$. \end{coro}

The proof is straightforward.\qed\msk

{\bf Examples.} Figure \ref{F-air} provides sixteen examples of three ray
points. Figure \ref{f-T3OP} shows five of these (three with a complete
set of orbit portraits).  The examples in this picture have the
property that  $q=p=3$,
with $(\alpha,~\beta,~\gamma)$ equal respectively to
$$ (67,68,49),\quad (70,71,17),\quad (76,77, 2)\quad{\rm over}~~ 78~.$$
Figure~\ref{F-2sf} contains  six similar examples.
 Figure~\ref{F-33jul} shows
 the Julia set of a period three example in $\cS_3$.
 
 Figure~\ref{F-44} provides an example with $p=q=4$, with
 $(\alpha,~\beta,~\gamma)$ equal respectively to $(82, 83, 28)$ over $240$,
 and with
orbit portrait of  the principal hyperbolic component:
$$\{\{3,28,2\},~\{9, 4, 6\},~ \{27,12, 18\},~\{1,36,54\}\}/80~.$$

 Figure~\ref{M53-rays} provides an example with $p=3$ and $q=5$
together with its orbit portraits.

\begin{rem}
The two primary rays always seem to belong to the same escape region; so that 
the primary face which contains {\sf all} of the associated Mandelbrot copy,
is all or part of a wake. In many cases this wake is minimal.
However in Figure \ref{F-Bcart}-right there two wakes of angular 
width $20/726$ attached to the bottom of the large B component; while in 
Figure~\ref{F-Boffcenter}-left there is one example of a wake 
of angular width $19/726$ attached to the right of the large A component. 
\end{rem}

%Notice that in most of these examples the root point of 
%the Mandelbrot set is also the root point of a component  type A or B.

\begin{quote}
  {\bf Conjecture.} \sf If $p=q$ then we conjecture that $\p$ is also a root
  point of a component of Type A or B which lies in one of the two side faces.
\end{quote}

\noindent Note that there can never be such an A or B component 
when $p\ne q$,
since the root point of an A or B component always has period $p$.
(Compare Figure \ref{M53-rays}.)

\begin{figure}[htb!]
  \centerline{\includegraphics[width=2.7in]{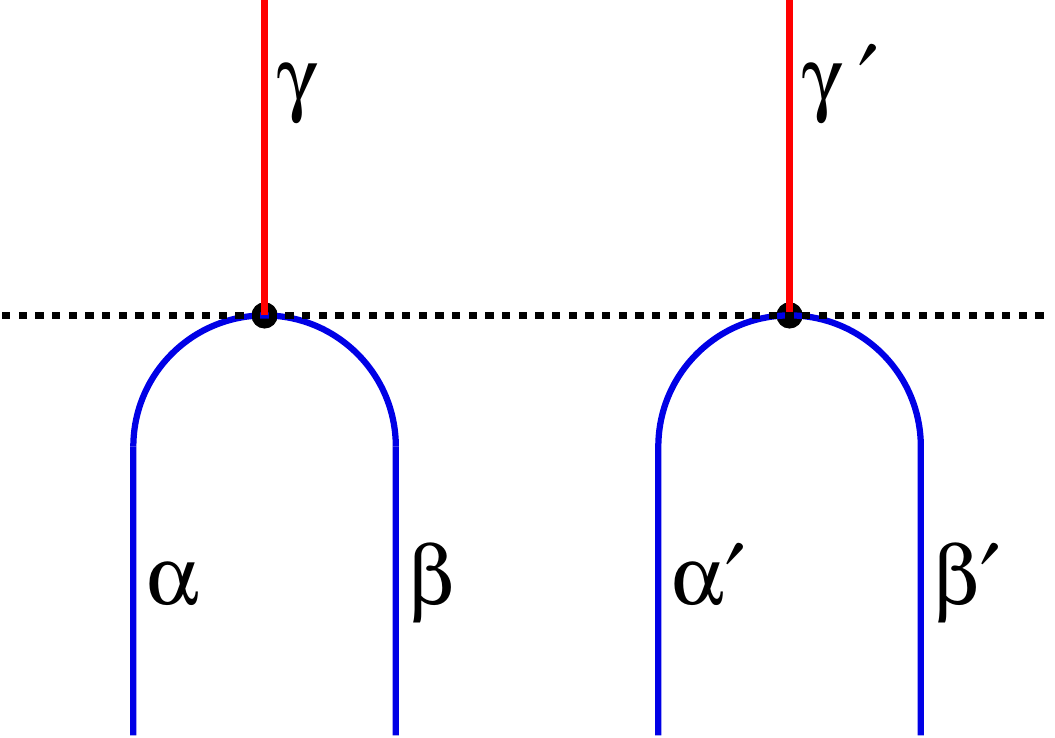}}
  \caption{\label{F-3rp}\sf Cartoon illustrating an adjacent pair of 3-ray
    triples, satisfying condition (1).}
\end{figure}
\bigskip

\begin{figure}[htb!]
  \centerline{\includegraphics[width=5in]{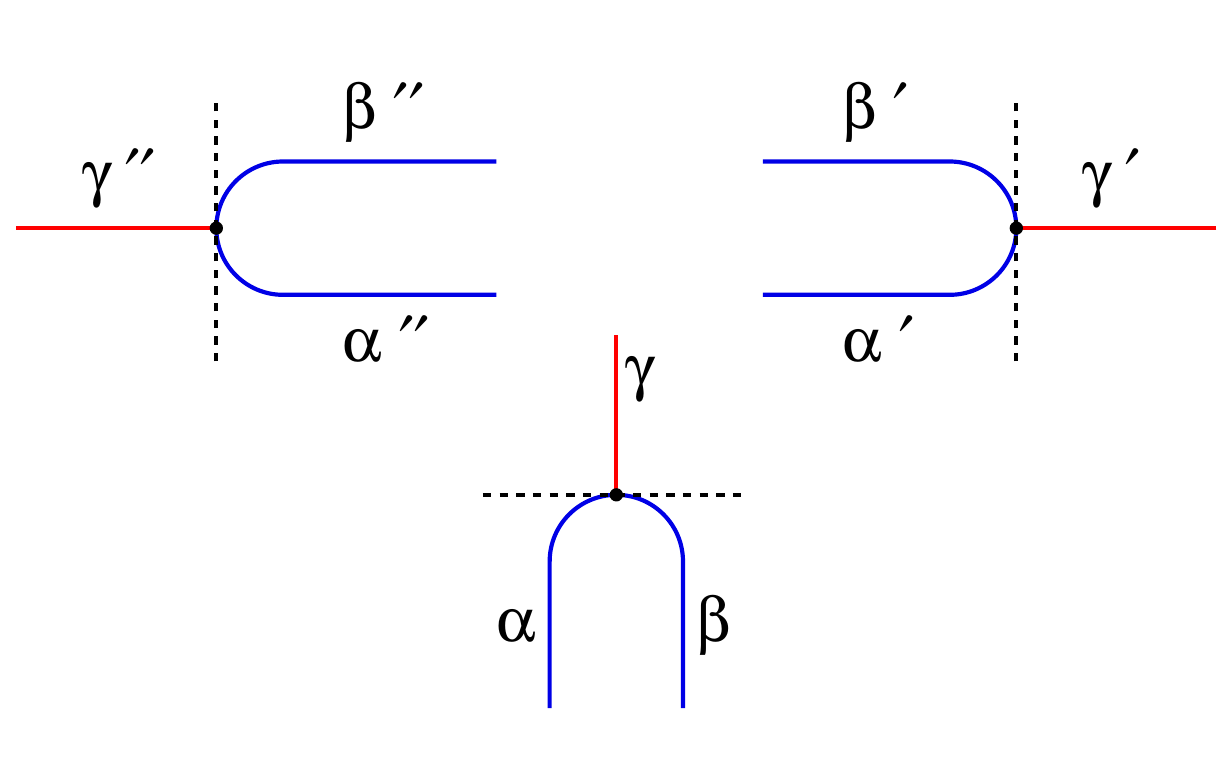}}
  \caption{\label{F-3rpa}\sf Cartoon illustrating two adjacent escape regions
 with 3-ray triples, satisfying condition~(3).}
\end{figure}
\msk

\noindent{\bf Neighboring 3-Ray Cases.} 
Now consider two adjacent 3 ray cases which are separated by a face which
is shared between two escape regions, as illustrated schematically in
Figures~\ref{F-3rp} and \ref{F-3rpa}. If we know the  three parameter angles
$\alpha,~\beta,~\gamma$,
then we can compute the orbit portraits of the three faces around the left
hand triple. Similarly if we know  $\alpha',~\beta',~\gamma'$,
then we	can compute the	orbit portraits	of the three faces around the right
hand triple. This means that 
the orbit portrait for the central face, with its 
dynamics, can be computed in two different ways. From the left it
is generated by the paired orbits
$$ (\alpha_1,~\gamma_1)\,\mapsto\, (\alpha_2,~\gamma_2)\,\mapsto
~\cdots~\mapsto~ (\alpha_p,~\gamma_p)~.$$
But from the right it is generated by the paired orbits
$$ (\gamma'_1,~\beta'_1)\,\mapsto\, (\gamma'_2,~\beta'_2)\,\mapsto
~\cdots~\mapsto~ (\gamma'_p,~\beta'_p)~.$$
Thus the unordered pair $\{\alpha_1, \gamma_1\}$ must be equal to 
$\{\gamma'_j,\,\beta'_j\}$ for some $j$.\bsk

{\bf Notation.}  It will be convenient to use the notation $\(\alpha\)$
for the grand
  orbit of an angle $\alpha$ of co-period $q$, recalling that each 
  such grand orbit contains a unique period $q$ orbit
  $\langle\alpha_1,~\alpha_2,~\cdots,~\alpha_q\rangle$,\quad   where
$\quad \alpha~\mapsto~\alpha_1~\mapsto~\cdots~\mapsto~\alpha_q\qquad
{\rm  under~ tripling} .$
\bsk

\begin{rem}\label{R-cond12}
\noindent{  \sf We believe that all four grand orbits can never be identical.
  Assuming this, it follows that either:}
\begin{description} 
\item[\qquad{\bf Condition (1):}]
  \quad $\(\alpha\) =\(\gamma'\)\quad\ne \quad
  \(\gamma\) =\(\beta'\)$,  \quad  {\sf or}
\item[\qquad{\bf Condition (2):}]
  \quad$\(\gamma\)=\(\gamma'\)\quad \ne\quad 
  \(\alpha\)=\(\beta'\)$, \quad {\sf or}
\item[\qquad{\bf Condition (3):}]
  \quad$\(\gamma\)=\(\gamma'\)\quad \ne\quad {\sf and} \quad
  \(\alpha\)=\(\alpha'\) \quad  {\sf or} \quad  \(\beta\)=\(\beta'\)$. 
\end{description}

\noindent{\sf In fact it seems that Condition (1) always applies in cases where
  the diagram is strictly like Figure~\ref{F-3rp},  meaning when the two
  primary wakes are in the same escape region. However we will see that with
  versions of the diagram which are only slightly modified, the case of
  Condition (2) can also occur.  (Figures~\ref{F-air} and
  \ref{F-Boffcenter}-left illustrate conditions (1) and (2)).
  Condition (3) is satisfied in the
case of Figure~\ref{F-3rpa}. (Figure~\ref{F-Bcart} illustrates Condition (3).)}

 Note that  Condition (1), exemplified in  Figure~\ref{F-3rp},
concerns only the four outer angles $\gamma$, $\alpha$, $\gamma'$ and $\beta'$.
Condition (3), exemplified in Figure~\ref{F-3rpa}, concerns the angles
$\gamma$, $\gamma'$, $\alpha$ and $\alpha'$ on the right; and
$\gamma$, $\gamma''$, $\beta$ and $\beta''$ on the left.
\end{rem}
 \msk

 There are three possible cases which correspond precisely to the diagram in
 Figure~\ref{F-3rp}:
Either there is a unique component of Type A or B occupying the middle of the
central face, or there is no component of Type A or B in this face. 
Condition (3) corresponds to the case in which there is no component of
Type A or B in the face.\msk

\begin{figure}[h!]
  \centerline{\includegraphics[width=2.8in]{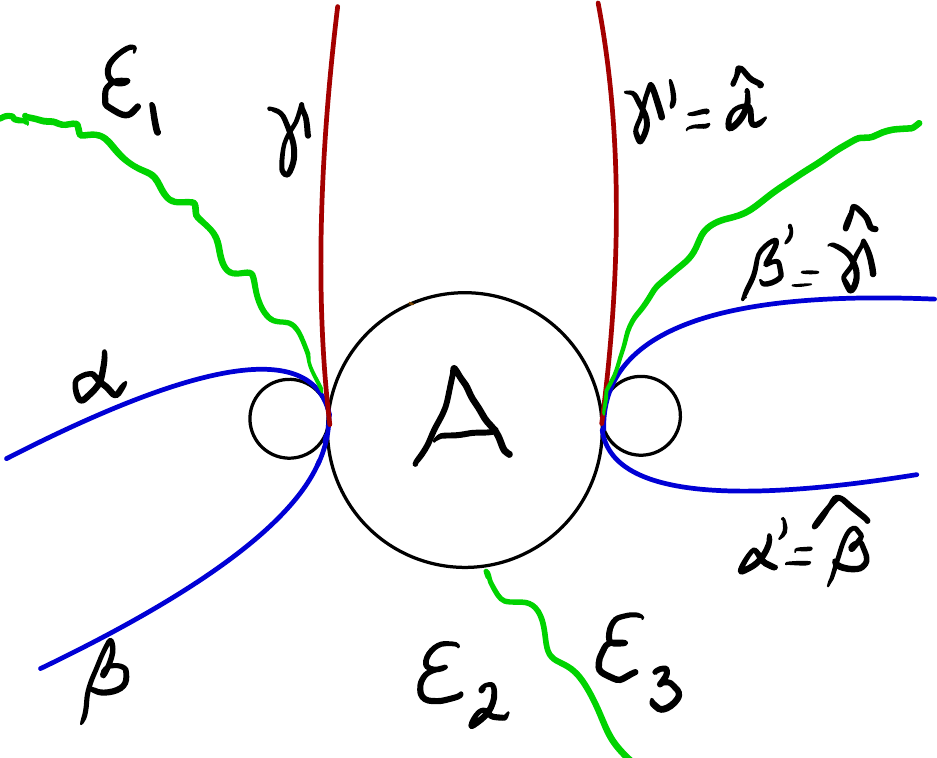}}
  \caption[A-centered face]{\label{F-Acent}\sf Cartoon of a central region
    dominated by
    a hyperbolic component of Type A. Here three rays with angles $\alpha,~
    \beta,~\gamma$ land at the left hand root point, and rays with the twin
    angles $\widehat\alpha,~\widehat\beta$ and $\widehat\gamma$ land in the
    same cyclic order around the right hand root point. (This configuration
    depends on the angles that we start with, if instead in 
    Figure~\ref{F-MC-crash} we set $\alpha=70/78$, $\beta=71/78$ and
    $\gamma=17/78$ we obtain $\alpha'=\widehat{\gamma}$,
    $\beta'=\widehat{\alpha}$ and  $\gamma'=\widehat{\beta}$.) }
  \end{figure}

  \noindent{\bf The A-Centered Case.~} In this case 
  the central face is centered at an A component, with three rays          
  landing at each root point.  (Compare Figure \ref{F-Acent}.)
  Since each root point of a Type A component in $\cS_p$ has  period $p$,
  we must have $q=p$. We conjecture that the three grand orbits
  $\(\alpha\),~\(\beta\),~ \(\gamma\)$  are always distinct, and that the
  cyclic order of $\alpha,~\beta,~\gamma$ is always the
  same as the cyclic order of their twins. %, we will prove the following.
  %is illustrated in Figure \ref{F-Acent}.

\begin{lem}\label{L-Apair}
  {\sf Assuming this conjectural description, the angles of the rays
    landing at each root point can be labeled as in
    Figure \ref{F-Acent}, with $\widehat\alpha=\gamma'$, $\widehat\beta=
    \alpha'$, and $\widehat\gamma=\beta'$. It follows that
    $$ \alpha_i=\gamma'_i\,,~~ \beta_i=\alpha'_i\,,~~{\rm and}~~
    \gamma_i=\beta'_i~.$$
    Therefore   $\(\alpha\)=\(\gamma'\)$,~~$\(\beta\)=\(\alpha'\)$, and
    $\(\gamma\)=\(\beta'\)$. Since we have assumed that $\(\alpha\)\ne 
    \(\gamma\)$, it follows that Condition (1) is satisfied.} 
\end{lem}

\begin{proof} From the three ray rule (Lemma~\ref{L-3ray}),
  we know that the orbit portrait
  of the left hand wake is the disjoint union of the relations
  $\alpha_i\simeq\beta_i\simeq\gamma_i$, while the orbit portrait of the
 central face is the disjoint union of the relations $\alpha_i\simeq \gamma_i$.
  Similarly from the right we see that the central orbit portrait is the
  union of the relations
  $\beta'_i\simeq \gamma'_i$. This is only possible if either
  $~~\beta'=\widehat\alpha$ and $\gamma'=\widehat \gamma~,~$ or
  $~~\beta'=\widehat\gamma$ and $\gamma'=\widehat\alpha$.  But in the first case
  the twins would be in the wrong cyclic order. This proves that
 $$\beta'=\widehat\gamma\quad{\rm and}\quad \gamma'=\widehat\alpha~,$$
 and consequently $\alpha'= \widehat{\beta}$ as asserted. \end{proof}\msk

Here is an example in $\cS_3$. 
For the A component in the upper right of Figure~\ref{F-air} (type A component
in Figure~\ref{F-MC-crash}) setting
$$ \alpha=43\mapsto \langle 25,23,17  \rangle ,\quad \beta=44\mapsto
\langle 2,6,18  \rangle,
\quad \gamma= 19 \mapsto \langle 5,15, 19  \rangle~,$$  
the twin angles having the same forward orbits and the same cyclic order
are \break
$\widehat\alpha=17,~\widehat\beta=70$ and $\widehat\gamma=71$. 
Here the denominator is  $3d(3)=78$ for $\alpha$, $\beta$ and $\gamma$;
and  $d(3)=26$ for the corresponding periodic angles.

As an example in $\cS_4$, in the top right of Figure~\ref{F-44} with denominator
  $3\,d(4)=240$, the twins of $4,~5,$ and $166$ are $\widehat 4=164$, 
  $\widehat 5=85$ and $\widehat{166}= 86$.

An example in $\cS_5$ is given by the large A component in 
Figure~\ref{F-Boffcenter}. In this case we have
%the component of type A is located towards the middle of the picture and we have
$$ \alpha=64\mapsto \langle 192,92,34,102,64  \rangle ,\quad \beta=65\mapsto \langle 195,101,61,183,65  \rangle,
\quad \gamma= 568 \mapsto \langle 10,30,90,28,84  \rangle~,$$   with
twin angles  $\widehat\alpha=548, ~\widehat\beta=307,$ and $\widehat\gamma
=326$. The co-periodic denominator is  $3\,d(5)=726$, while the
periodic denominator is  $d(5)=242$.\msk

\begin{figure}[htb!]
  \centerline{\includegraphics[width=2.7in]{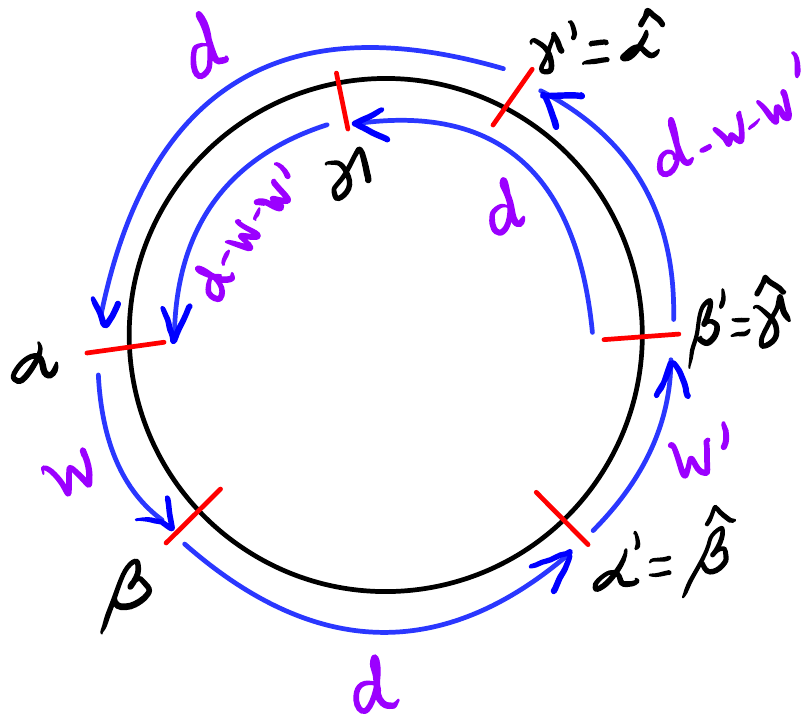}}
  \caption[Angular distances around an A-centered face]{\label{F-cartaw}\sf Cartoon illustrating angular distances around
    an A-centered face, following the notation of Figure~\ref{F-Acent}.} 
\end{figure}

\begin{rem}\label{R-angwidth} 
  Recall that parameter angles can be thought of as integers modulo $3\,d$
  where $d=3^q-1$. By the \textbf{\textit{angular distance}}
  $\Delta=\Delta(\theta, \eta)$,
  measured in the positive direction, between two angles $\theta$ and
  $\eta$ of co-period $q$   we will mean the unique integer $0\le\Delta<3\,d$
  such that $\Delta\equiv \eta-\theta$ modulo $3\,d$.

  By the angular \textbf{\textit{width}} $w$ of the $(\alpha,\beta)$-wake we will
    mean the angular distance $\Delta(\alpha,\,\beta)$.
    In all  of the examples we have seen, this width satisfies $w\equiv 1$
    modulo three, and $w<d$. This implies
  that $  \alpha\equiv 1$ and $\beta\equiv 2$ modulo 3. Therefore
  $$\widehat\alpha=\alpha-d\qquad {\rm but}\qquad \widehat\beta=\beta+d~,$$
  so that $\Delta(\widehat\alpha,\,\widehat\beta)=2\,d+w$.
    It follows easily that
$$\Delta(\beta',\,\gamma')~=~\Delta(\widehat\gamma,\,\widehat\alpha)~=~d-w-w'~.$$
  % angular distance between $\beta'$ and $\gamma'$  is $d-w-w'$.
(Compare Figure \ref{F-cartaw}.) A completely analogous argument shows that
$\Delta(\widehat\gamma,\,\gamma)=d$, and hence that
$$ \Delta(\gamma, ~\alpha)~=~ d-w-w'~.$$
In most of the cases we have seen, $w=w'=1$; but in the $\cS_5$ example described 
above we have $w=1$ and $w'=19$.    \end{rem}

\noindent{\bf  The Intermediate Case.} 
By this we mean the case where there is {\sf no 
component of Type A or B in the central face.}
%e will call this the \textbf{\textit{intermediate}} \change
In all such examples that we have seen, each of the two vertices
is a root point of an A or B component.
As an example, in the upper right of Figure~\ref{F-air} we have such a case
lying between an A component in the airplane region of $\cS_3$
and a B component in the $100-$ region, with
$$ \alpha=70,~~ \beta=71,~~ \gamma=17, \quad{\rm and}\quad
\alpha'=76, ~~ \beta'=77, ~~ \gamma'=2~,$$
with denominator ~~78~. For an example in $\cS_5$, lying between
two B components, see the left side of Figure~\ref{F-Boffcenter}-left. 
In this case.
$$ \alpha=28,~~ \beta=29,~~ \gamma=670 \quad{\rm and}\quad
\alpha'=61, ~~ \beta'=62, ~~ \gamma'=574~,$$
with denominator ~~726~.

In  the examples above, and conjecturally in all such examples,
the four grand orbits $\(\alpha\),~\(\beta\),$ $\(\alpha'\), ~ \(\beta'\)$
are disjoint from each other, but $\(\alpha\)=\(\gamma'\)$ and $\(\gamma\)
=\(\beta'\)$. Thus we believe that Condition (1) is always satisfied
in the Intermediate Case when the two primary wakes belong to the same
escape region.\ssk

The symmetry relation $~\beta-\alpha~=~\beta'-\alpha'$ is satisfied in the
examples above; but it is not  always satisfied when the primary
wakes belong to different escape regions, as exemplified in
Figure~\ref{F-Bcart}-right.   In this Figure, consider the triads
determined by \hbox{$\alpha=518/726$}, \hbox{$\beta=520/726$}, $\gamma=515/726$;
and $\alpha'=548/726$, $\beta'=568/726$,  $\gamma'=521/726$.
Here $\beta-\alpha=2$, and $\beta'-\alpha'=20$. In the Intermediate
Case, when the primary wakes belong to different escape regions, as exemplified
in this Figure, Condition (3) holds.
\msk

\begin{figure}[htb!]
  \centerline{\includegraphics[width=2.8in]{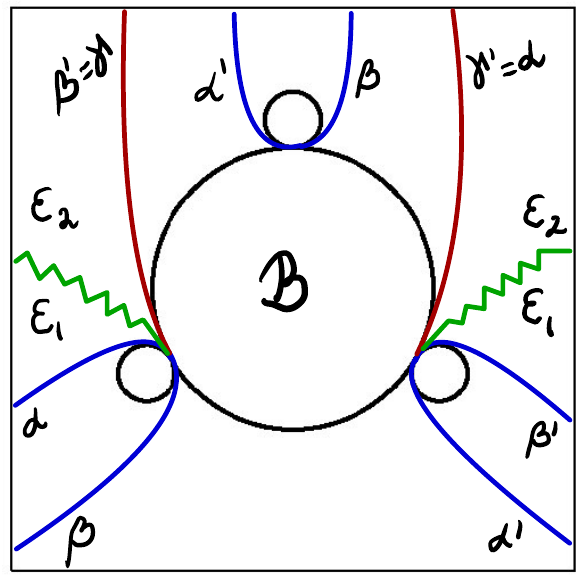}\quad \includegraphics[width=2.7in]{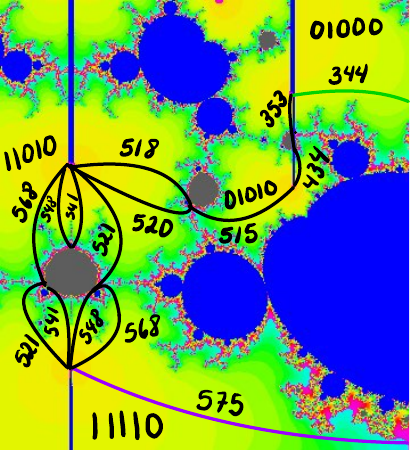}}
  \caption[B-centered face]{\label{F-Bcart} \sf  On the left, cartoon of a
type B component  between two escape regions $\E_1$ and $\E_2$, showing 3 rays
    landing on each of its two bottom root points, plus an additional 
two rays landing  on its top root point.  Here $\gamma=\beta'$ and
$\gamma'=\alpha$. 
 In all  3 ray examples that we have seen, the four grand orbits 
    $\(\alpha\),~\(\beta\),~\(\alpha'\),~\(\beta'\)$ are all distinct. 
(This is true even in cases where there are more than two escape regions.) On
the right,  an example in $\cS_5$. In this case the two wakes 
    below the B component have angular width $20/726$, so they are far
    from minimal.  }
\end{figure}
\medskip

\noindent{\bf The B-Centered Case.} Examples with a component of type B in the
middle of the central face are more complicated and interesting. These always
seem to be organized in a very special way as in Figure \ref{F-Bcart}-left
(except that in some cases there are more than two surrounding escape regions).
In the examples we have seen,
the four grand orbits $\(\alpha\),~\(\beta\),~\(\alpha'\)$ and $\(\beta'\)$
are all different. Furthermore 
each of the four angles $\alpha,~\beta,~\alpha'$ and $\beta'$
is the parameter angle for two different rays which  belong to different
escape regions, as indicated in the figure. The angles $\gamma$ and $\gamma'$
of the two  secondary rays always seem to satisfy
$\gamma=\beta'$ and $\gamma'=\alpha$. 

It clearly follows that
 $\(\alpha\) =\(\gamma'\)~\ne ~  \(\gamma\) =\(\beta'\)$ 
in all such examples, so that Condition (1) is satisfied.

Another interesting property of the B case is that %the following 
symmetry condition 
$$ \beta-\alpha~~=~~\beta'-\alpha'\qquad{\rm in} \quad \R/\Z~,$$
is %either
satisfied %(or~ {\bf nearly}~ satisfied) 
in every case where the primary wakes belong to the same escape region.
See the examples below.
\msk

{\bf B-Centered Examples.} At the top of Figure~\ref{F-air} we have 
$$(\alpha,~\beta,~\alpha', \beta')~~=~~(49,~50,~67,~68)/78~,$$
with $\cE_1$ the airplane region in $\cS_3$ 
and $\cE_2$ the $010-$ region. In this case, both primary wakes belong
to the airplane region.
 There is a similar example at the right center of Figure~\ref{F-air} 
 % the corresponding angles are$(76,~77,~1,~2)/78$ with $\cE_1$
 between the airplane and the $100-$ region. The corresponding orbit portraits
 are shown in Figure \ref{f-T3OP}.
 
 For a slightly different
 $\cS_3$ example, at the top of the $110$ region on the left of
 Figure~\ref{F-T3S3},
 the corresponding angles are $(56, ~58,~ 59, ~61)/78$. 
 Thus both wakes have width two (although they are still minimal since the
 intermediate angles $57$ and $60$ are not co-periodic). As noted above, 
 we believe that such a wake of even width can never occur in an A-centered
 case.

In Figure~\ref{F-44}  (top), $\cE_1$ is
the Kokopelli region of $\cS_4$ and $\cE_2$ is a $0010$ region.
The angles are $(28,~29,~82,~83)/240$,  and both primary wakes belong to
the Kokopelli region. Figure \ref{F-Bcart}-right
shows an example in $\cS_5$ where  both of the wakes below the central
type B component  have width 20.  These primary wakes belonging to the
11110 escape region are certainly not minimal.

 For type B-centered examples as corners of different escape regions see
Figure \ref{F-Boffcenter}-left. This figure shows two examples in $\cS_5$ where
the B component is surrounded by four different escape regions.   The
top left example  with angles $(670,~674,~28,~29)/726$ doesn't satisfy the,
 symmetry relation,  and the wake determined by $670$ and $674$
is certainly not minimal. Immediately below  is the B component surrounded by
$(61,~62,~572,~574)/726$. In this case again  the symmetry relation is not
satisfied, however the wake spanned by the angles $572$ and $574$ is still
minimal, since the intermediate angle $573$ is not co-periodic. \bsk

\begin{figure}[h!]
  \centerline{\includegraphics[width =3.5in]{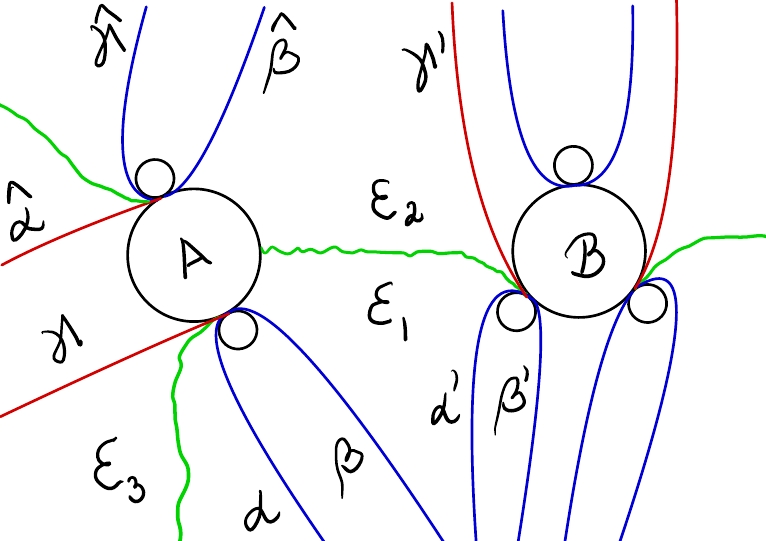}}
  \caption[off-center A component]{\label{F-Aoff} \sf Cartoon showing an off-center A component, as well as the adjacent B component.}
\end{figure}      

{\bf Examples with  A Off-Center.} We now consider cases which do not quite
correspond to the requirements of Figure \ref{F-3rp}. As an example,
Figure \ref{F-Aoff} can represent the face on the upper left of
the airplane region of $\cS_3$, as shown in Figure~\ref{F-air}. In this case
$$\alpha=46,~\beta=47, ~\gamma=22\,,\quad{\rm and}\quad\alpha'=49,~\beta'=50,
~\gamma'=68 \,,$$
while $\quad\widehat\alpha=20,~\widehat\beta=73,~\widehat\gamma=74$,~~ all with
denominator $78$. Note that a mirror image case is shown on the upper right
of the same figure.

On the other hand, Figure \ref{F-Aoff} can also represent
the face of $\cS_5$ which is shown on the bottom of
Figure~\ref{F-Boffcenter}-left. In this case we have
$$\alpha=64,~\beta=65, ~\gamma=568\,,\quad{\rm and}\quad\alpha'=91,~\beta'=92,
~\gamma'=272 \,,$$
while $\quad\widehat\alpha=548,~\widehat\beta=307,~\widehat\gamma=326$,~~ all
with denominator 726.

In both cases we have
$\quad \(\alpha\)=\(\widehat\alpha\)=\(\beta'\)\quad{\rm and}\quad
    \(\gamma\)=\(\widehat\gamma\)=\(\gamma'\)$,~  so that Condition (2) is
    satisfied. In fact we conjecture that it holds in all such cases.
\bsk

\begin{figure}[htb!]
  \centerline{\includegraphics[width=3.3in]{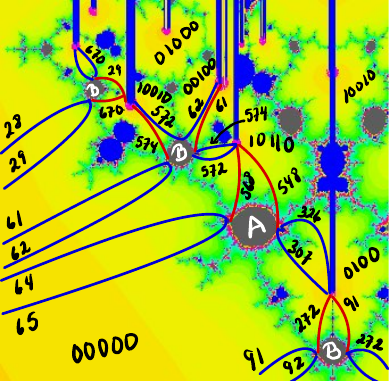}\qquad
    \includegraphics[width=2.7in]{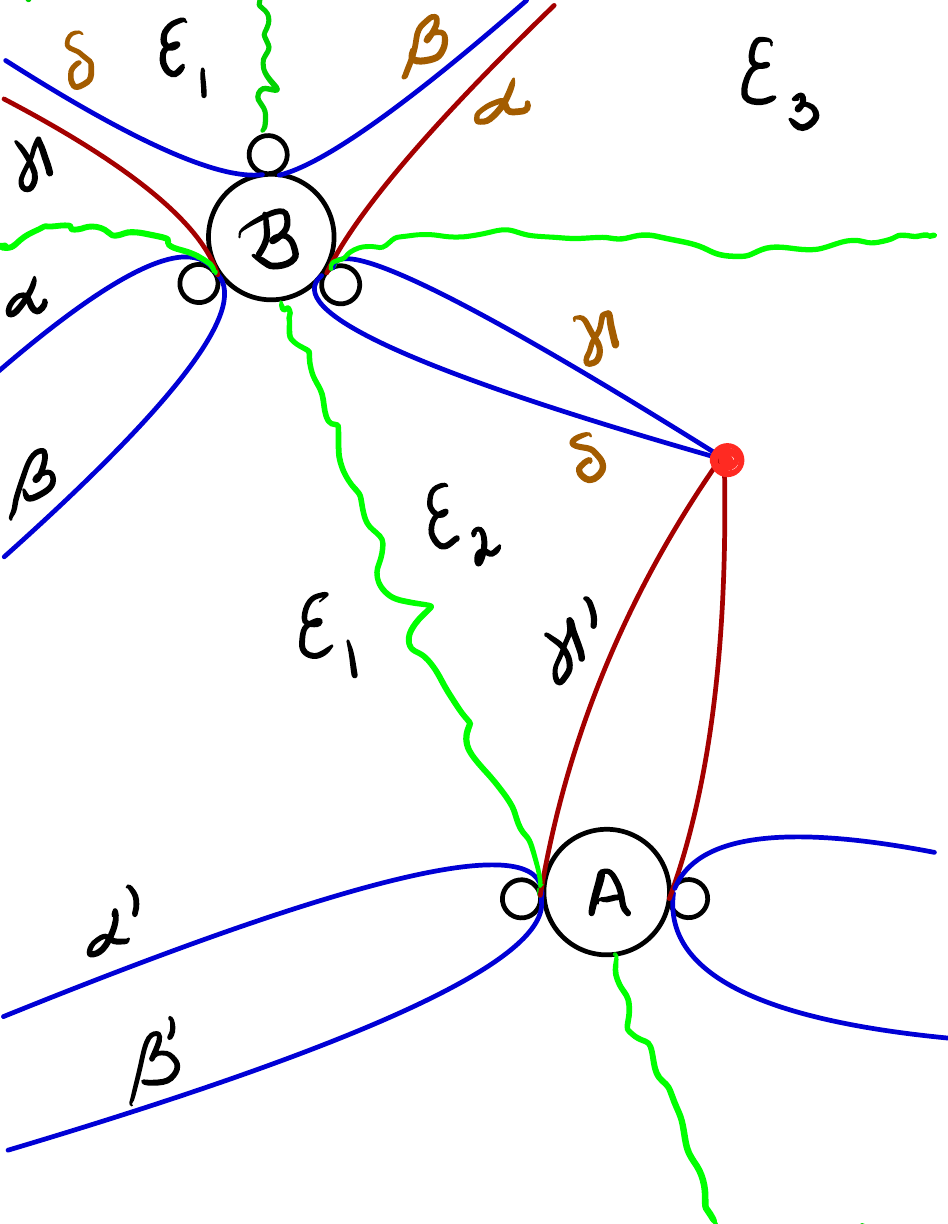}}
  \caption[Off-center B component]{\label{F-Boffcenter} \sf On the left, a region
    in $\cS_5$.
    Roughly in the middle there is a B component contained in a face which extends
    to the left, and which has one root point of a larger A component on its
    boundary.
    This is the only example we have seen of a B off-center face;  but we 
    suspect there are many more for larger values of $p$. On the right, a cartoon
    illustrating this face. The common denominator is $3(3^5-1)=726$.} 
  \end{figure}

  {\bf An Example with B off-center.} Figure \ref{F-Boffcenter}-left 
  illustrates  an example in $\cS_5$ with
  $$ \alpha=61,~\beta=62,~\gamma=574,\quad{\rm and}\quad
  \alpha'=64,~ \beta'=65, ~\gamma'=568,~ ({\rm with~denominator}~ 726). $$
  Here
  $$61 \mapsto \langle 183, 65, 195, 101, 61\rangle \qquad{\rm and}\qquad
  65\mapsto \langle 195, 101, 61, 183,65\rangle$$
  (with denominator 726 for parameter angles, or 242 for dynamic angles).
 It follows that $\(\alpha\)=\(\beta'\)$. Similarly
$$ 574\mapsto \langle 28, 84, 10, 30 , 90\rangle\qquad{\rm while}\qquad
568\mapsto \langle 10, 30, 90, 28, 84\rangle~,$$
so that $\(\gamma\)=\(\gamma'\)$. Thus Condition (2) is satisfied.

\msk

\subsection*{\bf The Four Ray Case.}  

\begin{figure}[ht!]
  \begin{center}
    \begin{overpic}[width=2.3in]{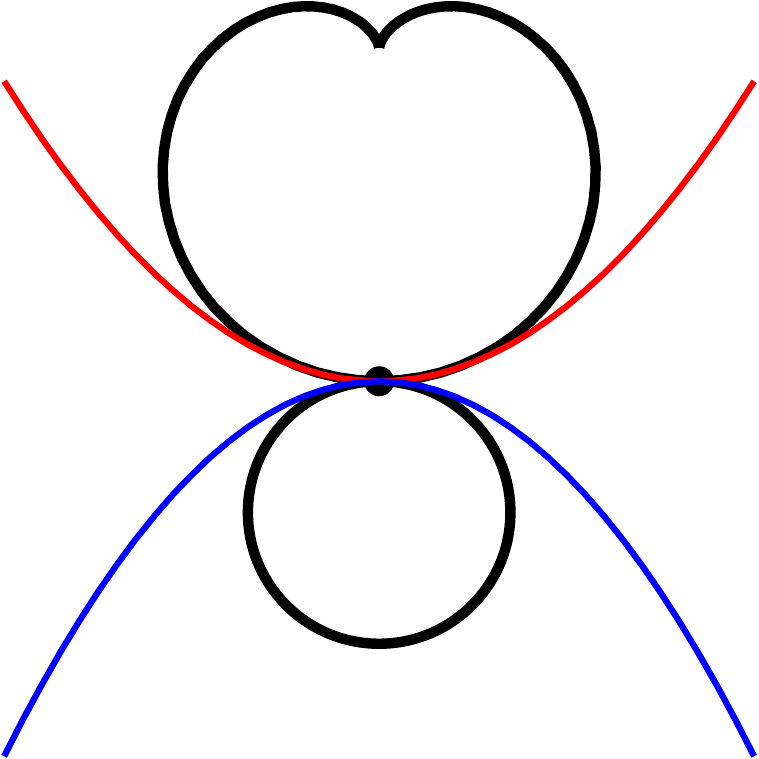}
      \put(-10,10){$\alpha$}
      \put(170,10){$\beta$}
      \put(170,150){$\gamma$}
    \put(-8,150){$\delta$}
    \put(40,5){${\mathcal F}_1$}
    \put(80,170){${\mathcal F}_3$}
    \put(80,90){${\mathfrak p}$}
    \put(70,50){$H_{\mathfrak p}$}
    \put(-10,70){${\mathcal F}_{4}$}
    \put(170,70){${\mathcal F}_2$}
    \end{overpic}
    \caption[Cartoon of 4-ray case]{\label{F-rc1}\sf Cartoon of the four
      faces surrounding a
  %\index{Figure~\ref{F-rc1}, cartoon of 4-ray case}
      parabolic vertex $\p$ in $\Tes_q(S_p)$ where four parameter rays land.
      As usual, face number one contains the component $H_\p$ which
      has $\p$ as root point. The opposite face is number 3 and the side faces
      are numbers 2 and 4.
%  Here $\cF_0$ is the primary face (as seen from $\p$),
%  containing the hyperbolic component
%  $H_\p$ which has $\p$ as root point; $\cF_2$ is the secondary face
%  containing the cusp component of the associated Mandelbrot copy, and
%  $\cF_{\pm 1}$ are the two side faces.
  The four parameter angles, with denominator $3(3^q-1)$, are denoted
  by $\alpha$ through $\delta$.}
\end{center}
\end{figure}
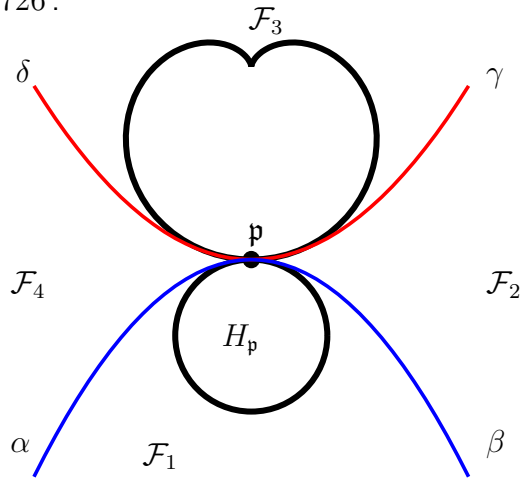

Now suppose that four edges of the tessellation $\Tes_q(\ocS_p)$ 
have $\p$ as a parabolic vertex, as illustrated in Figure \ref{F-rc1}.
Arguing just as in the two and three ray cases, we see that the
orbit relations in the number one face are generated by
\begin{equation}\label{E-4R0}
\alpha_j~\simeq~ \beta_j~\simeq~ \gamma_j~\simeq~ \delta_j
\qquad{\rm for~all}\qquad j \in \Z/q\end{equation}
(together with any background relations). However in all four ray cases
that we have seen:\ssk

%{\bf(a$'$)} There are no background relations (just as in the 3-ray case).\ssk
%and\ssk

{\bf(a$'$)} We have\begin{equation}\label{E-GOR}
 \lg\alpha\rg~=~ \lg\gamma\rg ~~\ne~~\lg\beta\rg~=~\lg\delta\rg~.\end{equation}
% $((\alpha))=((\gamma)) \ne ((\beta))=((\delta))$.\ssk

\noindent In fact there is a somewhat more precise observation:\ssk

{\bf(b$'$)} There is a non-zero $k\in\Z/q$ such that
\begin{equation}\label{E-4R1}
 \gamma_j=\alpha_{j+k}\quad{\rm and}\quad \beta_j= \delta_{j+k}\qquad
 {\rm for~all}\quad j\in\Z/q~. \end{equation}

\noindent
In all four rays examples we have seen with $p=q$, there are no background
relations; and $\p$ is also a root point of two components of Type A or B in the
two side faces. (See Figures~\ref{F-S2rays} and \ref{F-s5rab-par}.) Of course
there can be no such A or B components when $p\ne q$
(Figures~\ref{F-q8ex} and \ref{F-43}). \medskip 

Assuming (\ref{E-GOR}) and (\ref{E-4R1}), it follows easily that the
distinguishing relations in the two side faces are generated by
\begin{equation}\label{E-4R2}
\alpha_j\simeq\gamma_j=\alpha_{j+k}\simeq\gamma_{j+k} {\rm ~~in~~} \cF_{2}\qquad{\rm and}\qquad
 \beta_j\simeq\delta_j=\beta_{j-k}\simeq\delta_{j-k} {\rm ~~in~~} \cF_{4}~.\end{equation}

%More explicitly we can write \rnote{I think an old version had this (??).}
%$$ \gamma_j=\alpha_{j+k}\quad{\rm and}\quad \delta_j=\beta_{j+k'}$$
%for all $j$ mod $q$, where $k$ and $k'$ are non-zero constants mod $q$.
%If $q$ is even perhaps $k$ and $k'$ are both equal to $q/2$; while
%if $q$ is odd then $k=(q\pm 1)/2$ and $k+k'=q$ (??).
%\rnote{Pure guesses. Correct as needed.}\msk

\noindent In fact, as we cross the $\alpha$ edge of the tessellation,
 the $\alpha_j$ and $\gamma_j$ dynamics rays will jump so that only the
 $\beta_j$ and $\delta_j$ rays will continue to land together. Similarly,
 after we cross the $\beta_j$ edge, only the $\alpha_j$ and $\gamma_j$
 rays will land together. An important consequence is the following.

\begin{prop}\label{P-amalg} In both the three ray case and the four ray case,
  if we assume the appropriate  grand orbit condition  $(\ref{E-GO3})$ or
  $(\ref{E-GOR})$, then it follows that the orbit portrait for the number
  one face  is the  amalgamation {\rm (Definition \ref{D-sum})}
of the orbit portraits for the two side faces.  \end{prop}

Finally we must consider the opposite face  $\cF_3$.
The Edge Monotonicity Conjecture~\ref{CJ-main} tells us which rays will
jump  as we cross the $\gamma$ or
$\delta$ ray into  $\cF_3$; however it does not tell us what they will jump to.
A better strategy is to cross
directly from $\cF_0$ to  $\cF_3$ through the point $\p$. The conclusion
as illustrated in Figure \ref{F-imp}, is that the $\alpha_1$ and $\delta_1$
rays will be pushed away from the $\beta_1$ and $\gamma_1$ rays. This will
leave only the distinguishing relations
$\alpha_1\si \delta_1$ and $\beta_1\si \gamma_1$ between these particular
angles, where the first two are definitely not equivalent to the second two.
Since equivalence must be preserved under angle tripling, and since the
equivalence relations must be a subset of those for the parabolic point,
this proves the following.

\begin{prop}\label{P-drel}
  The distinguishing relations for the opposite face $\cF_3$ 
are generated by
$$ \alpha_j~\si~\delta_j\quad{\rm and}\quad\beta_j~\si~\gamma_j
\quad{\rm for~all~~~} j~.$$
\end{prop}

As we cross any secondary ray in either direction, note that all of the
distinguishing relations get destroyed.\medskip

Thus we obtain quite explicit  
descriptions of how the period $q$ orbit portrait changes
as we circle around any parabolic vertex $\p$ of the tessellation
$\Tes_q(\ocS_p)$ in all of the two, three, and four ray cases,
provided that the
appropriate conditions $(\ref{E-GO3})$ or $(\ref{E-GOR})$ are satisfied.
 All that we need to know is the parameter angles for
the rays landing at $\p$, or even the dynamic angles for the rays landing
at the critical value $F(-a_F)$. 

\begin{figure}[htb!]
  \centerline{\includegraphics[width=1.8in]{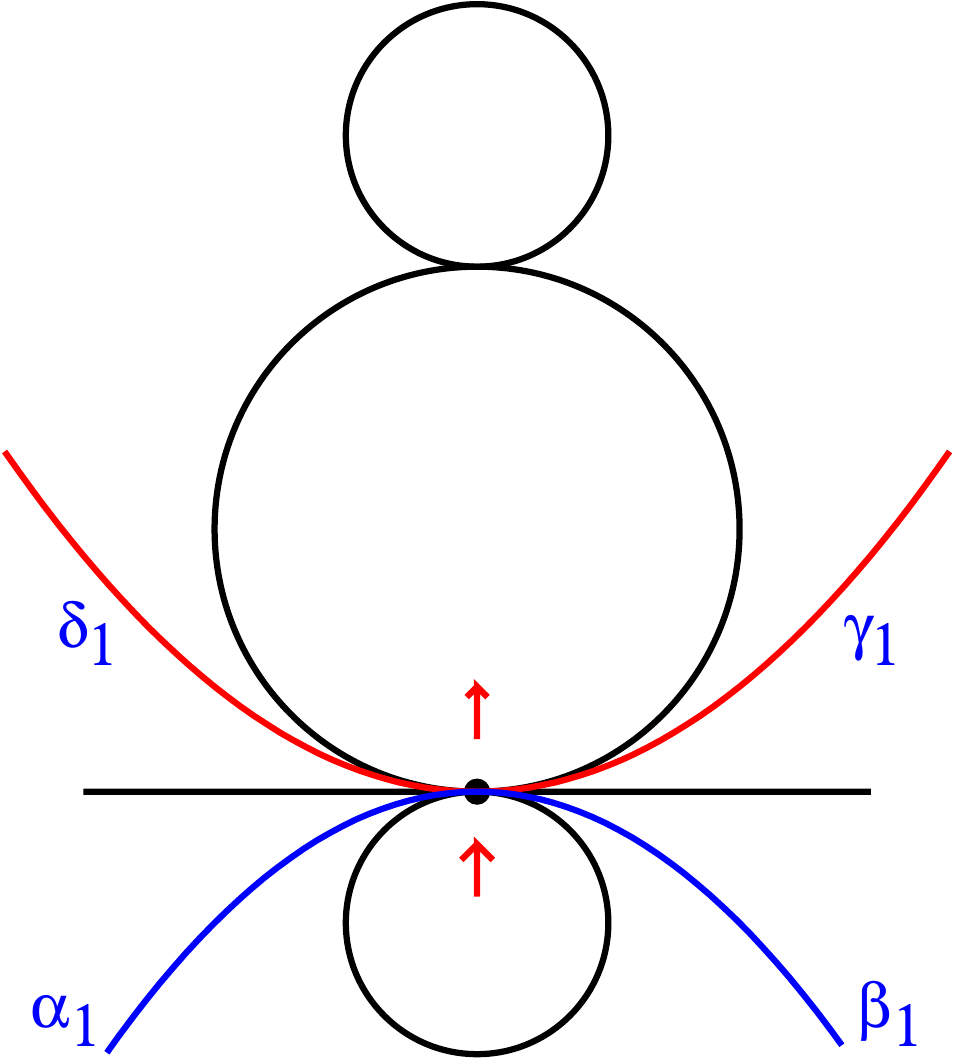}\qquad
    \includegraphics[width=1.8in]{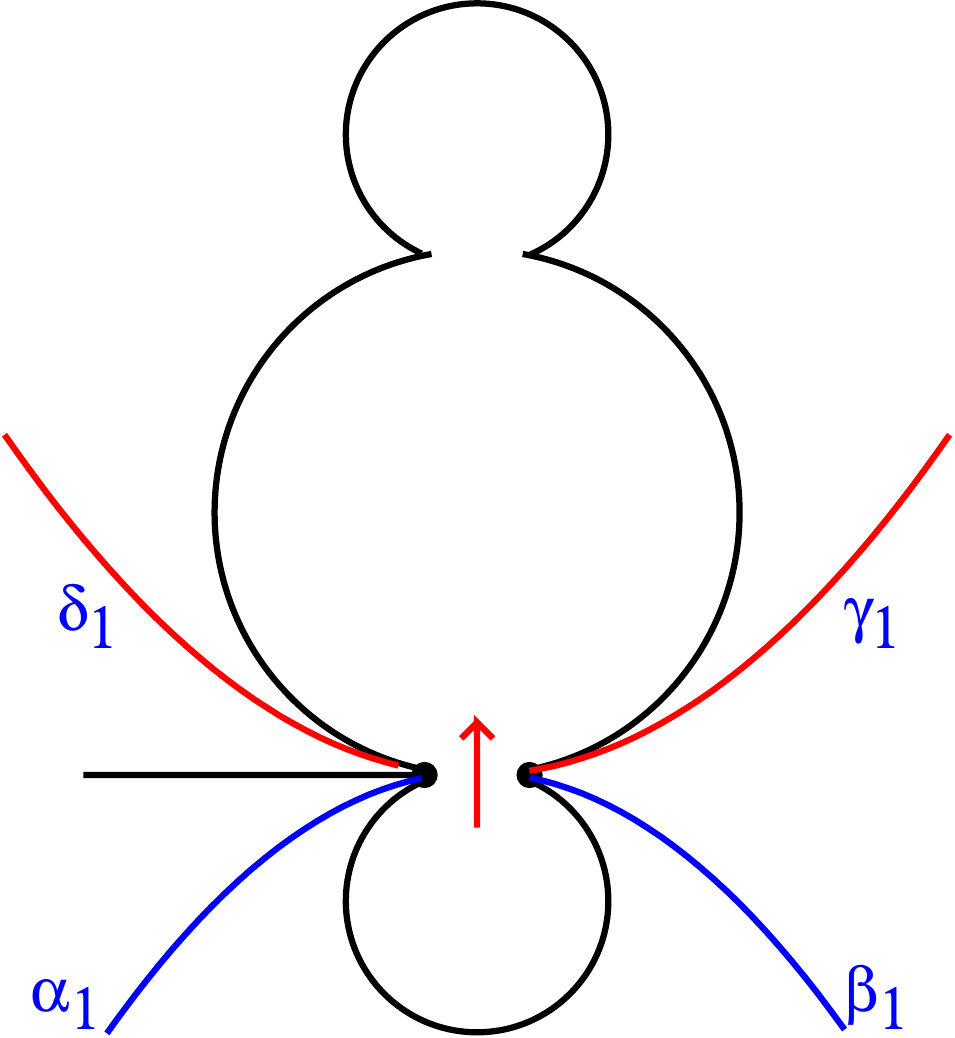}}
  \caption[Orbit portrait as parabolic point splits into a pair of periodic repelling points]{\sf Cartoon illustrating the change of orbit portrait, as a
    parabolic point splits into a pair of repelling points. Compare
    Figures~\ref{F-q8ex} and \ref{F-43}. However note that
    the primary wake is on top in Figure \ref{F-43}. \label{F-imp}} %and \ref{F-rabpics}. 
\end{figure}

\begin{rem} In all of the four ray cases we have seen, the
  $\alpha$ and $\beta$ rays belong to the same escape region, 
and bound a minimal wake. In fact the angles $\alpha$
  and $\beta$ always seem to be immediately contiguous in the sense that $\beta
  =\alpha+1/3(3^q-1)$. (Thus if $q$ is large the difference is VERY small.)
\end{rem}
\msk

\begin{figure}[htb!]
\centerline{\includegraphics[width=3.3in]{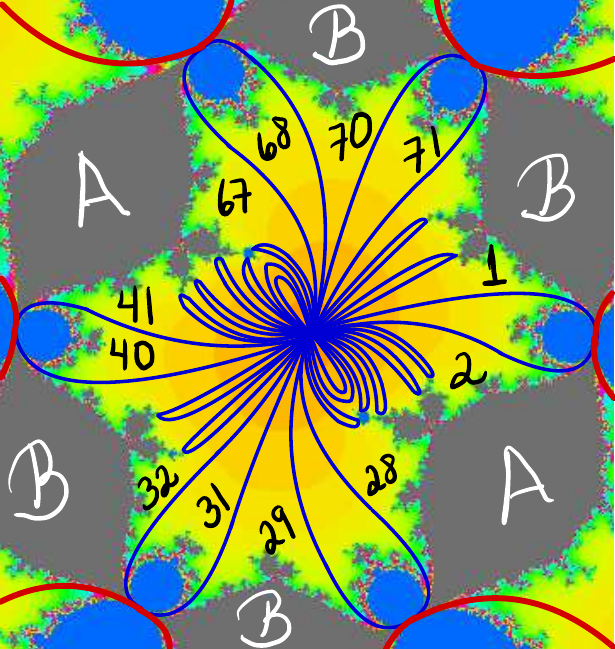}}
\caption[Detail of the $1/3$-rabbit region in $\cS_3$]{\label{F-3rab} \sf The
  $1/3$-rabbit region in $\cS_3$. There are
  $2(3^3-1)=52$ rays of co-period three in this region. Twelve of these
  are connecting rays, labeled by their angles  (multiplied by the common
  denominator of $78$). 
These form $6$ primary wakes, each with the
smallest possible angular width $1/78$. Of the remaining $36$ rays of
co-period $3$, note that $28$ are associated with Type A boundary components,
while only $8$ are associated with Type B components. Each subset of $\E(1/3)$
bounded by the two connecting rays around an A component has
angular width $26/78=1/3$; while the analogous 
subsets adjacent to a B component all have much smaller angular width.
The rotation $\I$ adds $39=78/2$ to each angle.}
\end{figure}

{\bf Note.} If $q$ is a prime number, then it is not hard to conclude,
for maps in $\cF_0$, that all of the the $\alpha_j$ and $\beta_j$ dynamic rays
must land together at the central fixed point. 
This is also true in  many cases where $q$ is not prime. For example
in each ``rabbit''
escape region of $\cS_p$, as discussed in  \cite[Rabbit Section]{BM}
% ~\ref{s-rab}, 
there are $2\,p$ shared boundary points. Each one is the landing point of four
parameter rays,  and all are parabolic of ray period $q=p$.
For the Julia sets corresponding to these boundary points,
the $2\,p$ primary dynamic rays all land on a central fixed
point. 
For an example where the $2\,q$ rays do not all land together,
 see Figure~\ref{F-43}-right with $q=4,~p=3$. 
\medskip

\begin{figure}[htb!]
  \centerline{\includegraphics[width=5in]{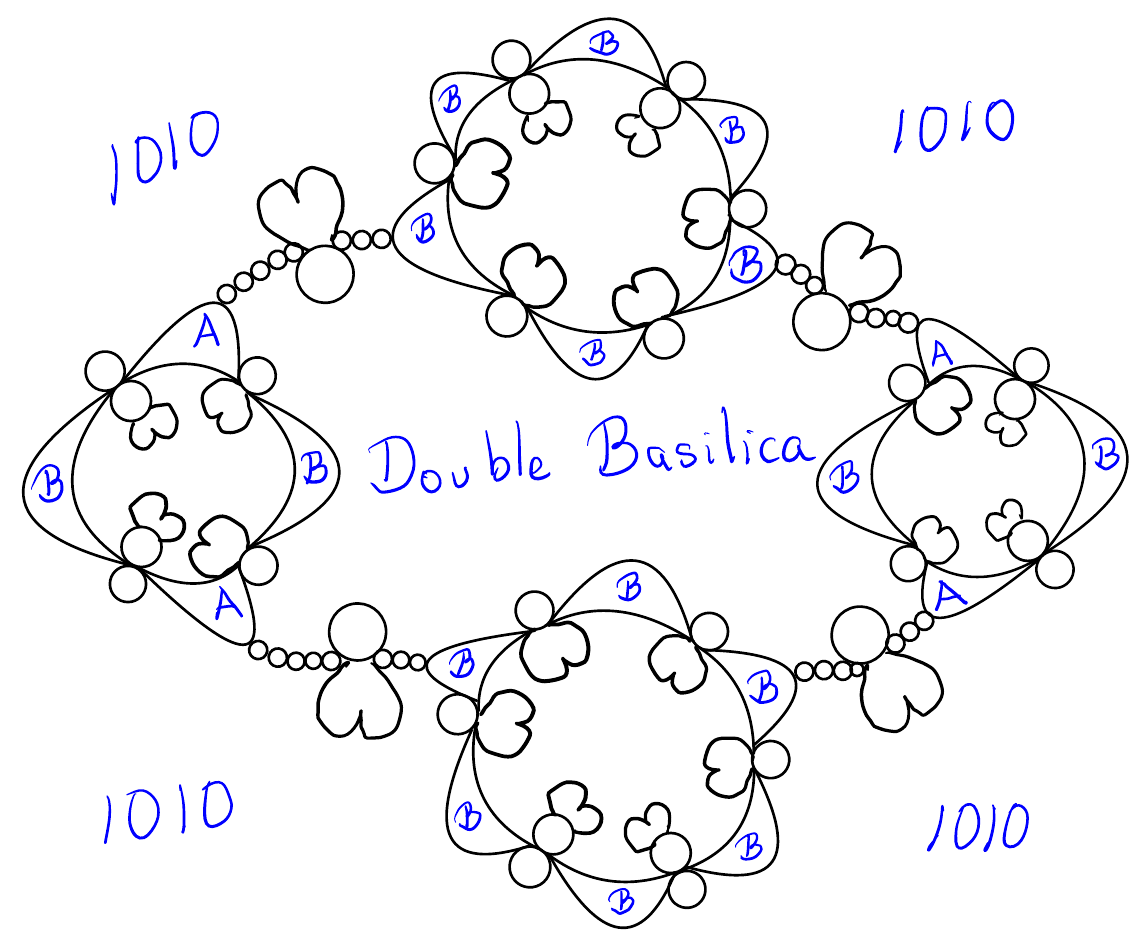}}
  \caption[Part of $\ocS_4$]{\label{F-4rings} \sf Cartoon illustrating
    the Double Basilica and $1010$ regions of $\ocS_4$ which are separated
    from the rest of $\ocS_4$ by four rings. All of the rest of $\ocS_4$
    can be reached by passing through any one of these rings; but 
    cannot be represented in any meaningful way in this figure.}

  \centerline{\includegraphics[height=2.5in]{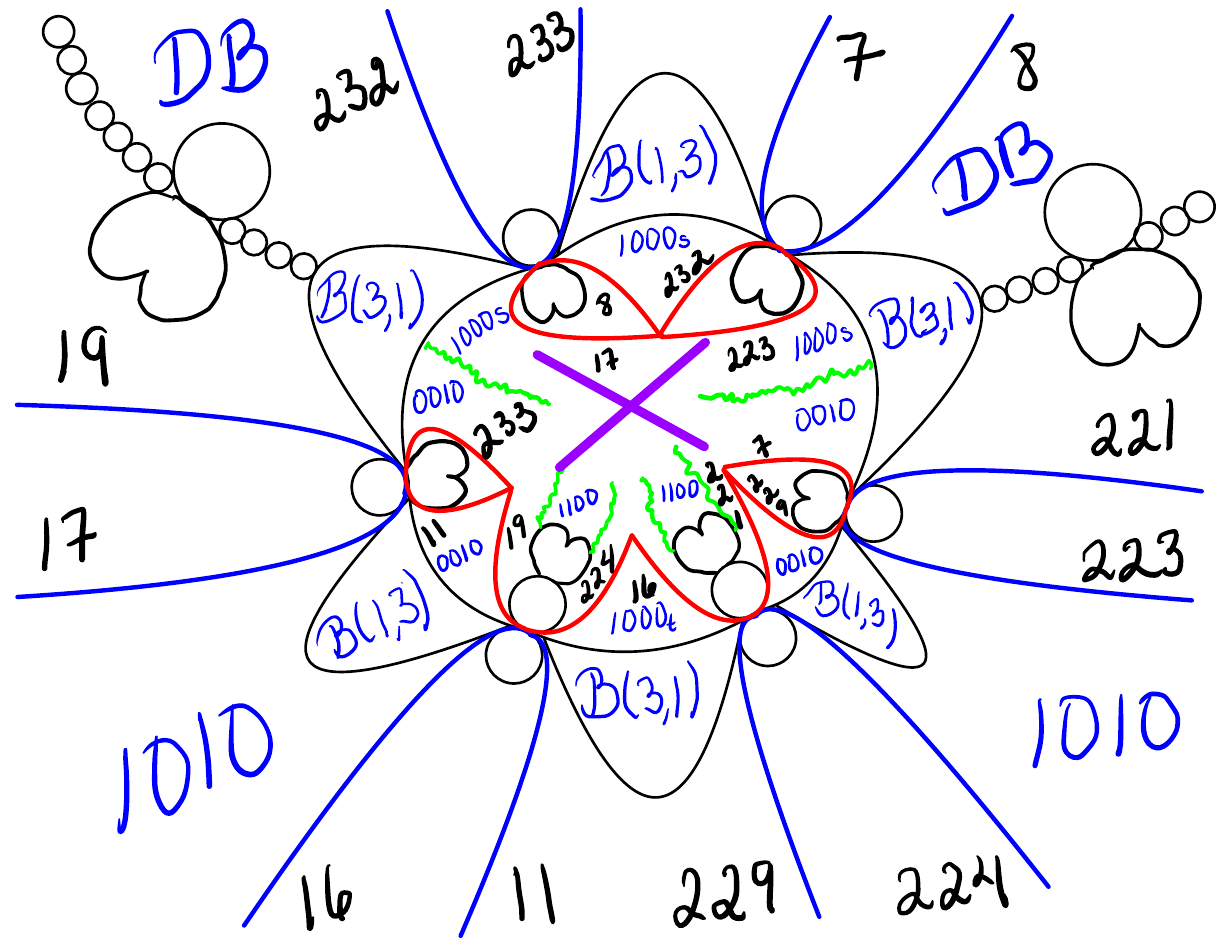}\quad
    \includegraphics[width=2.2in]{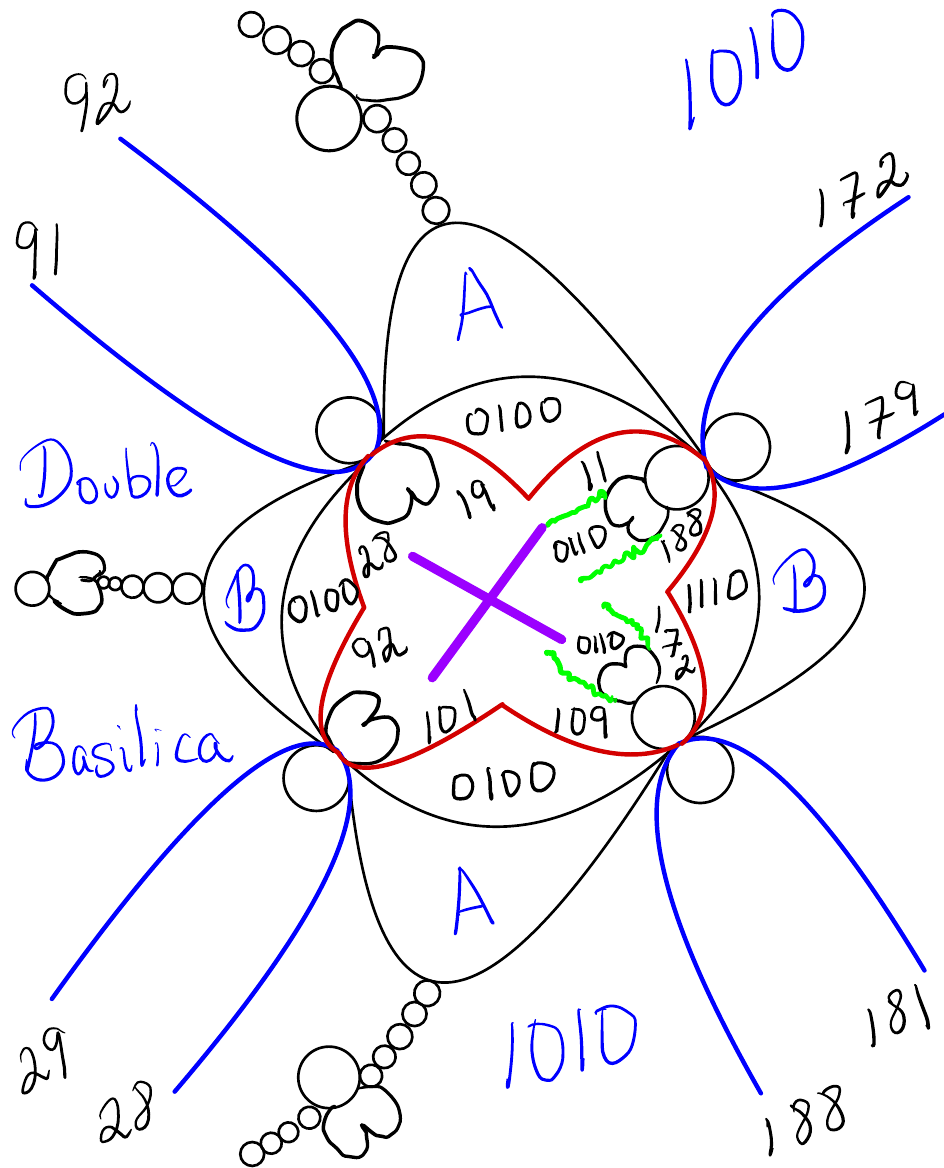}}
  \caption[Two rings]{\label{F-2rings} \sf On the left: a more detailed
    picture of the bottom ring in Figure \ref{F-4rings}. As we circle 
    once around the outside of
    this ring, note that the corresponding angles circle twice
    around $\R/\Z$.   On the right: a more detailed picture of the right
    hand ring. The type B components in this ring are of type B(2, 2), i.e.,
  they are weakly self-dual.}
\end{figure}

\begin{rem}[\bf Rabbit and Non-Rabbit Examples]{\label{R-RNR}
    There is an essential difference between four ray points
    with $p=q$ and those with $p\ne q$. The cases with $p\ne q$
    always seem to be isolated. (Compare Figures~\ref{F-43}-left %\ref{f-M(43)bdryair}
    and \ref{F-q8ex}-left.) There can never be an associated A or B component
when $p\ne q$ since every root point of an A or B component in $\cS_p$ has
period $p$.\ssk

The case $p=q$ is much more interesting.
        In every case that we have seen, 
        the parabolic point $\p$ is part of a compact connected
        \textbf{\textit{ring}} consisting of the closures of $2n$
        components of type A or B, together with their $2n$ common
        parabolic root points, where $n\ge 2$.        
Most of the examples we have seen are rings
    around the boundary of a rabbit region. Compare Figures~\ref{F-S2rays},
    \ref{F-s5rab-par}, \ref{F-3rab} and \ref{F-4rab}. In these
    cases, $n$ is always equal to $p$, and there are always two diametrically
    opposite components of Type A, together with $2n-2$ components of Type B.
    This case will be studied in more detail in \cite{BM}.

    \begin{figure}[htb!]
      \centerline{\includegraphics[width=3.8in]{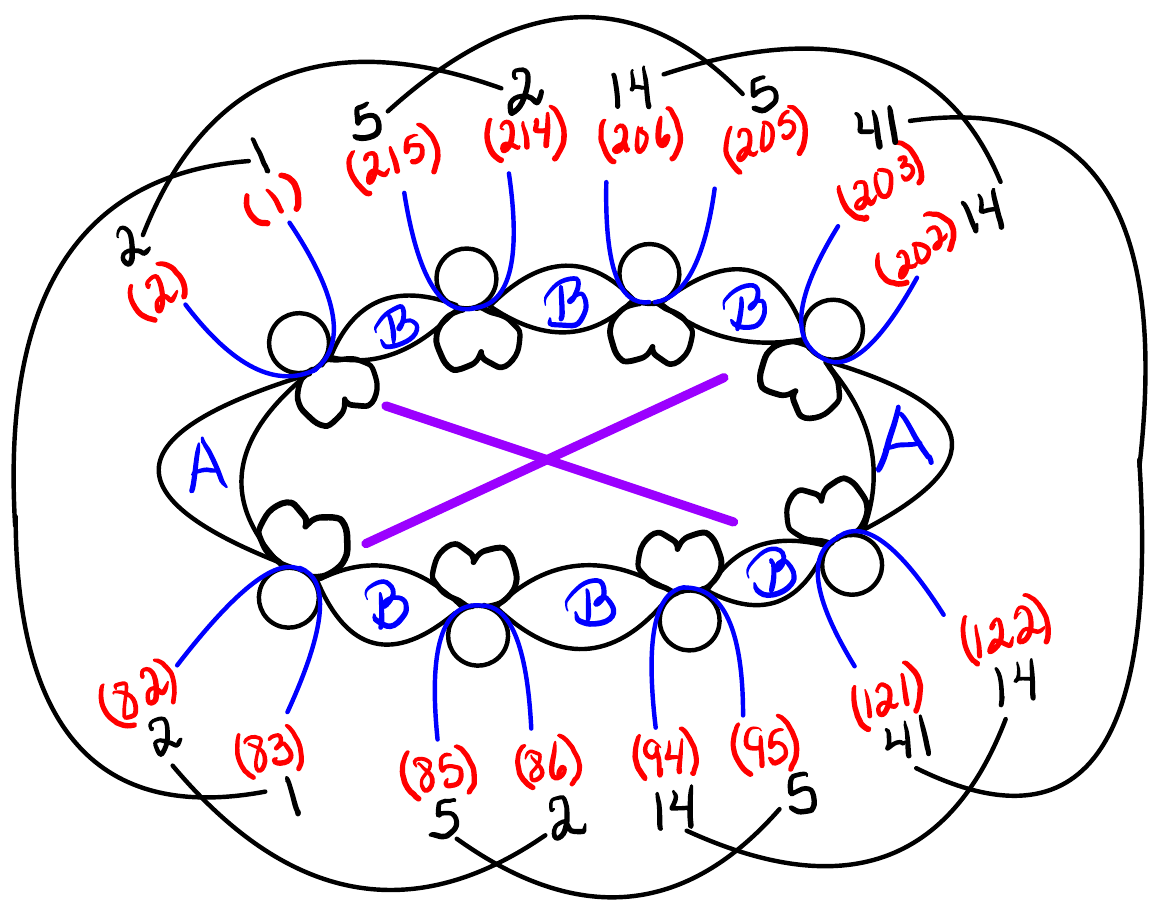}}
      \caption[Cartoon illustrating the ring around the $1/4$-rabbit region in $\cS_4$]{\label{F-4rab}\sf Cartoon illustrating the ring around the
  $1/4$-rabbit region. Here the rabbit region is outside of the ring, while the
        rest of $\ocS_4$ is not shown. The actual parameter angles are shown
        in red, while labels for the associated grand orbits are shown
        in black. (Compare Definition~\ref{D-GO}.) Note that each A or
        B component is associated with a corresponding grand orbit; and that
        complex conjugate components share the same grand orbit.}
      \end{figure}

    Figure \ref{F-4rings} illustrates four rings in $\cS_4$ which are not
   associated with rabbit regions, but rather are associated with both 
the double basilica and 1010 regions. Two of these
  rings have $n=2$ with pattern ABAB, while the other two have $n=3$
    with pattern BBBBBB. Two of these four rings are shown in more detail
    in Figure \ref{F-2rings}. The large $\bf X$ in the middle of each ring
    indicates that there is no meaningful way to fill in the center. We could
    proceed by analytic continuation along any path into the center (as long
    as we don't hit an ideal point). However different paths to the 
same point would lead to very different parts of $\cS_4$.\bsk

  For each of these four rings choose a smoothly embedded circle which 
    passes through the $2n$ parabolic points, connecting them through the
    $2n$ A or B components.

    \begin{lem}[Splitting $\ocS_4$ ]\label{L-S4split}
       If we cut $\ocS_4$ open along all four
      of these embedded circles, then it will be split as the union
      of two connected Riemann-surfaces-with-boundary, $S$ and $S'$, which
      intersect only along the four common boundary circles. Here the surface
      $S$, as shown in Figure~$\ref{F-4rings}$, has 
      genus zero and contains the double basilica escape region and the
      $1010$ escape region; while the other surface $S'$ has
      genus twelve and contains all of the other $18$ escape regions.
 The surface $S$ contains all of the primary rays
 which land at one of the parabolic boundary points; while $S'$ contains all
 such secondary rays.\end{lem}

\begin{figure}[htb!]
\centerline{\includegraphics[width=3in]{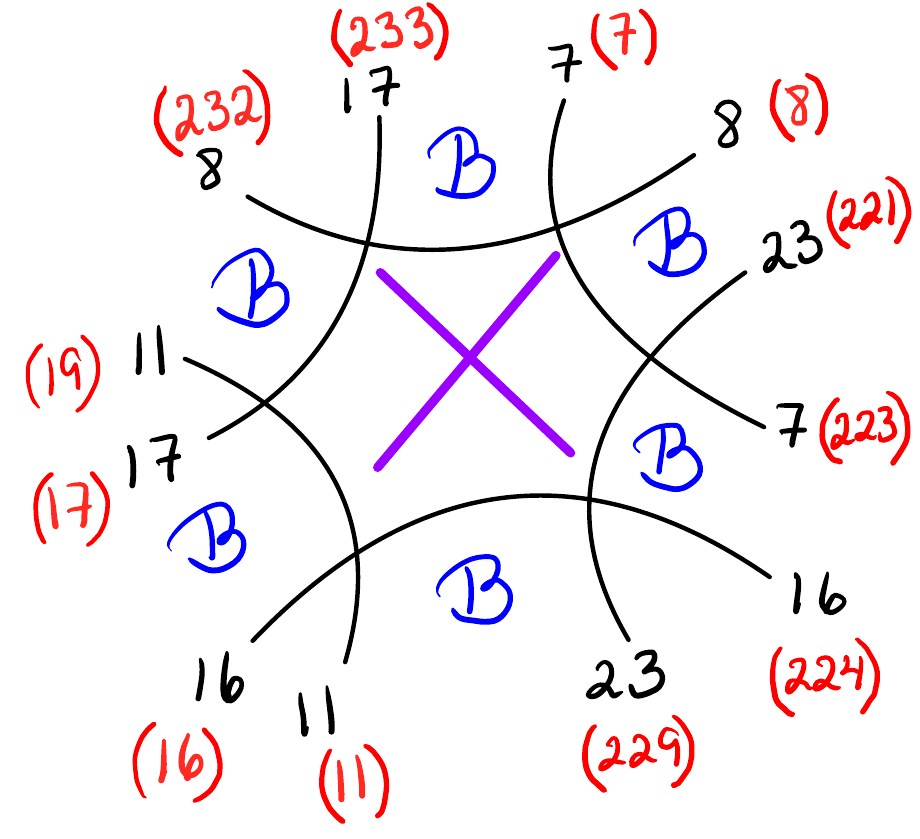}}
\caption{\label{F-Bring} \sf Angles and grand orbits for Figure
  \ref{F-2rings}-left, with notation as in Figure \ref{F-4rab}.}
\end{figure}

\begin{proof}[\bf Proof of Lemma \ref{L-S4split}]
  The surface $S$ is clearly connected, and of genus zero.
  To show that $S'$ is connected, we first consider the quotient $S'/\I$
  where $\I$ is the standard involution ($=$180 degree rotation).
Figure 11 of \cite{BKM}, together with the accompanying text, gives a complete
description of the quotient $\ocS_4/\I$. In this figure, the notation
$\bf d$ stands for the Double Basilica
region, while $\bf 5$ stands for the $1010$ region,
and $\bf 7$ for the $1110$ region. Examining this figure, it is not hard to
check that if we remove the $\bf d$ and $\bf 5$ regions, what is left is
still connected. This proves that $S'/\I$ is connected. Since the $\bf 7$
region is $\I$-invariant, it follows easily that $S'$ itself is connected.

Let $g$ be the genus of $S'$. The computation of $g$ is a fairly
straightforward Euler characteristic computation.  Let $~S'\cup{\rm disks~}$
be the closed Riemann surface of the same genus $g$ which is obtained
by pasting a closed disk into each of the four boundary circles. Then
the Euler characteristic $$\chi(S'\cup{\rm disks}) =\chi(S')+4$$ is equal
to $2-2g$, or in other words $\chi(S')= -2-2g$. Since $S$ has genus zero,
a similar computation shows that $\chi(S)=-2$. Furthermore, since
$\chi(S\cap S')=0$ it follows that
$$ \chi(\overline S_4)~=~ \chi(S)+\chi(S') ~=~ -2+(-2-2g)~=~ -4-2g~.$$
since the Euler characteristic of $\overline S_4$ is $~2-(2\times 15) =-28$,
we can solve for $g=12$.  \end{proof}}\end{rem}\bsk

\begin{figure}[htb!]
  \centerline{\includegraphics[width=2.3in]{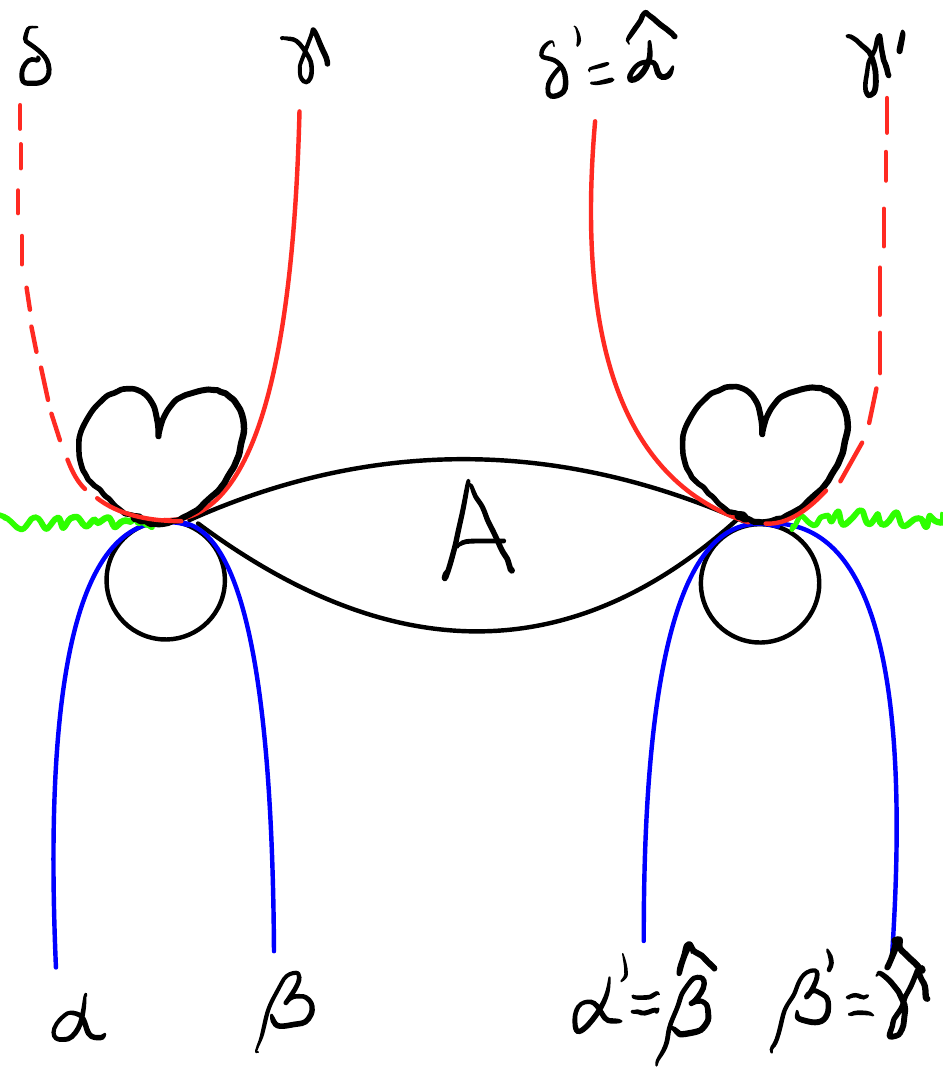}\qquad\quad
    \includegraphics[width=2.3in]{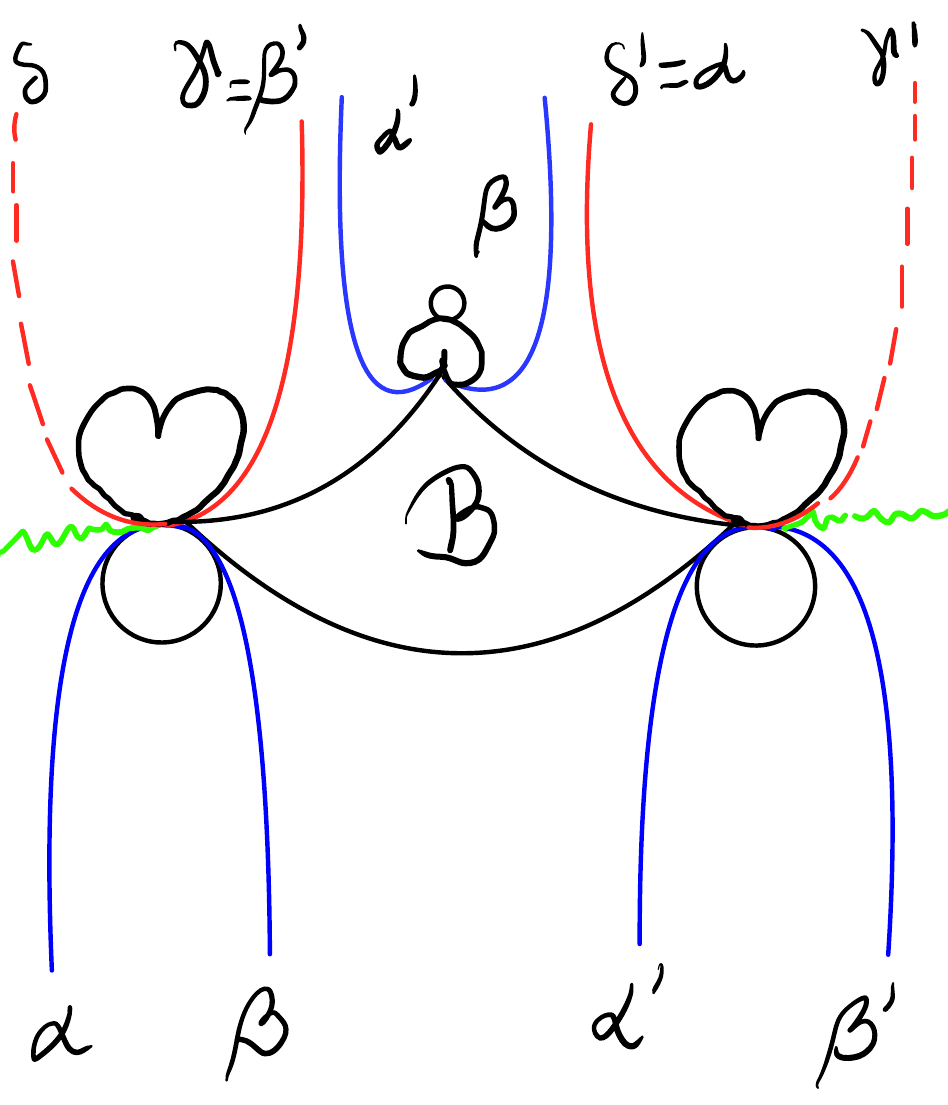}}
  \caption{\label{F-AorB} \sf On the left: two four ray points separated
    by an A component. On the right: separated by a B component}
\end{figure}

  \begin{rem}[\bf Neighboring Four Ray Points]\label{R-AorB}
  As we circle around such a ring, we repeatedly encounter four ray points
  separated by an A or B component. First consider the case of an A component,
  as illustrated in Figure \ref{F-AorB}-left. If the $\alpha,\, \beta,\,
  \gamma,$ and $\delta$ rays land on the left root point, as illustrated,
  then the twin rays $\widehat\alpha,~\widehat\beta,~\widehat\gamma$, and
  $\widehat\delta$ must land in the same cyclic order around the right
  hand root point. (Compare Remark~\ref{R-dual}.)
In fact they always seem to land as indicated, with
  $$ \alpha'=\widehat\beta,~~ \beta'=\widehat\gamma,~~
~  {\rm and}~~ \delta'=\widehat\alpha~.$$
  Recall that $\alpha$ and $\gamma$ always belong to the same grand orbit,
  which also contains $\widehat\gamma=\beta'$. It follows
 that $\alpha$ and $\beta'$ belong to the same grand orbit: 
 $\(\alpha\)=\(\beta'\).$
   Similarly, since $\beta$ and $\delta$ belong to the same orbit, and
  $\alpha'=\widehat{\beta}$, consequently  $\(\alpha'\)=\(\beta\)$.
Assuming as in Equation (\ref{E-GOR}) that $\(\alpha\)\ne\(\beta\)$ it 
follows that 
$$\(\alpha\)=\(\gamma\)=\(\beta'\)=\(\delta'\) ~~\ne~~
\(\alpha'\)=\(\gamma'\)=\(\beta\)=\(\delta\) ~.$$

The key to understanding this situation is to look at the Julia set
for the parabolic root points. Since the two root points of an A
component are dual to each other, they have Julia sets which are identical,
except for the choice of which component to label as $a$ and which to
label as $-a$.

Nearly all of the examples we know are for rabbit regions. As one 
specific example, let us concentrate on the
$1/3$-rabbit region as shown in Figure~\ref{F-3rab}, and on the A component in
the  upper left of this Figure.  
Figure \ref{F-3rabjul} illustrates the Julia set for either of its two
root points. (See Figure~\ref{f-s2par} for an analogous Julia set for the $1/2$
rabbit, and Figure~\ref{F-4wall} for the  $1/5$ rabbit.)

Given any co-periodic angle $\theta$ with denominator $3(3^p-1)$, it is useful
to consider the corresponding periodic angle $3\theta$ with denominator
$3^p-1$. Note that two co-periodic angles $\theta$ and $\theta'$ satisfy 
$3\theta=3\theta'$ if and only if they are twins.

In our example, the root point of each $\pm 2a$ component 
in the Julia set is the landing point of six dynamic rays.
However, only four of the six dynamic
angles correspond to parameter angles for rays landing at the specified
root point of A.  Below is a table describing the angles for the six
dynamic rays landing at the left hand root point in Figure~\ref{F-3rab}
(or \ref{F-3rabjul}),  and specifying which four are the angles of parameter
rays landing on a root of A.
$$\begin{matrix}& \gamma & & & \delta & \alpha &\beta\\
  {\rm angles:}&  16& 19 & 22& 31 & 40 & 41 & /78 \\
  \times ~3: &    16 & 19& 22& 5 & 14 & 15  & /26
\end{matrix}$$
Here is the corresponding table for the right hand root points.
$$\begin{matrix} & \alpha'& \beta'& \gamma'&&& \delta'\\
  {\rm angles:}& 67 & 68 & 71 & 74 & 5 & 14 & /78\\
  \times ~3: &  15& 16& 19& 22& 5 & 14 & /26
\end{matrix}$$
Evidently the last row of the second table is just the last row
of the first table permuted cyclically. It follows that the angles
in the second table are just the twins of the angles in the first
table, again permuted cyclically. In particular
$$ \alpha'=\widehat\beta,\quad \beta'=\widehat\gamma,\quad
{\rm and}\quad \delta'= \widehat\alpha~.$$
However the twin $\widehat\delta$ is an angle which does not correspond
to any parameter ray landing at a root point of this A component, and the
same is true for $\widehat{\gamma'}$.

\begin{figure}[htb!]
  \centerline{\includegraphics[width=3.5in]{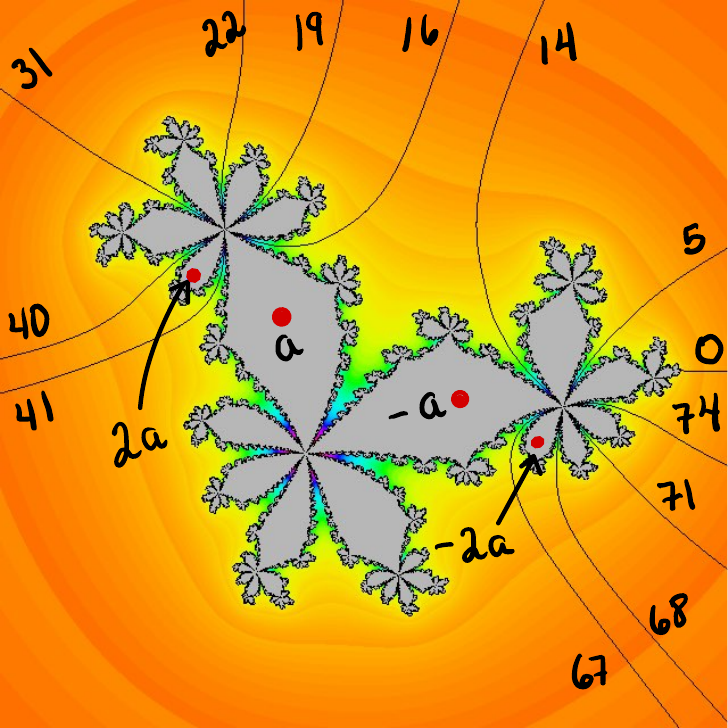}}
  \caption{\label{F-3rabjul} \sf Julia set for the root point of an
    A component in a $1/3$-rabbit, with angles modulo $78$,
     illustrating the tables above.}
  \end{figure}

  Similar remarks apply to any A component for an $n/p$ rabbit ring
  with $p\ge 3$. However for $p=2$, as in Figure~\ref{f-s2par}, there are only
  four rays landing at the root points of the $\pm 2a$ components. Therefore,
  in this case we get the extra relation $\gamma'=\widehat\delta$. Similarly
  we get this extra relation for the left and right rings in
  Figure~\ref{F-4rings}, in a case with $p=4$.  This is probably the
  case since in  both cases previously described  the type B components are
  of type B(2, 2), hence weakly self-dual. 

\bsk

  The B case, as shown in Figure \ref{F-AorB}-right, is quite different.
  There are now ten angles to compare, rather than eight, and twin angles
  do not play any role,  since  type B components  of types
  B(1,3) or B(3,1) are never weakly self-dual.
  Each of the four angles in the lower half of the figure
  is duplicated once in the upper half. The two outer angles above, $\delta$
  and $\gamma'$  do not occur below. However 
  $\delta$ belongs to the same grand orbit as $\beta$, and $\gamma'$
  belongs to the same grand orbit as $\alpha'$. Thus we have three grand orbits
  $$ \(\alpha\)=\(\gamma\)=\(\beta'\)=\(\delta'\),\qquad \(\beta\)=\(\delta\),
  \quad  {\rm and}\quad \(\alpha'\)=\(\gamma'\)~,$$
  which are conjecturally always distinct.\msk

  In the special case of a ring consisting only of components of type
  B, as in Figure~\ref{F-2rings}-left, there is a simple explicit description.
  Let $m$ be the number of B components in the ring. (Perhaps this
  number is always even, but we are not at all sure of this.) Let us
 label the components in cyclic order as $B_1,~B_2,~\cdots B_{m}$. Let $O_j$
  be the common grand orbit for the $\alpha$ and $\beta'$ angles of $B_j$.
  Here $j$ should be understood as an integer modulo $m$. Then as we go around
  the ring, the four grand orbits associated with $B_j$ are respectively
  $$  O_j,~ O_{j-1},~ O_{j+1},~ O_j ~.$$
Each $O_j$ occurs twice, and there are a total of $m$ grand orbits,
conjecturally all distinct.
\msk

  In all other cases that we have seen, there are two diametrically opposite
   A components and $~n-2~$ B components. We will numbers the
  components as
  $$A_0,~B_1,~ \cdots,~ B_{n-1},~ A_n, ~B_{n+1},~\cdots,~B_{2n-1}~.$$
  {\sf  Then conjecturally the grand orbit $O_j$ associated with $A_j$ or
 $B_j$ satisfies  $O_j=O_{2n-j}$ for all $j$, but no other equalities.
  Thus in this case there are $n+1$ distinct grand  orbits.}
Compare Figures~\ref{F-S2rays} and \ref{F-2rings}-right for the case $n=2$,
Figure~\ref{F-3rab} for $n=3$, and \ref{F-4rab} for $n=4$.\msk

However there is one difference between the rings surrounding rabbit regions
and the similar rings in Figure~\ref{F-2rings}.  Each rabbit region in $\ocS_p$
is invariant under the $180^\circ$ rotation $\I$ of $\ocS_p$. Hence each
parameter angle $\theta_j$  around such a rabbit ring satisfies 
$$ \theta_{j+p} ~=~ \theta_j \pm 3(3^p-1)/2~.$$
On the other hand, $\I$ interchanges the left and right rings of
Figure~\ref{F-2rings}.
Therefore for each angle $\theta_j$ in one of these rings, the ray of angle
$\theta_j \pm 3(3^p-1)/2$ lands in the opposite ring. A similar remark
applies to the upper and lower rings in Figure~\ref{F-2rings}.
\end{rem}

\begin{rem}[\bf Three versus Four] There is a strong resemblance
  between the four ray cases  described above and the corresponding three
  ray cases of Figures~\ref{F-Acent}
  and \ref{F-Bcart}. More explicitly, if we delete the $\delta$ and
  $\gamma'$ rays in Figure~\ref{F-AorB} and change the label $\delta'$ to
  $\gamma'$, then we will get the equivalent of Figures~\ref{F-Acent} and
  \ref{F-Bcart}. However there is one essential difference: There are always
  more distinct grand orbits in the three ray case.
  \end{rem}

\begin{figure}[htb!]
\begin{minipage}{3.8in}
  \centerline{\includegraphics[width=3.8in]{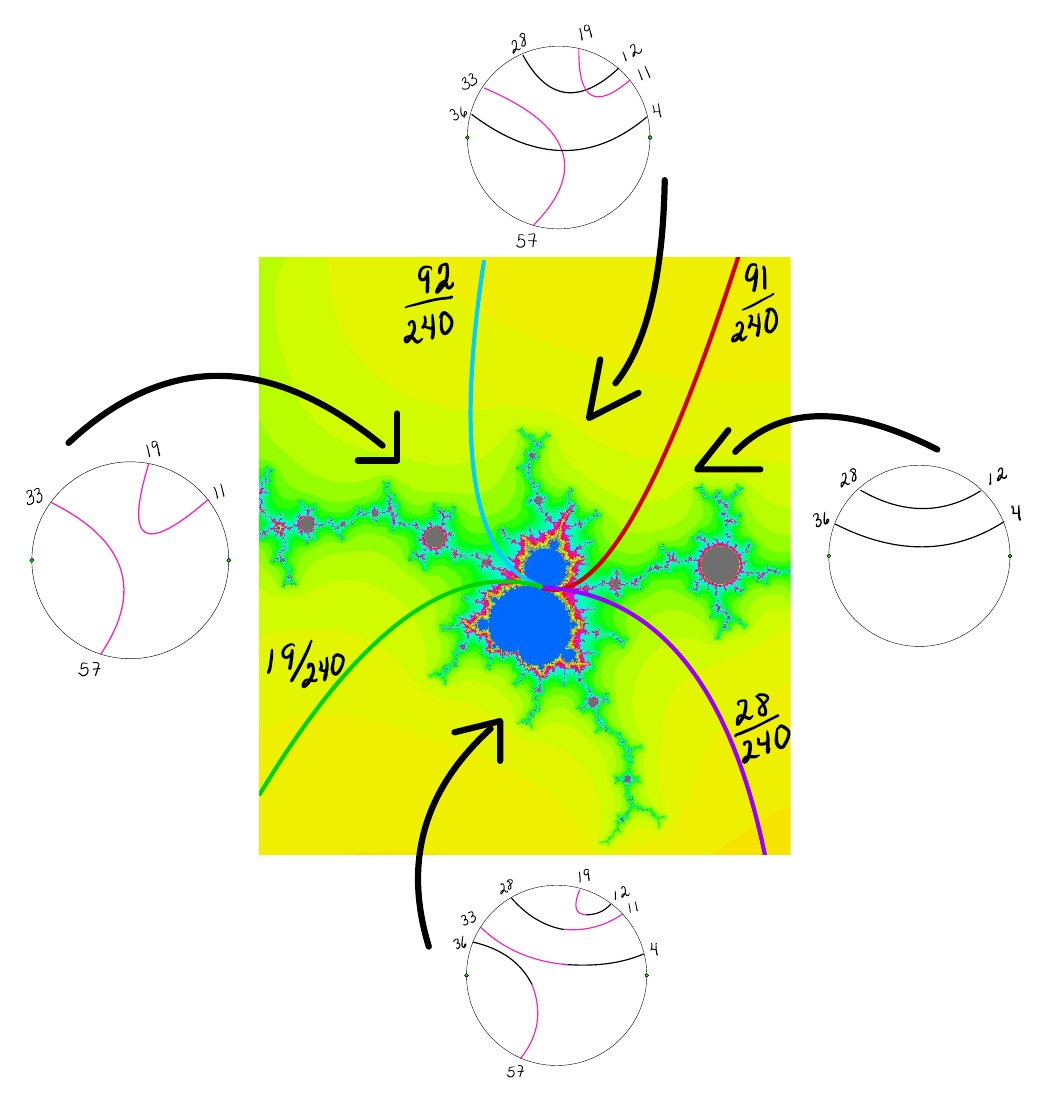}}
\end{minipage}\,\begin{minipage}{3in}
    \centerline{\includegraphics[height=3in]{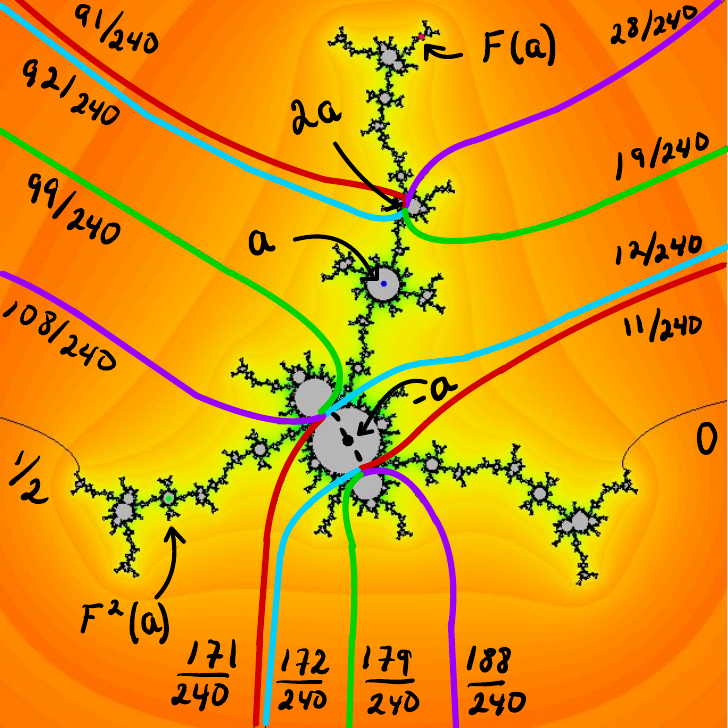}}
    \end{minipage}
    \caption[Period 4 orbit portraits for Mandelbrot copy across the boundary between $010+$ and airplane region]{\label{F-43} \sf {\bf On the left:}
Period 4 orbit portraits for  a small  Mandelbrot copy which lies 
across the bound\-ary between $010+$ region below and the airplane
region above. (This copy is barely visible in Figure~\ref{F-t2s3} as the end
point of the  2 ray in the  $010+$ region to the lower 
middle right, and is somewhat more visible, at the lower left of
Figure~\ref{F-air}, just above the 7 and 8 rays.)
Here $(\alpha,\beta,\gamma,\delta) =(91,~92,~19, ~28)/240$, while
$(\alpha_1,\beta_1,\gamma_1,\delta_1)~=~(11,\,12,\,19,\,28)/80$.\quad\break
{\bf On the right:} Julia set for the center point of the $1/2$ limb of the
Figure on the left. (Compare Figure \ref{F-HT2}-left which shows the same Julia
set with the orbit of $-a$ also marked.) 
Here the co-periodic rays  $(\alpha,\, \beta,\, \gamma,\, \delta)$ 
 land at the root point of $U(2a)$. The
 corresponding period $4$ rays $(\alpha_4,\,\beta_4,\,\gamma_4,\,
 \delta_4) = (171,\,12,\,99,\,108)/240$  land at the root point of
 $U(-a)$, which is periodic of period $2$. Their twins 
 $(\widehat\alpha,\,\widehat\beta,\,\widehat\gamma,\,\widehat\delta) =
 (11,\, 172,\,179,\, 188)/240$  land at the diametrically 
 opposite boundary point of $U(-a)$. The iterate $F^{\circ 2}$ maps $U(-a)$
 to the neighboring Fatou component to its upper left with degree two, mapping
 both the root point  and its opposite point to the common root point.
The point $-2a$ lies between the 172 and 179 rays. 
 Here the coloring
 % ,  interchanging red and green rays, and also interchanging blue and purple rays,
 is chosen  to emphasize kneading walls (Definition \ref{D-wall}). 
  The red or  blue kneading walls yield a zero kneading sequence,
corresponding to the airplane region;
  while the green or purple walls yield $010$. }
\end{figure}

\begin{rem}[\bf Small Mandelbrot Copies]\label{R-sMc}
The space $\cS_3$ contains many Mandelbrot copies
which lie on the boundary between two different escape regions.
  More precisely there are three possibilities: Such a copy $M$ lies:

  \begin{quote}  \textbf{\textit{On one side}} of the boundary if its cusp point
    is the only point which is a boundary point of two escape regions
  (Figure~\ref{M53-rays}). It lies:

\textbf{\textit{Across}} the boundary if the boundary cuts across at a
parabolic point which bisects M (Figures \ref{F-43} and 
\ref{F-q8ex})  or:

\textbf{\textit{Along}} the boundary
if there are many components of $M$ which have boundary points on both
of the two escape region boundaries.  (See Figure  \ref{f-M(53)bdry}.)
\end{quote}\msk

\noindent
In fact any part of the boundary which shows up as a thin wiggly line
in our large scale pictures will contain many such copies. Compare
Figures~\ref{f-M(53)bdry} and \ref{M53-rays}. Note that the 
size of a Mandelbrot copy of period $q$ in $\cS_p$
tends to shrink exponentially as $q$ increases with fixed $p$;
so most of these copies are extremely small.
\end{rem}

\begin{figure}[ht!]
  \begin{center}
  \begin{minipage}{3in}\includegraphics[width=3in]{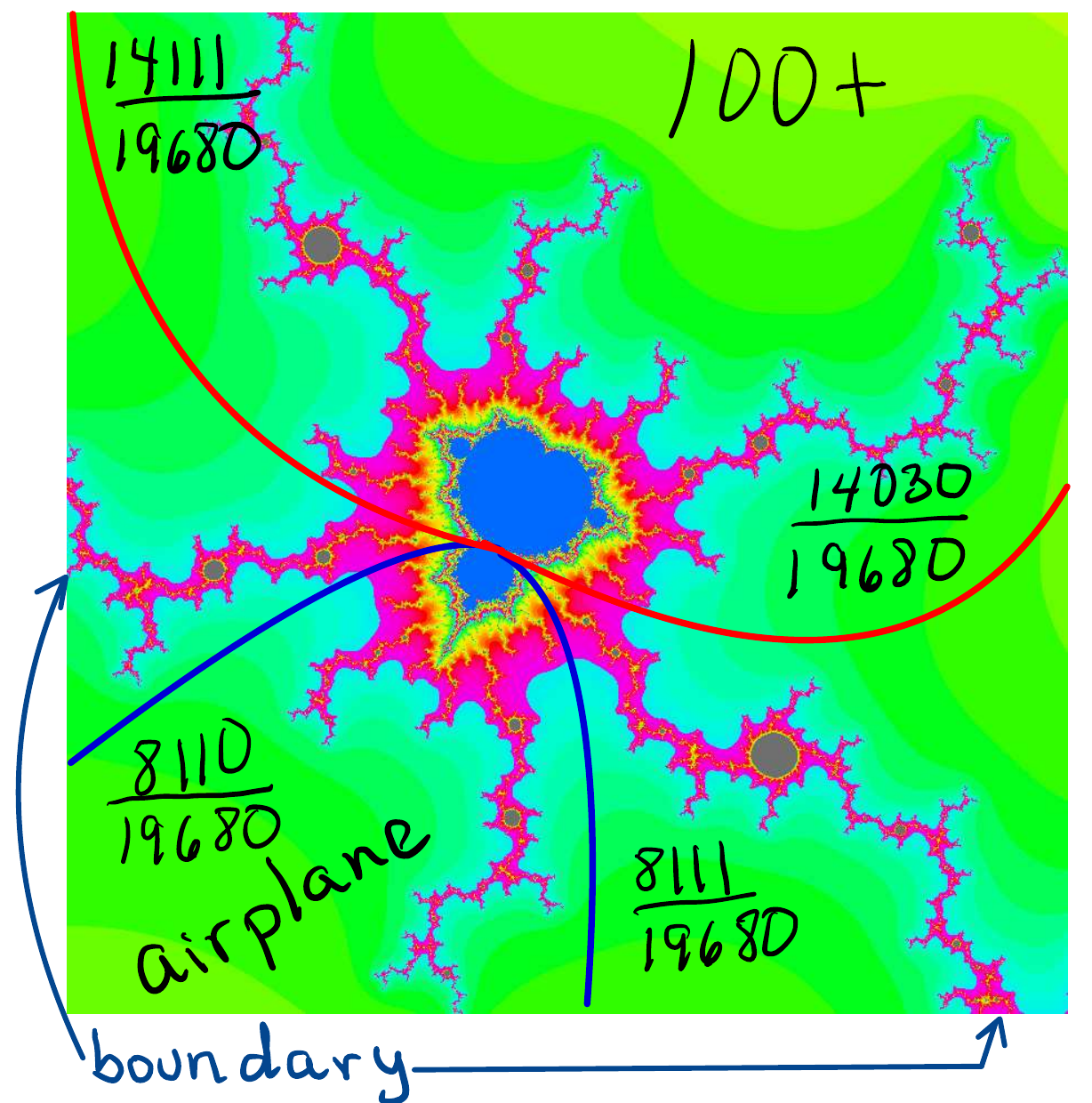}\end{minipage}\quad
  \begin{minipage}{3in} \includegraphics[width=3in]{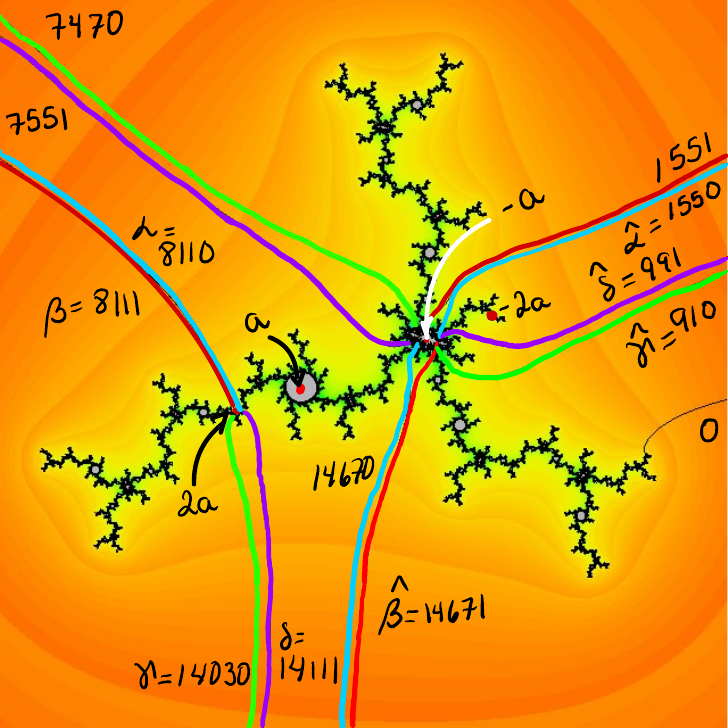}\end{minipage}
  \vspace{.4cm}
    \centerline{\includegraphics[width=2.7in]{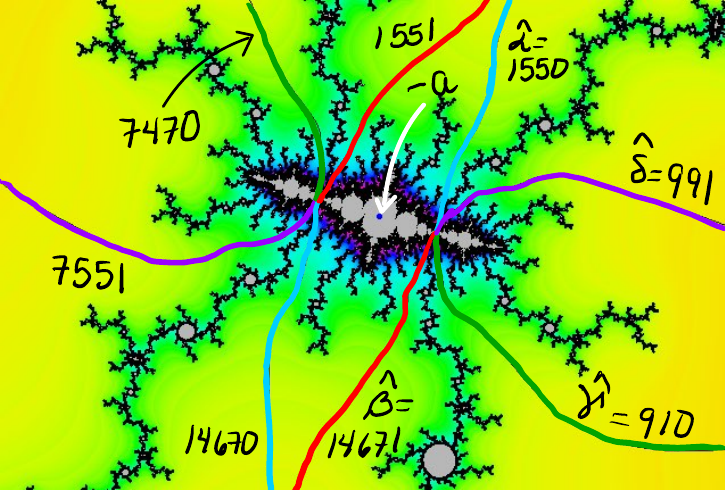}}
  \end{center}
\caption[Small Mandelbrot copy of period 4 between the airplane and $100+$ regions in $\cS_3$.]{\label{F-q8ex} \sf On the left: a small Mandelbrot copy of period 4
across the boundary between the airplane and $100+$ regions of $\cS_3$. On
the right: Julia set for the center point of its $1/2$ limb. Below: a magnified
picture around the $-2a$ component. The denominator is $3(3^8-1)=19680$. The
periodic dynamic rays land at the point of period two which lies
between $-a$ and
$F^{\circ 2}(-a)$; while associated co-periodic rays land at the pre-image
of the fixed point which lies to the right of $-a$. The twin co-periodic
rays land at the root point of the $2a$ component.}
\end{figure}
\medskip

\begin{rem}
There is something  remarkable about Figures~\ref{F-43} and \ref{F-q8ex}.
Note that there are four periodic parameter rays landing at a periodic
point close to the $-a$ component,  and four corresponding co-periodic rays landing
 at a pre-image of this point  on the opposite side of $-a$.
These are precisely the twins of the angles of the rays landing at the $2a$
root point. This is exactly what happens in examples such as
Figure~\ref{F-3rabjul}; but
in such cases there is an explanation in terms of the two root points
of an associated A component. Here there is no 
such A component. How general is this behavior, and why does it occur?
\end{rem}\msk

\begin{conj}[\bf Across the Boundary Conjecture]\label{conj-atbc} 
  If a Mandelbrot copy in $\cS_p$ cuts across the boundary between two escape
  regions at a parabolic point of period $q$, then there are 4 parameter rays
  landing on $\p$ . In the Julia set for $\p$ there are 4 corresponding
  triads. Now suppose that $p\ne q$. The  periodic dynamic rays
  $\theta_q$ land at the point of period at least two  which lies
between the marked critical point $a$ and the free critical point $-a$. The
co-periodic rays $\theta$ land at the root point of the hyperbolic component
containing $2a$, and  their corresponding  twins rays $\widehat{\theta}$
land on a preimage of $-a$. Here the mid-point between the periodic point
of period $q \geq 1$, where the periodic rays $\theta_q$ land and the landing
point of the twin rays $\widehat{\theta}$ is always
the free critical point $-a$.  When the
parabolic point $\p$ is self dual, as in the case of a type A component,  the
twin rays land precisely at the root point of $-2a$. 
  \end{conj}

{\bf Note.} In all of the parabolic Julia sets we have seen with four or more
dynamic rays landing at a common  point, 
there are four parameter rays landing at the corresponding point
in parameter space.
On the other hand there can certainly be three dynamic rays
landing at a common parabolic point in  examples where no three 
co-periodic parameter rays land together. See for example the 
landing point of the $62$ and $64$ rays in Figures $\ref{F-air}$. (We believe
that this happens whenever the two parameter rays belong to the same grand
orbit, and $q$ is odd.)

\bsk

\setcounter{lem}{0}
\section{Counting} \label{s-count}
This section will be concerned with the numbers of vertices, edges, and faces
in the period $q$ tessellation $\Tes_q(\overline\cS_p)$. We will give explicit
formulas for the number of edges, and (conjecturally at least) for
the number of parabolic vertices. The number $\N_p$ of ideal vertices ($=$
number of escape regions) is more difficult; but has been computed for
$p\le 26$ by DeMarco and Schiff \cite{DM-S}. The number of faces
seems to be even more difficult to compute, so we know it only for small values
of $q$ and $p$. (See Table \ref{t3} at the end of this section.)

\begin{definition}\label{D-dp}
One important number is the degree 
$\d_p$ of the affine curve $\cS_p$, which is defined uniquely by the recursive
formula
$$ 3^{n-1} ~=~\sum_{p|n} \d_p .$$
The first few values are shown in  Table~\ref{t1}.
Note that $\d_p\sim 3^{p-1}$ as $p\to\infty$.

\begin{table}[h!]
  \begin{center}
  \begin{tabular}{ccccccccccc}
    \\
$p:$ & 1     & 2& 3& 4 & 5 & 6 & 7& 8&9&$\cdots$\\[1ex]
$\d_p:$ & 1&2& 8 & 24 & 80 & 232 &728&2160& 6552&$\cdots$\\
\end{tabular}
\caption{\sf The degree $\d_p$ of $\cS_p$.   \label{t1}}
\end{center}
\end{table}

\noindent This number $\d_p$ is also equal to the number of hyperbolic
components of Type A in $\cS_p$. (See \cite[Lemma 5.4]{M4}.)
For each $m,n>0$ with $m+n=p$, the number of components of
Type B$(m,n)$ in $\cS_p$ is also equal to $\d_p$ 
(See Theorem \ref{T-Bcount}.). Furthermore, the number of  escape regions
in $\cS_p$ \textbf{\textit{counted with multiplicity}} is also equal to $\d_p$.
(Compare \cite[Remark\,5.5 and Corollary 5.12]{M4}.)
\end{definition}

The actual number $\N_p$ of escape regions in $\cS_p$ is much more difficult
to determine. For relatively small values of $p$ it has been computed by
DeMarco and Schiff \cite{DM-S}. Here are a few values. 
\begin{table}[h!]
\begin{center}
\begin{tabular}{ccccccccccc}
$p:$& 1&2&3&4&5&6&7 & 8&9&$\cdots$\\[1ex]
$\N_p:$& 1&2&8& 20&56&144&404&1112&3120& $\cdots$\\
\end{tabular}
\caption{\sf The number of escape regions in $\cS_p$.
 \label{t2}}
\end{center}
\end{table}

Note that the ratio
$\d_p/\N_p\ge 1$ can be described as the \textbf{\textit{average multiplicity}}
of the various escape regions in $\cS_p$.  (Compare Figure~\ref{F-aven} and
Remark~\ref{R-AM}.)
\smallskip

For any tessellation of a compact surface into simply-connected faces,
the classical Euler characteristic formula takes the form
$\quad\v-\e+\f~=~\chi({\ocS})$, 
where $\v$ is  the number of vertices,  $\e$ is the number of edges,
and $\f$ is the number of faces.
However, in all of our  tessellations ${\Tes}_q(\overline\cS_p)$ with
$1<n\ne q$, some of the faces are not simply-connected, 
so a correction term is needed. %(Compare Corollary \ref{C-notcon}.) 
The appropriate corrected formula is
\begin{equation}\label{E-eul}
  \v\,-\,\e\,+\,\f~=~\chi(\ocS)~+~ {\rm rank}\big(H_1(\F^{\,\cup})\big)
  ~=~\chi(\ocS)~+~ \sum_k{\rm rank}\big(H_1(\F_k)\big)~,
\end{equation}
using homology with rational coefficients,\footnote{There is no
  torsion in the cases we consider, so it doesn't
  matter what coefficient field we use. } 
 where $\F^{\,\cup}=\bigcup\F_k$ is the union of the open faces. 

 Here the Euler characteristic of the compactified moduli space
 is related to the number $\N_p$ of escape regions by the formula
 \begin{equation}\label{E-Eu}
    \chi(\overline\cS_p)~=~2-2\,{\bf g}~=~  \N_p  +(2-p)\d_p~. \end{equation}
 Thus the genus
 $${\bf g}~=~ 1+\big((p-2)\d_p-\N_p\big)/2 ~>~ (p-3)\d_p/2 $$
 
 \noindent
 grows very rapidly with $p$.
 (See ~\cite[Theorem~7.2]{BKM}.) Here the term $(2-p)\d_p$ can be described
 as the Euler characteristic of the open manifold $\cS_p$, or of the
 connectedness locus which is a nested intersection of
 deformation retracts of $\cS_p$.
\medskip

{\bf Angles.} The number of periodic angles of period $q$
can easily be computed as
\begin{equation}
  {  \rm \#~period}~q~~=~\begin{cases} 2 \qquad\;  {\rm if}\quad q=1~,\\
    3\,\d_q\quad{\rm if} \quad q>1~.\end{cases}
\end{equation}
To obtain the number of co-periodic angles $\theta\pm 1/3$ we must multiply
these numbers by two. Here is a list of the first few values

\begin{table}[htb!]
\begin{center}          
  \begin{tabular}{cccccccccc}  \\
    $q:$ & 1     & 2& 3& 4 & 5 & 6 & 7& 8&$\cdots$\\[1ex]
    periodic angles:& 2&6&24&72&240& 696&2184& 6480&$\cdots$\\
co-periodic: & 4& 12 & 48 & 144 & 480 & 1392&4368& 12960&$\cdots$\\
\end{tabular}
\caption{\sf The number of angles of period or co-period $q$.
  \label{t1a}}
\end{center}
\end{table}
\medskip

{\bf Edges.} The number $\e$ of edges 
in our tessellation is equal to
the product of the number of angles of co-period $q$,
and the number of escape regions counted with multiplicity.
Thus it is given by
\begin{equation}\label{E-edge} \e~=~\e_q(\cS_p)~=~\begin{cases}
   4\,\d_p \qquad \; {\rm for}\quad q=1~,\quad{\rm but}\\
  6\,\d_q\,\d_p \quad{\rm for}\quad q>1~.
\end{cases}\end{equation}
\medskip

%characteristic. However we had shown in ~\cite[Theorem~7.2]{BKM}~ that
%$~~\chi(\overline\cS_p)= \N(\cS_p  characteristic. However we had shown in ~\cite[Theorem~7.2]{BKM}~ that
%$~~\chi(\overline\cS_p)= \N(\cS_p)+(2-p)\d_p~.$}  )+(2-p)\d_p~.$}
%(\cS_p)+(2-p)\d_p~.$}

{\bf Vertices.} We must distinguish between ideal vertices and parabolic
vertices. The number of ideal vertices is clearly equal to the number $\N_p$ of
escape regions, as listed above.

To compute the number of parabolic vertices, we proceed as follows. By
definition a hyperbolic component in $\cS_p$ is of Type D$(q,p)$ if the
free critical orbit for its center map is periodic of period $q$, and
is disjoint from the marked critical orbit. (Here the use of a double index is
convenient, since  there is a natural one-to-one correspondence between
components of type D$(q)$  in $\cS_p$ 
and their dual components of Type D$(p)$ in $\cS_q$ with $p\neq q$.)\medskip

\begin{lem}\label{L-MC} If we assume the Conjecture ${\rm MC1}$ of 
  Remark~$\ref{R-MandelC}$, then it follows that there is a one-to-one
  correspondence between components of Type D$(q)$ in $\cS_p$ 
 and parabolic vertices in $\Tes_q(\ocS_p)$.  \end{lem}

\begin{proof} This conjecture implies that no two components of
  Type D can have the same root point, and the conclusion follows
  immediately.\end{proof}\ssk

\begin{lem}\label{L-Dcount} For every $p$ there are precisely $3\,\d_p^{\,2}$
  hyperbolic components of Type either {\rm A, B,} or {\rm D}$(p,p)$ in the
   curve $\cS_p$. Similarly, for $q\ne p$, there are precisely $3\,\d_q\,\d_p$
   components of Type {\rm D}$(q,p)$ in $\cS_p$.
   \end{lem}

   The proof in \cite[Lemma 5.8]{M4} involves the dual curve $\cS'_q$
   consisting of all pairs $(a,v)\in\C^2$ such that the critical point $-a$
   has period exactly $q$ under the usual map\break $F(z)=z^3-3a^2z +(2a^3+v)$.
   This dual curve
   has degree $3\d_q$ in $\C^2$. It is important to note that the two curves
   $\cS_p$ and $\cS'_q$ intersect transversally.
   This transversality can be proved using Thurston’s basic construction of
   rational  maps from topological data. (See \cite{DH2}.) A purely algebraic
   proof of transversality has been provided by Silverman \cite{Si}.  The
   number of intersections between these  two curves is then counted using
   Bezout's Theorem. When $p\ne q$, the intersection
points are precisely the centers of components of Type D$(q,p)$; but for $q=p$
the centers of A and B components are also included in this count.\qed
\medskip

As an immediate consequence, assuming Conjecture~{\rm MC1}, %\ref{Cj-Man},
we have
the following formula for the number of parabolic vertices in $\Tes_q(\ocS_p)$.
\begin{equation}\label{E-q-ne-p}
  {\bf v}_{\rm par}~=~3\, \d_q\,\d_p\qquad{\rm  whenever}~~~ q\ne p~.
\end{equation}
On the other hand, for $q=p$ we need to count the number of components of Type
A and B in order to compute ${\bf v}_{\rm par}$. The number of components
of Type A in $\cS_p$ is precisely equal to $\d_p$. See \cite[Lemma 5.4]{M4}
and its proof. For Type B we proceed as follows.

\subsection*{\bf Counting Type B Components.}
By definition a hyperbolic component in $\cS_p$  is of Type B$(m,n)$ if its
center point $F$ satisfies $F^{\circ m}(a) =-a$ and $F^{\circ n}(-a)=a$, where
$m+n=p$.

\begin{theo}\label{T-Bcount}
For each $m, n>0$ with $m+n=p$ the number of components
of Type {\rm B}$(m, n)$ in $\cS_p$ is equal to the degree $\d_p$ of
the curve $\cS_p$.
\end{theo}
\medskip

The proof proceeds as follows.\medskip

Let $\cS\cong\C^2$ be the full space of maps $F=z^3-3a^2z+2a^3+v$.
For any $m,n\ge 0$ with $m+n>0$, let $\cB(m,n)\subset \cS$ be the finite
subset consisting of all $F$ which satisfy the equations
$$ F^{\circ m}(a)=-a\qquad{\rm and}\qquad  F^{\circ n}(-a)=a~.$$
Note that $\cB(m,n)\subset \cB(km, kn)$ for any $k>0$.
If $m+n=p$,  the subset $\cB(m,n)\cap\cS_p$ will be denoted by $\cB_0(m,n)$.
Evidently $\cB_0(m,n)$ is precisely the set of
all center points of hyperbolic components of Type B$(m,n)$ in
$\cS_{m+n}$.
\medskip

{\bf Caution:} The sets $\cB(0,p)$ and $\cB(p,0)$ are identical,   
and should not be distinguished from each other.
This set $\cB(0,p)=\cB(p,0)$ can be described
as the set of all unicritical maps $F(z)=z^3+v$ for which $F^{\circ p}(0)=0$.
In particular, $\cB(1,0)=\cB(0,1)$ consists of a  single map $F(z)=z^3$
in $\cS_1$, and is contained in every $\cB(m,n)$.
\medskip

\begin{figure}[t!]
\centerline{\includegraphics[width=1.8in]{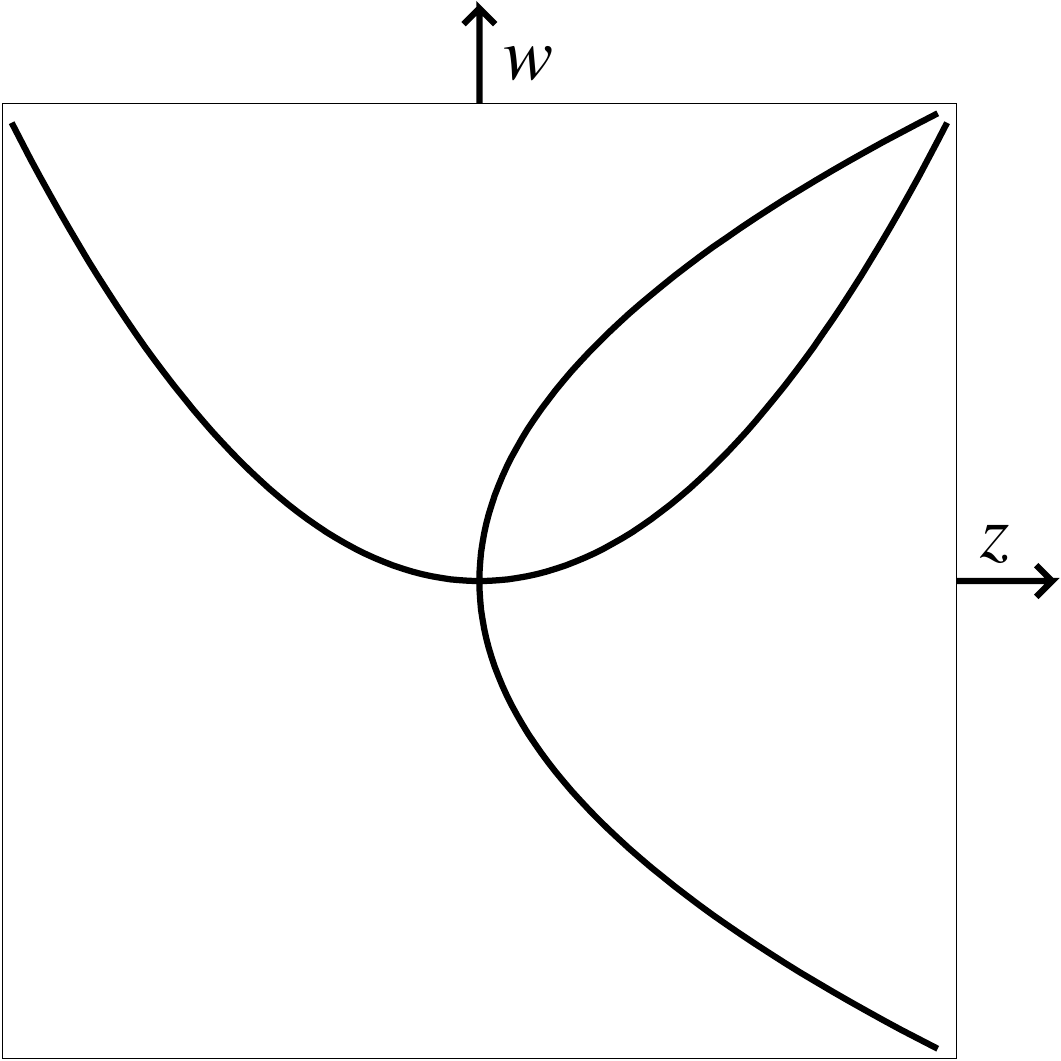}}
\caption[Two transverse curves]{\label{F-Bgraf}\sf Two transverse curves.}
\end{figure}

\begin{lem}\label{l-bcount1}
The curves $F^{\circ m}(a)=-a$ and $F^{\circ n}(-a)=a$ intersect
transversally within $\cS$, and have precisely $3^{p-1}$
intersection points.\end{lem}
\medskip  

\begin{proof} We must first prove transversality. Each intersection
point is the center of a hyperbolic component in $\cS\cong\C^2$.
The paper \cite{M4} provides a canonical biholomorphic model for such
hyperbolic components. In this particular case, the model is
a polynomial map from the disjoint union of two copies of $\C$
to itself. If $z$ is the coordinate for the first copy and $w$
for the second, then the map sends $z$ to $w=z^2$ and $w$ to
$z=w^2$. If we restrict to the case where $z$ and $w$ are real,
then the graphs of these two maps are shown in Figure~\ref{F-Bgraf}.
Evidently these graphs intersect transversally, so the corresponding
complex graphs also intersect transversally.

The conclusion then follows easily from Bezout's Theorem. (Note
that the equation $F^{\circ m}(a)=-a$ has degree $3^{m-1}$  when
$m>0$, since it can be written as $F^{\circ  m-1}(v)=-a$.
Bezout's Theorem must actually be applied in the complex projective
plane which is obtained from $\cS$ by adding a line at infinity;
but it is not hard to check that there are no intersection points
on this line at infinity.)
\end{proof}
\medskip

\begin{lem}\label{l-bcount2}
Suppose that $m'\le m$ and $n'\le n$. Then
$\cB(m', n')\subset \cB(m, n)$  if and only if
\begin{equation}\label{E-bcount}
 m\equiv m'\quad{\rm and}\quad  n\equiv n'\qquad{\rm modulo}\quad p'=m'+n'~.\end{equation}
In particular, whenever $\cB(m', n')\subset \cB(m, n)$ it follows that 
$p=m+n\equiv 0 \quad({\rm mod}~p')$.
\end{lem}
\medskip

The proof is straightforward. Any map $F\in \cB(m', n')$ certainly satisfies
$F^{\circ p'}(a)=a$. Given (\ref{E-bcount}),
since $m\equiv m' ~({\rm mod}~ p')$, we can write $m=m'+kp'$  for some $k$;
and it follows that $F^{\circ m}(a)=F^{\circ m'}(a)=-a$.
On the other hand, given that $\cB(m', n')\subset \cB(m, n)$, choose some
$F\in\cB_0(m',n')$. Then $F^{\circ  j}(a)=-a$ if and only if
$j\equiv m'~~({\rm mod}~p')$. But
$F$ must also belong to $\cB(m,n)$, hence
$F^{\circ m}(a)=-a$, which proves that $m\equiv  m'~~({\rm mod}~p')$.
\qed

%example if $F(a)=-a$ and $F(-a)=a$ so that $F\in \cB(1, 1)$, then it follows
%that $F^3(a)=-a$ so that $F\in S(3,1)$.)\qed
\bigskip

%Let $\cB_0(m, n) = \cB(m, n)\cap \cB_p$ be the subset   
%of $\cB(m, n)$ consisting of maps for which the marked critical point $a$
%has period exactly $p=m+n$. Evidently $\cB(m, n)$ is precisely the set of
%center points of components of Type B$(m,n)$.

\begin{lem}\label{l-bcount3}
For any divisor $p'$ of $p$ we can write
$$   \cB(m, n)\cap \cS_{p'}~=~\cB_0(m',  n')~,$$
where the integers $m'$ and $n'$ are uniquely determined by the
conditions that
$$m'+n'=p',\quad m'\equiv m\,,\quad{\rm and}\quad n'\equiv n~~~{\rm modulo}
~p'~.$$
\end{lem}
\msk

\begin{proof}
This follows immediately from Lemma~\ref{l-bcount2}. In fact, for each $p'$  
there is a unique
$0<m'\le p'$ with $m'\equiv m\quad ({\rm mod}~p')$. Then $n'=p'-m'$
must satisfy $n'\equiv n\quad ({\rm mod}~p')$. 
\end{proof}
\medskip

\begin{lem}\label{l-bcount4}
The set $\cB(m,n)$ is the disjoint union of the sets
$\cB(m,n)\cap \cS_{p'}$ as $p'$ varies over all divisors of $p$. Since
$\cB(m,n)$ has $3^{p-1}$ elements, and since $\cB(1,0)$ has $\d_1=1$ element,
it follows inductively from the defining formula
$$ 3^{p-1}=\sum_{p'|p} \d_{p'} $$
that $\cB_0(m,n)$ has $\d_p$ elements.
\end{lem}\medskip

{\rm The proof is straightforward; and Theorem \ref{T-Bcount} follows 
immediately.}\qed

\medskip
 \begin{theo}\label{T-v-count}
   Assuming the Mandelbrot Copy Conjecture~{\rm MC1}, it follows that the
   number of parabolic vertices in $\Tes_q(\ocS_p)$ is equal to $3\d_q\d_p$ if
   $p\ne q$, but is equal to $3\d_p^{\;2}-p\,\d_p$ if $p=q$.
   \end{theo}

\begin{proof}  We have shown that
$\cS_p$ contains $(p-1)\d_p$ components of Type B plus $\d_p$ components
of Type A, yielding a total of $p\,\d_p$ components of Types A and D.
Combining this with Lemmas \ref{L-MC} and  \ref{L-Dcount}, 
the conclusion follows.\end{proof}
\medskip

\begin{rem}\label{R-face} If we combine Theorem \ref{T-v-count} with the Equations
  (\ref{E-Eu}) and (\ref{E-edge}), then we can solve for the number $\f$ of faces,
  yielding
$$ \f~=~\begin{cases} 3\,\d_q\,\d_p+2\,\d_p-p\,\d_p+{\rm rank}\big(H_1(\F^{\,\cup})\big)
    & {\rm if}\quad p\ne q>1\\
    3\,\d_p^{\,2}~~+~~2\,\d_p \quad+~{\rm rank}\big(H_1(\F^{\,\cup})\big)
    & {\rm if}\quad p= q>1~; \end{cases}$$
  with similar formulas when $q=1$.   (Thus the $\N_p$ terms cancel out.)
  When $p=q$ it seems possible that all faces are simply connected,
  so that the last term is always zero. But for $p\ne q$ and $p>1$, 
 there are always faces surrounding the rabbit regions
 which are not simply connected by %Corollary \ref{C-notcon},
 \cite[Corollary 4.3]{BM}, so that the last
  term is always strictly positive.
\end{rem}\medskip

{\bf Counting Non-Simply Connected Faces.}
To compute the rank of $H_1(\F^\cup)$, note first that the inclusion map
$$i:\F^\cup\to \ocS_p$$ induces a homomorphism
$~ i_*:H_1(\F^\cup)\to H_1(\ocS_p).$ It follows that
\begin{equation}\label{E-H1}
\rank\big(H_1(\F^\cup)\big)~=~ \rank\big({\rm kernel}(i_*)\big)
~+~\rank\big({\rm image}(i_*)\big).\end{equation}

Alternatively we can compute these two summands in terms of
the \textbf{\textit{1-skeleton}}  \hbox{$X=\ocS_p\ssm\F$} of $\Tes_q(\ocS_p)$,
that is, the union of all vertices and edges.
\medskip

More generally, let $S$ be any closed Riemann surface of genus $\g$, 
and let $X$ be any non-vacuous finite graph which is embedded in $S$.
(In our application, $X$ will be the 1-skeleton and its complement
will be the union $\F^\cup$ of faces.) %$\F$ will be the complement $S\ssm X$.)
\medskip

\begin{theo}[\bf Lefschetz Duality]\label{t-duality} 
  The kernel of the natural homomorphism
  \begin{equation}\label{e-dual}
h:    H_1(S\ssm X)~~\to~~ H_1(S)
    \end{equation}
has rank equal to one less than the number of connected components of $X$.
Furthermore the sum of the ranks of the images of the homomorphism
$(\ref{e-dual})$ and 
of the corresponding homomorphism $H_1(X)\to H_1(S)$ is always equal to the
rank $2\,\g$ of $H_1(S)$.
\end{theo}

\begin{proof}
  Here coefficients in the rational field $\Q$ are to be understood.
Hence  any cohomology
group $H^j(*)$ can be identified with the dual ${\rm Hom}\big(H_j(*),\,\Q\big)$
 of the corresponding homology group.

 The Lefschetz Duality Theorem\footnote{See for example \cite{Ha} or \cite{Sp}.
   The Lefschetz Theorem is an amalgamation of Poincar\'e and Alexander
   duality.}
 then asserts that the 
following two exact sequences are dually paired to each other.\vspace{-.3cm}

$$\vcenter{\vbox{
\xymatrix@-1pc{H_2(S)\ar[r] & H_2(S,\,S\ssm X)\ar[r]& H_1(S\ssm X) \ar[r]& H_1(S) \ar[r]&H_1(S,\,S\ssm X)\\
  H_0(S)& \ar[l] H_0(X)& \ar[l] H_1(S,\,X)& \ar[l] H_1(S)& \ar[l] H_1(X)}
}}$$

\noindent That is, each group or homomorphism above is naturally isomorphic to
 the dual of the group or homomorphism below it. In particular, the second
homomorphism above has rank equal to the rank of the kernel of the homomorphism
$(\ref{e-dual})$.
This is equal to the rank of the second homomorphism below, which is easily
seen  to be one less than the number of connected components of $X$.

 On the other hand, the rank of $H_1(S)$, either above or below, is equal
 to the sum of the ranks of the homomorphisms to and from it. This implies
 that it is equal to the sum of the ranks of $H_1(S\ssm X)\to H_1(S)$
 and $H_1(X)\to H_1(S)$, as asserted.
 \end{proof}
\medskip

As an immediate corollary, if all faces of the tessellation are
simply connected, then it follows that the 1-skeleton is connected.
(It seems quite possible that
all faces are simply connected whenever either $p=q$ or $p=1$; but we
have no proof of this.) Combining the above formulas, this would imply that
the number of faces is precisely
$$ \f~=~ 3\d_p^{\,2}+ 2\,\d_p $$
when $p=q>1$. 
However \cite[Corollary 4.3]{BM}  implies that the 1-skeleton is not
connected in all other cases.

\medskip

For all cases with $p,\,q\,\le 3$ see Table \ref{t3}. 
In four of these cases, the 1-skeleton is not connected, 
which means that there is at least one face which is not
simply-connected.  For example  $\Tes_2(\cS_3)$ has 40 simply-connected faces,
but the two faces surrounding the rabbit and co-rabbit have the topology 
of an annulus, so that the rank of $H_1(\F^\cup)$ is two.
Here is a table showing numbers of vertices, 
faces and edges for small values of $p$ and $q$.
\begin{table}[htb!]
  \begin{center}
\begin{tabular}{cccccccccc}
   $q$ & $p$ && vertices & edges& faces&&  $\chi(\cS_p)$&
     ${\rm rank}\big(H_1(\F^\cup)\big)$\\
    &&& $\N_p+\v_{\rm par}$ & $\e$  & $\f$ && $2-2\g$&
   $|$kernel($h$)$|$ +$|$image($h$)$|$\\[3ex]
  1&1&& 1+2& 4&3&&2& 0+0\\
  2&1 &&1+6&12&7&&2&0+0\\
3&1&& 1+24 & 48&25&&2&0+0 \\   [3ex]
  1 & 2  && 2+6 & 8   &   3&&2& 1+0    \\
  2 & 2 && 2+8 & 24  & 16&&2& 0+0 \\
3&2&& 2+48 & 96  & 49&&2 & 1+0\\  [3ex]
1&3&& 8+24&32 & 9&&0& 7+2\\
2&3&& 8+48& 96  &42&&0&  2+0\\
3&3&&8+168&384&208&&0&0+0\\
\end{tabular}

\caption{\sf Statistics of $\Tes_q(\cS_p)$ for small $q$ and $p$.
  The two summands in the third column refer to ideal and
  parabolic vertices; while the two summands in the last
  column refer to the kernel and image 
  of the homomorphism $h:H_1(\F^{\,\cup})\to H_1(\cS_p)$.
  Note that
  \centerline{$ \v-\e+\f~=~ \chi(\cS_p)+{\rm rank}\big(H_1(\F^\cup)\big)~~$}
  in all cases.\label{t3}}
\end{center}
\end{table}

\begin{figure}[h!]
  \centerline{\includegraphics[width=2.8in]{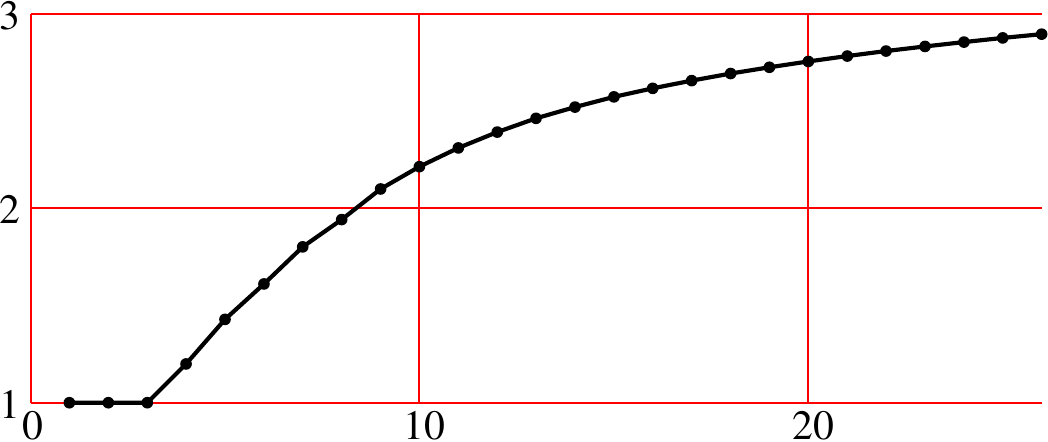}}
  \caption[Graph of the ratio $\d_p/\N_p$ as a function of $p$]{\sf Graph of
    the ratio $\d_p/\N_p$ as a function of $p$  (based on numbers from
    \cite{DM-S}). \label{F-aven}}
\end{figure}
\medskip

\begin{rem}[{\bf Average Multiplicity of the Escape Regions}]\label{R-AM}
Another interesting number associated with $\overline\cS_p$ is the ratio
$r(p)=\d_p/\N_p$
which measures the average of the  multiplicities of the $\N_p$
escape regions.  This seems to increase rather slowly as a function of $p$.
(See Figure \ref{F-aven}.) We have  
no idea of how to estimate its behavior for large $p$. (However,
it seems reasonable to guess that the function $p\mapsto r(p)$
is monotone, with slope tending to zero as $p\to\infty$. It also seem likely,
but not at all certain, that it is unbounded.) 
The Euler characteristic of $\overline\cS_p$
can be computed as $\chi=(1/r(p)+2-p)\d_p$, so it follows that  
$$ p-3 ~\le ~ -\chi(\overline\cS_p)/\d_p ~<~ p-2
\qquad{\rm for~all}~~ p~.$$

Furthermore it follows that $r(p)\to\infty$ as $p\to\infty$ if and only if

\centerline{$p\,+\,\chi(\overline\cS_p)/\d_p$\quad converges to~~~  
  +2~~  as $~p\to\infty $.}\medskip

\noindent   Compare \cite[Table 1]{DM-S}; and note that
$\d_p=3^{p-1}+O(3^{p/2})$, so that $3^{p-1}$ is a very good approximation
to $\d_p$ for large $p$. \end{rem}

\bigskip

\setcounter{lem}{0}
\section{Misiurewicz Maps and Similarity}\label{s-Mis}
Recall that we use the term Misiurewicz map for any $F\in\cS_p$ such that
the orbit of the free critical point $-a$  eventually maps to a repelling
periodic orbit.  Since $F(2a)=F(-a)$, the orbit of $2a$ eventually
maps to the same repelling periodic orbit. But the orbit of the free 
co-critical point $2a$ for a Misiurewicz map can never contain a critical
point. This leads to a great deal of local self similarity. \medskip

Consider for example Figure \ref{F-S2cheb}. Here both $-a$ (at the landing
point of the 5 and 11 rays) and $2a$ (at the landing point of the 17 ray) map to
the free critical value (at the landing point of the 15 ray), which maps to a
repelling fixed point (at the landing point of the 9 ray which is fixed
under tripling mod 18). It follows that the Julia set in a small neighborhood
of $2a$ is conformally isomorphic to the Julia set in a neighborhood of the
free critical value, which is conformally isomorphic to the Julia set in the
neighborhood of this repelling fixed point. 
On the other hand, since $8\mapsto 6\mapsto 0$ under tripling mod $18$, the tiny
configuration at the endpoint of the $8$ ray maps conformally onto the larger
configuration  at the end of the $6$ ray, which maps conformally onto the much
larger configuration at the end of the $0$ ray.

\begin{figure}[t!]
  \begin{center}
    \begin{overpic}[width=3.3in]{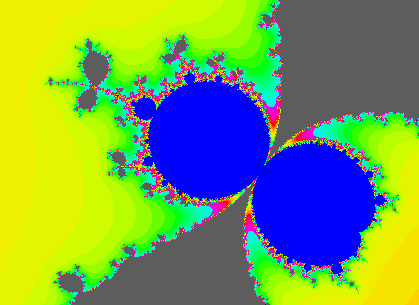}
 \put(215,20){\bf 10}
    \put(15,50){\bf basilica}  
\end{overpic}\,\begin{overpic}[width=3.15in]{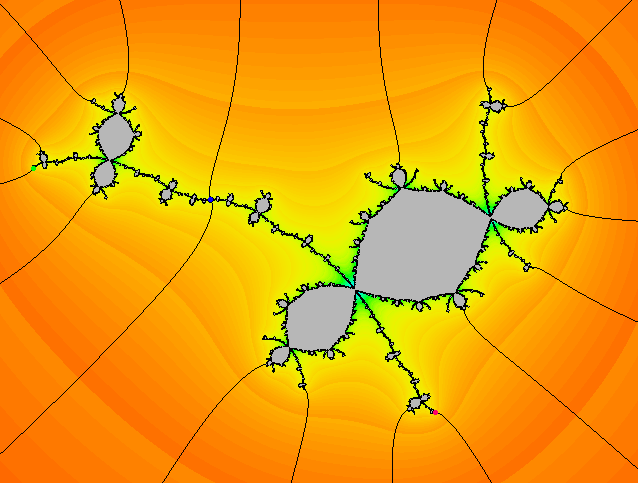}
  \put(221,96){\bf 0}
  \put(221,130){\bf 1}
  \put(207,163){\bf 2}
  \put(179,163){\bf 3}
  \put(139,163){\bf 4}
  \put(89,163){\bf 5}
  \put(47,163){\bf 6}
  \put(9,163){\bf 7}
  \put(1,133){\bf 8}
  \put(1,110){\bf 9}
  \put(1,82){\bf 10}
  \put(1,22){\bf 11}
  \put(43,1){\bf 12}
  \put(107,1){\bf 13}
  \put(143,1){\bf 14}
  \put(175,1){\bf 15}
  \put(214,16){\bf 16}
  \put(214,65){\bf 17}
  \put(145,82){$\mathbf a$}
  \put(73,103){$\mathbf{-a}$}
\end{overpic}
\end{center}
 %\centerline{\includegraphics[width=3.3in]{chebpar.png}\quad
   % \includegraphics[width=3.2in]{S2-cheb-J.png}}
\caption[Parameter space and Julia set for a Chebyshev map in $\cS_2$.]{\label{F-S2cheb}\sf On the left a copy of $\M$ in $\cS_2$, between the
  basilica region to the left and the $(1,0)$ region to the right.
  The Chebyshev point is the landing point of the $17/18$ parameter ray
  at the leftmost tip of this copy. 
  (Compare Figure~\ref{f2} in Section \ref{s-tess}
  for a picture showing the entire connectedness
    locus in $\cS_2$.) On the right is the Julia set for this Chebyshev
    point.   The common denominator $18$ 
    is to be understood for all angles. Here  $+a$ is in the
    largest component; while $2a$ is at the landing point of the 17 ray;
    and $-a$ is at the meeting point of the $5$ 
    and $11$ rays. Both $2a$ and $-a$ map to the landing point of the $15$
    ray; which maps to the fixed point at the end of the $9$ ray. }
\end{figure}
\bigskip

\begin{figure}[htb!]
  \centerline{\includegraphics[height=2.5in]{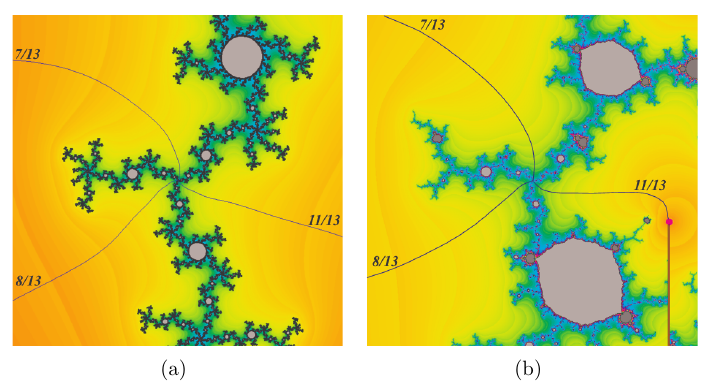}}
  \vspace{-.4cm}
\caption[External rays in Julia set and parameter space for a Misiurewicz
map]{\textsf{{\rm(a)} Julia set picture of width $0.96$
    centered at $2a$ for a Misiurewicz map\break
    $F_0\in\cS_4$. ~ In this example, $2a$ is a fixed point. 
 {\rm(b)} Corresponding parameter space picture, centered at
$F_0$ and magnified by the complex
factor of $\frac{ds}{dt}(F_0)$.~~ Here $t$ is the usual canonical 
coordinate, and $s(F)=2a_F-z_F$ 
is a dynamically defined local coordinate. (See \cite[Corollary 6.4]{BM}.)
 Estimated parameter rays are labeled by the parameter angle.
The parameter values are
$~a=-0.185084 -0.358121\,i,\quad v = -0.629660-0.752726\, i\,.$
The two left parameter rays lie  
in the Kokopelli escape region, 
 while the right hand ray lies in the
 $0010$ region with multiplicity $\mu_h=2$. (The vertical line on the right
 is a slit, terminating at the infinite point for this region.) Conjecturally
 the left and right figures would look more and more alike
as we zoom in towards the central point. 
In other words, there is a Tan Lei type similarity between parameter space
and dynamic plane around this point. (Compare Conjecture~\ref{C-TLS}.)
}
\smallskip\label{f5}}

\bigskip

\centerline{\includegraphics[height=2.5in]{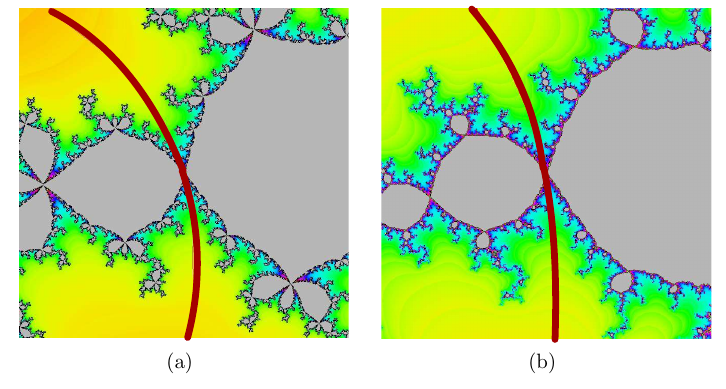}}
\vspace{-.4cm}
\caption[Misiurewicz example from $\cS_4$ in both dynamical and parameter planes]{\textsf{ Another Misiurewicz example 
  from $\cS_4$, again with the dynamical picture  
in {\rm(a)} and the parameter picture, magnified by a factor of
$ds/dt$,
 in {\rm(b)}. In both figures, the rays of angle
 33/80 and 57/80 land at the central point. In the dynamical picture,
 this point  $2a$ has period two. 
 Note that $11\mapsto 33\mapsto 19\mapsto 57\mapsto 11$ under tripling mod
 80. In the parameter picture this central point
lies between the ``double-basilica'' region above, and the 1010-region
below. (Here the parameters are:~ $a=-0.1996 -0.0353\,i$ and
$v=0.2284+ 1.0521\,i$.) 
\smallskip\label{f6}}}
\end{figure}

\begin{conj}[\bf Tan Lei Similarity]\label{C-TLS}  
\rm  For quadratic maps, 
Tan Lei proved that, under increasing magnification, the
connectedness locus $\M$ around a Misiurewicz point $F\in \M$,
and the filled Julia
set of $F$ at the corresponding repelling point, look more and more alike.
(See \cite{TL}, \cite{Ka2}, \cite{RL}.)

\begin{quote}\sf We conjecture that the analogous statement holds in $\cS_p$:
  Under iterated magnification,
the connectedness locus around a Misiurewicz map
$F\in\cS_p$  looks more and more alike the filled Julia set of
$F$ around the free co-critical point $2a_F$ (up to a fixed
scale change and rotation).\end{quote}

\noindent In fact, since the points in the forward orbit of $2a$ all have locally
conformally isomorphic Julia sets, we could very well substitute any one
of these these points for $2a$. For example, in Figure \ref{F-S2cheb},
the neighborhood of the Chebyshev point in parameter space looks very much like
a neighborhood of the post-critical fixed point in the Julia set.\medskip

For other examples, see Figures
\ref{f5}, and \ref{f6}; as well as Figure \ref{F-S2juls} in Section \ref{s-rays}.
Also compare the right hand tip of Figure~\ref{F-percase} with
Figure~\ref{F-S2rays} in Section \ref{s-tess}.\medskip

Closely related is the observation that a dynamic ray of given angle for
a Misiurewicz map  $F\in\cS_p$ seems to land at the point $2a_F$ if and only if
there is a parameter ray in $\cS_p$ of the same angle which lands at $F$.  
Furthermore the dynamic and parameter rays land in the same cyclic order,
even though the various parameter rays may be contained in different
escape regions. (See Figure \ref{f5}.) \medskip

We can make this similarity between Julia set and parameter space
more explicit as follows. For $F$ in a small neighborhood of any Misiurewicz
point $F_0\in\cS_p$, it is not hard to see that there is a unique pre-periodic
point $z(F)$ which varies as a holomorphic function of $F$, with
$z(F_0)=2a_{F_0}$.

\begin{quote}\sf We further  conjecture that the correspondence 
  $$ F~\mapsto~s(F)~=~2a_F-z(F) $$
  is univalent throughout some neighborhood of $F_0$, so that $s(F)$
  provides a dynamically defined holomorphic local coordinate for  $\cS_p$
  throughout this neighborhood. Using this coordinate,
  we conjecture that, under iterated magnification, $\cS_p$ near $F_0$ looks
  more and more like the dynamic plane for $F_0$ near $2a_{F_0}$, without
  any need for a preliminary scale change and rotation.
\end{quote}

\noindent
In practice, since $s(F)$ is defined only locally, it is more convenient to
use this procedure to choose the correct preliminary scale change and rotation.
This was done for Figures \ref{f5} and \ref{f6}. (The computation was relatively
easy since for these two cases $2a_{F_0}$ itself is a periodic
point, so that $z(F)$ is also periodic.) For more on the study of self-similarity
see \cite{BY}.)
\end{conj}\bigskip

\appendix
%%%%%%%%%%%%%%%%%%%%%%%%%%%%%%%

\setcounter{lem}{0}
\section{Parabolic Landing Stability (in some cases)}\label{a-para-stab} 

In general, a parameter ray landing at
 a parabolic point will not deform continuously as we perturb the map.
As an example, the zero ray for the quadratic map
$z\mapsto z^2+1/4$ will jump discontinuously under a small perturbation of
the map. (Compare Figure~\ref{f4}.) Nevertheless, the actual landing point
of the ray will vary continuously in the quadratic case, and for some 
higher degree cases.
\begin{figure}[h!]
\centering
\subfigure[\label{f4-a}]{\includegraphics[height=2.2in]{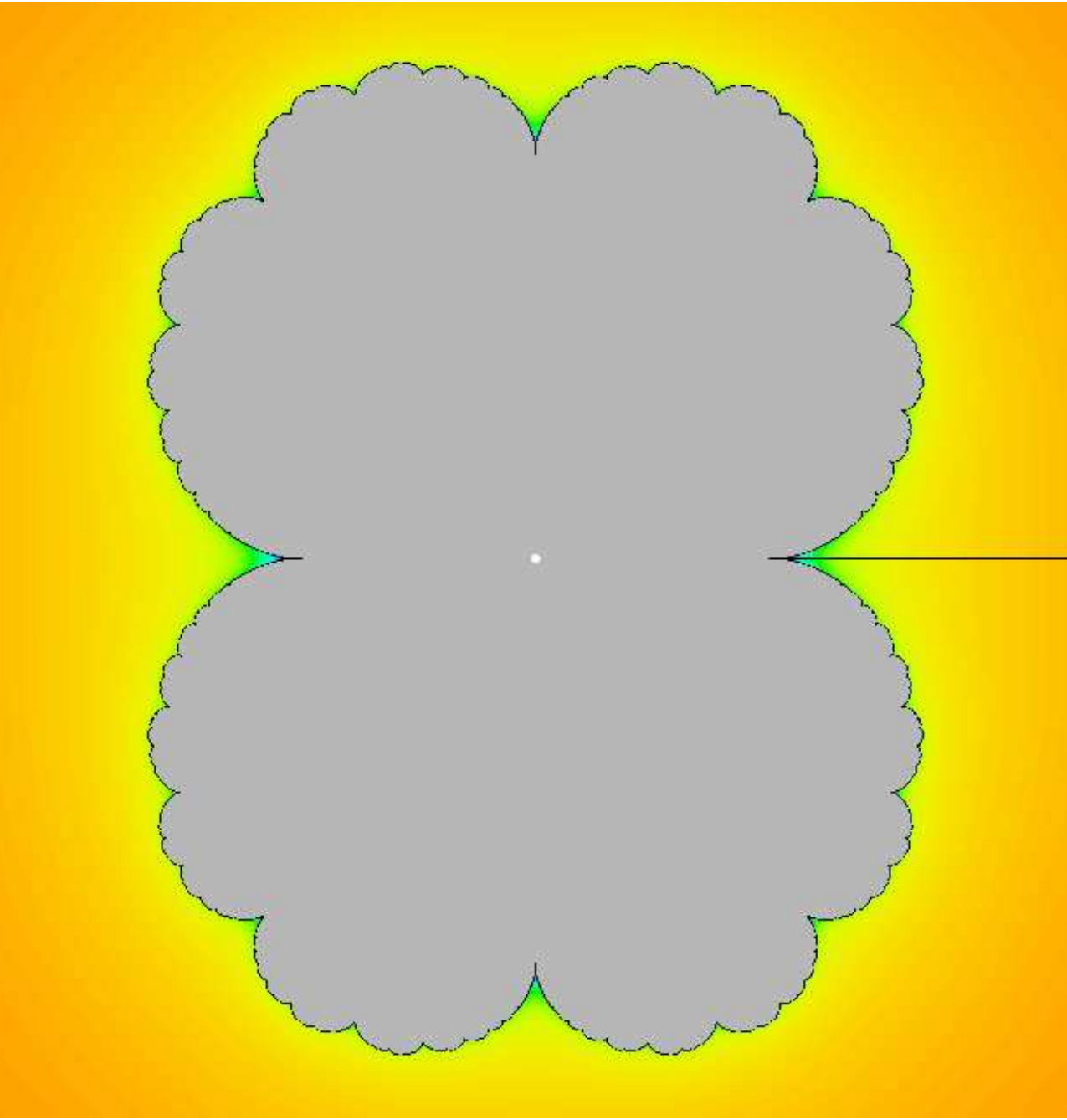}}\quad
\subfigure[\label{f4-b}]{\includegraphics[height=2.2in]{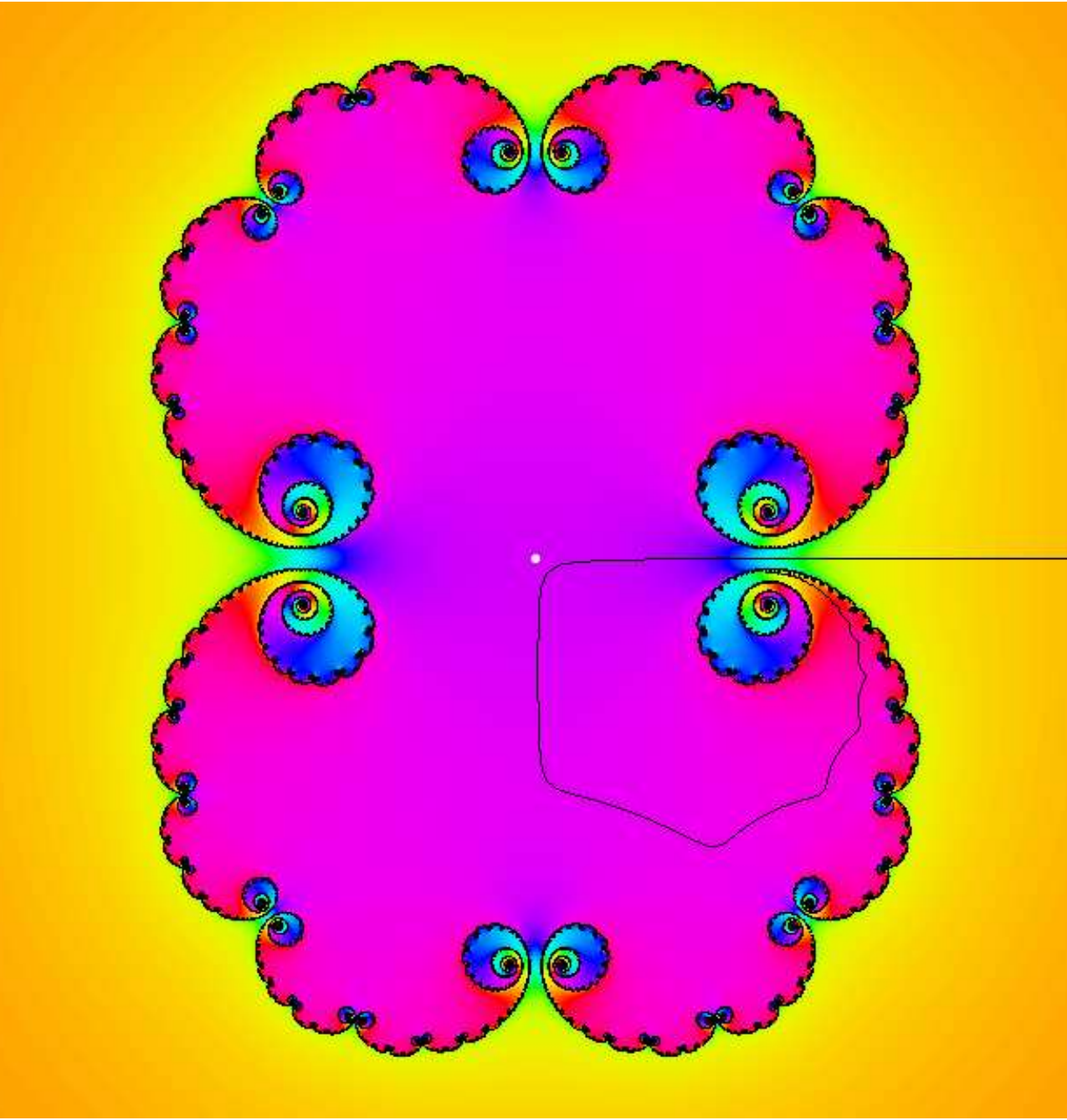}}
\caption[Pictures of the zero ray $~\cR_{F_0}(0)~$ for the quadratic map
\hbox{$~F_0(z)=z^2+1/4~$} and for $~F~$ close to $~F_0$]{\textsf{ {\rm(a)} The
    zero ray $~\cR_{F_0}(0)~$ for the quadratic map \hbox{$~F_0(z)=z^2+1/4~$}
    lands at the parabolic fixed point.
 {\rm(b)} For suitable $~F~$ close to $~F_0$,~ the zero ray passes between
the two fixed points, then ricochets off the critical point and off a sequence
of precritical points, eventually landing on the lower fixed point.
\label{f4}}}
\end{figure}\medskip

We will say that  a holomorphic one-parameter family of monic 
polynomial maps $F_s$ of degree $\d>1$ has only \textbf{\textit{one
free critical point}} if all but one of the critical
points (if there is more than one) is periodic or eventually periodic;
and if this one is a simple critical point. 
To simplify the notation, we will assume that the parameter $s$ varies
over some open subset of $\C$.\msk

\begin{figure}[htb!]
\centerline{\includegraphics[width=2.5in]{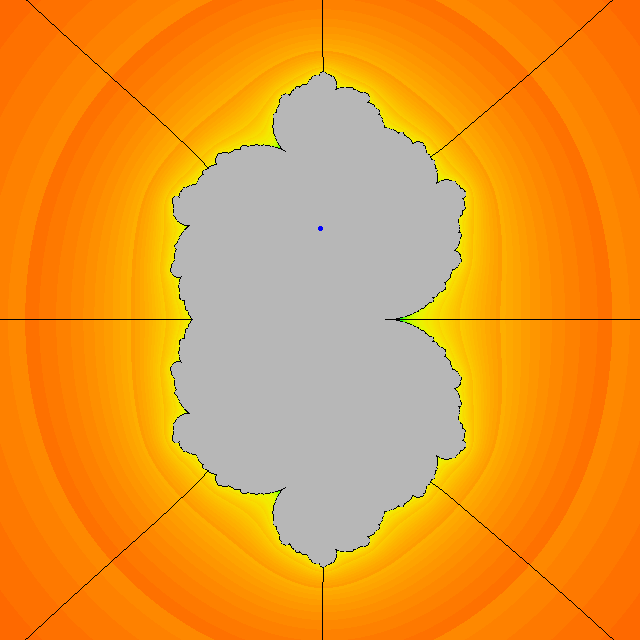}\quad
  \includegraphics[width=2.5in]{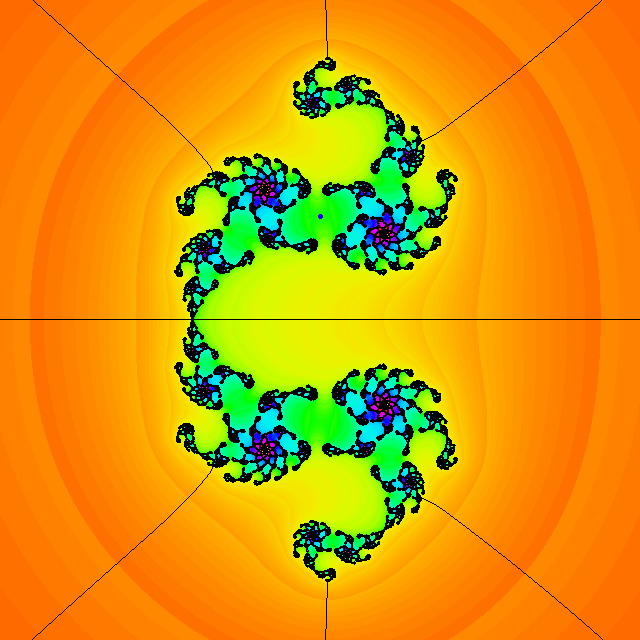}}
\caption[Julia sets for two maps in the family $~F_s(z) ~=~ z^3 +  z^2 + (1+s)z +s~$]{\sf Julia sets for two maps in the family \hfill{}\break
 \centerline{$~F_s(z) ~=~ z^3 +  z^2 + (1+s)z +s~$.}
Note that the ``one free critical point'' condition is not satisfied. 
On the left: For $s=0$ the zero ray lands at a parabolic fixed point, while
the $1/2$ ray lands at a repelling fixed point. On the right, the case $s=0.2$.
For all $s\ge 0$, the two critical points are complex conjugates. 
For $s>0$ the parabolic fixed point splits into a pair of complex
conjugate fixed points,
and the zero ray squeezes through the gap between them to land at the 
same point as the $1/2$ ray. Both critical orbits escape to infinity,
so the Julia set is totally disconnected. 
The top critical point is marked in both figures.
\label{pls-f1}}
\end{figure}\bsk

\begin{theo}[\bf Parabolic Landing Stability Theorem]\label{pls-T1} 
{\sf Let $\theta$ be an angle of period $q$ under
multiplication by $d$, and suppose that the dynamic ray of angle
$\theta$  for $F_0$ lands at a parabolic point  $z_0$. If there is
only one free critical point,  and if we exclude the special case
where the $\theta$ ray for $F_s$
crashes into a critical point of $F_s^{\circ q}$, then the landing point for
this ray varies
continuously with $s$ throughout some neighborhood of $s=0$.
In the exceptional case where this ray does crash into a critical point,
the same statement is true for the limit rays as we approach $\theta$ from
the left or from the right. (The two limits differ; but converge to each
other as $s\to 0$.)}\footnote{This result contrasts with Douady and
Hubbard's proof that repelling landing points always deform continuously
under perturbation. See Lemma \ref{L-land-stab}.}
\end{theo}\msk

%(The corresponing statement when the landing point is
%repelling was proved by Douady and Hubbard. See Lemma \ref{L-land-stab}.)
%In the parabolic case, 
Here the condition that there is at most one free critical
point, or some similar condition, is essential. (See  Figure~\ref{pls-f1}.)
For an example of the exceptional case, see 
Figure \ref{F-near-para}-left.
\bsk

\begin{coro}[\bf Upper Semi-Continuity of Orbit Portraits]\label{usc}
Now suppose that the maps $F_s$ are in $\cS_p$, with degree $\d=3$.   If 
$\theta$ and $\theta'$ are periodic under tripling, and if the $\theta$
and $\theta'$ dynamic rays land together for maps $F_s\in\cS_p$ 
with $s$ arbitrarily close to zero, then they land together for $s=0$.
\end{coro}

\begin{proof}
This follows immediately.
\end{proof}
\bigskip

The proof of Theorem~\ref{pls-T1} begins as follows. The angle 
$\theta$, of period $q$ under multiplication by $d$, will be fixed throughout
the discussion.
Suppose that we have a one parameter family for which the
Theorem is false. Then there exists some small neighborhood 
$U$ of the point $z_0$ and an infinite sequence of parameter points $s_j$
converging to $s=0$ such that the landing point
$\widehat{z}_j$ of the $\theta$ ray
$\cR_j$ for each $F_{s_j}$ lies outside of $U$. After passing to an infinite
subsequence, we may assume that these landing points converge to
some point  \hbox{$\widehat z\not\in U$}.\msk
%\end{description}

{\bf Notations:} We will use the notation
$f_s$ for the iterate $F_s^{\circ q}$ of degree $\d^q$.  Evidently $\widehat z$
must be one of the finitely many fixed points of $f_0$.
The parabolic point $z_0$ is the root point of some number $n\ge 1$
of parabolic basins which map to themselves under $f_0$. In 
general\footnote{The only exception that we know is the map $F(z)=z^3+z$,
with $q=1$ but $n=2$. }
$n$ is some divisor of $q$, and $z_0$ is a point of period $q/n$ under
$F_0$.\smallskip

Let $B$ be the union of the closures of these $n$ basins.
First note the following.

\begin{lem}\label{pRPls-L1} The root point $z_0$ is the only fixed
point of $f_0$ in the set $B$. Since $\widehat z$ is a fixed
point not equal to $z_0$, it follows that $\widehat z\not\in B$.
\end{lem}

\begin{proof} By a Theorem of Roesch and Yin \cite{RY}, every polynomial
  parabolic basin is a Jordan domain. Each of the $n$ parabolic basins in $B$
  contains a
single critical point of $f_0$, and hence  maps onto itself with degree two.
Thus each basin boundary must wrap around itself with degree two. 

Let $\Gamma$ be the boundary of one of the $n$ parabolic basins in $B$.
Since $\Gamma$ is contained in the Julia set, it cannot contain any attracting
points. Furthermore since $F_0$ can have only one cycle of parabolic points,
there cannot be any parabolic point other than the common root point in $\Gamma$.
Thus we need only prove that there is no repelling fixed point in $\Gamma$.

If we think of $f_0:\Gamma\to\Gamma$ as a one-dimensional topological
dynamical system, then the root point is a topologically repelling fixed point.
But it is not hard to check that an orientation preserving circle map with $k$ fixed
points, all repelling, must have degree $k+1$. (Each interval between two
fixed points cannot map into itself, and hence must cover itself twice, and
cover the rest of the circle once.) Since the degree is known to be two,
it follows that $k=1$; and hence 
the root point is the only fixed point of $f_0$ in $B$.
\end{proof}
\msk

\begin{figure}[htb!]
\centerline{\includegraphics[width=2.8in]{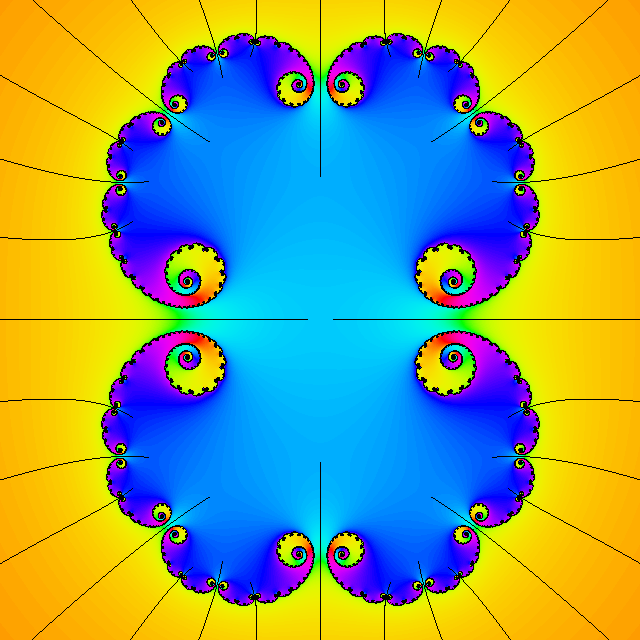}\quad
\includegraphics[width=3.2in]{ 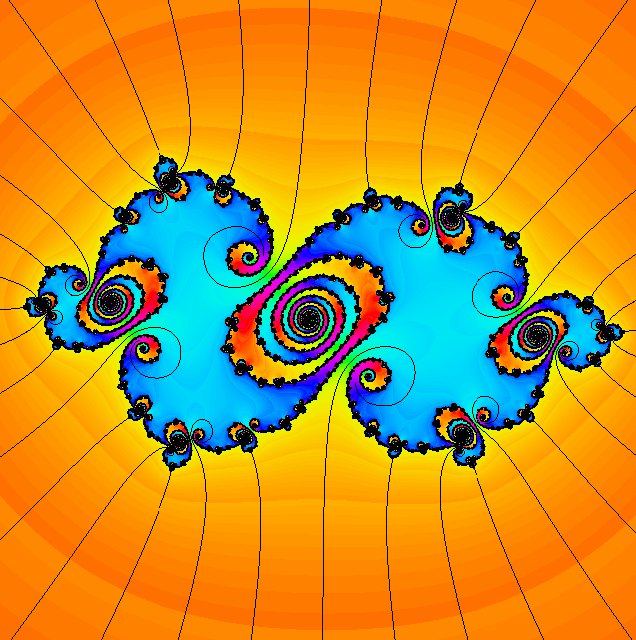}}
\caption[Approximation to the standard cauliflower map]{\label{F-near-para}\sf On
  the left: an approximation $z\mapsto z^2+0.27$
  to the standard cauliflower map. All rays of angle $n/2^k$
  penetrate from the outside into the central domain; but no rays from the
  inside pass to the outside. On the right: an approximation to the parabolic
  map $z\mapsto z^3-z$, for which the $1/4$ and $3/4$ lays land at the central
  parabolic point. Again many rays from the outside penetrate to the
  interior; but no rays pass in the opposite direction. (The illustrated map is
  $z\mapsto z^3+(-1.00878700+.160147\,I )z+ (-.097287+ .008229\,I)$.)}
    \end{figure}\medskip

    Note the following statement:
\begin{quote}
    {\sf Every polynomial map with a
      parabolic point can be approximated arbitrarily closely by
      a map of the same degree with disconnected Julia set.}
  \end{quote}

\noindent We are grateful to Misha Lyubich for pointing out that this
 follows from his paper \cite{Ly}, together with Ma\~n\'e, Sad
and Sullivan \cite{MSS}, and Levin \cite{Le}. In fact it follows from
these papers that the \textbf{\textit{bifurcation locus}}
(consisting of maps where the Julia set
does {\sf not} vary continuously in some neighborhood) is nowhere dense, and
coincides with the closure of the set of parabolic maps. It also coincides
with the set of maps which are \textbf{\textit{active}}, in the sense
that the family of iterated forward images of some critical point is not a
normal family. A map in the 
connectedness locus is active if and only if its
Julia set becomes disconnected under some small perturbation.\msk

% \change   More generally, {\sf the bifurcation locus in the space of
%      polynomials is nowhere dense, coincides with the closure of parabolic
%      parameters, and with the set of parameters where one of the critical
%      points is ``active''\footnote{An \textit{\textbf{active critical point}}
%        can be made escaping under a small perturbation of the map. \change}
%      (i.e., where the sequence
%    $$  \lambda \mapsto f_\lambda^n (c_\lambda)$$
 %   is not normal).} 
%    \end{quote}

\begin{lem}\label{L-many-rays}  Suppose that $F$ is a small perturbation
of a parabolic map $F_0$ with a set $B$ of parabolic basins of period $q$,
as described above. Then there must be an angle $\theta$ of period $q$ under
multiplication by $d$ such that the dynamic ray of angle $\theta$ for $F_0$
lands at the parabolic root point, but the corresponding ray for $F$ penetrates
into the former parabolic basin, as illustrated in Figure~$\ref{F-near-para}$.
It follows that all iterated preimages of this ray under multiplication by
 $d^q$ also penetrate. These preimages are everywhere dense among rays
landing on $B$, and none of then is invariant under the action of $F^{\circ q}$.
  \end{lem}

  \begin{proof} Since the ray is invariant under $F^{\circ q}$, it must either
    land at an $F^{\circ q}$-invariant point or else crash into a critical
    or precritical point. In either case, it must penetrate into
    the parabolic basin of $F_0$. (Compare Petersen and Zakeri \cite{PZ}.)
\end{proof}

\begin{figure}[htb!]
\centerline{\includegraphics[width=2.5in]{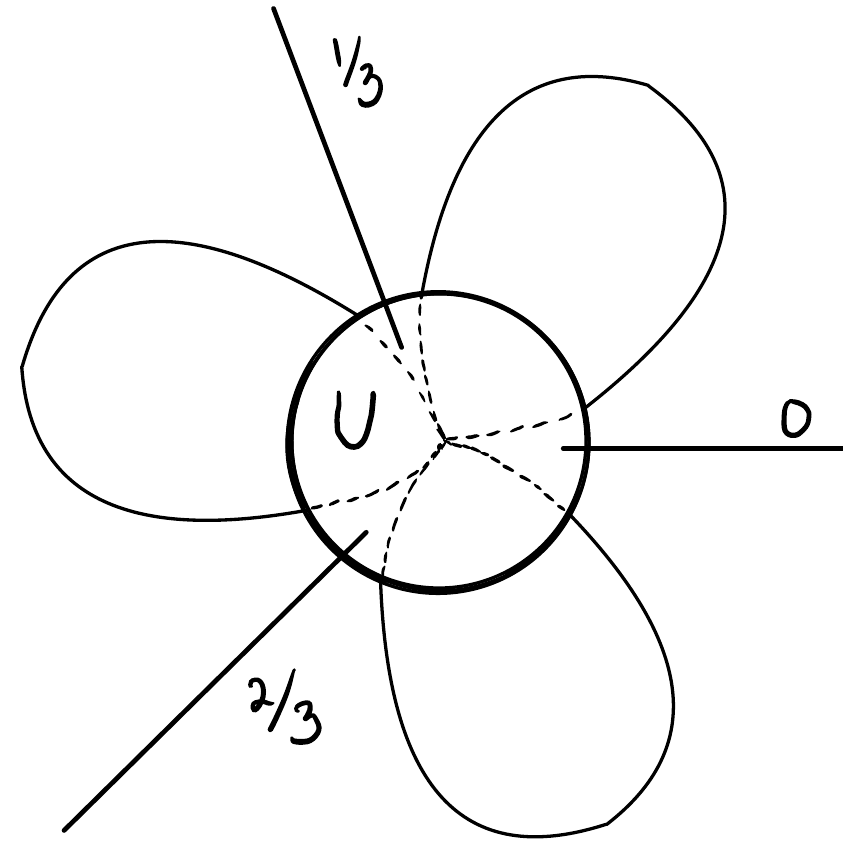}}
\caption{\label{F-tri-leaf}\sf Cartoon of a parabolic map with rotation
  number $1/3$.}
\end{figure}
  
\begin{proof}[{\bf Proof of Theorem \ref{pls-T1}}]
Figure \ref{F-tri-leaf} illustrates a typical example, with $n=3$. Note
that the $n$ loops of $B$ are contained in attracting petals, while the $n$
$f_{j_s}\!$-invariant rays enter through repelling petals. If $f_{s_j}$ is close
enough to $f_0$, then these invariant parameter rays will certainly cross
into $U$, and must eventually land at some repelling fixed point of $f_{s_j}$.
There are no such fixed points in $B\ssm U$,
so if the landing point is not in $U$ then they somehow
escape from $B$. They certainly cannot go back out the way they came in, through
one of the $n$ repelling petals. The only other possibility is that they cross
through the boundary of $B$ in one of the $n$ attracting petals. But this is
impossible. In the limit as $f_{s_j}$ tends to $f_0$, this boundary maps
bijectively to itself without fixed points (other than the root point in $U$).
This looks impossible: If a ray first crosses out of $B$ at a point 
$b$ then it should already have crossed out previously near some pre-image of $b$. 
\msk

Here is an alternative  more explicit argument. 
If an $f_{s_j}$-invariant ray crossed from inside
to the complement of $B$, then it would have to cross between two very close
rays which are entering $B$ from the outside, and which are moved somewhere
else by  $f_{s_j}$. This is clearly impossible; which completes the proof.
\end{proof}
\bsk \goodbreak

\setcounter{lem}{0}
\section{Embedded Trees}\label{a-embdtree}

 By an \textbf{\textit{embedded tree}} we will mean a contractible 
  simplicial complex of dimension zero or one 
  which is topologically embedded in $\C$. However two such trees will be
  regarded as completely equivalent if there is an orientation preserving
  homeomorphism of the plane which takes one to the other.
\begin{figure}[htb!]
  \centerline{\includegraphics[width=3.1in]{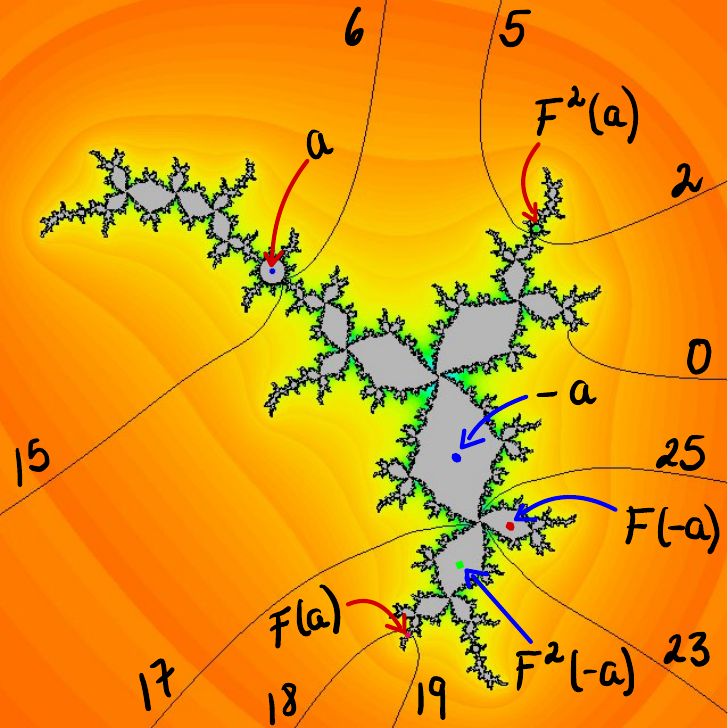}\quad\qquad
    \includegraphics[width=3in]{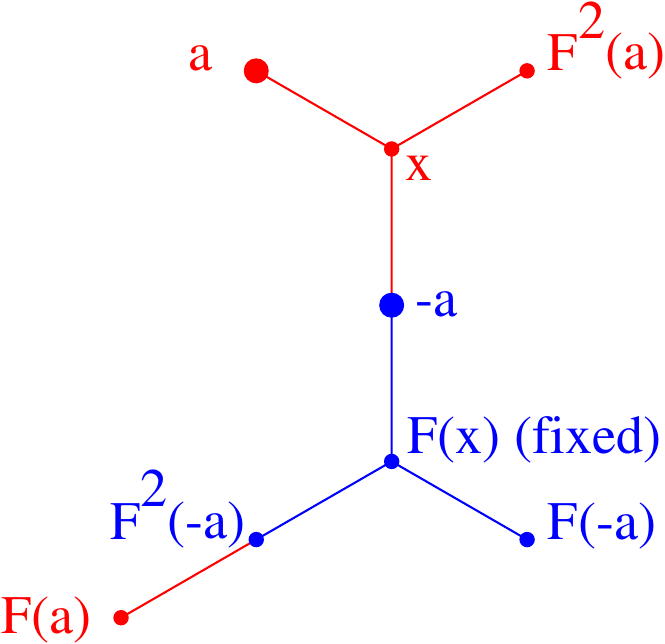}}
  \caption[Julia set for the center
    point of the $2/3$ limb of a Mandelbrot copy associated with the $1/3$
    rabbit and a corresponding Hubbard tree]{\label{F-J+HT}
    \sf  On the left, Julia set for the center
    point of the $2/3$ limb of a Mandelbrot copy associated with the $1/3$
    rabbit.  (All angles over 26.) On the right, a corresponding
    Hubbard tree. For further discussion see Remark \ref{R-figs} below.  }
  \bigskip
\bigskip
  \bigskip
  
 \centerline{\includegraphics[width=3.4in]{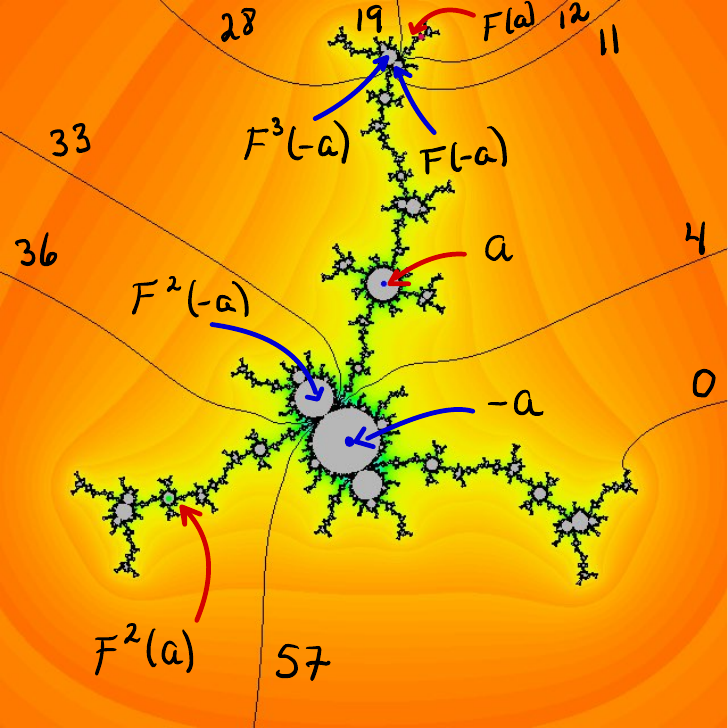}\quad\qquad\includegraphics[width=2.4in]{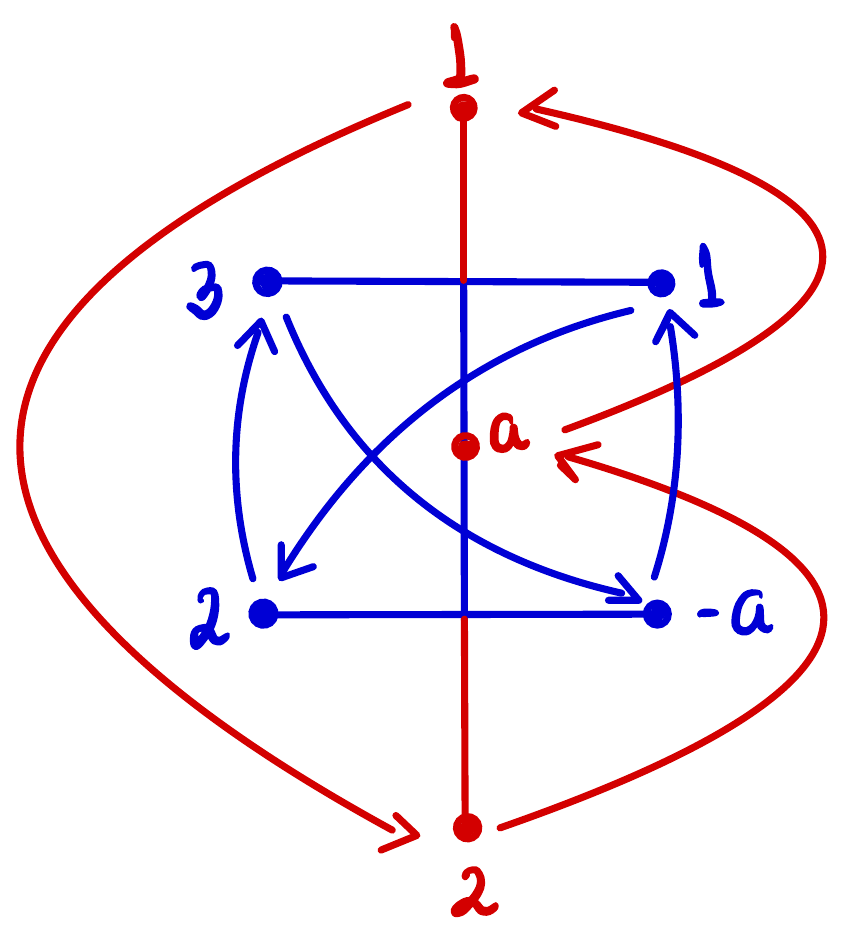}}
 \caption[Julia set for the basilica point]{\label{F-HT2}\sf  On the left, Julia
   set for the ``basilica point'' of  a very small Mandelbrot copy in $\cS_3$.
   (See Figure~\ref{F-43} in Section \ref{s-near-para}.) On the right,  the
   corresponding  Hubbard tree $T$  is quite different from the previous one.
   The sub-tree $T'$ is again colored blue. (All angles over $80$.)}
\end{figure}

\goodbreak

  The \textbf{\textit{ $($weak$)$ Hubbard tree}} $T$ of a critically finite
  polynomial map $F$ is an   embedded tree together with some additional
  structure. More precisely:

  \begin{itemize}
  \item[(1)] The underlying set $|T|$ is  the smallest subset of $\C$ which
  contains all of the critical orbits,
  and also contains the regulated path (see \cite{DH1} or \cite {Po})
  joining any
  two critical or postcritical points within the filled Julia set of $F$.

 \item[(2)]  The vertices of this
  tree are these critical or postcritical points, together with any
  branch points.

\item[(3)]  Furthermore:
\begin{itemize}
\item[(a)] the map from vertices to vertices must be specified;
\item[(b)] the critical points must be labeled; and
\item[(c)] their multiplicity must be specified if it is greater than one.
\end{itemize}
\end{itemize}
\msk

  {\sf Note:} We are using this weak definition of Hubbard tree
  in order to relate maps of two different degrees. However this 
  loses a basic property of more usual definitions (as given for example
in \cite{Po} or \cite{M4}). This tree
  does not determine the map up to conformal conjugacy. For example
  the cubic maps $z\mapsto z^3+i$ and $z\mapsto z^3-i$ have the same weak tree,
  yet do not belong to the same conformal conjugacy class. \msk
%  have identical trees in this sense. For the tree to determine the
%  conjugacy class, more information must be included. Compare \cite{BM},
%  as well as \cite{Po}.)\msk

  Figures~\ref{F-J+HT} and \ref{F-HT2} illustrate two examples for the tree
  $T$ associated to a critically periodic point in a Mandelbrot copy. They
  suggest the following two statements:\msk
\begin{itemize}
\item[{\bf 1.}] The tree $T'$ for the associated critically periodic point in
  $\M$ naturally embeds as a subtree of $T$.\msk

\item[{\bf 2.}] The ``quotient'' $T/T'$ in which $T'$ is collapsed to a
  point is naturally isomorphic to the tree $T''$ for the main hyperbolic
  component of the copy $M$.\msk
\end{itemize}

However, even if both of these statements are true, we would still be
far from knowing how to construct the tree $T$ for any given component of any
Mandelbrot copy in $\cS_p$. \msk

\begin{rem}\label{R-figs} Here are further comments about Figures~\ref{F-J+HT}
  and \ref{F-HT2}.
  
 In Figure \ref{F-J+HT},  the tree $T'$ for the
associated quadratic map is colored blue, while the rest of the cubic tree
is colored red. Note that the blue part has rotation number $2/3$.
If we collapse the blue part to a point, we essentially get the tree $T''$ for
the main hyperbolic component. However there is a minor technicality: When we
collapse $T'$ to a point, both vertices of the edge joining the branch point
 $x$ to $-a$ map to the same point. This is not allowed in a Hubbard
tree, so we must also collapse this edge to this point. In this way we get
the required tree for the principle hyperbolic component, with rotation
number $1/3$.  
 \end{rem}

 \bigskip
 
 \goodbreak
 
%%%%%%%%%%%%%
    
\bigskip

\bigskip
\goodbreak

\listoffigures
%\printindex

\vspace{1.5cm}

\parbox{2.5in}{Araceli Bonifant \\
              Mathematics Department,\\
              University of Rhode Island,\\
              Kingston, R.I.,  02881.\\
              e-mail: bonifant@uri.edu\\
           } 
\parbox{3in}{John Milnor \\
               Institute for Mathematical Sciences,\\
               Stony Brook University,\\
               Stony Brook, NY. 11794-3660.\\
               e-mail: jack@math.stonybrook.edu}


\begin{thebibliography}{99999}
  %  \bibitem[A]{A} L.~V. Ahlfors, ``Complex Analysis'' Third Edition,
  %    International Student Edition, (1979).
      
\bibitem[AK]{AK} M. Arfeux and J. Kiwi. {\it Irreducibility of periodic curves
    in cubic polynomial moduli space.\/}  Proceedings of the London Mathematical
  Society {\bf 127}\, (3)\, (2023)\, 792--835.
  \url{https://doi.org/10.1112/plms.12553}
  % arXiv:2012.14945 [math.DS]
  \newpage
  

\bibitem[BH1]{BH1}
  B. Branner and J. H. Hubbard.
  \textit{The iteration of cubic polynomials part I:  The global topology of
    parameter space.\/} Acta Math., \textbf{160}\, (1988)\, 143--206.
  

%\bibitem[BH2]{BH2}
%B. Branner and J. H. Hubbard,
%{\it The iteration of cubic polynomials II,
%patterns and parapatterns,} Acta Math., \textbf{169}\, (1992)\, 229--325.

\bibitem[BKM]{BKM}
A. Bonifant, J. Kiwi and J. Milnor.
{\it Cubic polynomial maps with periodic critical orbit, part II:
escape regions.\/} Conform. Geom. Dyn. {\bf 14}\, (2010)\, 68--112. 
{\it Errata:} Conform. Geom. Dyn. {\bf 14}\, (2010)\, 190--193. (Also
available in: ``Collected papers of John Milnor. VII. Dynamical systems 
(1984-2012),'' ed. A. Bonifant. Amer. Math. Soc. (2014)\, \hbox{477--526}.)

\bibitem[BMS]{BMS} A.\,Bonifant, J.\, Milnor and S.\,Sutherland. {\it Parabolic
 implosion and the relative Green's function.\/} In preparation,

\bibitem[BM]{BM} A. Bonifant and J. Milnor.  {\it Cubic polynomial maps
    with periodic critical orbit, part IV: Zero-kneading regions.\/}
  In preparation.

\bibitem[BY]{BY} A. Bonifant and B. Young.  \textit{Self-similarity of Misiurewicz
  maps in the cubic parameter curves.\/}   In preparation.
  
\bibitem[DM-S]{DM-S} L.\,DeMarco and A.\,Schiff. {\it The geometry of the space
of critically periodic curves in the space of cubic polynomials.\/}
Exp. Math. {\bf 22}\, (2013)\, 99--111.

\bibitem[DH1]{DH1}  A. Douady and J. H. Hubbard.
{\it Exploring the Mandelbrot set. 
The Orsay Notes $(1983)$.\/} (Available in 
\url{http://www.math.cornell.edu/~hubbard/OrsayEnglish.pdf}  


\bibitem[DH2]{DH2}  A. Douady and J. H. Hubbard. 
{\it A proof of Thurston's topological characterization of rational 
  functions.\/} Acta Math. {\bf 171}\, (1993)\, 263--297.

\bibitem[DH3]{DH3}  A. Douady and J. H. Hubbard. 
  {\it On the dynamics of polynomial-like mappings.\/}  Ann. Sci. \'Ec. Norm.
  Sup. {\bf 18}\, (1985)\, 287--344.
%\bibitem[EMS]{EMS} D. Eberlein, S. Mukherjee, and D. Schleicher,
%    {\it Rational parameter rays of the multibrot sets,}
%    Dynamical systems, number theory and applications, 49--84,
%    World Sci. Publ., Hackensack, NJ, 2016 (ArXiv 1410.6729 2015)

%  \bibitem[EY]{EY} A. Epstein and M. Yampolsky, \change
%    {\it Geography of the cubic connectedness locus: Intertwining surgery,\/}
%    Annales Scientifiques de l'\'Ecole Normale Sup\'erieure
%    {\bf 32}\, (1999)\, 151--185.
%    \url{https://doi.org/10.1016/S0012-9593(99)80013-5}
    %Issue 2, March–April\, (1999)\, 151--185.
    
%\bibitem[Fa]{Fa} D. Faught, {\it Local connectivity in a family of cubic
%polynomials,\/} Thesis, Cornell 1992.

%\bibitem[FK]{FK} H. Farkas and I. Kra, ``Riemann Surfaces'', Springer Verlang
%  1992.

%\bibitem[G]{G} L.~R. Goldberg, {\it Fixed points of polynomial maps. I.
%    Rotation subsets of the circle,\/} Ann. Sci. \`Ec. Norm. Sup\'e r.
%    {\bf 25}\, (1992)\, 679--685.

%\bibitem[GM]{GM}  L. Goldberg and J. Milnor, {\it  \ednote{DELETE ?}
%Fixed points of polynomial maps. II. Fixed point portraits,\/}
%Ann. Sci. \'Ecole Norm. Sup. {\bf 26}\, (1993)\, 51--98.
\bibitem[EE]{EE} J.-P. Eckmann and H. Epstein. {\it  Scaling of Mandelbrot sets
    generated by critical point preperiodicity.\/} Comm. Math. Phys.
  {\bf 101}\, (1985)\, 283--289.
  
\bibitem[H1]{H1} P. Ha{\"\i}ssinsky. {\it Chirurgie parabolique.\/}
 C. R. Acad. Sci. Paris, {\bf 327} \, (1998)\, 195--198.

\bibitem[H2]{H2} P. Ha{\"\i}ssinsky. {\it D\'eformation $J$-\'equivalence de
 polyn\^omes g\'eom\'etriquement  finis.\/}  Fundamenta Math. {\bf 163}\,
(2000)\, 131--141.

\bibitem[Ha]{Ha} A. Hatcher. ``Algebraic Topology.'' Cambridge U. Press 2002.

\bibitem[Ka1]{Ka} T. Kawahira. {\it Note on dynamically stable perturbations
    of parabolics.\/} {\bf 1447}\, Complex Dynamics,\, Kyoto University
  Research Information Repository, RIMS Kokyuroku,\, August (2005) \, 90--107.

\bibitem[Ka2]{Ka2} T. Kawahira. {\it Zalcman functions and similarity
    between the Mandelbrot set, Julia sets, and the tricorn.\/}
  Analysis and Mathematical Physics {\bf 10}\,  (2020)\, 1--18.
 % arXiv:1907,13488v2.

\bibitem[Le]{Le} G.M. Levin. {\it Irregular values of the parameter of
    polynomial mappings.\/} Uspekhi Mat. Nauk SSSR
{\bf 36}\, (6)\, (1981)\, 219--220. 
  
\bibitem[Ly]{Ly} M. Yu. Lyubich. {\it Some typical properties of the
    dynamics of rational maps.\/} Russian Math. Surveys {\bf 38}\, (5)\,
  154--155. \url{https://doi.org/10.1070/RM1983v038n05ABEH003515}

  
%\bibitem[K]{K1}
% J. Kiwi, {\it Rational laminations of complex polynomials\/}, in
%``Laminations and Foliations in Dynamics, Geometry and Topology,'' ed.
%Lyubich et al., Contemporary Math. {\bf 269}, AMS (2001)\, 111--154.

%\bibitem[LV]{LV} O. Lehto and K.~I. Virtanen, ``Quasiconformal Mappings in
%the Plane,'' Springer Verlag 1973.

%\bibitem[LP]{LP} G. Levin, and F.Przytycki, {\it External rays to periodic
%    points\/}, Isr. J. Math. {\bf 94} (1995) 29-57,

\bibitem[MSS]{MSS} R. Ma\~n\'e, P. Sad and D. Sullivan.  {\it On the dynamics
    of rational maps.\/} Ann. Sci. \'Ecole Norm. Sup. {\bf 16}\, (2)\,
  (1983)\, 193--217. \url{http://eudml.org/doc/82115}

\bibitem[Mc]{Mc} C. McMullen. {\it The Mandelbrot set is Universal.\/}
pp. 1--17 of ``The Mandelbrot Set, Theme and Variations", London Math. Soc.
Lecture Note Series {\bf 274}, Cambridge U. Press 2000.

%\bibitem [McS]{McS} C. McMullen and D. Sullivan,  {\it Quasiconformal
%homeomorphisms and dynamics III: The Teichm\"uller space of a holomorphic
%dynamical system,\/} Adv. Math. {\bf 135}\, (1998)\, 351--395.

%\bibitem[K2]{K2}
% J.~Kiwi, {\it Puiseux series polynomial dynamics and iteration of
%complex cubic polynomials\/},  Ann. Inst. Fourier (Grenoble)  {\bf 56} (2006)
% 1337--1404.

\bibitem[M1]{M1} J. Milnor.
 {\it Hyperbolic components in spaces of polynomial maps,\/}
 with an appendix by A. Poirier. {\it Contemp. Math.,\/}\, {\bf 573}
 ``Conformal dynamics and hyperbolic geometry,'' Amer. Math. Soc. (2012)\,
 183--232. (Also available in: ``Collected papers of John Milnor. VII.
 Dynamical systems  (1984-2012),'' ed. A. Bonifant. Amer. Math. Soc. (2014)\,
 527--576.)  
 

\bibitem[M2]{M2} J. Milnor. {\it Periodic orbits, external rays,
 and the Mandelbrot set.\/}  ``G\'eometrie Complexe et Syst\`emes 
Dynamiques.''  Ast\'erisque\, {\bf 261}\, (2000)\, 277--333. (Also available
in: ``Collected papers of John Milnor. VII. Dynamical systems (1984–2012),''
ed. A. Bonifant. Amer. Math. Soc. (2014)\, 171--229.)   


\bibitem[M3]{M3}  J. Milnor. ``Dynamics in One Complex Variable.''
3rd edition, Princeton University Press, 2006.

\bibitem[M4]{M4}  J. Milnor. {\it Cubic Polynomial Maps with 
Periodic Critical Orbit, Part I.\/} ``Complex Dynamics Families
and Friends'', ed. D. Schleicher, A K Peters (2009)\,  \mbox{333--411}. (Also 
available in: ``Collected papers of John Milnor. VII. Dynamical systems 
(1984–2012),'' ed. A. Bonifant. Amer. Math. Soc. (2014)\, 409--476.)   

%\bibitem[MT]{MT} J. Milnor and C. Tresser, {\it On entropy and monotonicity
% for real cubic maps\/}, Comm. Math. Phys. {\bf 209} (2000) 123-178.

%\bibitem[P]{P} I. N. Pesin, {\it Metric properties of Q-quasiconformal maps,\/} Mat. Sbornik {\bf 40(82)}, N3 (1956)\,  281--294 (Russian).

\bibitem[Po]{Po} A. Poirier. \textit{Hubbard Trees.\/}
Fund. Math.\, {\bf 208}\, (2010)\, 193--248.

\bibitem[PZ]{PZ} C. Petersen and S. Zakeri. {\it Periodic points and smooth
    rays.\/} Conf. Geom. and Dyn.\, {\bf 25}\, (2021)\, 170--178.

%\hspace{-1em}\DOI{10.4064/fm208-3-1}.

%\bibitem[PR]{PR} F. Przytycki and S. Rohde, {\it Rigidity of holomorphic
%Collet-Eckmann repellers,\/} Ark. Mat. {\bf 37}\, 357--371.

\bibitem [RL]{RL} J. Rivera-Letelier. {\it On the continuity of Hausdorff
dimension of Julia sets and similarity between the Mandelbrot set and
Julia sets.\/} Fund. Math.\, {\bf 170}\, (2001)\,  287--317.

%\bibitem[Ro]{Ro} P. Roesch, {\it Hyperbolic components of polynomials
%with a fixed critical point of finite order,\/} Ann. Sci. \'Ecole Norm. Sup. {\bf 40}\, (2007)\, 901--949.

%\bibitem[Ro]{Ro} P. Roesch, {\it Cubic polynomials with a parabolic point,\/}  Ergodic Theory Dynam. Systems {\bf 30}\, (2010)\, 1843--1867.

\bibitem[RY]{RY} P. Roesch and Y. Yin. {\it The boundary of bounded polynomial
    Fatou components.\/} C. R. Math. Acad. Sci. Paris\, {\bf 346}\, (2008)\, 
877--880.  

\bibitem[Sch]{Sch} D. Schleicher. {\it Rational external rays of the
    Mandelbrot set.\/} Asterisque\, {\bf 261}\, (2000)\, 405--443.  
  
\bibitem[Si]{Si} J. H. Silverman. {\it An algebraic approach to certain
    cases of Thurston rigidity.\/} Proc. of the Amer. Math. Soc.\, {\bf 140}\,
  (10)\, (2012)\, 3421--3434.                     

%\bibitem[Su]{Su} D. Sullivan, ``Seminar on conformal and hyperbolic geometry,''  (Notes) IHES (1982).
  \newpage
  
 \bibitem[Sp]{Sp} E. Spanier. ``Algebraic Topology'', McGraw-Hill, 1966.
%\bibitem[RY]{RY} P. Roesch and Y. Yin, {\it The boundary of 
%bounded \ednote{not cited} polynomial Fatou components,\/} C. R. Acad. Sci.
% Paris {\bf 346} (2008) 877-880.

%\bibitem[Si]{Si} J. H. Silverman, 
% ``The Arithmetic of Dynamical Systems,'' \ednote{not cited} Springer 2007.

\bibitem[TL]{TL} Tan\_Lei. {\it Similarity between the 
Mandelbrot set and Julia sets.\/} Comm. Math. Phys.\, {\bf 134}\, (1990)\, 
587--617.

%\bibitem[TLY]{TLY} Tan\_Lei and  Y.~Yin, {\it Local connectivity of the Julia
%    set for geometrically finite rational maps,\/}  Science in China, Series A~(1)
%  {\bf 39} \, (1996)\, 39--47.

%\bibitem[XG]{XG} Xiaoguang Wang, {\it Hyperbolic components and cubic
%polynomials,\/} arXiv:1710.03955 [math.DS],
%https://arxiv.org/abs/1710.03955~~https://arxiv.org/abs/1710.03955 .


%\bibitem[Z]{Z} S. Zakeri, ``Rotation Sets and Complex Dynamics'', Springer
%  2018.
%$$***$$
\end{thebibliography}
  \end{document}